\documentclass[journal]{IEEEtran}
\usepackage[notcite,notref]{}
\usepackage{amsmath,graphicx}
\usepackage{amssymb}
\usepackage{graphicx}
\usepackage{tabularx}
\usepackage{color}
\usepackage{times}
\usepackage{epsfig}
\usepackage{array}
\usepackage{threeparttable}
\usepackage{algorithm}
\usepackage{color}
\usepackage{lineno}
\usepackage{bm}
\usepackage{subfigure}
\usepackage{cases}
\usepackage{epsfig,amssymb,subfigure,version}
\usepackage{amssymb,version,graphicx,fancybox,mathrsfs}
\usepackage{graphicx,mathrsfs,fancyhdr,amsthm}
\usepackage{subfigure,slashbox,rotating,version,fancybox,pifont,booktabs,epsfig}
\usepackage{multicol,url,multirow}
\usepackage{setspace} 

\graphicspath{{figs/}}
\newtheorem{thm}{Theorem}

\newtheorem{rem}{ Remark}[section]
\newtheorem{prop}{\bf Proposition}[section]
\newtheorem{lem}{\bf Lemma}[section]
\newtheorem{assump}{\bf Assumption}[section]

\begin{document}
%
\title{
A general framework for denoising phaseless diffraction measurements
}
%
%

\author{Huibin Chang,
Stefano Marchesini
\thanks{H. Chang is  with Department
of Mathematical Sciences, Tianjin Normal University, Tianjin, 300387 China, e-mail:changhuibin@gmail.com.
He is currently a visiting scholar of the Center For Applied Mathematics for Energy Research Applications at Lawrence Berkeley National Laboratory.
}
\thanks{S. Marchesini is with Computational Research Division, Lawrence Berkeley National Laboratory, Berkeley.}
\thanks{
}
}

\markboth{Journal of \LaTeX\ Class Files,~Vol.~11, No.~4, December~2012}%
{Shell \MakeLowercase{\textit{et al.}}: Bare Demo of IEEEtran.cls for Journals}

\maketitle

\begin{abstract}
We propose a general framework to recover underlying images  from noisy phaseless diffraction measurements
 based on the alternating directional method of multipliers and the plug-and-play technique.
 The algorithm consists of three-step iterations: (i) Solving a generalized least square problem with the maximum a posteriori (MAP) estimate  of the noise, (ii) Gaussian denoising  and (iii) updating the multipliers.
The denoising step utilizes higher order filters such as total generalized variation and nonlocal sparsity based filters including nonlocal mean (NLM) and Block-matching and 3D filtering (BM3D) filters.
The multipliers are updated by  a symmetric technique to increase convergence speed.
 The proposed method with low computational complexity is provided with theoretical convergence guarantee, and it enables recovering  images  with sharp edges, clean background and repetitive features from noisy phaseless measurements.
  Numerous numerical experiments for Fourier phase retrieval (PR) as  coded diffraction and ptychographic patterns are performed to verify the convergence and efficiency, showing that our proposed method outperforms the  state-of-art PR algorithms without any regularization and those with total variational regularization.
\end{abstract}


%
\IEEEpeerreviewmaketitle

\section{Introduction}

\IEEEPARstart{P}hase retrieval (PR) plays a central role in the X-ray diffraction imaging in vast industrial and scientific applications, such as  crystallography and optics in \cite{Walther1963,marchesini2003x,Harrison:93,Miao2008,Shechtman2014}, \emph{etc}, and its goal is to reconstruct an object from phaseless measurements, \emph{i.e.} only pointwise magnitudes of the Fourier transform  (FT) of the object are available. Throughout the paper we discuss  the  PR from the noisy phaseless measurements in a discrete setting.

Assume we have an underlying  2-dimensional unknown object $u$  with  $n_1\times n_2$ pixels.  We denote $u$ in terms of a vector of size $n=n_1 \times n_2$  by a lexicographical order, i.e.
 $$u:\Omega=\{0,1,\cdots,n-1\}\rightarrow  \mathbb{C}^{}.$$
We call classical PR problem \cite{wen2012} when
one is given the magnitudes of the Fourier transform of $u$, \emph{i.e.},
$|\mathcal{F} u|^2$, where $|\cdot|^2$ denotes the pointwise  square of the absolute value of a vector, and $\mathcal{F}:\mathbb{C}^{n}\rightarrow \mathbb{C}^{n}$ denotes the discrete Fourier transform (DFT), and
the problem consists of  retrieving the underlining image $u$.
  A general PR problem \cite{Waldspurger2012,jaganathan2015phase} can be obtained as follows
\begin{equation*}\label{eqGPR}
{\text{To find ~}u\in \mathbb C^n,~~} s.t.~~|\mathcal A u|^2= b,
\end{equation*}
with the phaseless measurement $b:\tilde \Omega=\{0,1,\cdots,m-1\}\rightarrow\mathbb R_+$,
if one can extend the DFT operator $\mathcal F$ to arbitrary linear operator $\mathcal{A}:\mathbb C^n\rightarrow \mathbb C^m$.

The measurements can be contaminated by noise, or miscalibration of the experimental geometry, or the illumination mask, which leads to further trouble for the algorithm design. Generally speaking, PR involves solving a quadratic inverse problem, which is ill-posed and challenging. Here we give a brief review for  the computational tools for it. In 1947 D. Gabor \cite{gabor1948new} combined a known ``reference'' signal which linearized the PR problem, and received the Noble prize in Physics in 1971, in the 1950s   Karle and Hauptman  \cite{karle1950phases}, as well as Sayre \cite{sayre1952squaring} developed methods to solve the PR problem for binary atomic crystals. Hauptman and Karle  received the Nobel prize in Chemistry in 1985.  In earlier work, phaseless measurements of X-ray and electron diffraction played a pivotal role in unequivocally showing the existance of atoms \cite{eckert2012disputed}, atomic structure of crystals by the Braggs, molecular components of life including DNA, and later RNA, and over 100,000 proteins structures to date awarded by more than 10 nobel prizes in physics, chemistry and biology. In 1969 Hoppe \cite{hoppe1969beugung} proposed to use multiple diffraction measurements from a scanning sample, and in 1990s Rodemburg, Bates, and Chapman independently experimentally demonstrated a linearized inversion scheme based on the Wigner Distribution Deconvolution method \cite{rodenburg:92,chapman1996phase}. The complexities of the Wigner Deconvolution method are very high, which is about the square of the number of unknowns, and therefore it has not gained popularity in the experimental community.
In 1972  Gerchberg and Saxton \cite{Gerchberg1972}, introduced an alternating projection algorithm for a problem whereby one records phaseless measurements of the the scattering amplitude, and sample transmission.  Fienup extended the method and applied it to the classical phase retrieval problem \cite{Fienup1982}. The connection between Fienup' Hybrid Input Output and the Douglas-Rachford algorithm was discussed in \cite{Bauschke2002,Bauschke2003}, and the connection between the Douglas-Rachford and the ADM algorithm was discussed in \cite{wen2012}.  Several variant heuristic algorithms popular in the optics community including \cite{Elser2003,Luke2005} were proposed to solve the classical PR problems, and one can  refer to  \cite{marchesini2007invited,wen2012} and the reference therein. These methods use projections onto nonconvex constraint set, and therefore it is very difficult to ensure the convergence theoretically. Recently, more work has focused on the convergence theory. A global convergence for  Gaussian measurements was  analyzed by Netrapalli {\em et al.} \cite{netrapalli2015phase}.  The convergence under a generic frame  was proved in  \cite{marchesini2015alternating}, and spectral methods including initialization by truncation phase synchronization and framewise phase-synchronization were employed to accelerate  projection methods for large scale ptychographic PR.  Newton-type algorithms were proposed by Qian \emph{et al.} \cite{qian2014efficient} and Zhong \emph{et al.}\cite{zhongnonlinear} to accelerate the convergence. A modified Levenberg-Marquardt  method was proposed \cite{maglobally} with global convergence guarantees.  A Wirtinger flow (WF) approach was proposed in \cite{candes2015phase} by comprised of an initialization  by spectral method and gradient descent, and was further improved by  truncated Wirtinger flow (TWF) \cite{chen2015solving}. A Douglas-Rachford (DR) algorithm for PR with oversampling measurements was proved to be locally and geometrically convergent by Chen and Fannjiang \cite{chen2015fourier}.  In order to deal with the  conventional nonconvex projection algorithms for PR, the PhaseLift \cite{candes2013phaselift} algorithm was proposed by Cand\'es {\em et al.} with the lift technique of semi-definite programming(SDP). The  PhaseCut algorithm  was proposed by Waldspurger \emph{et al.}~\cite{Waldspurger2012}, where  the PR problem was convexified by separating phases and magnitudes.  A computationally tractable low-rank factorization method using lift technique of SDP  was proposed in \cite{horstmeyer2015solving} for solving ptychographic PR.


Sparse prior information of the unknown $u$ can be efficiently incorporated into PR in order to increase the quality of reconstructed image, especially when the data is noisy and incomplete.   The \textsc{shrink-wrap} algorithm generalized the regression model by slowly shrinking the size of the support of non-zero coefficients of $u$ \cite{marchesini:2003} and has been applied to several ground breaking experiments using X-ray Free Electron Lasers \cite{chapman2006femtosecond}. Similar methods based on hard thresholding, inversion \cite{chargeflip}, and $L_1$ soft thresholding have been proposed over the years, e.g. \cite{moravec2007compressive,marchesini2009ab}.  SDP-based methods were proposed to solve a sparse PR problem  \cite{ohlsson2012cprl,li2013sparse}. Directly extension of the Fieup's method \cite{Fienup1982} was proposed in  \cite{mukherjee2014fienup} by an additional sparsity constraint  of the $L_0$ norm for $u$. The $L_1$ regularization based variational method \cite{yang2013robust} was applied to the classical PR problem. An efficient local search method for recovering a sparse signal for the classical PR was presented in \cite{shechtman2014gespar}. A  probabilistic method  based on the generalized approximate message passing algorithm was proposed in \cite{schniter2015compressive}.   The shearlet and total variation sparsity regularization methods were considered in \cite{loock2014phase,chang2015,chang2016} by assuming the objects possess a sparse representation in the transform domain. Dictionary learning methods were proposed to reconstruct the image in \cite{tillmann2016dolphin,qiu2016undersampled}, where a dictionary was automatically learned from the redundant  image patches.

In this paper, we consider the PR problem for  the sparse images $u\in \mathbb C^n$ (or $\mathbb R^n$) in which the measured data $f\in \mathbb R^m_+$ are corrupted due to severe noisy measurements as
\[f=\mathrm{Corrupt}(|\mathcal A u|^2),\]   and a general minimization problem driven by regularization method of the underlying image can be established as
\begin{equation}\label{model}
\min\limits_{u} \Upsilon(u):=\mathbf B(|\mathcal A u|^2,f)+\lambda {\mathbf R}(u),~~s.t.~u\in \mathscr K
\end{equation}
where $u\in \mathscr K$ is an underlying image that we want to reconstruct from magnitude data $f$, $\mathbf  R(u)$ is the sparse promoted term, $\mathbf B(|\mathcal A u|^2,f)$ is the data fitting term deduced by the  maximum a posteriori (MAP) of the noise distribution,  $\lambda$ is the positive parameter to balance the sparsity and data fitting terms. Note that the set $\mathscr K$ is introduced to add some condition as positivity \cite{Shechtman2014} or the box constraint \cite{chang2015} and support set. Stimulated by the alternating directional method of multipliers \cite{Wu&Tai2010,boyd2011distributed} (ADMM) and the plug-and-play technique \cite{venkatakrishnan2013plug,li2014universal}, we  reformulate it to a novel framework for PR with theoretical convergence guarantee, which consist of three steps: Solving a generalized least square problem with the maximum a posteriori (MAP) estimate  of the noise, Gaussian denoising  and updating the multipliers. In order to implement  the denoising step, we utilize the higher order total variation such as total generalized variation and nonlocal sparsity based filters including nonlocal mean and Block-matching and 3D filtering (BM3D) filters in order to further improve the image quality compared with the work in \cite{loock2014phase,chang2015,chang2016}. Moreover, a symmetric  technique  is applied to the multiplier updating to order to increase the convergence speed.

The rest of this paper is organized as follows.
In section \ref{sec:model}, we will show the general framework starting from the ADMM for solving \eqref{model}. In section \ref{secAlg}, detailed numerical algorithms will be given, and the convergence of the proposed method is proved as well. In section \ref{sec:experiments}, the numerical experiments are performed to demonstrate the effectiveness of the proposed methods.  Conclusions and future works are given in section \ref{sec:conclusions}.

\section{General framework}\label{sec:model}

\subsection{MAP of noise}
We  consider two important types of noise as Poisson and Guassian noise. The measurements are usually contaminated by the Poisson noise for photon-counting, which follow Poisson distribution as follows
 \begin{equation*}
 \mathrm{Pr}_\mu(n)=\dfrac{e^{-\mu}{\mu^n}}{n!}, \qquad n\geq 0,
 \end{equation*}
 where $\mu$ is the mean and standard deviation.
The counted number of photons at pixel located at index $i$, denoted as $f(i)$, follows i.i.d. Poisson distributions with $h(i)$ being the ground-truth value  as \begin{equation}
 f(i)\stackrel{\mathrm{ind.}}{\sim}\mathrm{Poisson}(h(i)),~\forall i\in \tilde\Omega.
\end{equation}
By  MAP for a clean image $h$, the denoising problem can be expressed as $\max \mathrm {Pr}(h(i)|f(i))$ if   $f$ is measured. One can readily have
\begin{equation}
\label{eq:MLE}
{\mathrm {Pr}}(h(i)|f(i))=\dfrac{{\mathrm {Pr}}(f(i)|h(i)){\mathrm {Pr}}(h(i))}{{\mathrm {Pr}}(f(i))},
\end{equation}
according to the Bayes' Law, and as a result, maximization of $\mathrm {Pr}(g(i)|f(i))$ is equivalent to maximization of $ {\mathrm {Pr}}(f(i)|g(i)){\mathrm {Pr}}(g(i))$. Therefore, we have
\begin{equation*}
{\mathrm {Pr}}(f(i)|h(i)) = \mathrm{Pr}_{h(i)}(f(i)) = \dfrac{e^{-h(i)}{h(i)^{f(i)}}}{(f(i))!}.
\end{equation*}
Finally, one should minimize the logarithm of the ${{\mathrm {Pr}}(f(i)|h(i)){\mathrm {Pr}}(h(i))}$  instead,  {\em i.e.}
\begin{equation*}\label{Bayes}
\begin{split}
\min\limits_{h\geq 0}&\sum\limits_{i\in\tilde \Omega} -\log{{\mathrm {Pr}}(f(i)|h(i))}-\log {\mathrm {Pr}}(h(i))\\
=\min\limits_{h\geq 0}&\sum\limits_{i\in\tilde \Omega} (h(i)-f(i)\log{h}(i))-\log {\mathrm {Pr}}(h(i)).
\end{split}
\end{equation*}
In a similar manner by replacing the Poisson distribution with Gaussian distribution, one can readily get the MAP of it, and we summarize them as follows
\begin{equation}\label{MAP}
\mathbf B(h,f):=
\left\{
\begin{split}
&\frac12\langle h-f\log h,\bm 1\rangle, \text{ for  Poisson noised $f$};\\
&\frac12\|h-f\|^2, \qquad~\text{ for Gaussian noised $f$, }
\end{split}
\right.
\end{equation}
where  $\bm 1$ denotes  a vector whose entries are all one, and  $\langle\cdot,\cdot\rangle$ denotes the $L^2$ inner product of two vectors.

\subsection{General framework for PR}

Regularization often plays an important role for noise removal, and therefore, one can  establish a general model \eqref{model} if the measurements are collected from PR based on the above MAP. We will show how to build up a general framework starting from the ADMM  step by step.

As a very popular solver, ADMM is  simple and efficient to solve the existing variation models for noise removal. It can be obtained   by introducing an auxiliary variable $\bm p=\nabla u$ as \cite{chang2015,chang2016} to prevent directly solving a non-differential minimization problems by splitting technique. However, in order to solve \eqref{model}, there is an equivalent form to formulate the ADMM, which can provide a quite different understanding of the algorithm as \cite{venkatakrishnan2013plug}. By introducing an auxiliary viable $u=v$, one can readily decouple the data fitting term and regularization term as follows:
\begin{equation}\label{model1}
\begin{split}
\min\limits_{u,v}&\mathbf B(|\mathcal Au|^2,f)+I_{\mathscr K}(u)+\lambda \mathbf R(v),\\\
 &s.t.\qquad~u=v,
\end{split}
\end{equation}
where the indicator function $I_{\mathscr K}(u)$ defined by
\[
I_{\mathscr K}(u)=\left\{
\begin{split}
0,&~~u\in \mathscr K;\\
+\infty,&~~\text{otherwise,}
\end{split}
\right.
\]
with constrained set $\mathscr K$.

In order to solve \eqref{model1}, the corresponding  augmented Lagrangian is introduced with the multiplier $\Lambda$ as
\begin{equation}\label{Lagrangian}
\begin{split}
\mathcal L_r(u,v;\Lambda)=&\mathbf B(|\mathcal Au|^2,f)+I_{\mathscr K}(u)+\lambda \mathbf R(v)\\
+\mathrm{Re}&(\langle u-v,\Lambda\rangle)+\frac{r}{2}\|u-v\|^2,
\end{split}
\end{equation}
 where   the parameter $r$ is  a positive constant, and $\mathrm {Re}(\cdot)$ denotes the  real part of a complex-valued number.
The ADMM with symmetric updating scheme  following \cite{he2014strictly} can be written in order to solve the saddle point problem
\[
\max\limits_{\Lambda}\min\limits_{u,v}\mathcal L_r(u,v;\Lambda)
\]
as
\begin{equation}\label{ADMM}
\left\{
\begin{split}
&u^{k+1}=\arg\min\limits_{u} \mathcal L_r(u,v^{k};\Lambda^{k}),\\
&\Lambda^{k+1/2}=\Lambda^{k}+r(u^{k+1}-v^{k}),\\
&v^{k+1}=\arg\min\limits_{v} \mathcal L_r(u^{k+1},v;\Lambda^{k+1/2}),\\
&\Lambda^{k+1}=\Lambda^{k+1/2}+ r(u^{k+1}-v^{k+1}),
\end{split}
\right.
\end{equation}
if provided with the $k^{th}$ iterative solution $(u^k,v^k,\Lambda^k).$
 The  above algorithm is slightly different to \cite{he2014strictly} since a relax parameter for the step size $r$ can not accelerate the convergence based on the numerical experiments reported in \cite{he2014strictly} and therefore we omit it for simplicity. In our numerical experiments, updating the multiplier with symmetric style can indeed speedup the proposed method (see details in Figure \ref{hist1}) compared with asymmetric style.

Let us analyze the first and the third subproblems one by one. First for the first subproblem, one needs to solve
\[
\begin{split}
u^{k+1}=&\arg\min\limits_{u} \mathcal L_r(u,v^k;\Lambda^k)\\
=\arg\min\limits_u&\mathbf B(|\mathcal Au|^2,f)+I_{\mathscr K}(u)+\frac{r}{2}\|u-(v^{k}-{\Lambda^k}/{r})\|^2,
\end{split}
\]
which we call it a generalized least square minimization problem with respect to $u$ with data term as $v^{k}-\frac{\Lambda^k}{r}$ in order to satisfy the statistical property of the noise.
Similarly, for the third subproblem with respect to the variable $v$,
\begin{equation}\label{denoiseModel}
v^{k+1}=\arg\min\limits_v  \lambda \mathbf R(v)+\frac{r}{2}\|v-(u^{k+1}+\Lambda^{k+1/2}/r)\|^2,
\end{equation}
which we call it a denoising step in order to  smooth the data $u^{k+1}+\Lambda^{k+1/2}/r$. If setting $\mathbf R(v)=\mathrm{TV}(v)$, the well-known total variation \cite{osher1990feature,alvarez1992image,malladi1995image,rudin1992nonlinear} based model was considered in for Gaussian noised data in \cite{chang2015} and Poisson noised data in \cite{chang2016}.

Naturally one can build other generalized least square forms in the first subproblem by the maximum a posteriori (MAP) estimate for other kinds of noise, and in this paper we only focus on the Gaussian and Poisson noise, which are common and challengeable for PR. One can also generalize the third step by arbitrary variational denoising method such as higher order TV \cite{LLT2006,bredies2010total,tai2011fast} instead of the traditional TV in order to promote the sparsity. Moreover, similar to the work in \cite{danielyan2012bm3d,venkatakrishnan2013plug,li2014universal}, some advanced filters as BM3D can be introduced to replace the variational image denoising methods.  Therefore the proposed general framework  is written in a unified form as
\begin{equation}\label{ADMMProximal}
\left\{
\begin{array}{lll}
\text{Step 1:}&u^{k+1}=\mathrm{Prox}_{\mathbf G/r+I_{\mathscr K}}\left(v^{k}-\frac{\Lambda^k}{r}\right),\\
\text{Step 2:}&\Lambda^{k+1/2}=\Lambda^{k}+r(u^{k}-v^{k+1}),\\
\text{Step 3:}&v^{k+1}=\mathrm{Prox}_{\sigma \mathbf{R}}\left(u^{k+1}+\frac{\Lambda^{k+1/2}}{r}\right) ,\\
\text{Step 4:}&\Lambda^{k+1}=\Lambda^{k+1/2}+ r(u^{k+1}-v^{k+1}),
\end{array}
\right.
\end{equation}
where the proximal operator is denoted as
\[\mathrm{Prox}_{s}(v)=\min\limits_u s(u)+\frac{1}{2}\|u-v\|^2,\] the operator $\mathrm{Prox}_{\sigma\mathbf{R}}$  denoises the data $v$, the noise level is directly denoted by $\sigma:=\frac{\lambda}{r}$ in \cite{rond2015poisson}, and \[\mathbf G(u):= \mathbf B(|\mathcal Au|^2,f).\]

In this paper, we will consider to implement this framework in two aspects. First, we consider two  different data fitting terms of  $\mathbf B(\cdot,\cdot)$  for  the Poisson and Gaussian  noise removal respectively. Second, for the denosing procedure in Step 3, two kinds of operators can be employed: One is the variational models with the explicit expression of $\mathbf G$   as ROF \cite{rudin1992nonlinear}, LLT \cite{lysaker2003noise}, nonlocal TV \cite{zhang2010bregmanized} and total generalized variation (TGV) \cite{bredies2010total}; The other kind is the filter based without the explicit expression of $\mathbf R$ such as the bilateral  filter   \cite{tomasi1998bilateral}, nonlocal means filter (NLM) \cite{buades2010image}, and BM3D filter \cite{dabov2007image}. Our proposed method is rather simple and flexible, and can transform the denoising the noisy PR to traditional PR task (without regularization or filter) and the image denoising task. Therefore, a group of denoising methods can be employed to further improve the start-of-art PR methods. It will be very efficient, since we combine some advanced filters, and on this sense, we introduce a group of new filters for PR.
\subsection{Connection with other related methods for PR}
Our proposed method belongs to the ``black box'' methods. One do not need to specifically design elaborate algorithms for the original optimization models, and only focus on designing efficient solvers for subproblems.
Actually similar ideas already exist in \cite{danielyan2012bm3d,li2014universal,li2016new} where a general framework for image deblurring and inpainting in image and transform domains were proposed. Compared with the work \cite{chang2015,chang2016}, one can design different and more efficient algorithm for the inner loop for the Step 1 and Step 3 of our proposed method to solve the generalized least square and  image denoising problems. Compared with the work \cite{tillmann2016dolphin,qiu2016undersampled}, more fast and efficient filter as BM3D can be incorporated for the denoising step in our framework, and furthermore Poisson noised data can be dealt with by our proposed method. In a recent work in \cite{metzler2016bm3d}, a general framework was also proposed based on generalized approximate message passing from noisy data. However, they consider a different problem with measurements $f$  as
\[
f=|\mathcal A u+\epsilon|
\]
where $\epsilon$ denotes the additive noise, and  one readily sees that the noise happens to  $\mathcal A u$.
    In summary, we aim to bridge the gap between the PR algorithm and the image processing methods in this paper. The core idea in this paper is quite close to the error reduction algorithm \cite{Gerchberg1972} for PR, and the alternating projection is computed between the magnitude constraint set and the sparsity promoted evolution.

\section{Numerical algorithm and convergence analysis}\label{secAlg}
In this section we will present the algorithm details in the following two subsections for the generalized least quare problem in the step 1 of \eqref{ADMMProximal} and the denoising problems in the Step 3 of \eqref{ADMMProximal} separately.
\subsection{Step 1 of \eqref{ADMMProximal}: Solving the generalized least square subproblems}
We focus on  solving the first subproblem this subsection. By setting $g=v^{k}-\frac{\Lambda^k}{r}$,  it becomes
\[u^{k+1}=\mathrm{Prox}_{\mathbf G/r+I_{\mathscr K}}(g)\]
There are many choices for solve the traditional PR problems. However, these methods can not be directly applied to such model with additional quadratic term. An ADMM can be used directly to solve this problem. Similarly to \cite{wen2012,chang2015}, we reformulate it by introducing $z=\mathcal Au$ as
\begin{equation}\label{LagrangianLS}
\min\limits_{u,z} \mathbf B(|z|^2,f)+I_{\mathscr K}(u)+\frac{r}{2}\|u-g\|^2, s.t. z=\mathcal A u,
\end{equation}
with the corresponding augmented Lagrangian as
\[
\begin{split}
 \hat{\mathcal L}_\eta(u,z;\hat \Lambda)&=\mathbf B(|z|^2,f)+I_{\mathscr K}(u)+\frac{r}{2}\|u-g\|^2\\
 +& \mathrm{Re}(\langle z-\mathcal A u,\hat\Lambda\rangle)+\frac{\eta}{2}\|z-\mathcal Au\|^2.
\end{split}
\]
The ADMM consists of three steps. For the subproblem with respect to the variable $u$, one shall compute
the problem as
\[
\min\limits_{u\in \mathscr K} \frac{r}{2}\|u-g\|^2+\frac{\eta}{2}\|\mathcal Au-(z+\hat\Lambda/\eta)\|^2.
\]
If the constrained set $K=\mathbb C^n$, similarly to \cite{chang2016}, the real and complex parts of minimizer of the above problem satisfies the  the following equations
\begin{equation}\label{subsubU}
\begin{split}
	&\left[
	\begin{matrix}
	\eta{\mathrm{Re}(\mathcal A^* \mathcal A)}+r\mathbf I& -{\mathrm{Im}(\mathcal A^* \mathcal A)}\\
	&\\
	{\mathrm{Im}(\mathcal A^* \mathcal A)}& \eta{\mathrm{Re}(\mathcal A^* \mathcal A)}+r\mathbf I
	\end{matrix}
	\right]
	\left[\begin{matrix}
	\mathrm{Re}(u)\\
	\\
	\mathrm{Im}(u)
\end{matrix}\right]\\
\\
&\qquad=
	\left[
	\begin{matrix}
	\eta\mathrm{Re}(\mathcal A^* \hat z)+r \mathrm{Re}(g)\\
	\\
	\eta\mathrm{Im}(\mathcal A^* \hat z)+r\mathrm{Im}(g)\\
	\end{matrix}
	\right],
\end{split}
	\end{equation}
	with $\hat z=z+\hat\Lambda/\eta$.

We can simplify the solution of above subproblem if the matrix $\mathcal A$ involves
Fourier measurements with masks $\{I_k\}_{k=1}^K$  for coded diffraction pattern as
	\begin{equation}\label{cdp}
	\mathcal Au=
	\left[
	\begin{matrix}
	\mathcal F(w_0\circ u)\\
	\mathcal F(w_1\circ u)\\
	\vdots\\
	\mathcal F(w_{K-1}\circ u)
	\end{matrix}
	\right],
	\end{equation}
where $\circ$ denotes the pointwise multiplication, $w_k$ is a (mask) matrix indexed by $k$, each of which is represented by a vector in $\mathbb C^{n}$ in a lexicographical order.
	Therefore we have $\mathcal A^* \mathcal A=\sum\limits_j w_j^*\circ w_j,$ which is a real-valued matrix. Finally we have
\begin{equation}\label{subsubUSimple}
u_{min}=\big(\eta{\mathcal A^* \mathcal A}+r\mathbf I\big)^{-1}\big(\eta\mathcal A^*(z+\hat\Lambda/\eta)+rg\big),
\end{equation}
since the diagonal matrix ${\eta\mathcal A^* \mathcal A}+r\mathbf I$ is non-singular. One can easily get the same equations as \eqref{subsubUSimple} for ptychographic PR, and we omit the details here. However, if $\mathscr K\not=\mathbb C^n$, it requires us to introduce additional variable to establish the ADMM as \cite{chang2015}. In order to simplify the algorithm, one can use an additional projection step for such case.

For the subproblem with respect to the variable $z$, the following subproblem should be considered
\[
\min\limits_{z}\mathbf B(|z|^2,f)+\frac{\eta}{2}\|z-z_0\|^2,
\]
with $z_0:=\mathcal Au-\hat\Lambda/\eta$.

When the measurement is contaminated by the Poisson noise, we can directly get the solution to the minimization problem
\begin{equation}\label{zsub}
\min\limits_{z}\frac12\langle |z|^2-2f\log|z|,\bm 1 \rangle+\frac{\eta}{2}\|z-z_0\|^2,
\end{equation}
following \cite{wu2011augmented,chang2016} in the following lemma.
\begin{lem}
The minimizer to \eqref{zsub} is
\begin{equation}\label{eqPoisson}
\begin{split}
z_{min}(j)=&\dfrac{\eta|z_0(j)|+\sqrt{\eta^2|z_0(j)|^2+4(1+\eta)f(j)}}{2(1+\eta)}
\\&\qquad\times\mathrm{sign}(z_0(j)),
\end{split}
\end{equation}
for all $j\in \tilde \Omega.$
\end{lem}

When the noise is Gaussian, a direct $L^2$ data fitting term can be obtained by the MAP as the first equation in \eqref{MAP}.
Hence the minimization problem should be considered as
\begin{equation}\label{LSG}
\min\limits_{z} \frac12\big\||z|^2-f\big\|^2+\frac{\eta}{2}\|z-z_0\|^2.
\end{equation}
We give a lemma to show how to compute such problem.
\begin{lem}
The minimizer to \eqref{LSG} is
\begin{equation}\label{subsubZ}
z_{min}(j)=\rho(j)\mathrm{sign}(z_0(j)),
\end{equation}
where $\mathrm{sign}(z_0(j))=\dfrac{z_0(j)}{|z_0(j)|}$, and
\begin{equation}\label{cubicsolver}
\begin{split}
&\rho(j)=\\
&\left\{
\begin{split}
&\sqrt[3]{\frac{\eta|z_0(j)|}{4}+\sqrt{D(j)}}+\sqrt[3]{\frac{\eta|z_0(j)|}{4}-\sqrt{D(j)}},~\mbox{if~} D(j)\geq 0; \\
&2\sqrt{{(f(j)-\eta/2)}/{3}}\cos\big(\arccos({\theta_j}/{3})\big),\qquad \mbox{otherwise,}
\end{split}
\right.
\end{split}
\end{equation}
 for all $j\in \tilde\Omega,$
 with \[
D(j)=\dfrac{(\eta/2-f(j))^3}{27}+\dfrac{\eta^2|z_0(j)|^2}{16},\]
and
\[\theta(j)=\dfrac{\eta|z_0(j)|}{4\sqrt{-(\eta/2-f(j))^3/27}}.\]
\end{lem}
\begin{proof}
One readily sees that the optimization procedure is independent for each entry of $z$. Therefore we only consider the problem for each entry  $z(i)$
as
\[
z_{min}(j)=\arg\min\limits_{z(j)} \frac12\big||z(j)|^2-f(j)\big|^2+\frac{\eta}{2}|z(j)-z_0(j)|^2.
\]
Then  we have $z_{min}(j)=\rho(j) \mathrm{sign}(z_0(j))$, where $\rho(j)\geq 0$. We focus on the determination of $\rho(j)$ as
\begin{equation}\label{phi}
\begin{split}
\rho(j)&=\arg\min\limits_{x\geq 0} \phi(x),
\end{split}
\end{equation}
with $$\phi(x):=\frac12(x^2-f(j))^2+\frac\eta2(x-|z_0(j)|)^2.$$
The stationary points of the above problem satisfy the following cubic equation as
\[
x^3+\big(-f(j)+\frac{\eta}{2}\big)x-\frac\eta2|z_0(j)|=0.
\]
The minimizer should be either  the stationary points of $\phi(x)$  or zero.
Based on a simple discussion by Vieta's formulas and the close solution to cubic equation \cite{cubic}, it only has one real non-negative root as \eqref{cubicsolver}, which is also the unique minimizer of \eqref{phi}. Therefore, we can compute the value the non-negative root $\rho(j)$ by \eqref{subsubZ}  to get the minimizer of  \eqref{phi}, that concludes to this lemma.
\end{proof}

In summary, we give the overall algorithm  for this subproblem of Step 1 in \eqref{ADMMProximal} as
\begin{equation}\label{leastSquareAlg}
\left\{
\begin{split}
\text{Step 1:}&\text{ Solve~} u_{k+1} {\text {~by \eqref{subsubU}}} \text{~with~} (z,\hat\Lambda):=(z_k,\hat\Lambda_k),\\
\text{Step 2:}&\text{ Solve~} z_{k+1} {\text {~by \eqref{subsubZ} and \eqref{eqPoisson}}} \\
              &\qquad\qquad \text{with~} (u,\hat\Lambda):=(u_{k+1},\hat\Lambda_k),\\
\text{step 3:}&\text { Update~ } \hat \Lambda_{k+1} \text{~as~}\\
& \qquad\hat\Lambda_{k+1}=\hat\Lambda_{k}+\eta(z_{k+1}-\mathcal A u_{k+1}),
\end{split}
\right.
\end{equation}
for the ${(k+1)}^{th}$ iterations if provided with $(u_k,z_k,\hat\Lambda_k).$

\subsection{Step 3 of \eqref{ADMMProximal}: Denoising subproblems with respect to the variable $v$}

As mentioned above, two  kinds of operators can be considered. First one  is the variational image methods as ROF, LLT, NLTV and TGV, and the other kind is the filter such as the bilateral  filter,  nonlocal means filter and BM3D filter.
\subsubsection{Variational methods}
One needs to compute the subproblem of the third step of \eqref{ADMMProximal} as
\begin{equation}\label{varDenoi}
\min\limits_{v\in \mathbb C^n } \sigma\mathbf{R}(v) +\frac{1}{2}\|v-v_0\|^2,
\end{equation}
where we rewrite \eqref{denoiseModel}  with $v_0:=v^{k}-\frac{\Lambda^k}{r}.$

First for the underling real-valued image, one can directly use the existing algorithm to solve the variational models. However, in order to deal with the complex-valued image, we should give precise definition of TV as
\[
\mathrm{TV}(v)=\sqrt{\sum\limits_j |(\mathrm{D}_x v)_j|^2+|(\mathrm{D}_y v)_j|^2},
\]
where $\mathrm{D}_x(\cdot),\mathrm{D}_y(\cdot)$ denote the forward difference operators of the complex-valued image  with respect to the $x$-direction and $y$-direction. For TV based denoising, we need to compute the minimization problem \eqref{denoiseModel} as
\[
\min\limits_{v\in \mathbb C^n } \sigma\mathrm{TV}(v) +\frac{1}{2}\|v-v_0\|^2,
\]
with the complex-valued variable $v_0$.
By introducing the variable $p=\nabla u:=(\mathrm{D}_x u, \mathrm{D}_y u)$, we have to solve the following saddle point problem as
\[
\begin{split}
\max_{\Psi}\min\limits_{p,v} &\quad\sigma\|p\|+\frac12\|v-v_0\|^2+\mathrm{Re}(\langle p-\nabla v,\Psi\rangle)\\
&\quad\qquad+\frac{\gamma}{2}\|p-\nabla v\|^2,
\end{split}
\]
with  a positive constant $\gamma$.
It consists of three steps w.r.t the variables $v,p$ and the update of the multiplier $\Psi$. For the subproblem w.r.t $v$,
one needs to solve
\[
\min\limits_{v} \frac12\|v-v_0\|^2+\frac{\gamma}{2}\|p+\frac{1}{\gamma}\Psi-\nabla v\|^2.
\]
The above minimization problem can be considered as minimizing of the real and the complex parts of the variable $v$ independently, and therefore by simply separating the real and complex parts of $v$ as \eqref{subsubU} and computing the stationary points, we have
\[
\left\{
\begin{split}
&(-\gamma \Delta+\mathbf I) \mathrm{Re}(v)=\mathrm{Re} (v_0)+\gamma\mathrm{Re}(\nabla (p-\frac1\gamma \Psi) ),\\
&(-\gamma \Delta+\mathbf I) \mathrm{Im}(v)=\mathrm{Im} (v_0)+\gamma\mathrm{Im}(\nabla (p-\frac1\gamma \Psi) ),
\end{split}
\right.
\]
where $\Delta u=\nabla\cdot \nabla u$, $\nabla \cdot p$ denotes the backward difference operator of $p$ which satisfies the adjoint relation as
\[
\langle \nabla v, p\rangle=-\langle v,\nabla\cdot p\rangle.
\]
Finally we obtain the uniform formula for the complex-valued image as
\begin{equation}\label{vsubproblem}
(-\gamma \Delta+\mathbf I) (v)=v_0+\gamma\nabla (p-\frac1\gamma \Psi),
\end{equation}
which actually has a same form for the real-valued image. Furthermore, if using the periodical boundary condition for variable $v$, fast Fourier transformations (FFTs) can be used to solve the above problem with optimal time complexity.

For the subproblem w.r.t the variable $p$, one needs to compute the minimizer $p_{min}$ by the soft thresholding as
\begin{equation}\label{psubproblem}
p_{min}=\max\{0,|\nabla v-\Psi/r|-\sigma/\gamma\}\circ\mathrm{sign}(\nabla v-\Psi/r).
\end{equation}
 In summary, combining the solutions $v_{k+1}$ of \eqref{vsubproblem} and the solution $p_{k+1}$ of \eqref{vsubproblem} with the update of the multipliers $\Psi$ by \[\Psi_{k+1}=\Psi_k+\gamma (p_{k+1}-\nabla v_{k+1}),\] a fast minimization algorithm for  the subproblem \eqref{denoiseModel} is established. For other kind of higher order TV models, one can readily build the efficient ADMMs just by considering the complex-valued image instead of the real-valued image, since the analysis of solving the above TV  model demonstrates the similarity of the cases between the real-valued and complex-valued image. Limited to the space, we do not give all the detailed algorithms for them.

\subsubsection{Filters based methods}
There are a considerable mount of filters for image denoising, and a very comprehensive and deep review was given in \cite{milanfar2013tour}, where it showed that the bilateral filter, boosting, kernel, and spectral method, and nonlocal means and so on are  deeply connected with the current popular iterative methods as the Bregman iterations \cite{goldstein2009split}. Among  these filters, we focus on NLM and BM3D filters in this paper, which promote the patch sparsity and can  be used for the complex-valued image directly.

The nonlocal mean filter can be defined as a linear combination of the nonlocal neighbours \cite{buades2010image} as
\[
\mathrm{NLM}(v)(x)=\sum\limits_{y\in \Omega_x} w(x,y) v(y),~\forall x\in \Omega,
\]
where $v$ is the noisy image, the nonlocal weight function $w(x,y)$ measures the similarity between two pixels satisfying $w(x,y)\geq 0, ~\sum\limits_{y\in\Omega_x} w(x,y)=1, ~\forall x \in \Omega,~y\in \Omega_x$, and the region $\Omega_x$ is the search window at point $x$ which is usually square and selected over the entire image. Nonlocal total variation regularized methods and fast algorithms stimulated by the nonlocal means were also developed \cite{gilboa2008nonlocal,zhang2010bregmanized} for image deblurring, inpainting and reconstruction.

Different to the idea employing the weighted average of nonlocal patches from the noisy image, BM3D enhanced the sparsity by grouping 2D image patches consisting of similar structures or features into 3D patches. Four steps includes (1) Analysis step: Grouping the similar patches into a 3D blocks and linear transformation of the 3D blocks, (2) Processing: Hard thresholding of the transform domain, and (3) Synthesis: Inverse 3D transformation of the shrinkage of the transform domain to the image domain. The variational methods combining the BM3D filter for image deblurring were proposed in \cite{danielyan2012bm3d} and inpainting \cite{li2014universal,li2016new} based on the regularization by $L^p$ norm of the transform domain. The Nash equilibrium problem  as a bilevel optimization problem was further presented, and therefore the method by a direct use of BM3D filter greatly improved the restoration results of traditional variational methods with the single-objective minimization problem.

\subsection{Theoretical analysis}
We will give the theoretical analysis of the proposed framework in \eqref{ADMMProximal}. The denoising subproblem is well defined for both the regularized term or the filters. First we analyze the convergence of ADMM for the generalized least square subproblem in \eqref{leastSquareAlg}. Readily one can infer that the functional $\mathbf G(u)$ is nonconvex, which leads to the main challenge for the convergence study. However, from the numerical performances, it is quite robust. Following \cite{wen2012}, we give the convergence results as follows.
\begin{prop}
Assuming that the multiplier of \eqref{LagrangianLS} exists, the sequences of the multipliers $\hat \Lambda_{k+1}-\hat \Lambda_{k}\rightarrow 0$ as $k\rightarrow \infty,$ and $\{u_k\}$ is bounded, then there exists an accumulative point of $\{u_k,z_k,\hat\Lambda_k\}$ satisfies the Karush-Kuhn-Tucker (KKT) condition of \eqref{LagrangianLS} (a stationary point).
\end{prop}
\begin{proof}
See the detailed proof in the Refs. \cite{wen2012,chang2015,chang2016}.
\end{proof}

 Here we give a very general assumption of the convex regularization  $\mathbf R(v)$ as follows.
\begin{assump}\label{assump}
The functional $\mathbf R(v)$ is proper, closed, and convex, and lower semi-continuous.
\end{assump}

We analyze the convergence of the ADMM for the denosing step of \eqref{ADMMProximal}  when variational regularization models are incorporated in.
\begin{prop}
The ADMM converges to the unique global minimizer of \eqref{varDenoi} for the denoising step of  \eqref{ADMMProximal}.
\end{prop}
\begin{proof}
Readily one can prove the unique existence of the minimizer to \eqref{varDenoi} under Assumption \ref{assump}. By lifting the dimension of the complex-valued variable $v$ by separating the corresponding complex  and real parts as \cite{chang2016}, one can conclude this proposition for this convex minimization problem following \cite{Wu&Tai2010,boyd2011distributed}.
\end{proof}

We have shown the convergence study of the algorithms for each subproblems, and  one can also derive the following convergence theorem for the general framework in \eqref{ADMMProximal} if the  convex variational regularization functional $\mathbf R(v)$ satisfies Assumption \ref{assump}.
\begin{thm}\label{thm1}
If the subproblems in Step 1 and Step 3 are exactly solved, the algorithm in \eqref{ADMMProximal} is convergent, i.e. the iterative sequences have an accumulative point $(u^*,v^*,\Lambda^*)$, which satisfies the KKT condition of  \eqref{Lagrangian} if the sequences $\{\Lambda^{k+1}-\Lambda^{k+1/2}\}$ or $\{\Lambda^{k+1/2}-\Lambda^{k}\}$ converge to zero and $\{u^k\}$ is bounded.
\end{thm}
\begin{proof}
One needs to prove $(u^*,v^*,\Lambda^*)$ satisfies that
\begin{equation}\label{KKTVar}
\left\{
\begin{split}
&0\in\nabla_u \mathbf G(u^*)+\partial_u  I_{\mathscr K}(u^*)+\Lambda^*, \\
&0\in \lambda \partial_v \mathbf R(v^*)-\Lambda^*,\\
&0=u^*-v^*.
\end{split}
\right.
\end{equation}
We finish the proof in two steps. In Step 1, we prove the boundedness of the iterative solutions.
First by the multiplier update steps, the sequence $\{v^k\}$ is bounded as a result of the boundedness of $\{u^k\}$. Therefore, the multiplier $\{\Lambda^{k+1/2}\}$ is bounded due to $(u^{k+1}+\Lambda^{k+1/2}/r)-v^{k+1}\in \frac\lambda r \partial_v\mathbf R(v^{k+1}).$ With the assumption of the convergence of $\{\Lambda^{k+1}-\Lambda^{k+1/2}\}$ or $\{\Lambda^{k+1/2}-\Lambda^{k}\}$, the boundedness of $\{\Lambda^k\}$ is derived. Therefore there exists a subsequence
$\{u^k,v^k,\Lambda^k\}$ (still use the same notation to represent this subsequence) and a triple $(u^*,v^*,\Lambda^*)$, such that
\[
\{u^k,v^k,\Lambda^k\}\rightarrow \{u^*,v^*,\Lambda^*\} \text{~as~} k\rightarrow \infty.
\]

In Step 2, we will prove the triple $(u^*,v^*,\Lambda^*)$ satisfies \eqref{KKTVar}. One can readily has $u^*=v^*$ by taking limit in the step of multiplier updates.
Readily one has
\[
\begin{split}
&\mathbf G(u^{k+1})+I_{\mathscr K}(u^{k+1})+\frac{r}{2}\|u^{k+1}-(v^{k}-{\Lambda^k}/{r})\|^2\\
&-\frac{r}{2}\|u-(v^{k}-{\Lambda^k}/{r})\|^2
\leq \mathbf G(u)+I_{\mathscr K}(u)~\forall~u
\end{split}
\]
By taking limits of $u^k,v^k,\Lambda^k$, and the lower semi-continuity of $\mathcal G$ and $I_{\mathscr K}$, one has
\[
\begin{split}
&\mathbf G(u^{*})+I_{\mathscr K}(u^{*})+\frac{r}{2}\|u^*-(v^{*}-{\Lambda^*}/{r})\|^2\\
\leq &\mathbf G(u)+I_{\mathscr K}(u)+\frac{r}{2}\|u-(v^{*}-{\Lambda^*}/{r})\|^2~\forall~u,
\end{split}
\]
which is the first relation of the KKT condition. The second one can be derived in a similar manner. That concludes to the theorem.
\end{proof}
The convergence of the proposed method combining the filters is much difficult than the above case, since one does not know the explicit form of $\mathbf R(v)$ for a general filters such as NLM and BM3D. The convergence is only limited to the fixed-point sense. We give the following assumption.
\begin{assump}
There exists a fixed point $(u^*,v^*,\Lambda^*)$ for the framework  \eqref{ADMMProximal}, i.e.
\begin{equation}\label{fp}
\left\{
\begin{split}
&u^*=\mathrm{Prox}_{\mathbf G/r+I_{\mathscr K}} (v^*-\Lambda^*/r), \\
&v^*=\mathbf D_{\sigma}(u^*+\Lambda^*/r),\\
&u^*=v^*,
\end{split}
\right.
\end{equation}
where
\[\mathrm{D}_{\sigma}(u):=\mathrm{Prox}_ {\sigma\mathbf R } (u)\]
  represents a filter.
\end{assump}
We can eliminate $v^*$ in order to get the following form for $(u^*,\Lambda^*)$
\begin{equation}\label{fixedPoint}
\left\{
\begin{split}
&u^*=\arg\min\limits_u \mathrm{Prox}_{\mathbf G/r+I_{\mathscr K}} (u^*-\Lambda^*/r), \\
&u^*=\mathbf D_{\sigma}(u^*+\Lambda^*/r).
\end{split}
\right.
\end{equation}

Using the condition as Theorem \ref{thm1}, we have the following theorem as follows.
\begin{thm}
Assume the filter is  continuous, i.e.
\[
\mathbf D_\sigma(\hat v_k)\rightarrow \mathbf D_\sigma(\hat v^*)\text{~if~} \hat v_k \rightarrow \hat v^* \text{as~} k\rightarrow\infty
.\]  If the subproblem in Step 1  are exactly solved, the algorithm in \eqref{ADMMProximal} is convergent, i.e. the iterative sequences have an accumulative point $(u^*,v^*,\Lambda^*)$, which is a fixed point  of  \eqref{fp} if the sequences $\{\Lambda^{k+1}-\Lambda^{k+1/2}\}$ or $\{\Lambda^{k+1/2}-\Lambda^{k}\}$ converge to zero and $\{u^k\}$ is bounded.
\end{thm}

\begin{proof}
It can be readily proved similarly to the proof of Theorem \ref{thm1}, and we omit the details here.
\end{proof}

\begin{rem}
In the above theorem, we assume that the filter operator is continuous. In real applications, the operator  possibly  relies on the image and therefore one can  hardly give a verification. In the work \cite{chan2016plug}, the operator is assumed to be bounded, and their proposed method  with adaptive steps is convergent to its fixed point if $\mathrm{R}(v)$ has bounded gradient, that can not apply to  the data fitting term in our model.
\end{rem}
%
%

\section{Experiments and discussion}\label{sec:experiments}

\subsection{Numerical implementation}
\label{subsec:comparisons}


In the numerical part, we set $\mathscr K=\mathbb R^n$  and $\mathbb C^n$ for real-valued and complex-valued image respectively in \eqref{model}.  Although we have given a  framework for  PR with arbitrary linear operator $\mathcal A$, we only show the performance  on Fourier measurements involved with two  types of patterns of linear operators $\mathcal A$: coded diffraction pattern (CDP) with random masks
 and ptychographic PR with zone plate lens. For the first pattern,  the octanary CDP is explored, and specifically   each
element of $I_j$ in \eqref{cdp} takes a value randomly among the eight candidates, \emph{i.e.}, $\{\pm \sqrt{2}/2, \pm \sqrt{2}{\mathbf i}/2, \pm \sqrt{3},\pm \sqrt{3}{\mathbf i}\}$. Set $K=2, 4$ for real-valued and complex-valued image respectively  as \cite{chang2016}. The setting of ptychographic PR is given as follows. The number of frames is $16\times16$ with sliding distance $16$ pixels and the frame size  $64\times 64$. The len and the generated  illumination are shown in Figure \ref{illumination}.
\begin{figure}
\begin{center}
\subfigure[]{\includegraphics[width=.11\textwidth]{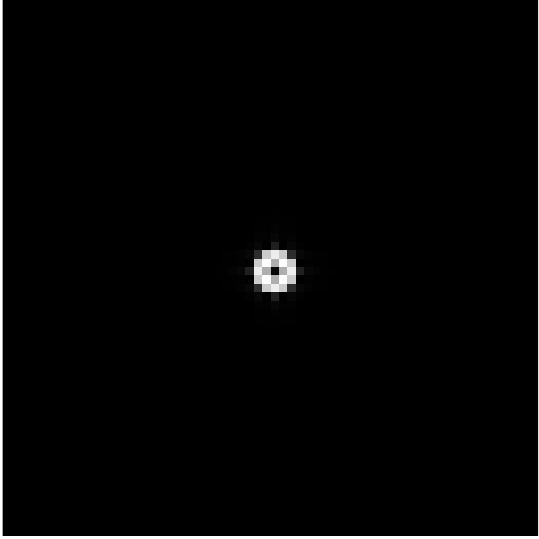}}\quad\qquad
\subfigure[]{\includegraphics[width=.11\textwidth]{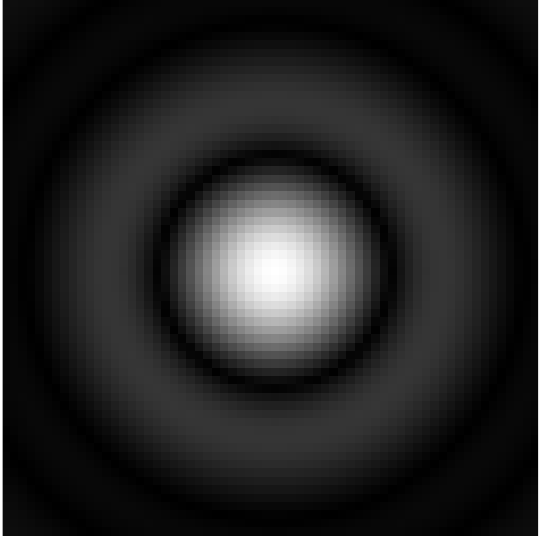}}
\end{center}
\caption{Lens in (a) and the generated illumination in (b) used for ptychographic PR. }\label{illumination}
\end{figure}

The ground truth images  consist of five images including four real-valued with $512\times 512$ pixels and one complex-valued with  $256\times 256$ pixels.
\begin{figure}
\begin{center}
\subfigure[]{\includegraphics[width=.11\textwidth]{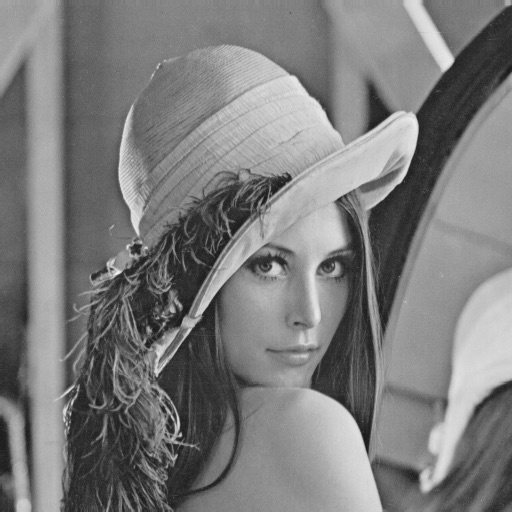}}
\subfigure[]{\includegraphics[width=.11\textwidth]{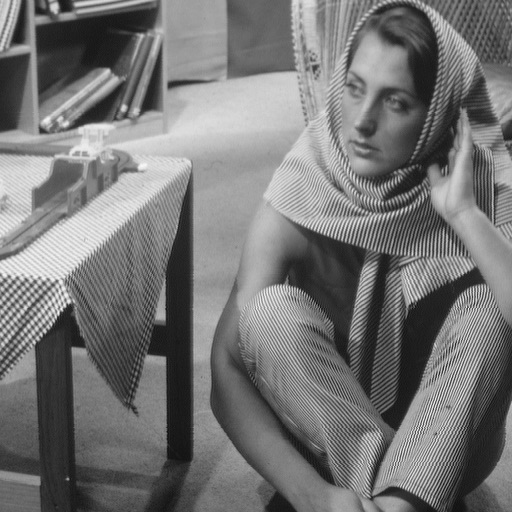}}
\subfigure[]{\includegraphics[width=.11\textwidth]{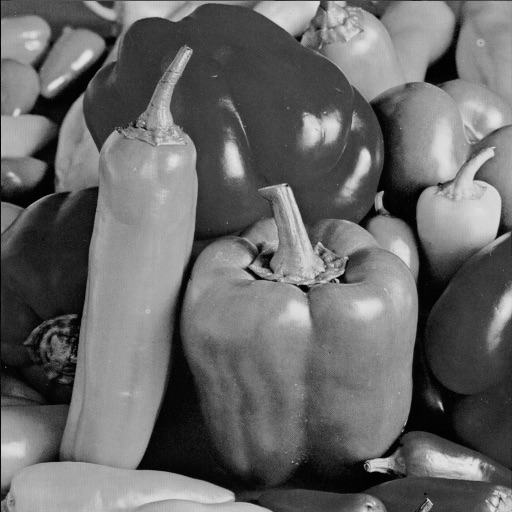}}
\subfigure[]{\includegraphics[width=.11\textwidth]{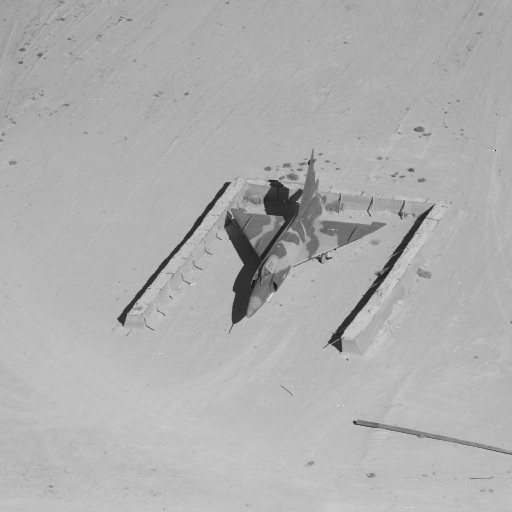}}\\
\subfigure[]{\includegraphics[width=.11\textwidth]{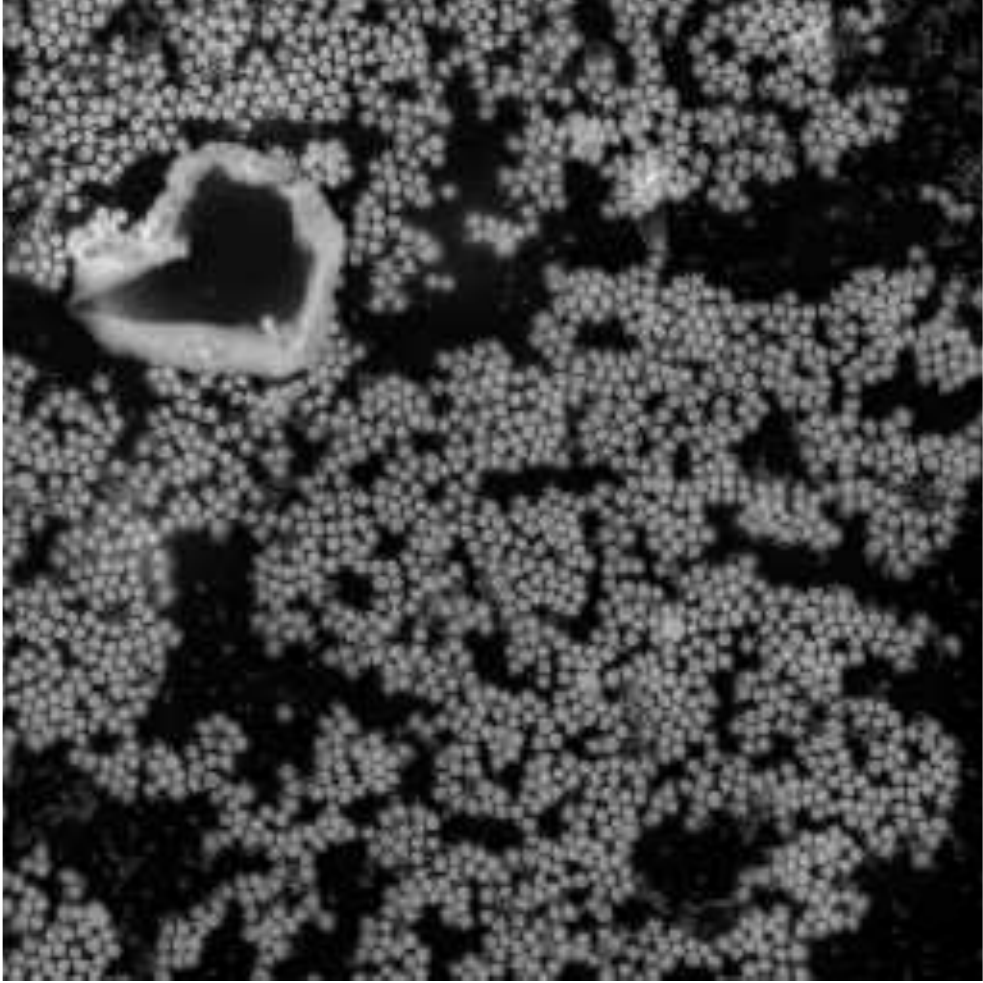}}\quad
\subfigure[]{\includegraphics[width=.11\textwidth]{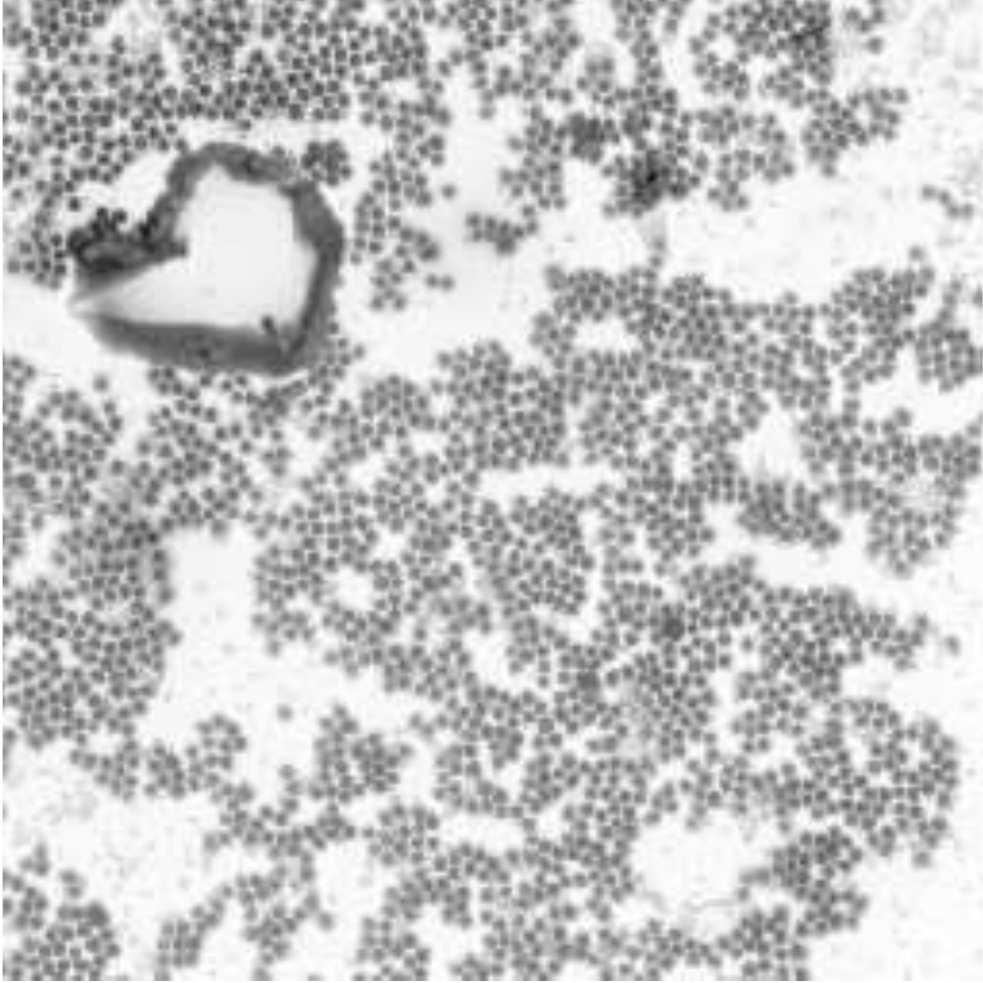}}\quad
\subfigure[]{\includegraphics[width=.11\textwidth]{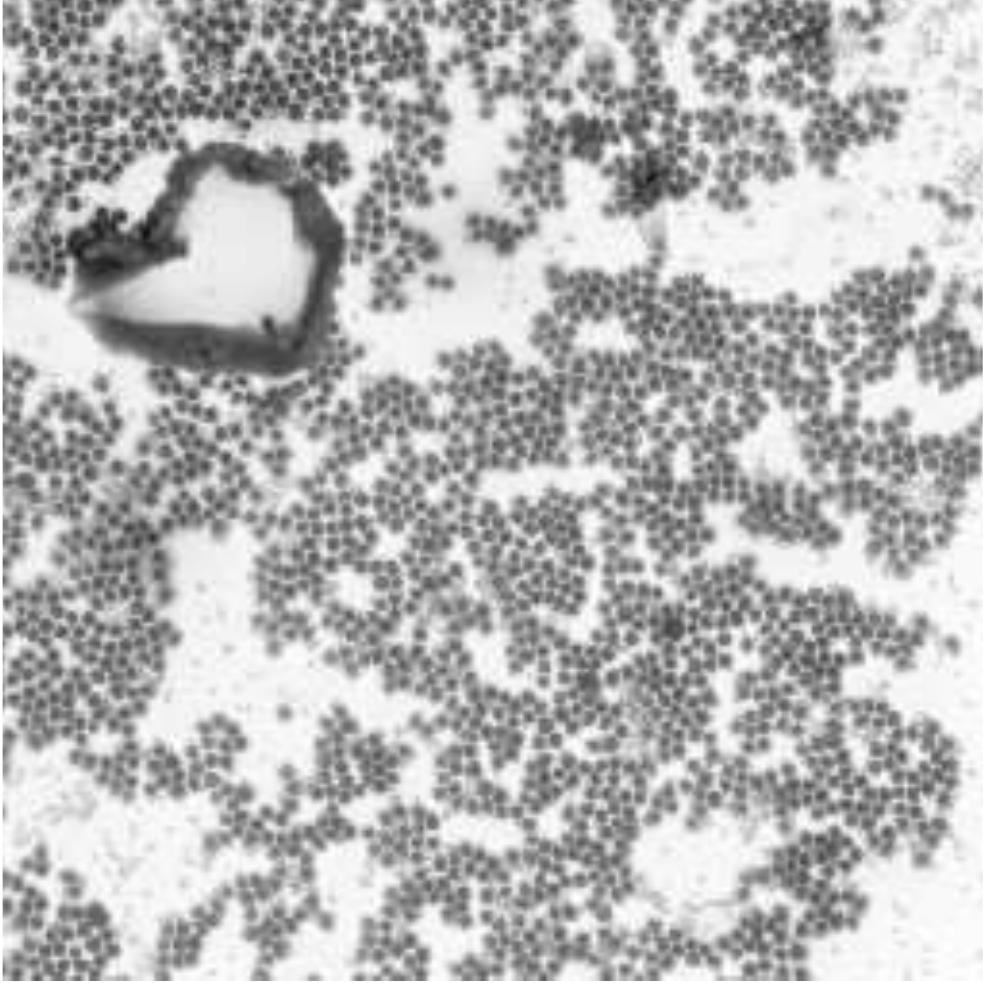}}
\end{center}
\caption{Ground truth images. Real-valued images ($512\times512$ pixels) ``Lena'' in (a), ``Barbara'' in (b), ``Peppers'' in (c), ``Plane'' in (d),   and complex valued image ($256\times 256$ pixels) with magnitude in (e), real and imaginary parts in (f) and (g) respectively. }\label{groundtruth}
\end{figure}

The Poisson noise is added to the clean image with different peak levels, and different scale image $u_\nu$ are used as $u_\nu=\nu u$  with the peak level $\nu$ for ground truth $u$. We measure the quality of the reconstructed the image $\tilde u$ by
\[
\mathrm{SNR}(\tilde u,u)=-20\log(\|c^*\tilde u-u\|/\|c^*\tilde  u\|),
\]
with the ground truth image $u$, and a scale constant
as
\[
c^*=\max\limits_{|c|=1} \|c\tilde u-u\|
.\]
 For Gaussian noise, SNR is also used to measure the noise level for the contaminated data.

There are many choices for PR without regularization, such as error reduction (ER)\cite{Gerchberg1972},  ADMM \cite{wen2012,chang2015,chang2016}, Wirtinger flow \cite{candes2015phase,chen2015solving} and so on, while we only use the ADMM  for comparison. For the methods with regularization, we compare with  TV based method in \cite{chang2016} for Poisson noise, and a variant method by inserting \eqref{subsubZ} for Gaussian noise. For our proposed methods, the second order TGV realization of $\mathbf R(v)$ and two filters including NLM, and BM3D are used. For simplicity, we use ``LS-PR'' to denotes the ADMM  for solving the original PR problem without any regularization, ``TV-PR'' for TV based methods in \cite{chang2016}, ``TGV-PR'' for our proposed method with TGV regularization, ``NLM-PR'' for our proposed method with NLM filtering, and ``BM3D-PR'' with BM3D filtering. It is quite difficult to set a unify fair stopping condition for such different algorithms, and the algorithms are all stopped if they reach a maximum iteration number $T$, where $T$ is select heuristically. For ``LS-PR'' and ``TV-PR'', set $T=50$, while for ``TGV-PR'', ``NLM-PR'' and ``BM3D-PR'', set $T=30$ (the iterations number are more than real number needed, and usually $10\thicksim15$ iterations are sufficient to produce satisfactory results). In the inner loop for solving the generalized least square problem, five inner iterations are adopted.  The other parameters  used in the experiments will be addressed in the following subsections.

All the tests are performed on a laptop with Intel i7-5600U\@2.6GHZ, and 16GB RAM. The implementation of the codes is in MATLAB. For our proposed methods, we solve the denoising subproblem with TGV with the package by S. Keiichiro {\footnote{\url{https://www.mathworks.com/matlabcentral/mlc-downloads/downloads/submissions/49717/versions/2/download/zip}}}, NLM filter  by J.V. Manj\'on  {\footnote{\url{https://www.mathworks.com/matlabcentral/mlc-downloads/downloads/submissions/13176/versions/1/download/zip}}} and BM3D by K. Dabov {\footnote{\url{http://www.cs.tut.fi/~foi/GCF-BM3D/BM3D.zip}}}.

 \subsection{Poisson noise removal}
\subsubsection{Coded diffraction pattern (CDP)}
For CDP, we first show the performance on real-valued  images in (a)-(d) of Figure \ref{groundtruth}  with different noise levels by setting the peak level $\nu\in\{3.0\times10^{-3},5.0\times 10^{-3},1.0\times 10^{-2}\}$, and see  the results in Figure \ref{poicdp3}-Figure \ref{poicdp1}. With a fixed noise level  the parameters of each compared methods are set to the same for each image, and see the detailed parameters in Table \ref{tab1}.

By observing the reconstructed images in  Figure \ref{poicdp3}-Figure \ref{poicdp1}, if the noise is not so heavy, all the methods work well as shown in Figure \ref{poicdp3}.  When noise level increases as in Figure \ref{poicdp2}, results by ``LS-PR'' show obvious noise, while all the regularized and filtering methods can work. If the noise level is very high as in Figure \ref{poicdp1}, the reconstructed results by ``LS-PR'' can not be acceptable, while ``NLM-PR'' and ``BM3D-PR'' outperforms the others. ``TV-PR'' with simple regularization by TV is effective to generate the noiseless images for different noise level, and meanwhile it introduces the staircase artifacts and break  some important information as texture; ``TGV-PR'' can suppress the staircase artifacts, and produce  resulted images with sharp edges. However, the repetitive structures can not be kept. It seems that ``NLM-PR'' and ``BM3D-PR'' can both deal with the texture parts well, e.g. the hair of ``Lena'' and the stripe structures of ``Barbara'', which outperform other compared methods. Meanwhile, it seems that ``BM3D-PR'' produces the cleanest background for ``Plane'' among all the compared methods.  The SNRs  are put in Figure \ref{snr2}, where one can readily see that our proposed ``TGV-PR'', ``NLM-PR'' and ``BM3D-PR'' outperform ``TV-PR'', and ``BM3D-PR'' gains the largest SNRs among them.  The average SNRs are $6.10, 20.32, 20.87, 21.98, 22.94$ dB for ``LS-PR'', ``TV-PR'', ``TGV-PR'', ``NLM-PR'' and ``BM3D-PR'' respectively, and ``BM3D-PR'' improves the image quality  with SNRs about 2dB higher  than that of ``TV-PR'' on average.
\begin{figure}
\begin{center}
\subfigure[]{\includegraphics[width=.08\textwidth]{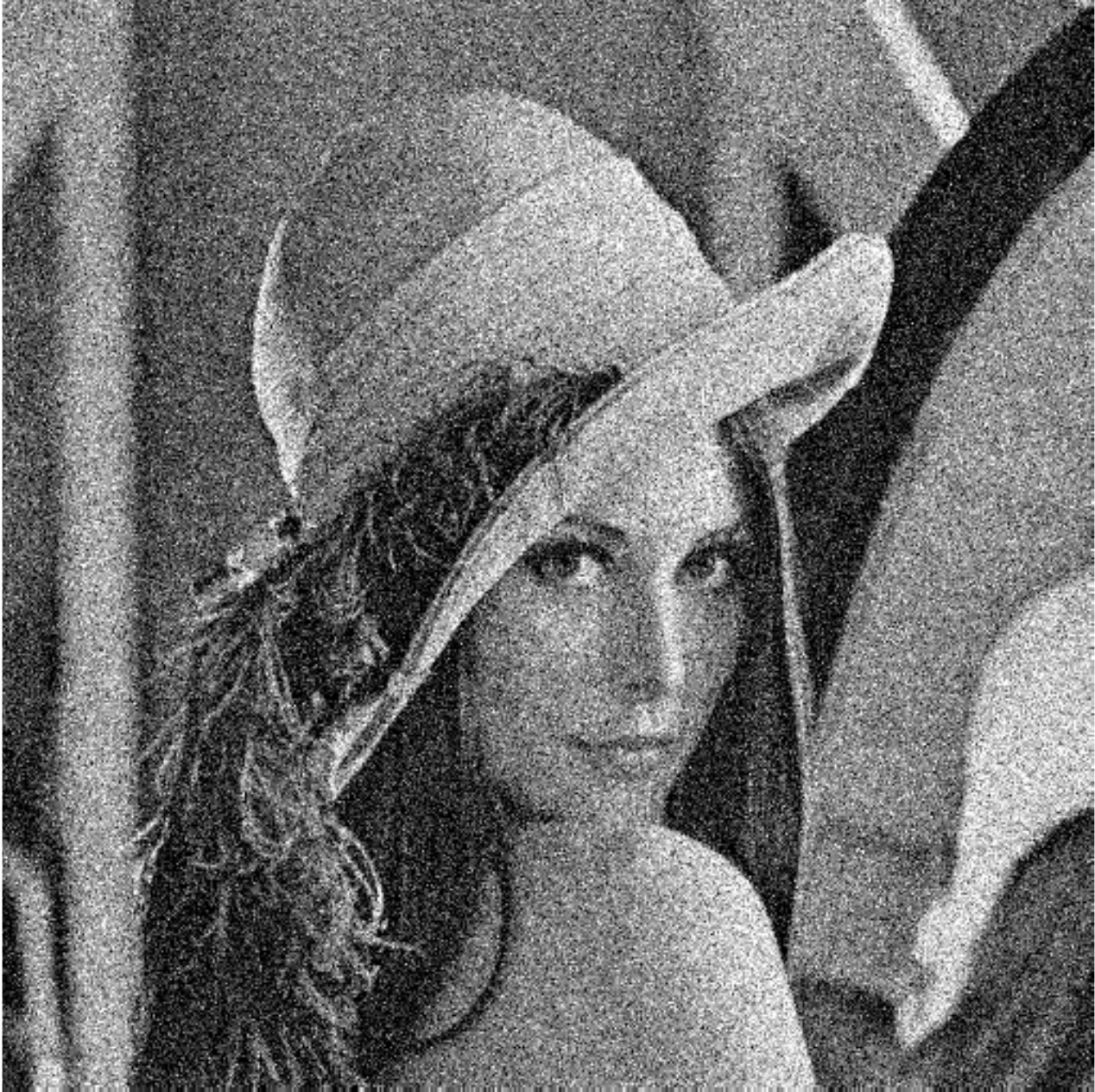}}~~
\subfigure[]{\includegraphics[width=.08\textwidth]{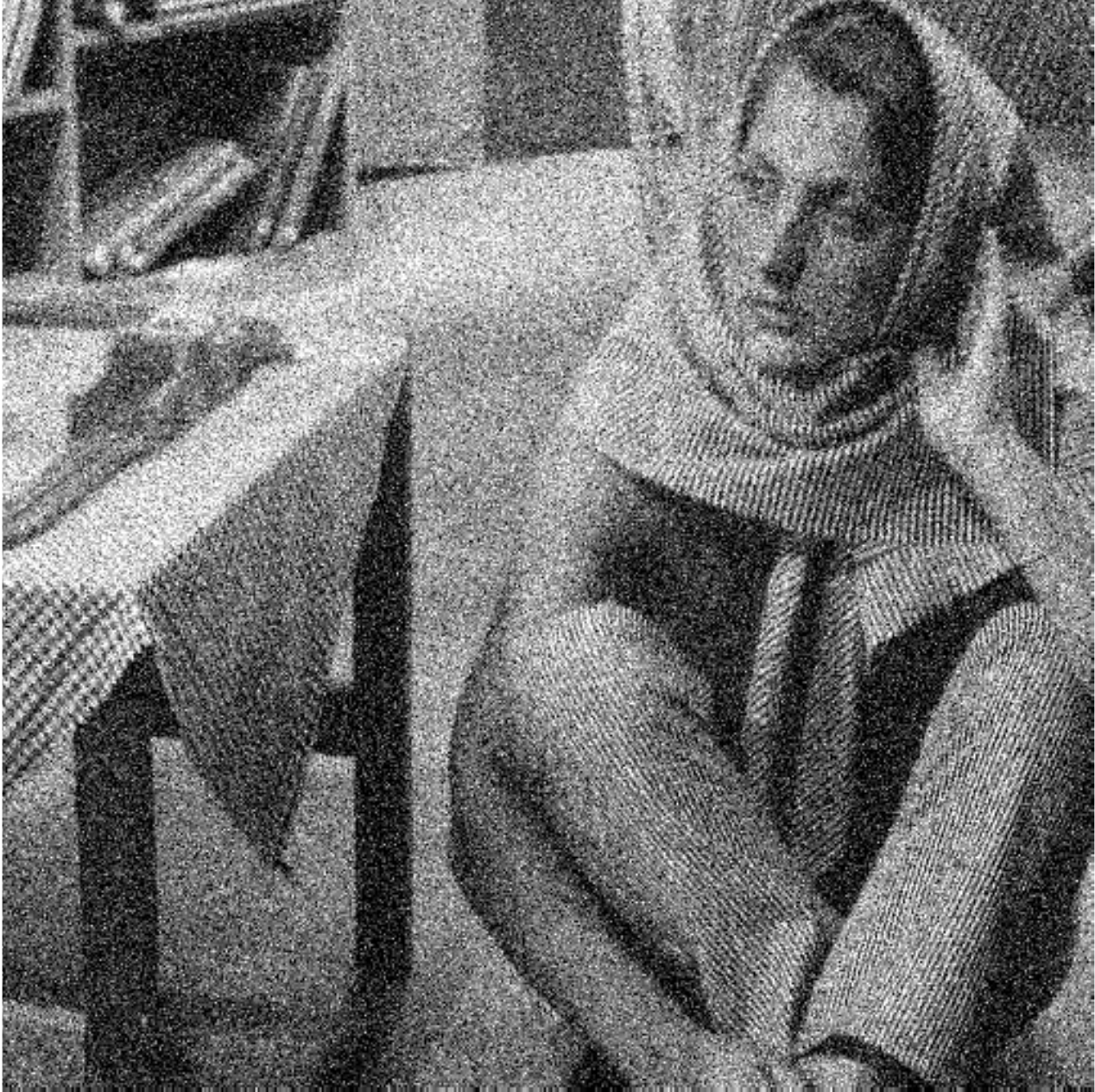}}~~
\subfigure[]{\includegraphics[width=.08\textwidth]{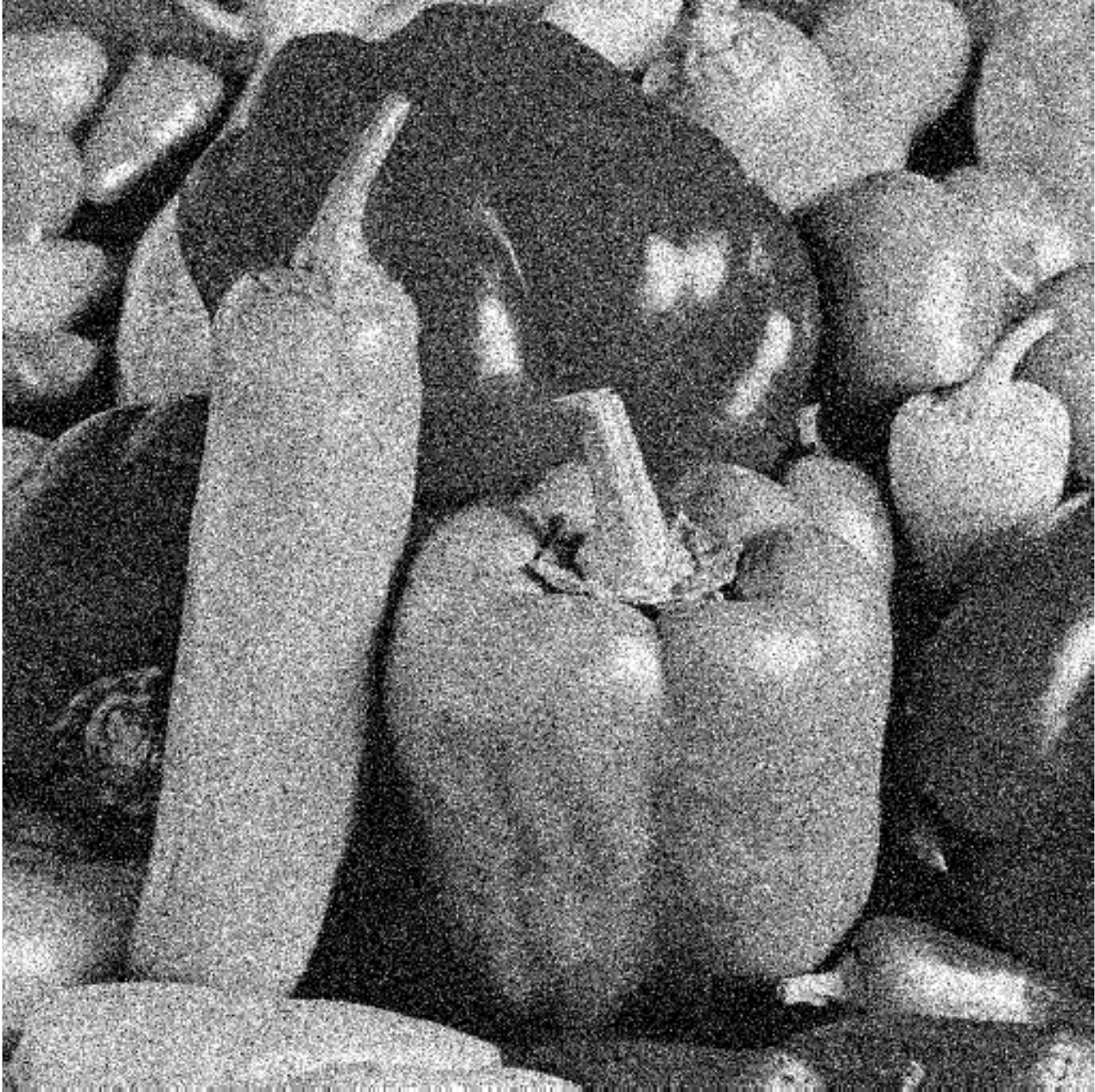}}~~
\subfigure[]{\includegraphics[width=.08\textwidth]{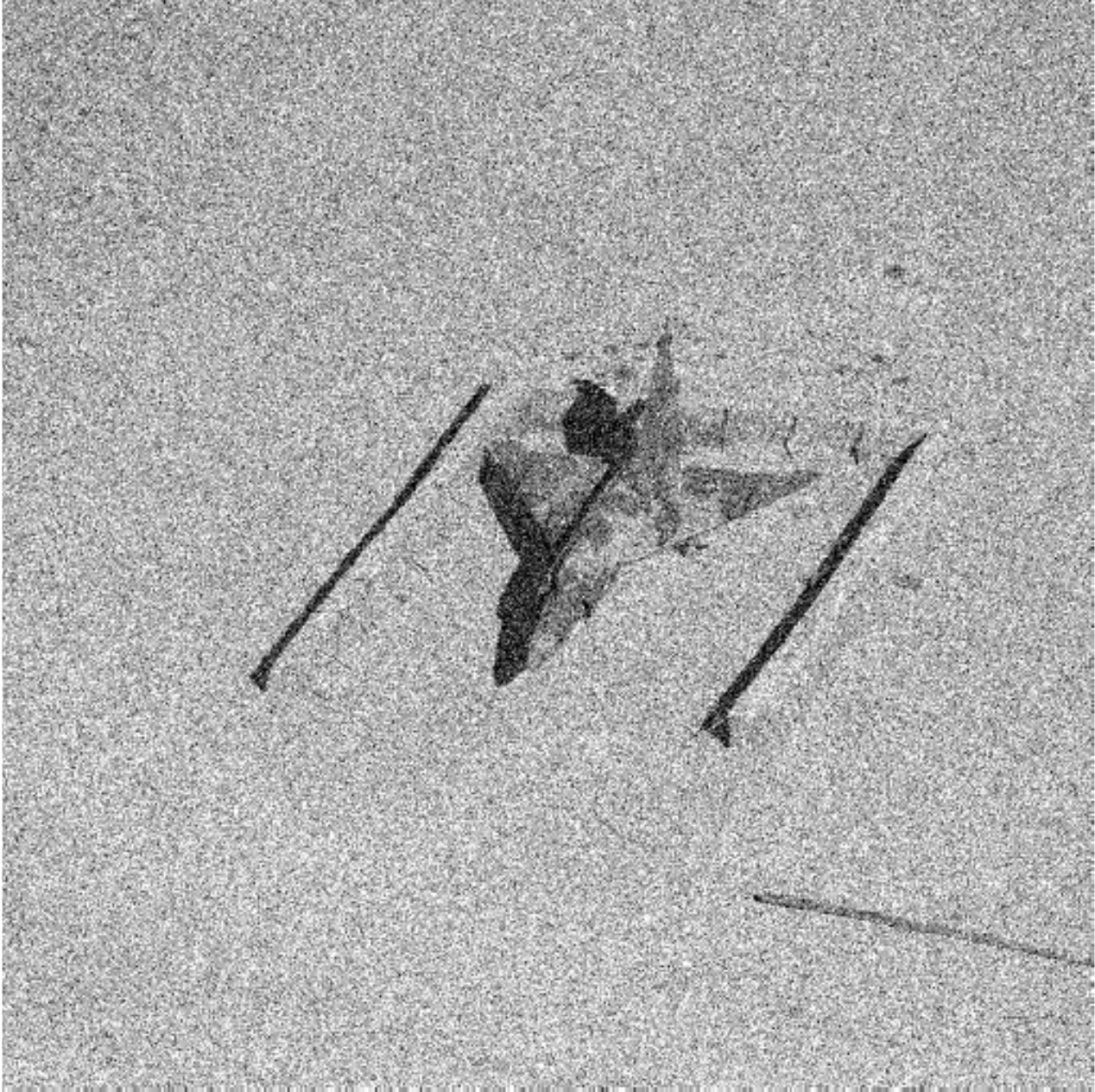}}\\
\subfigure[]{\includegraphics[width=.08\textwidth]{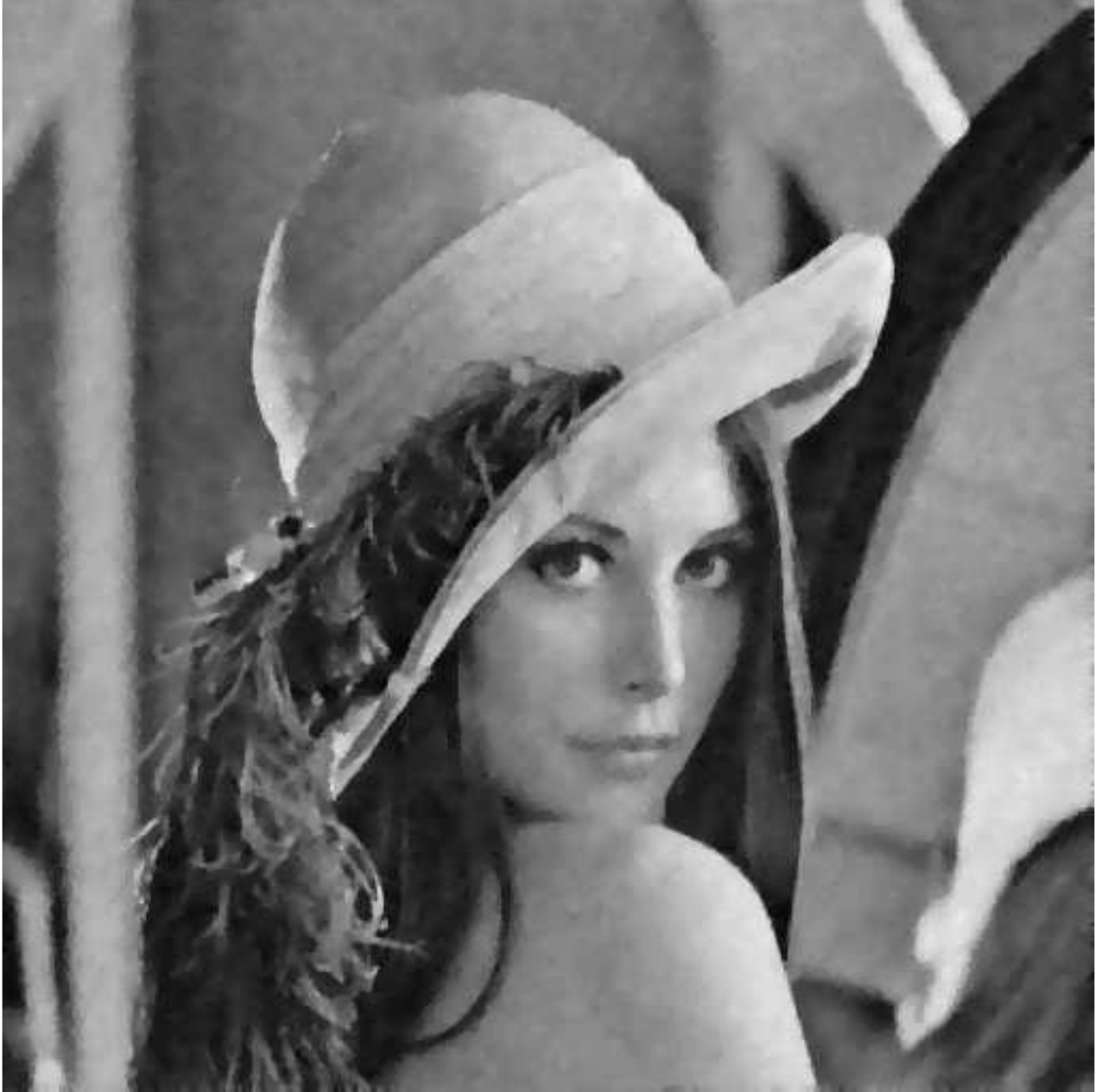}}~~
\subfigure[]{\includegraphics[width=.08\textwidth]{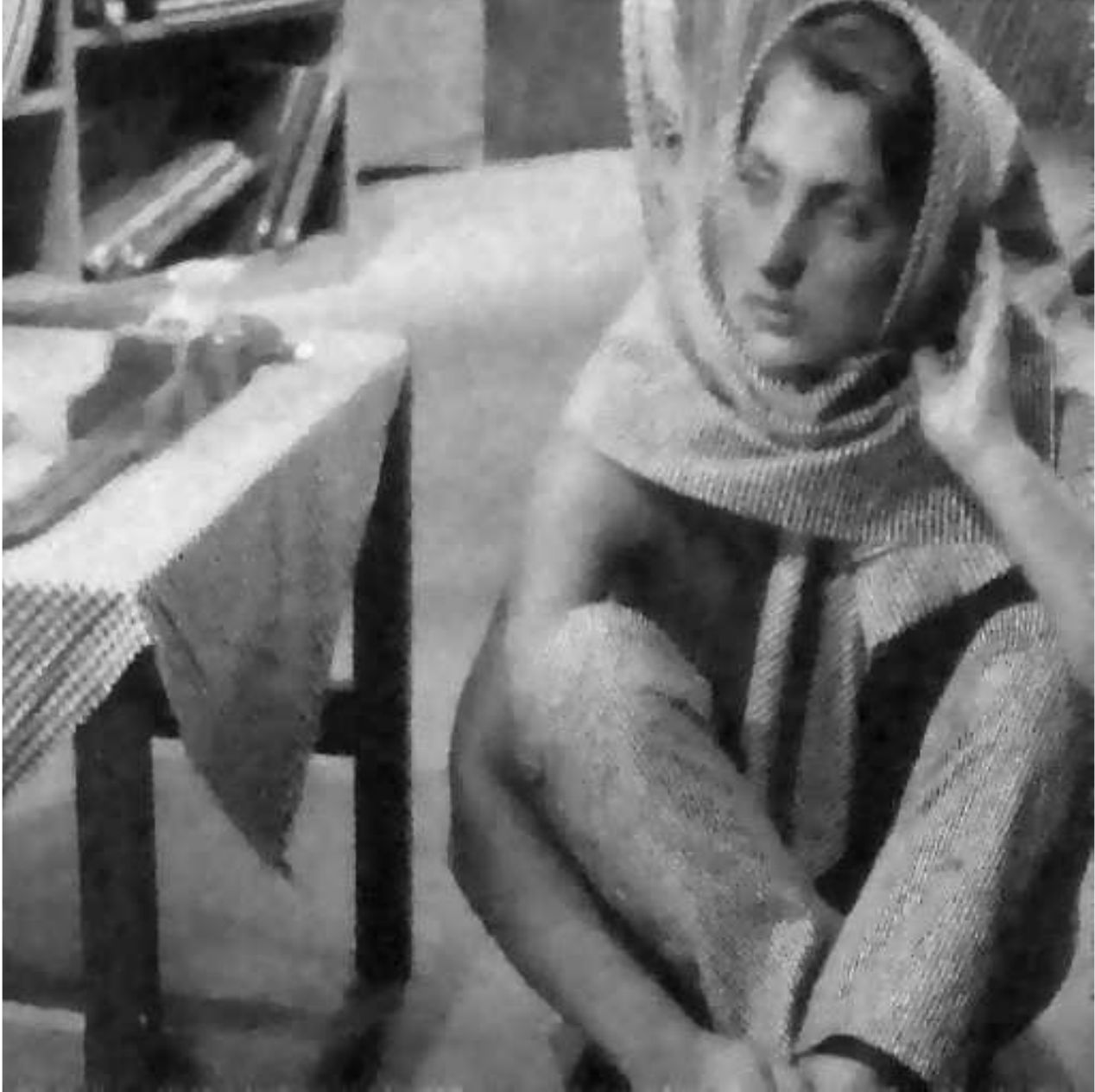}}~~
\subfigure[]{\includegraphics[width=.08\textwidth]{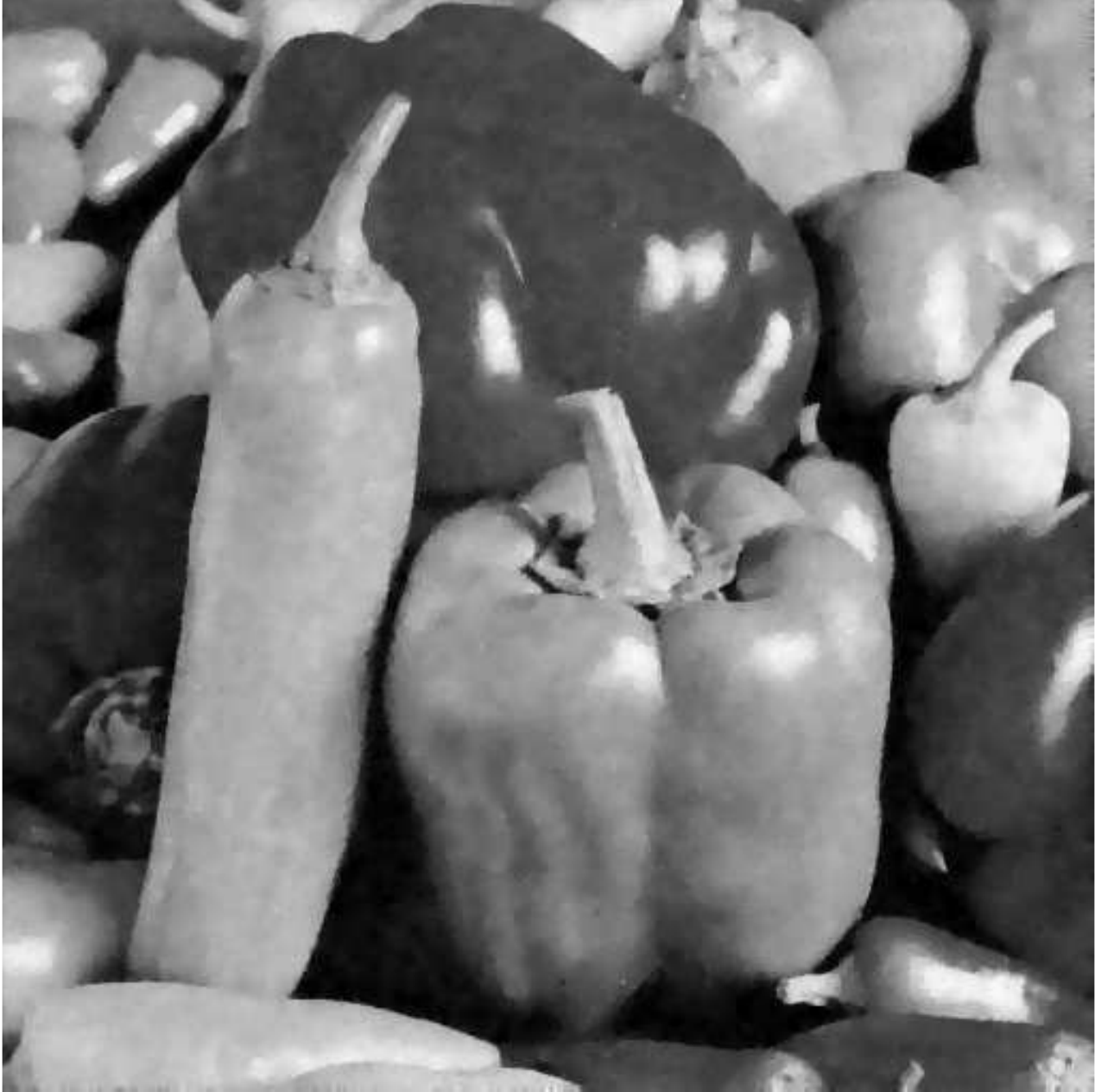}}~~
\subfigure[]{\includegraphics[width=.08\textwidth]{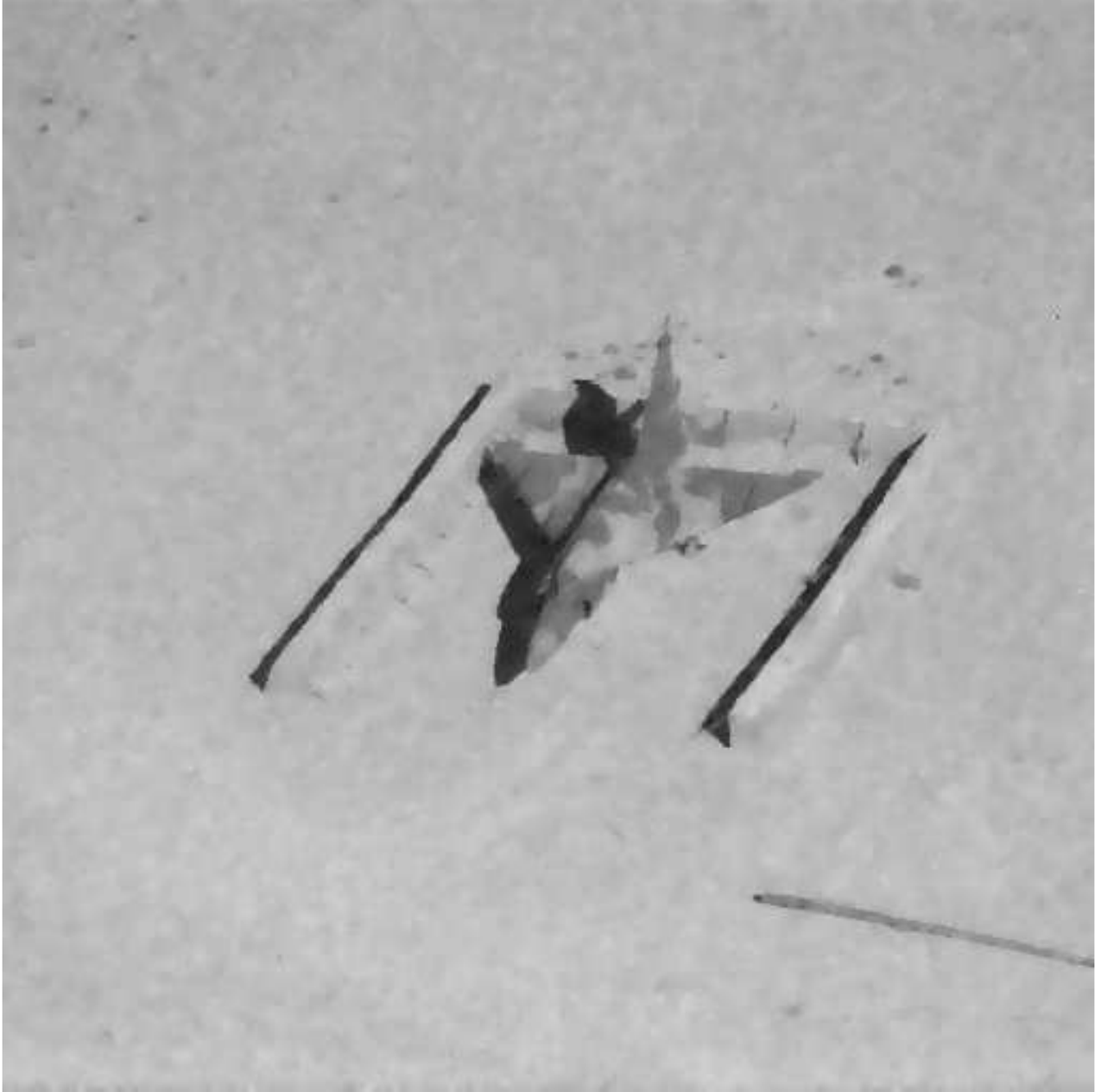}}\\
\subfigure[]{\includegraphics[width=.08\textwidth]{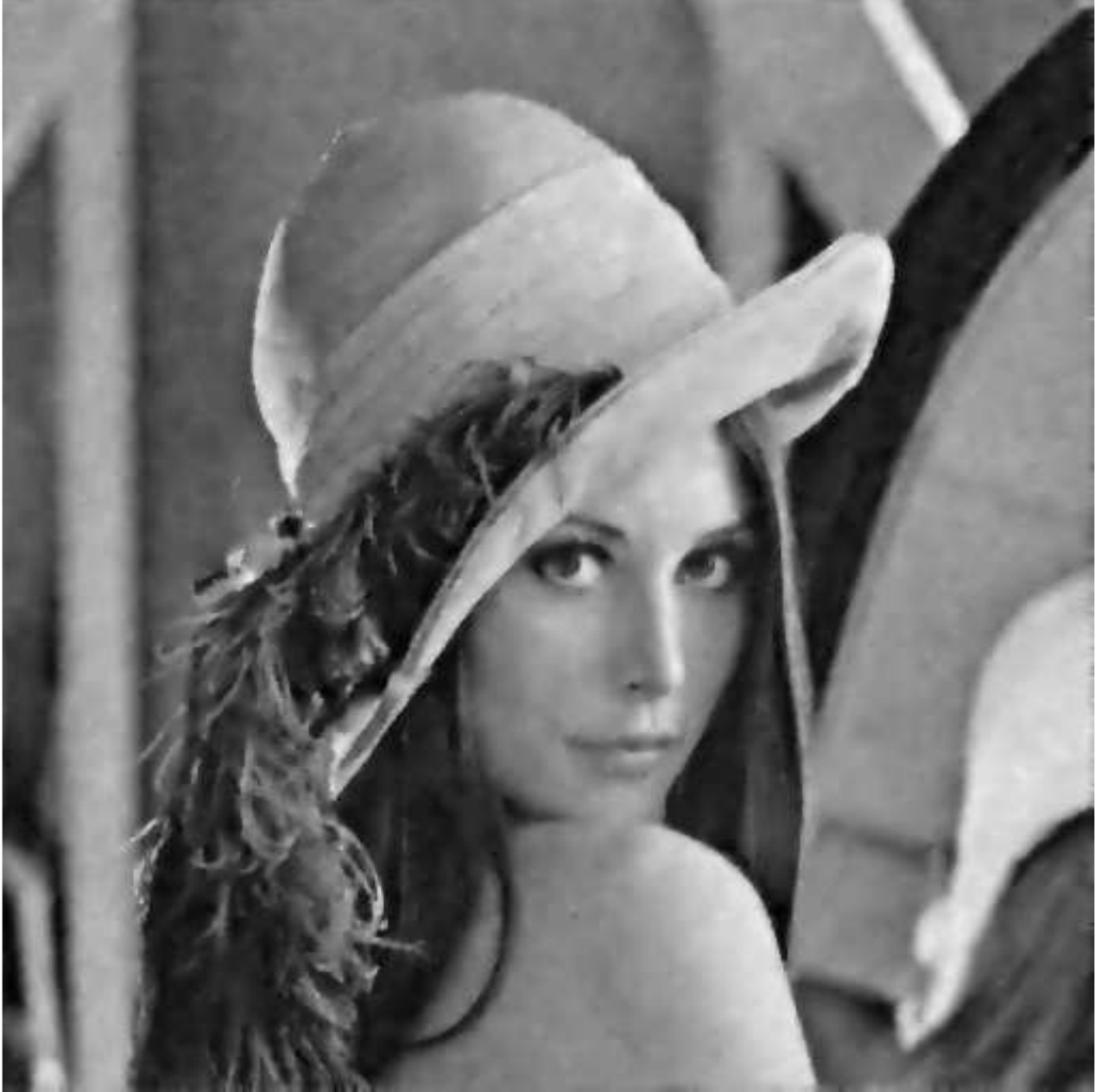}}~~
\subfigure[]{\includegraphics[width=.08\textwidth]{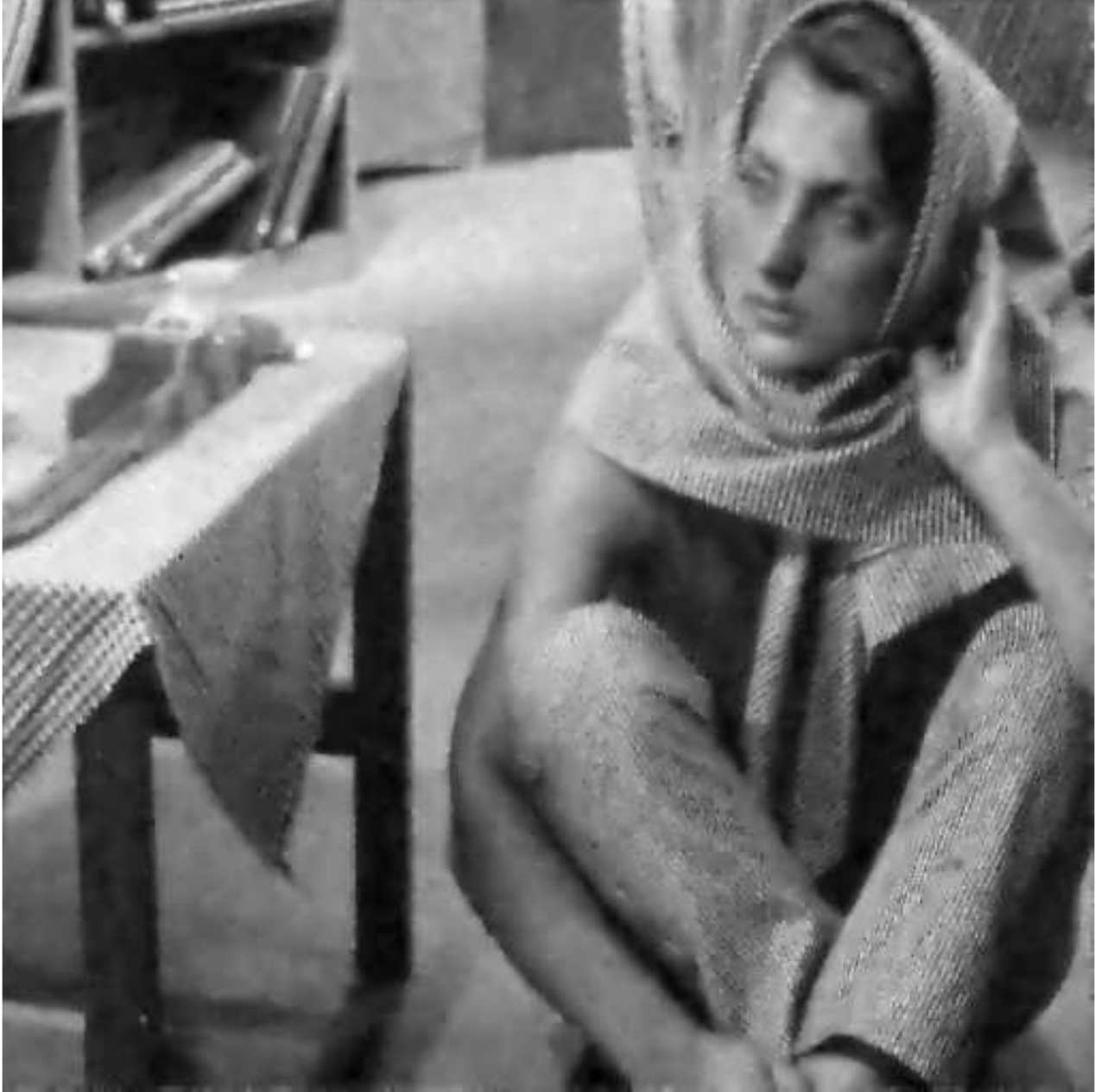}}~~
\subfigure[]{\includegraphics[width=.08\textwidth]{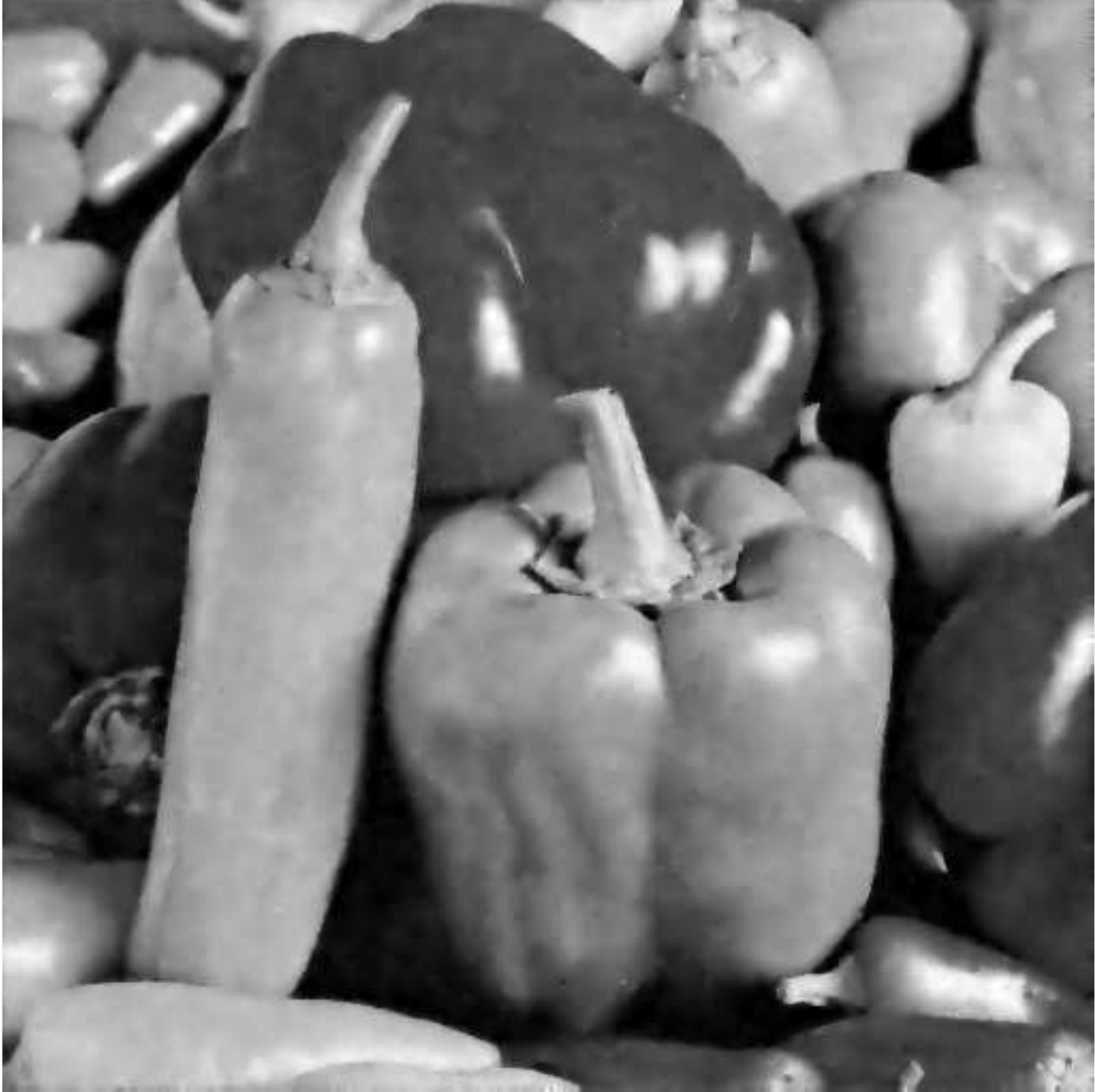}}~~
\subfigure[]{\includegraphics[width=.08\textwidth]{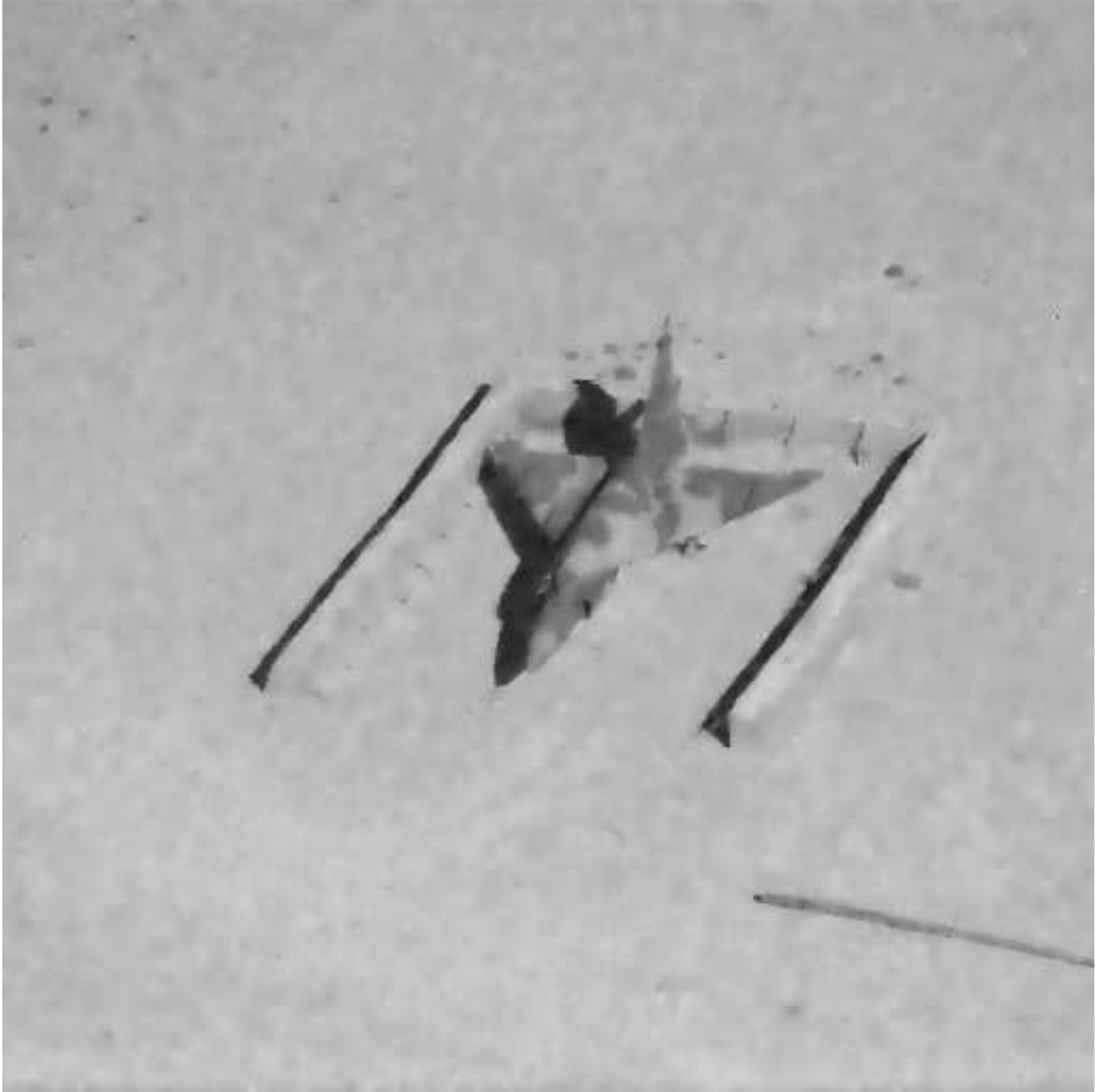}}\\
\subfigure[]{\includegraphics[width=.08\textwidth]{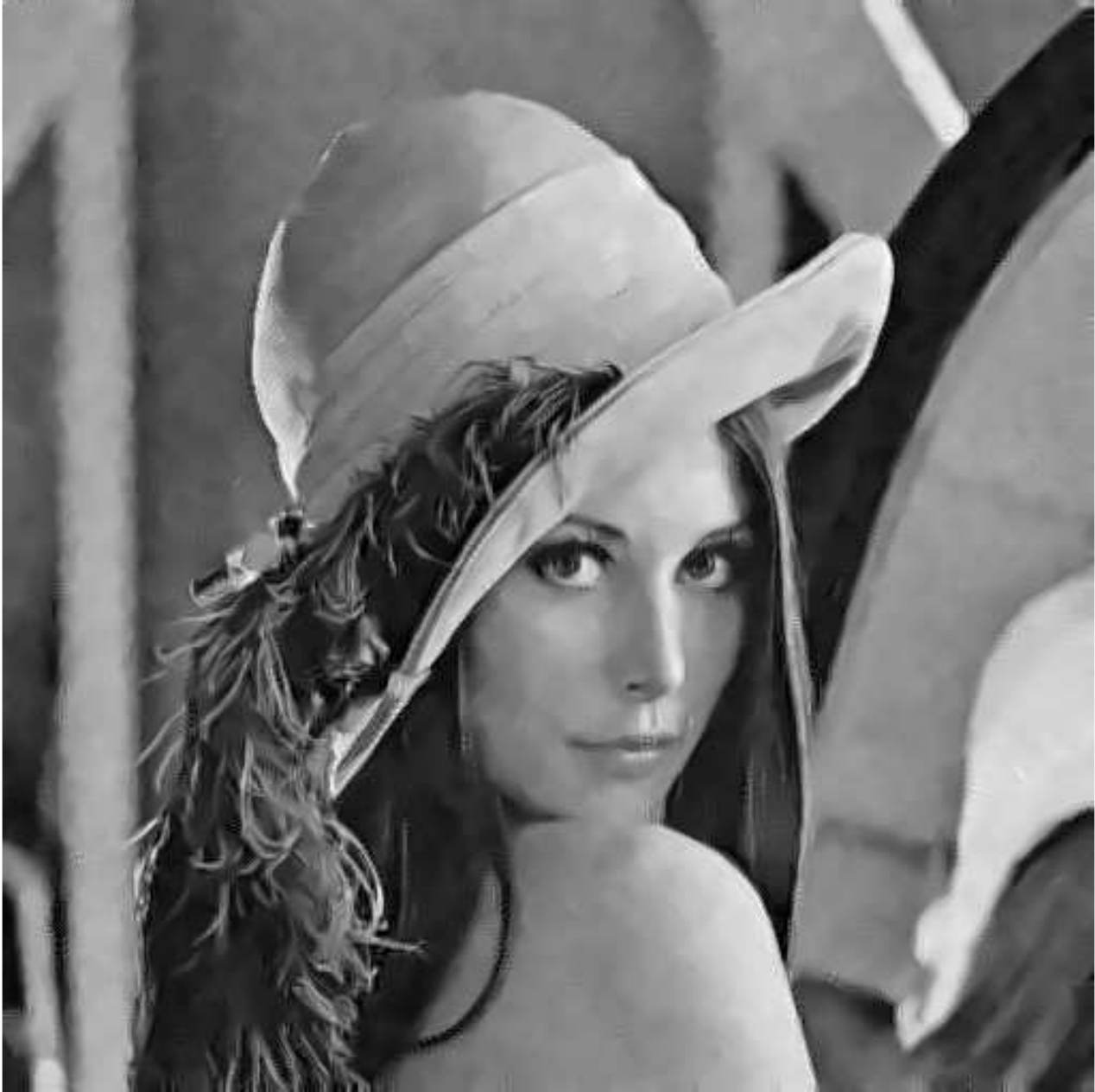}}~~
\subfigure[]{\includegraphics[width=.08\textwidth]{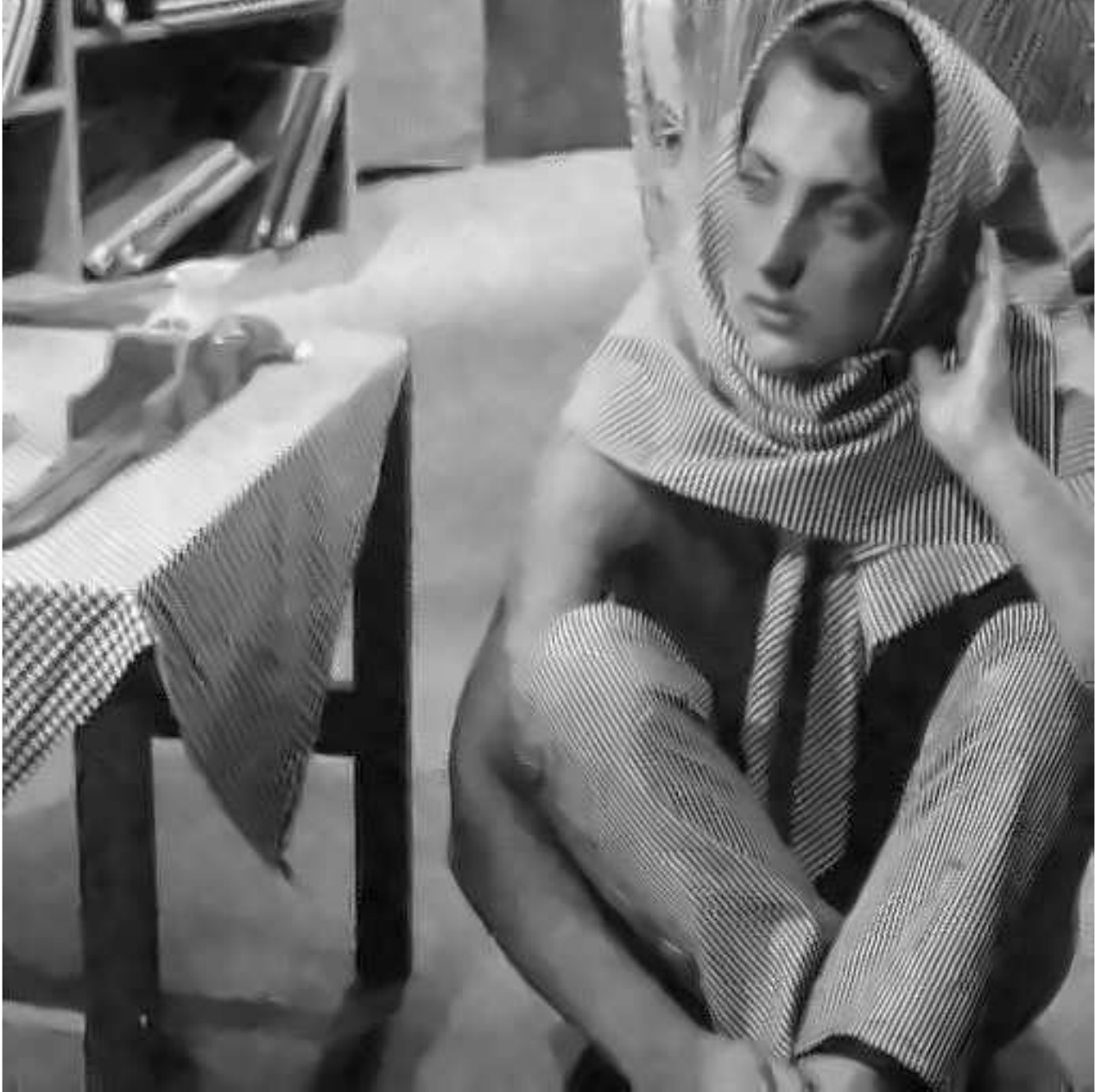}}~~
\subfigure[]{\includegraphics[width=.08\textwidth]{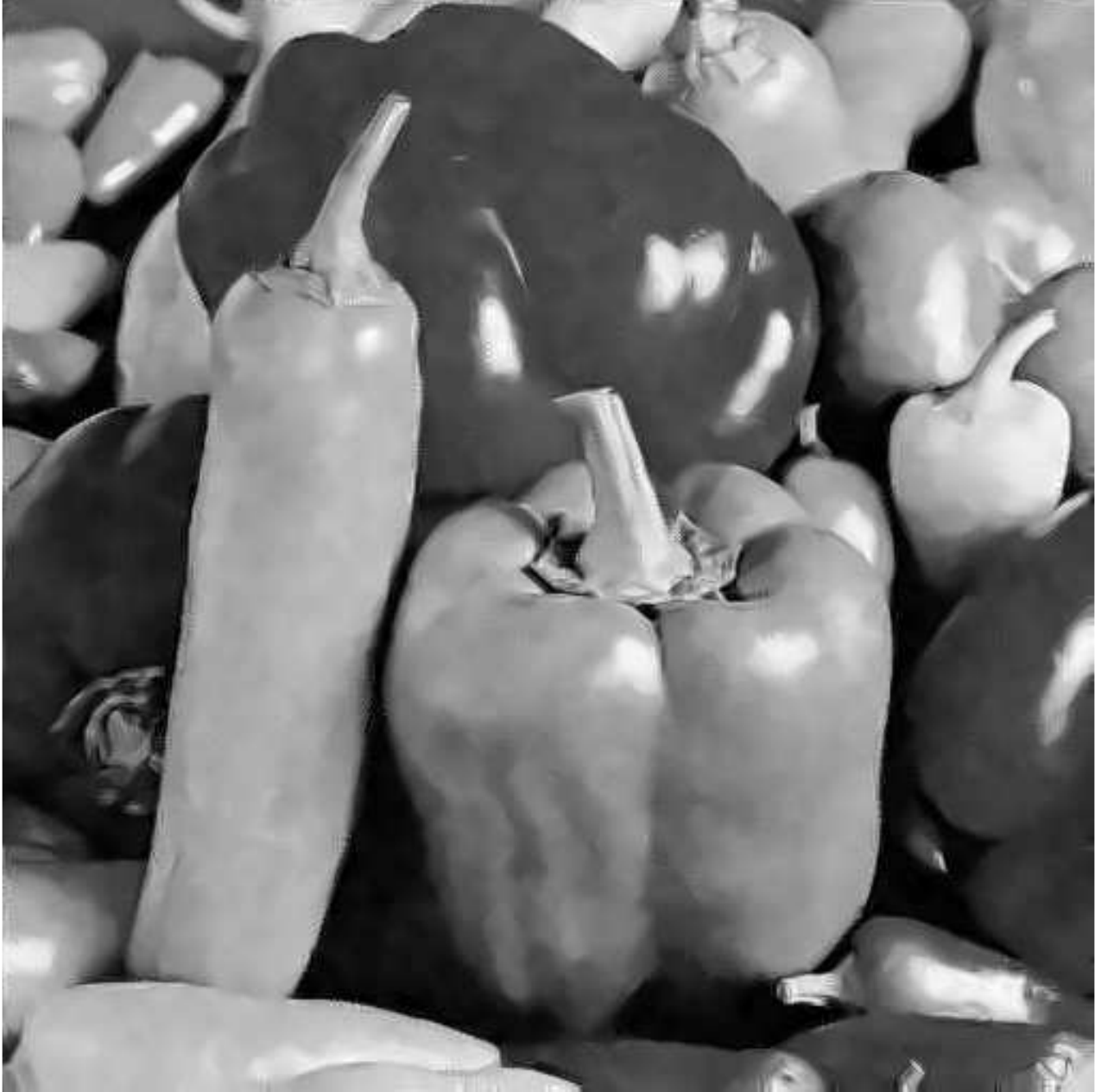}}~~
\subfigure[]{\includegraphics[width=.08\textwidth]{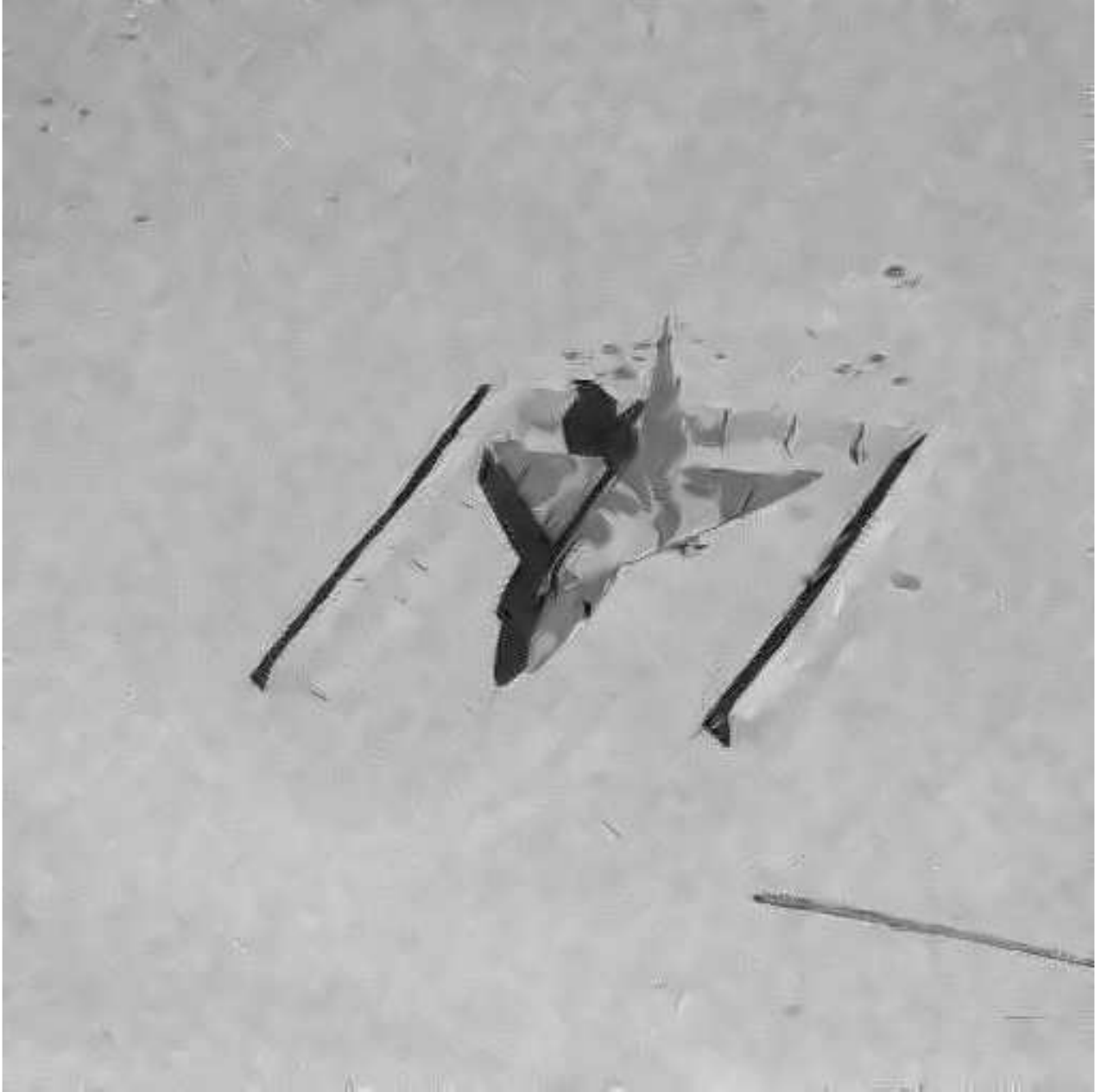}}\\
\subfigure[]{\includegraphics[width=.08\textwidth]{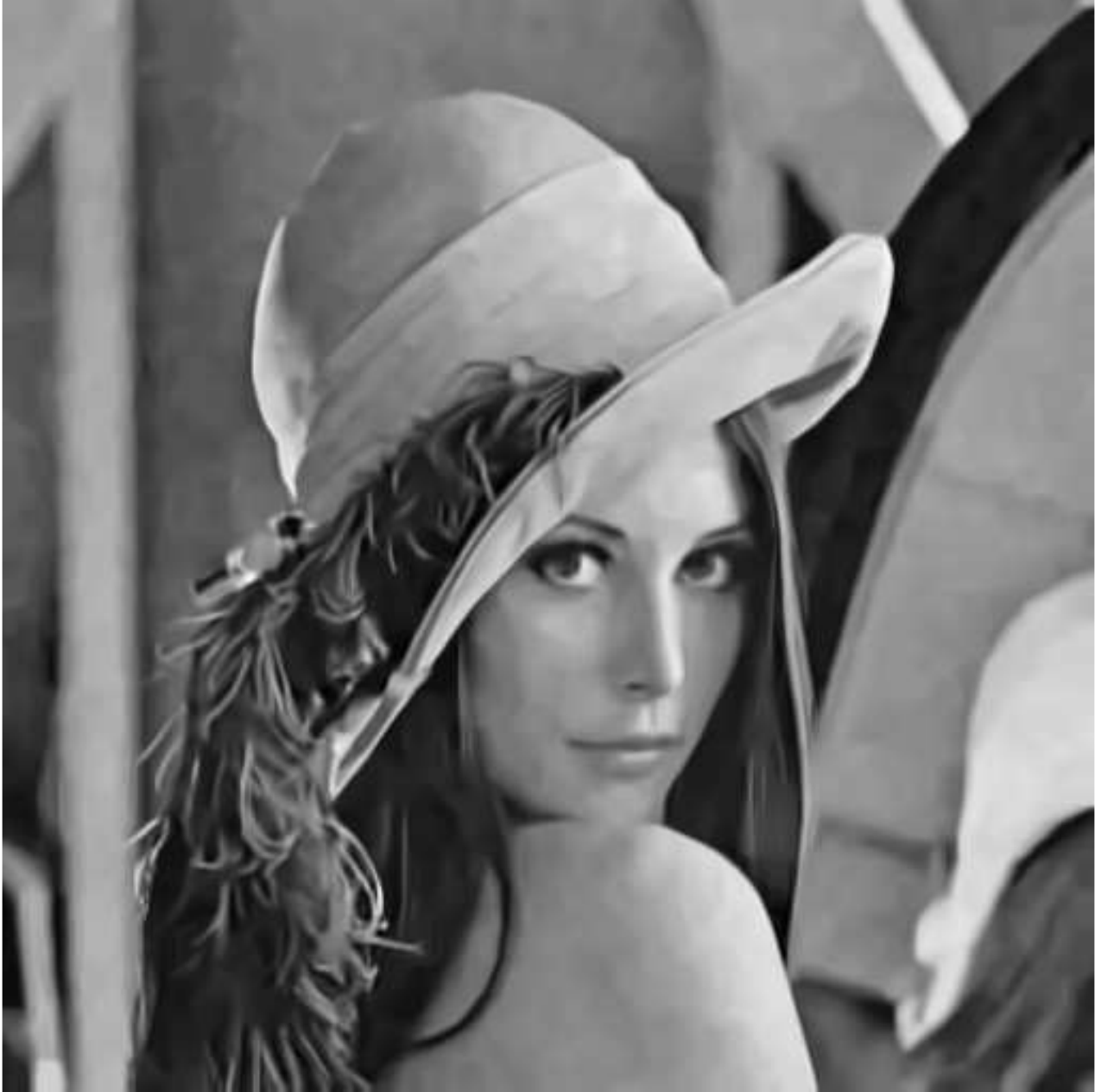}}~~
\subfigure[]{\includegraphics[width=.08\textwidth]{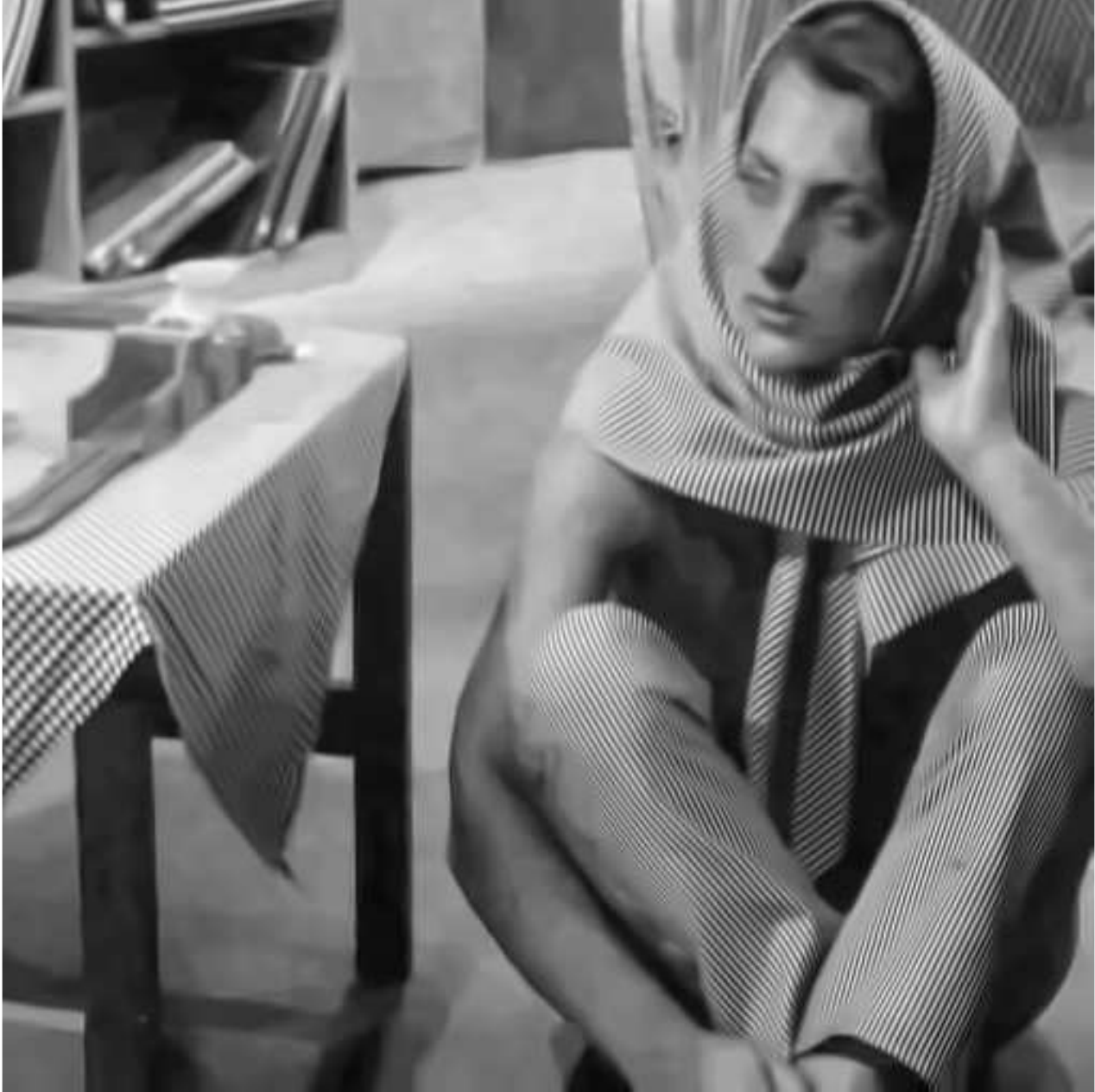}}~~
\subfigure[]{\includegraphics[width=.08\textwidth]{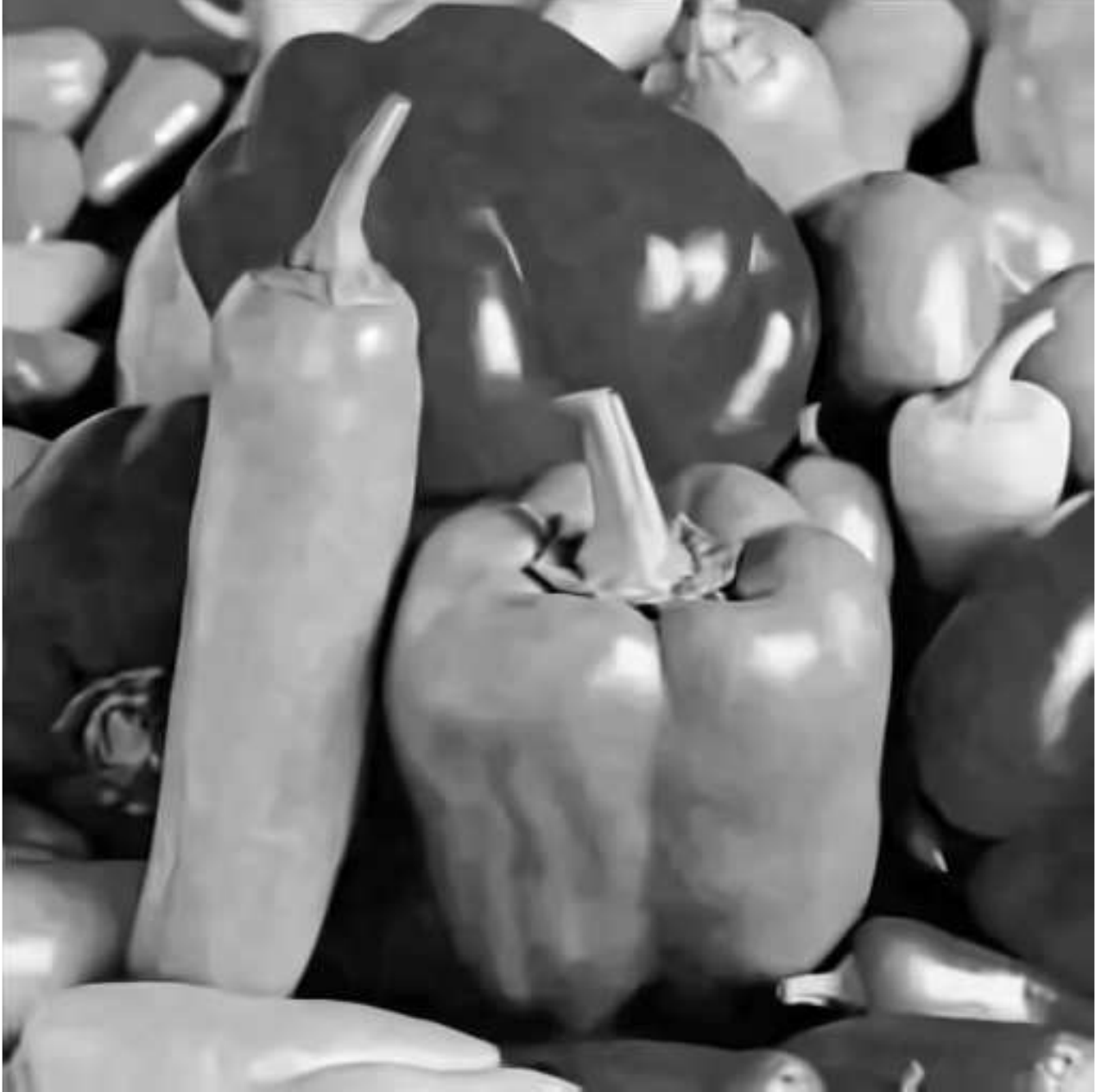}}~~
\subfigure[]{\includegraphics[width=.08\textwidth]{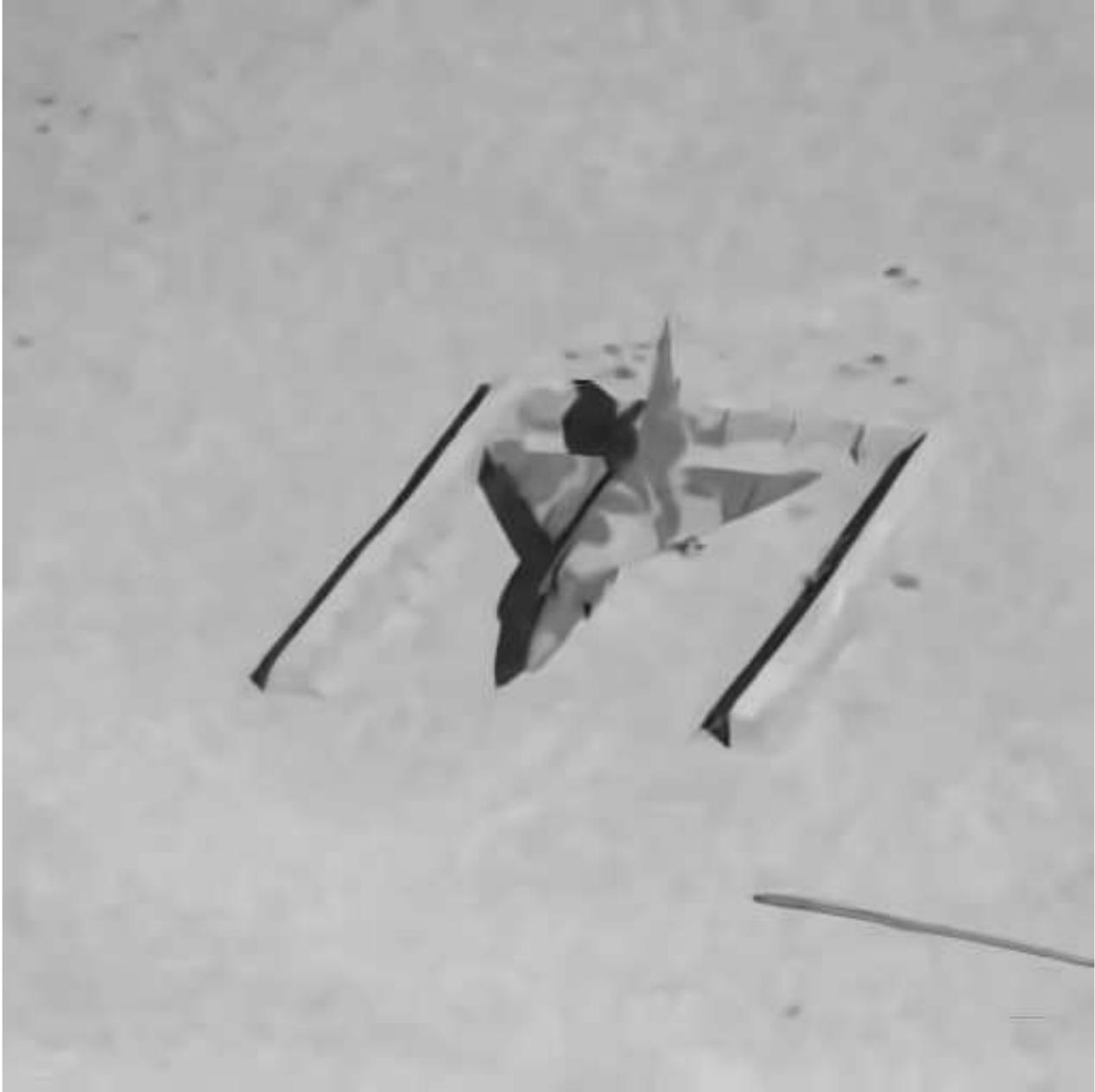}}
\end{center}
\caption{PR with CDP. Peak level $\nu=1.0\times10^{-2}$ for Poisson noise. First row: ``LS-PR''; Second row: ``TV-PR''; Third row: ``TGV-PR''; Fourth row: ``NLM-PR''; Fifth row: ``BM3D-PR''.}
\label{poicdp3}
\end{figure}

\begin{figure}
\begin{center}
\subfigure[]{\includegraphics[width=.08\textwidth]{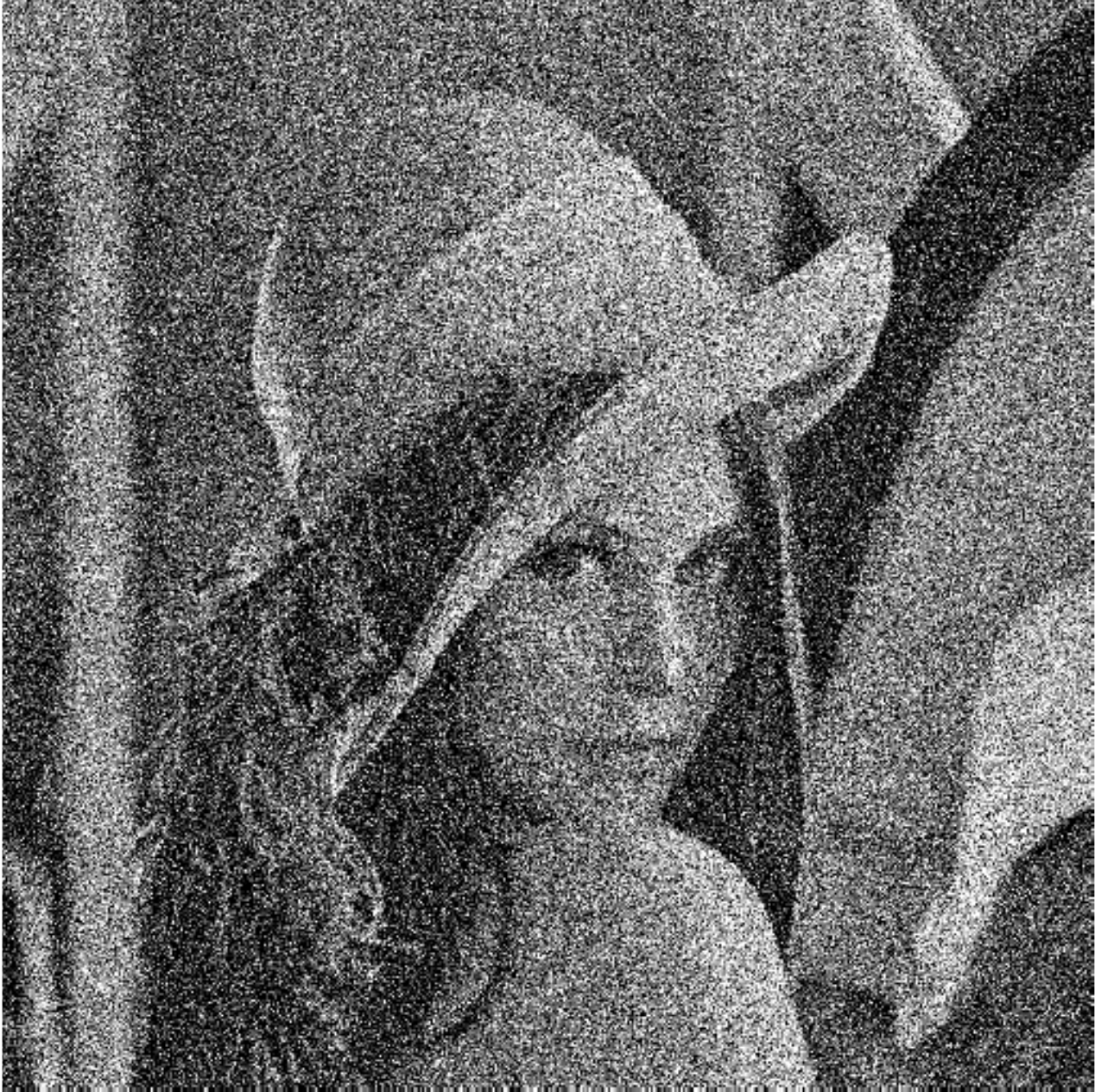}}~~
\subfigure[]{\includegraphics[width=.08\textwidth]{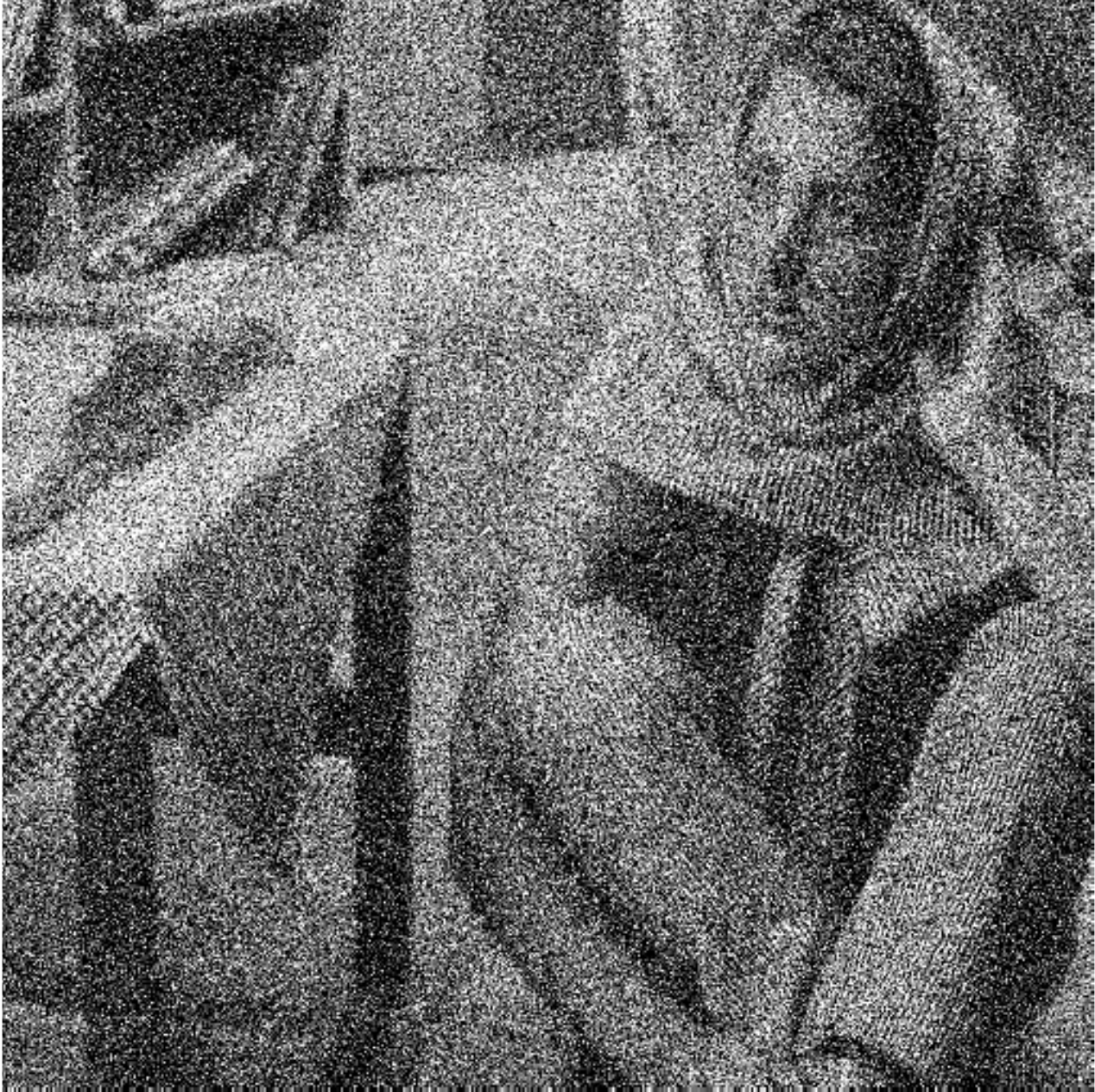}}~~
\subfigure[]{\includegraphics[width=.08\textwidth]{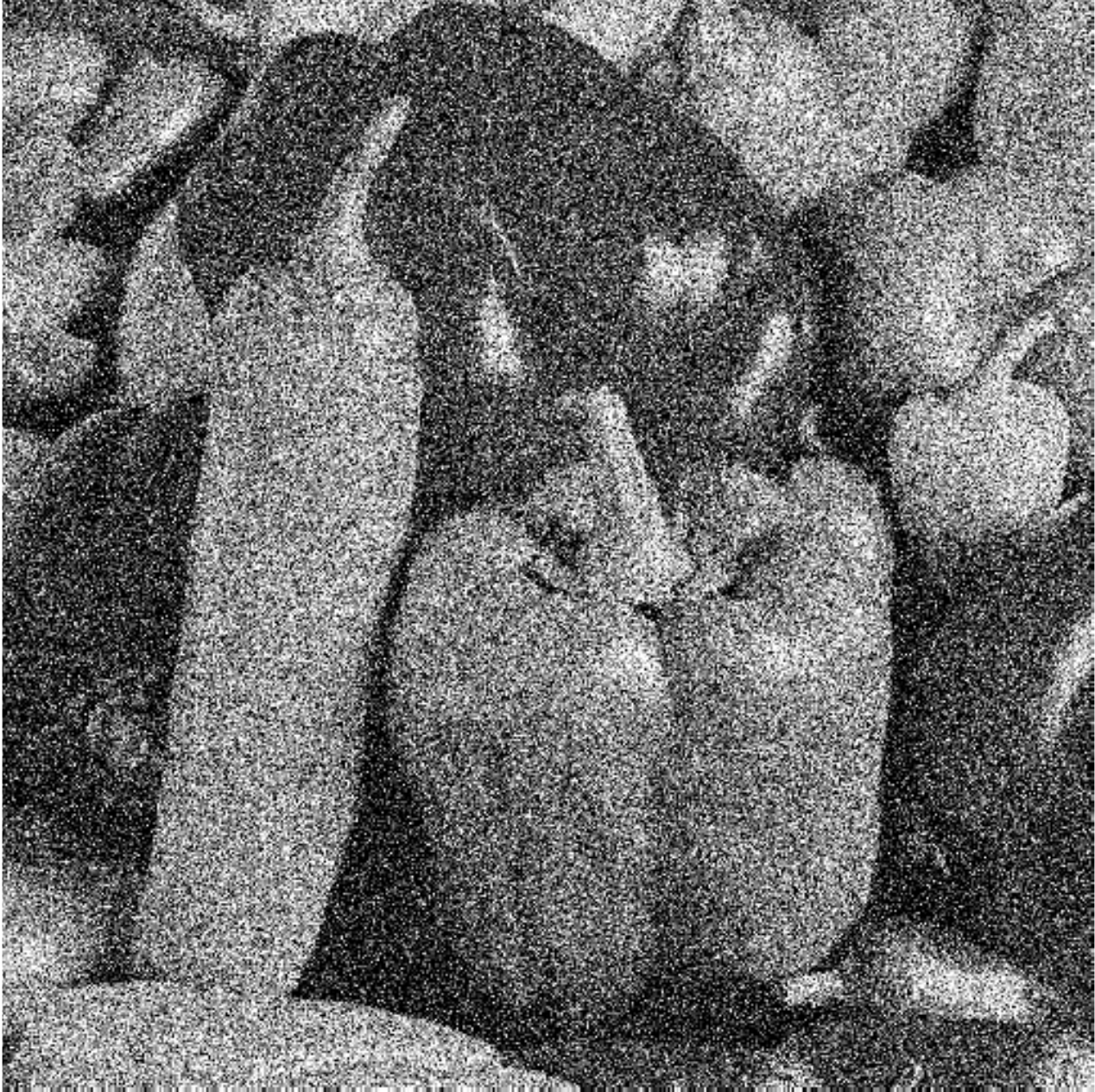}}~~
\subfigure[]{\includegraphics[width=.08\textwidth]{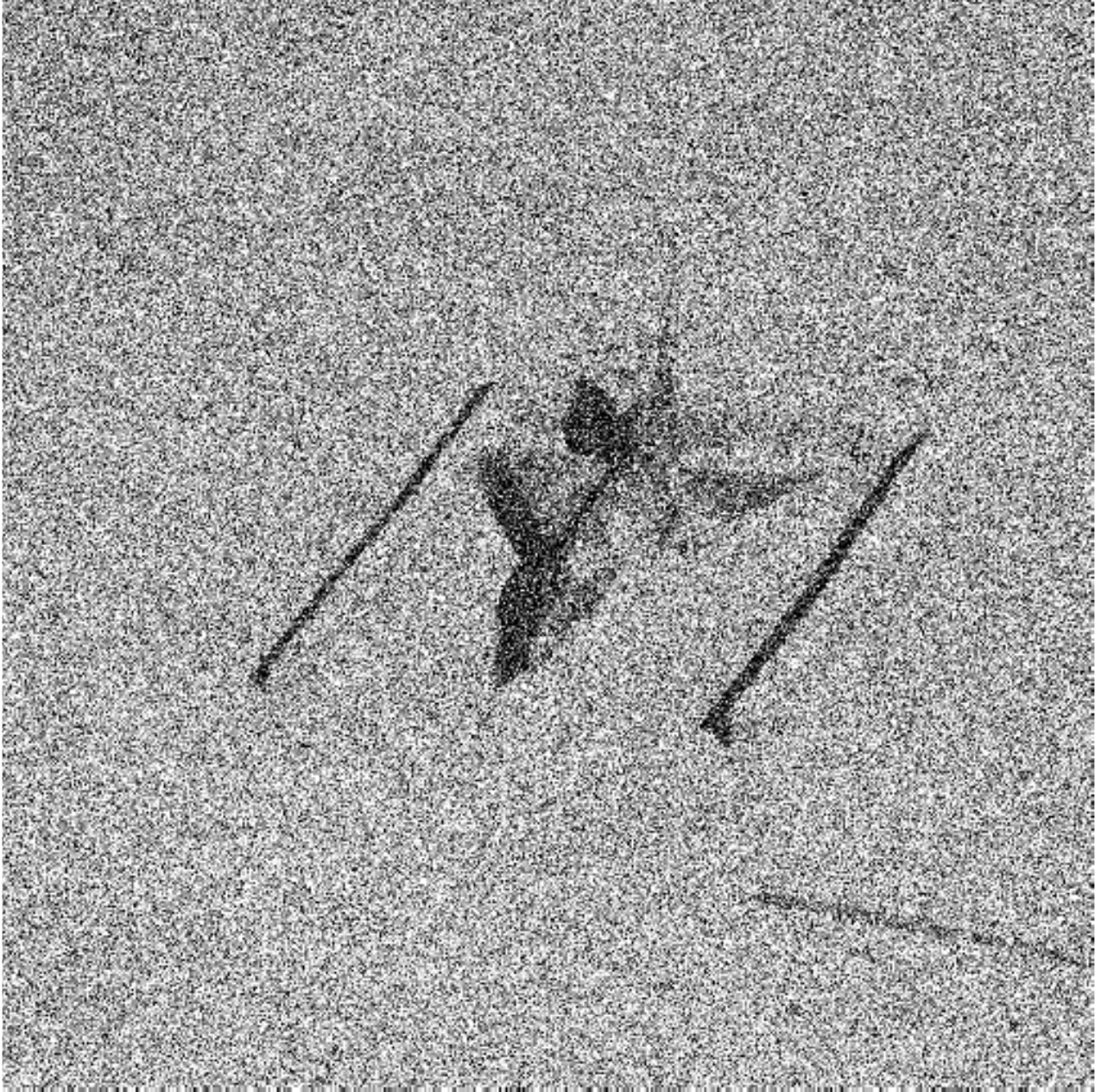}}\\
\subfigure[]{\includegraphics[width=.08\textwidth]{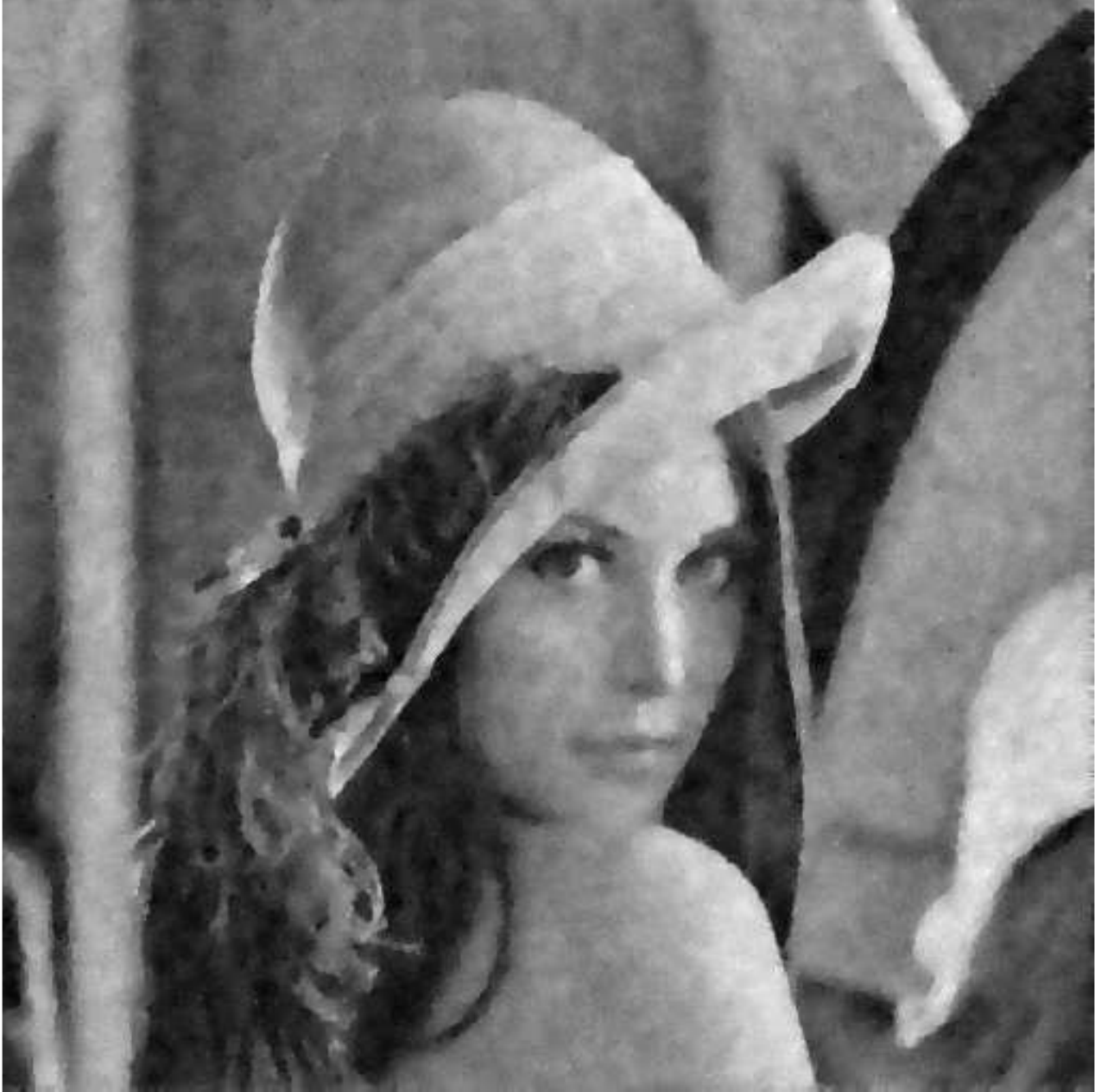}}~~
\subfigure[]{\includegraphics[width=.08\textwidth]{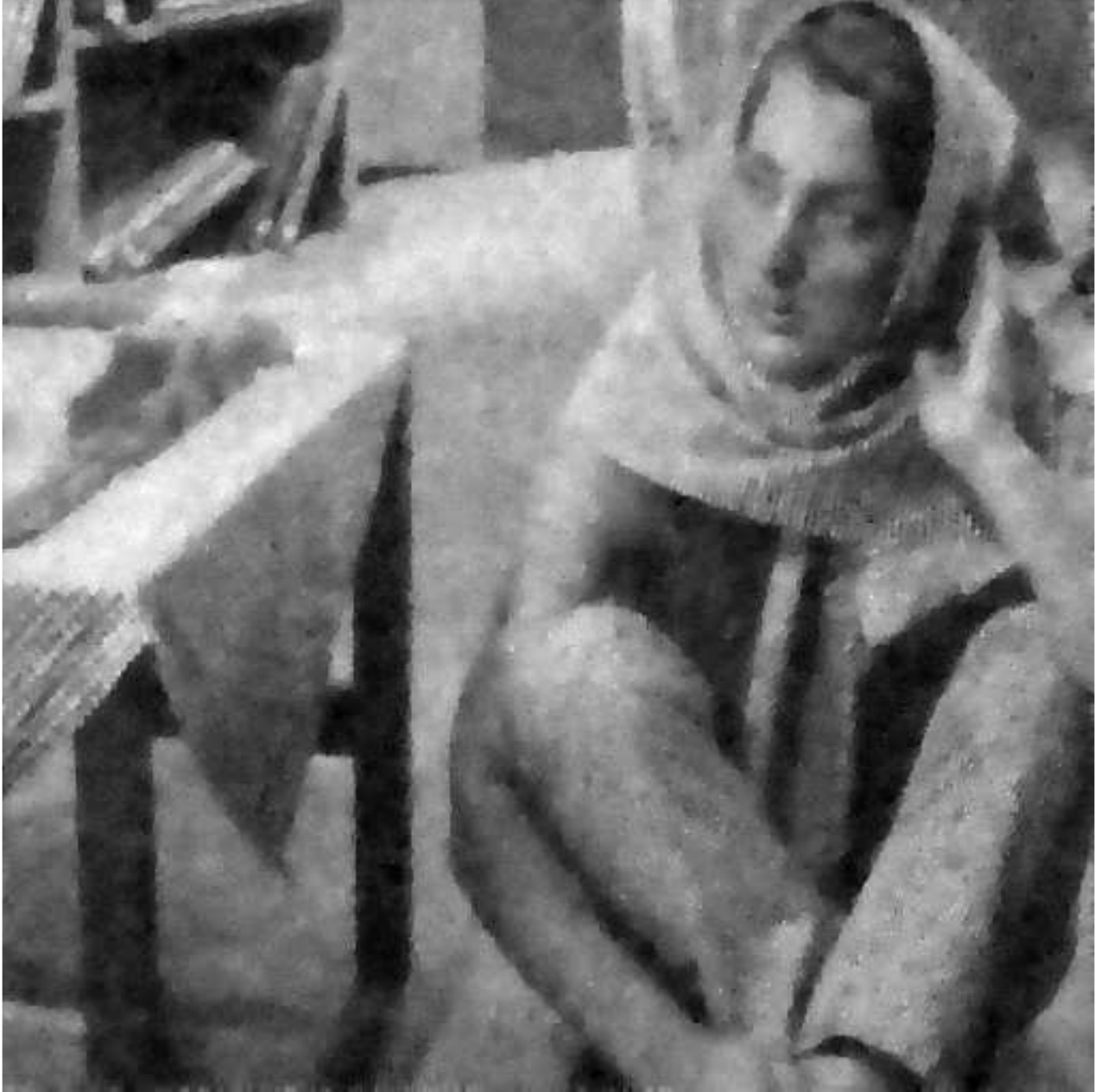}}~~
\subfigure[]{\includegraphics[width=.08\textwidth]{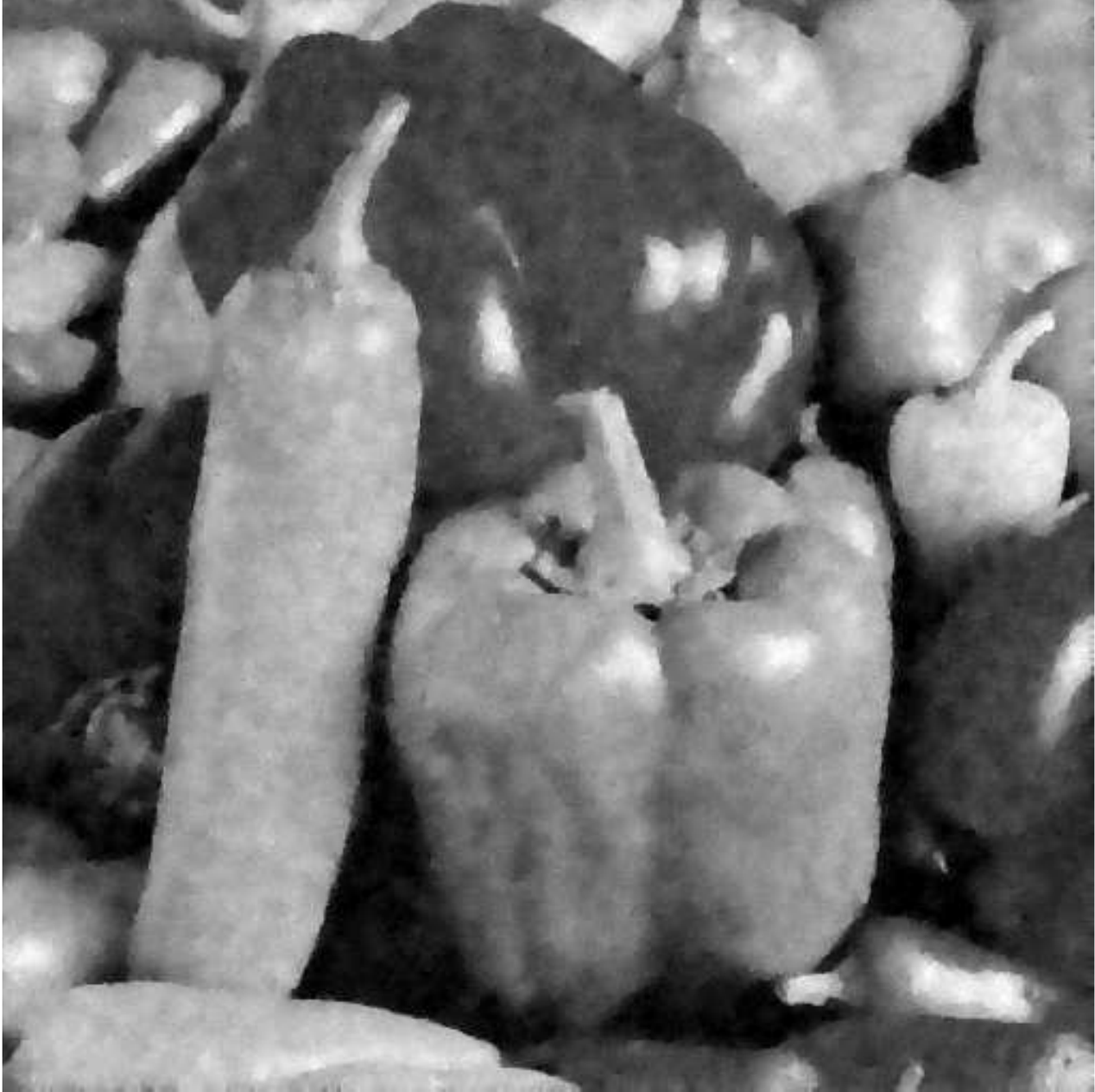}}~~
\subfigure[]{\includegraphics[width=.08\textwidth]{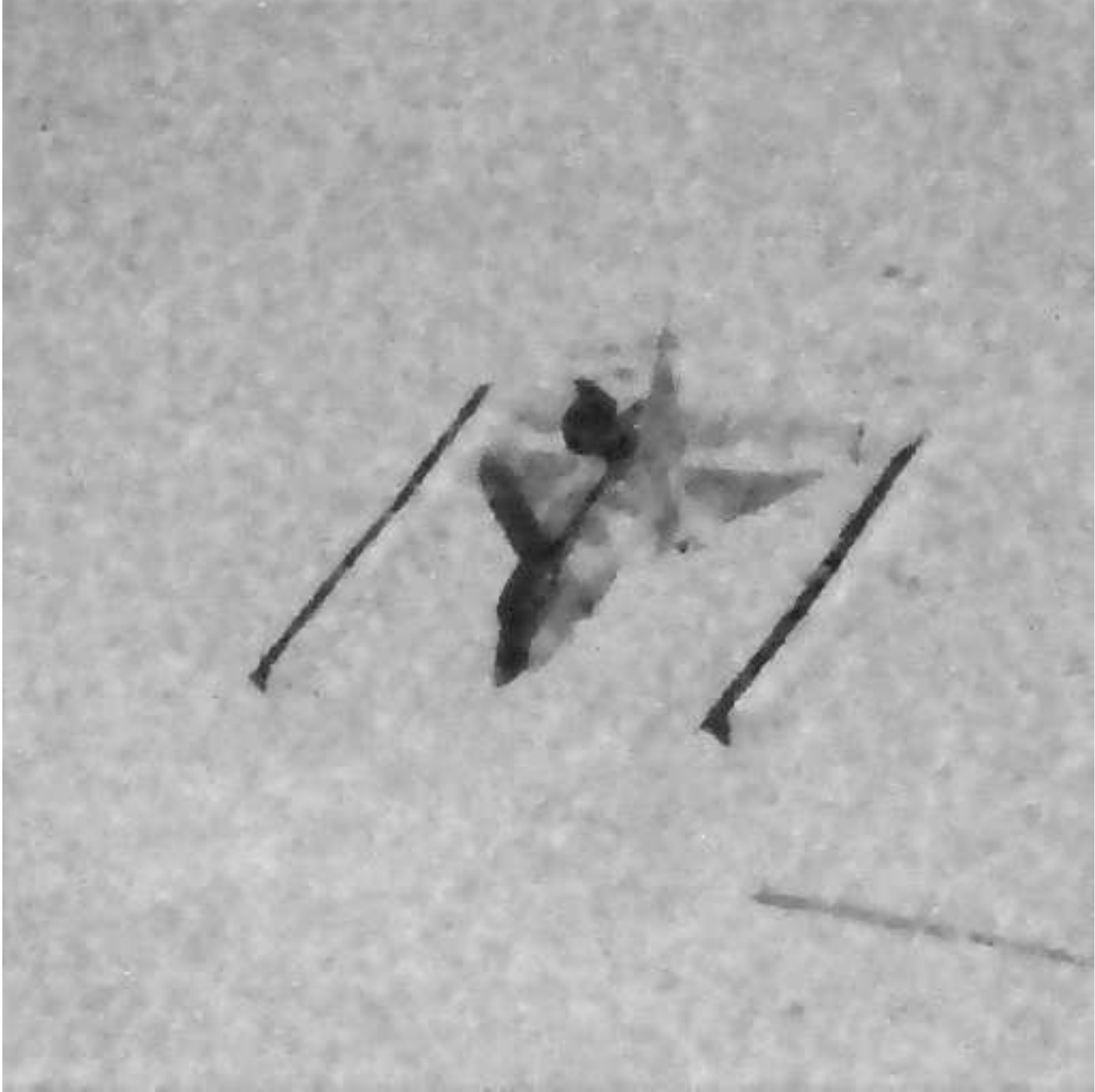}}\\
\subfigure[]{\includegraphics[width=.08\textwidth]{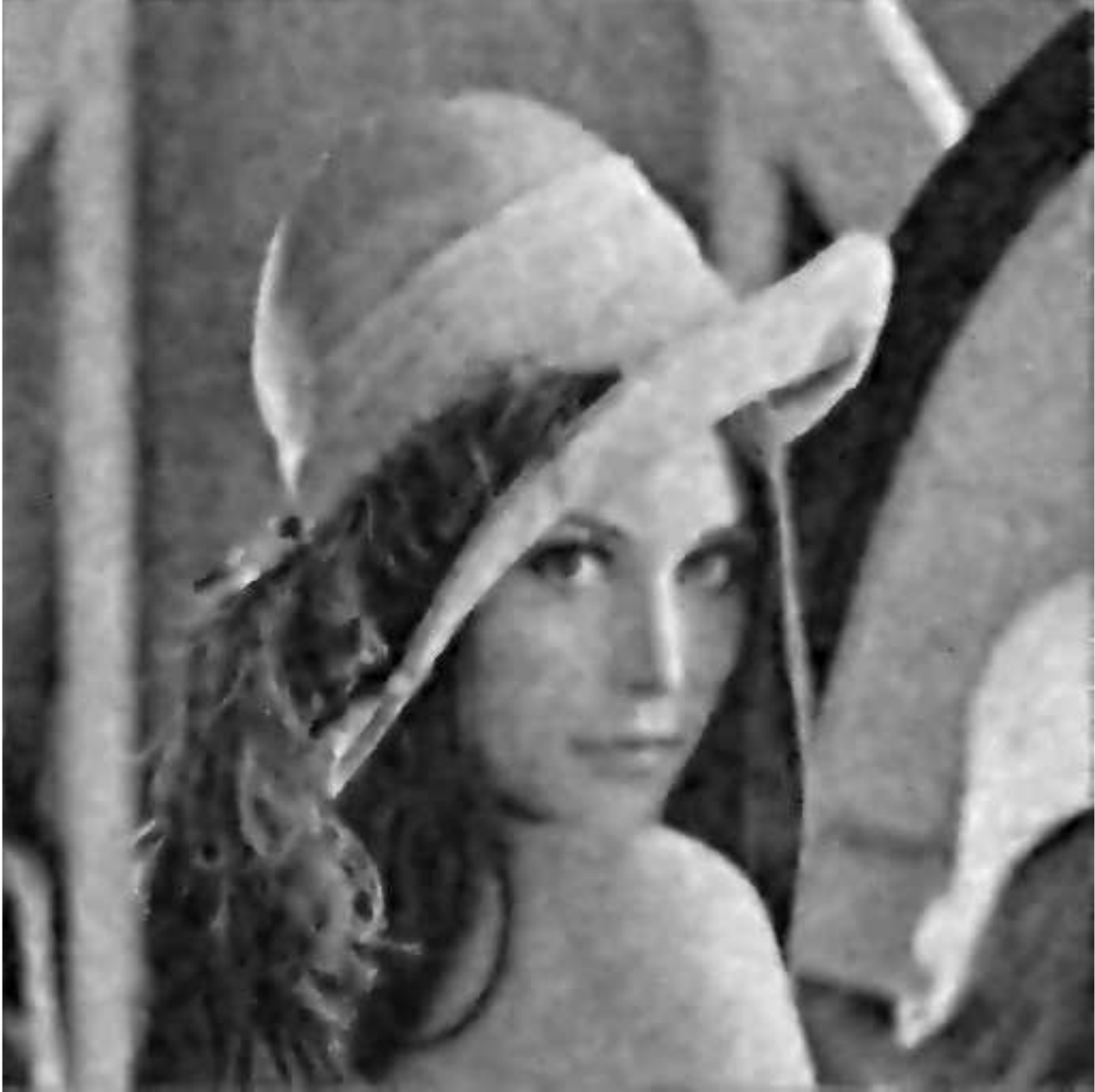}}~~
\subfigure[]{\includegraphics[width=.08\textwidth]{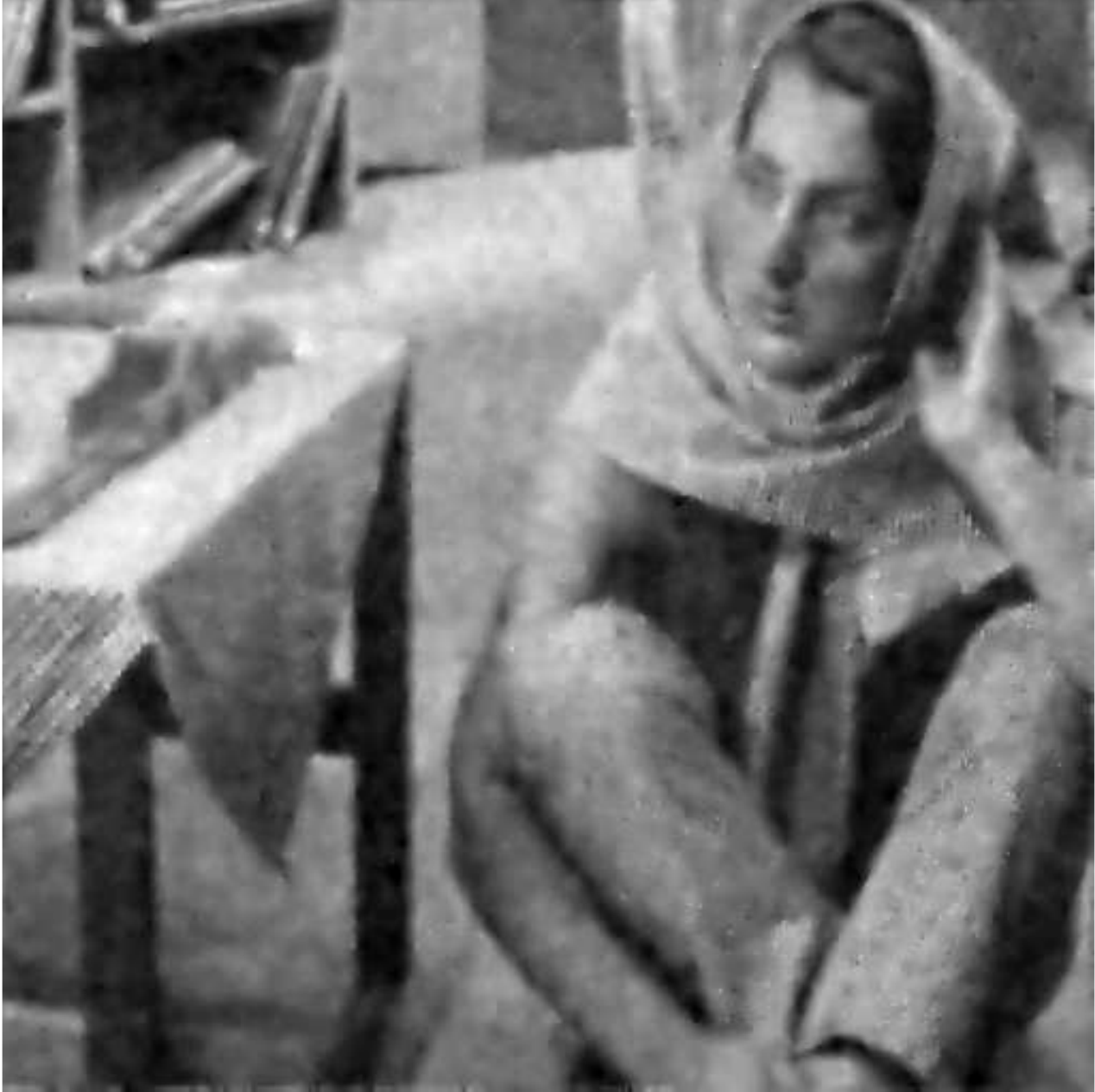}}~~
\subfigure[]{\includegraphics[width=.08\textwidth]{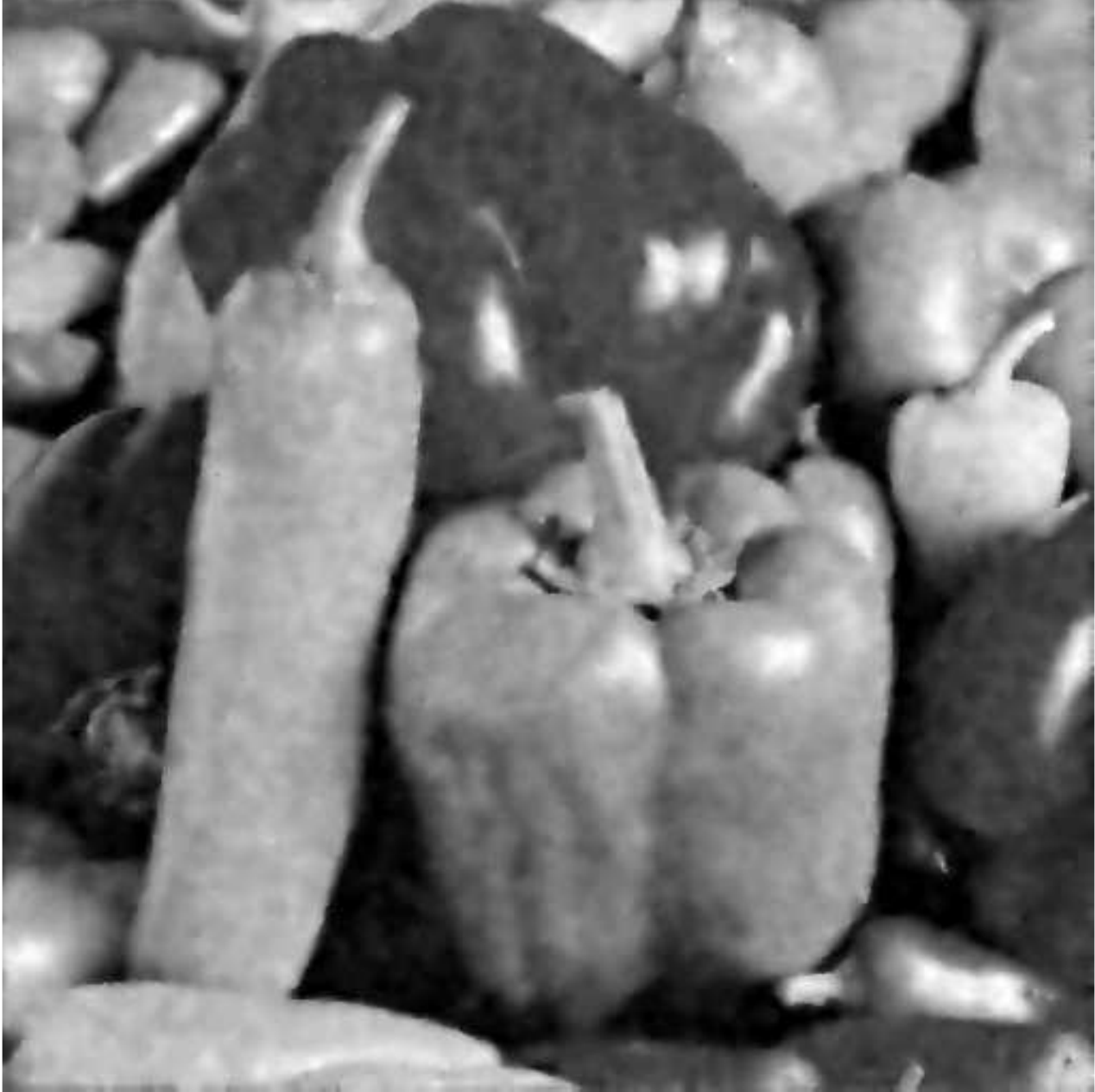}}~~
\subfigure[]{\includegraphics[width=.08\textwidth]{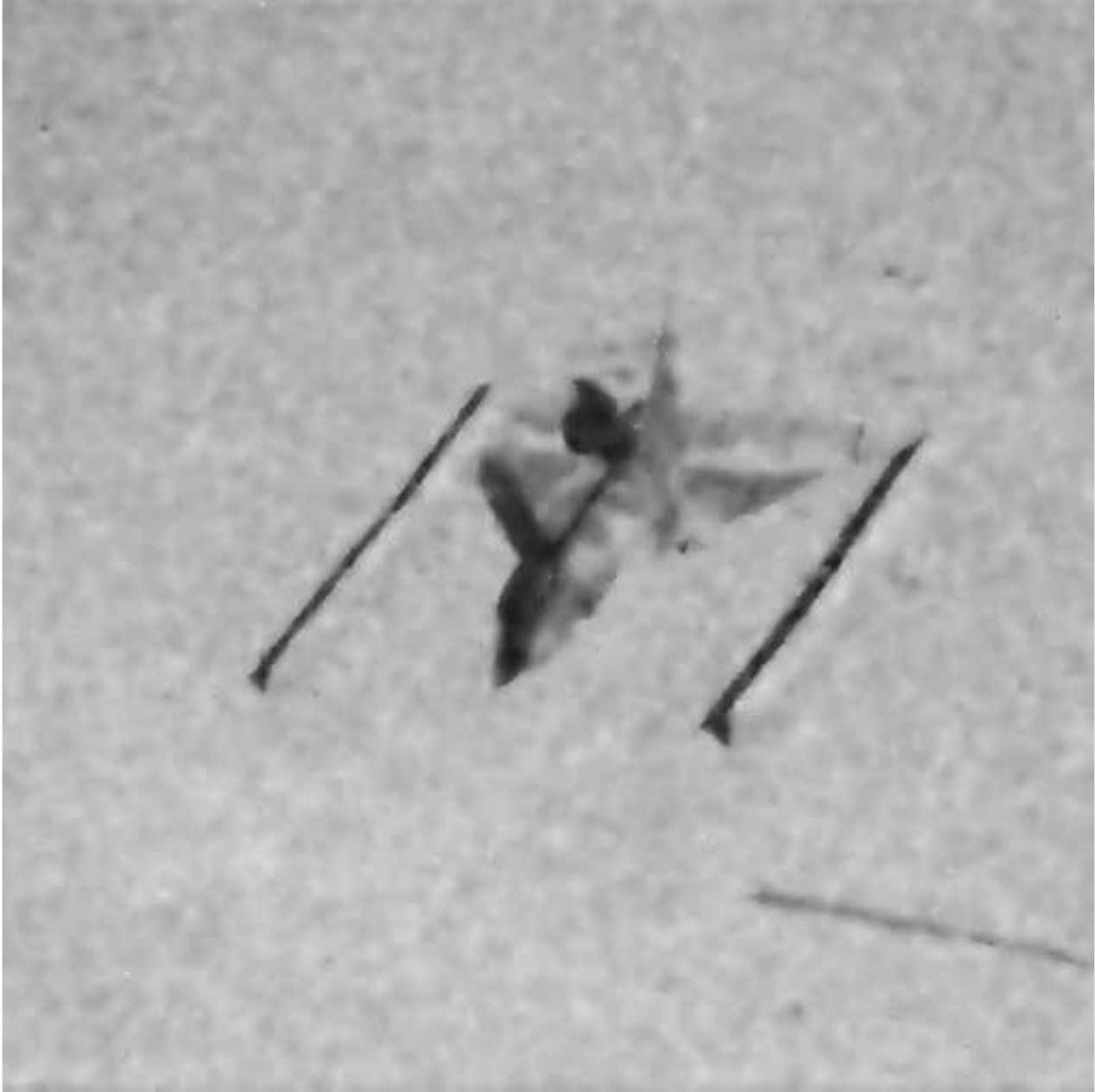}}\\
\subfigure[]{\includegraphics[width=.08\textwidth]{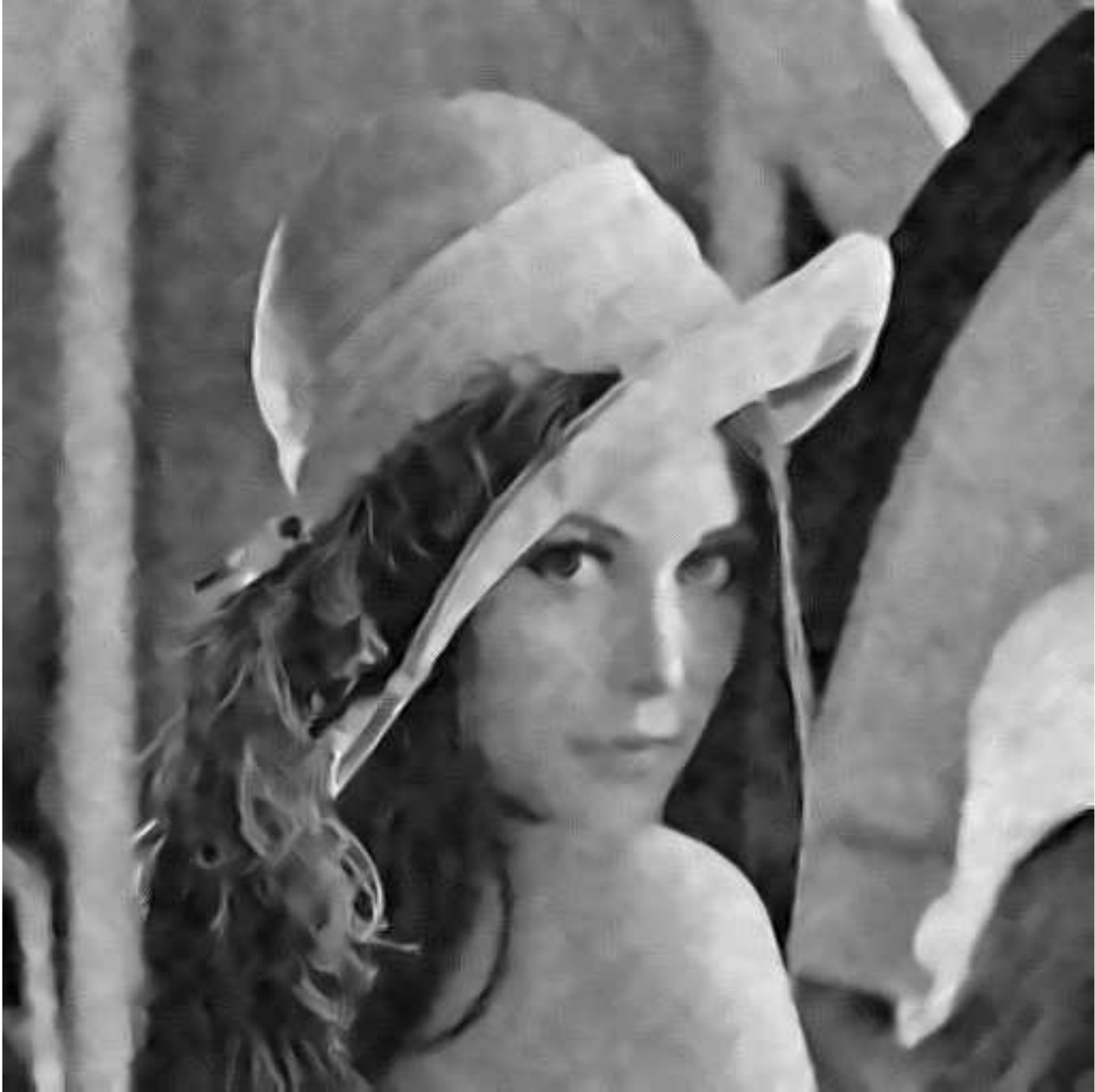}}~~
\subfigure[]{\includegraphics[width=.08\textwidth]{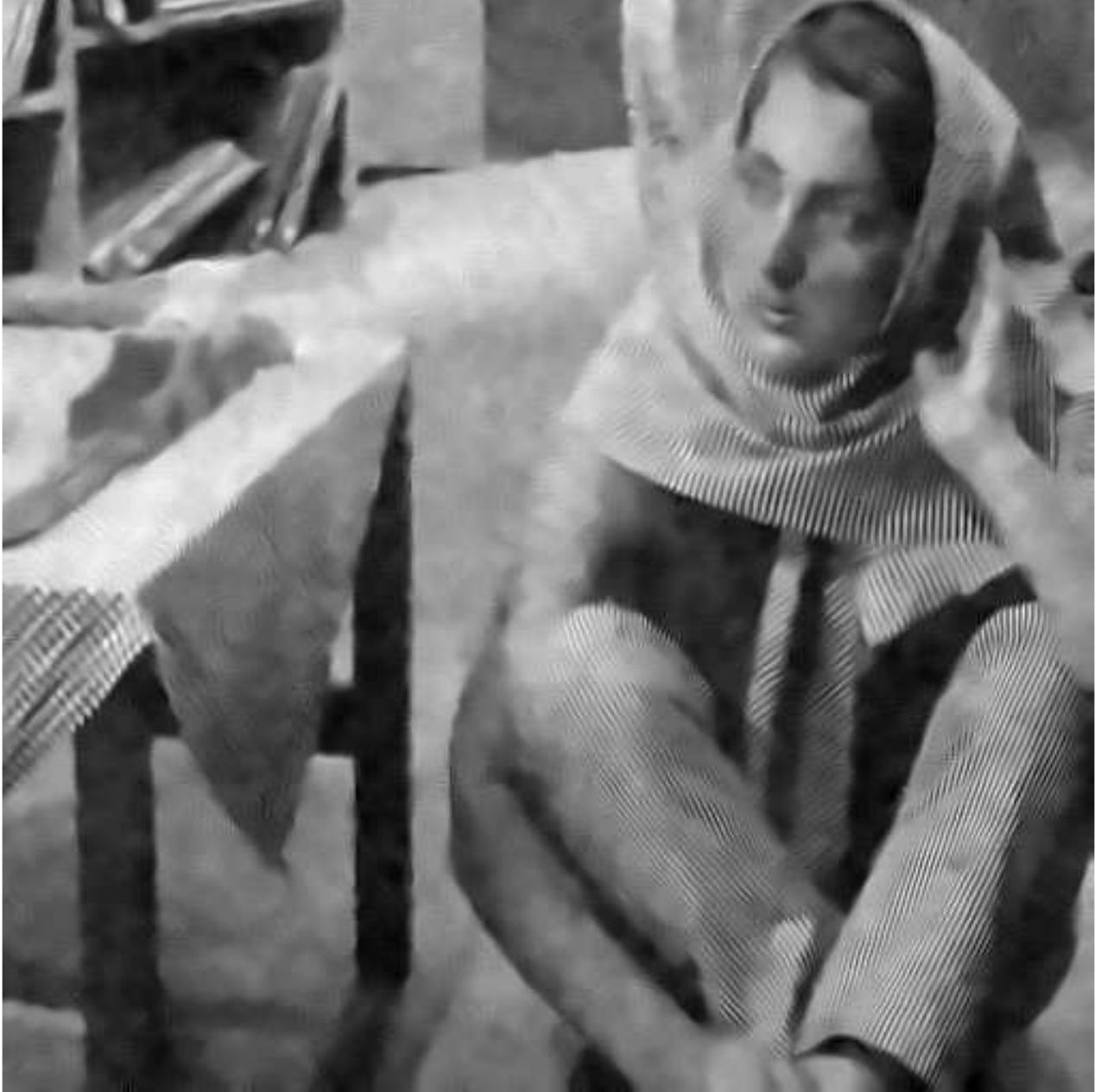}}~~
\subfigure[]{\includegraphics[width=.08\textwidth]{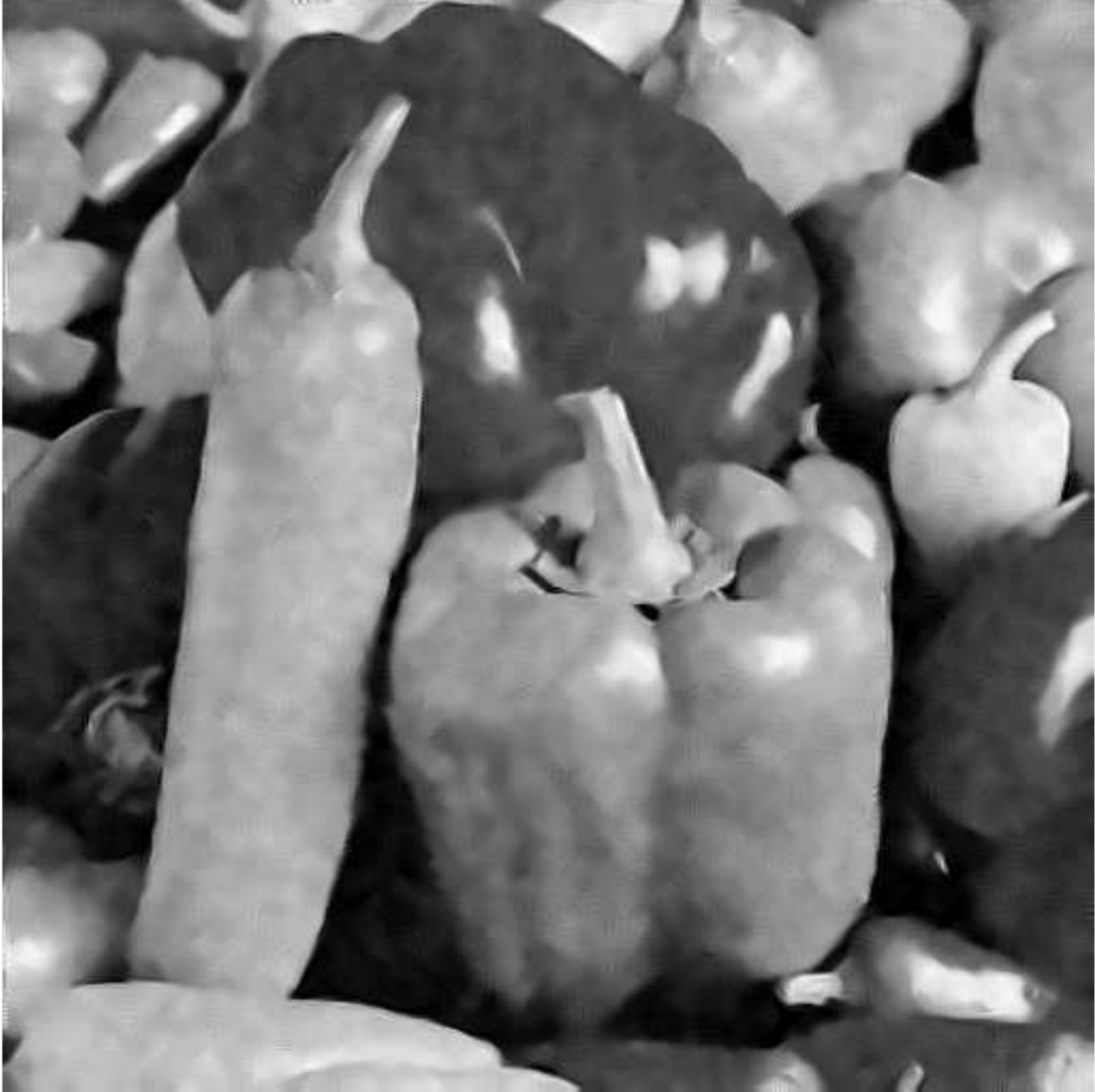}}~~
\subfigure[]{\includegraphics[width=.08\textwidth]{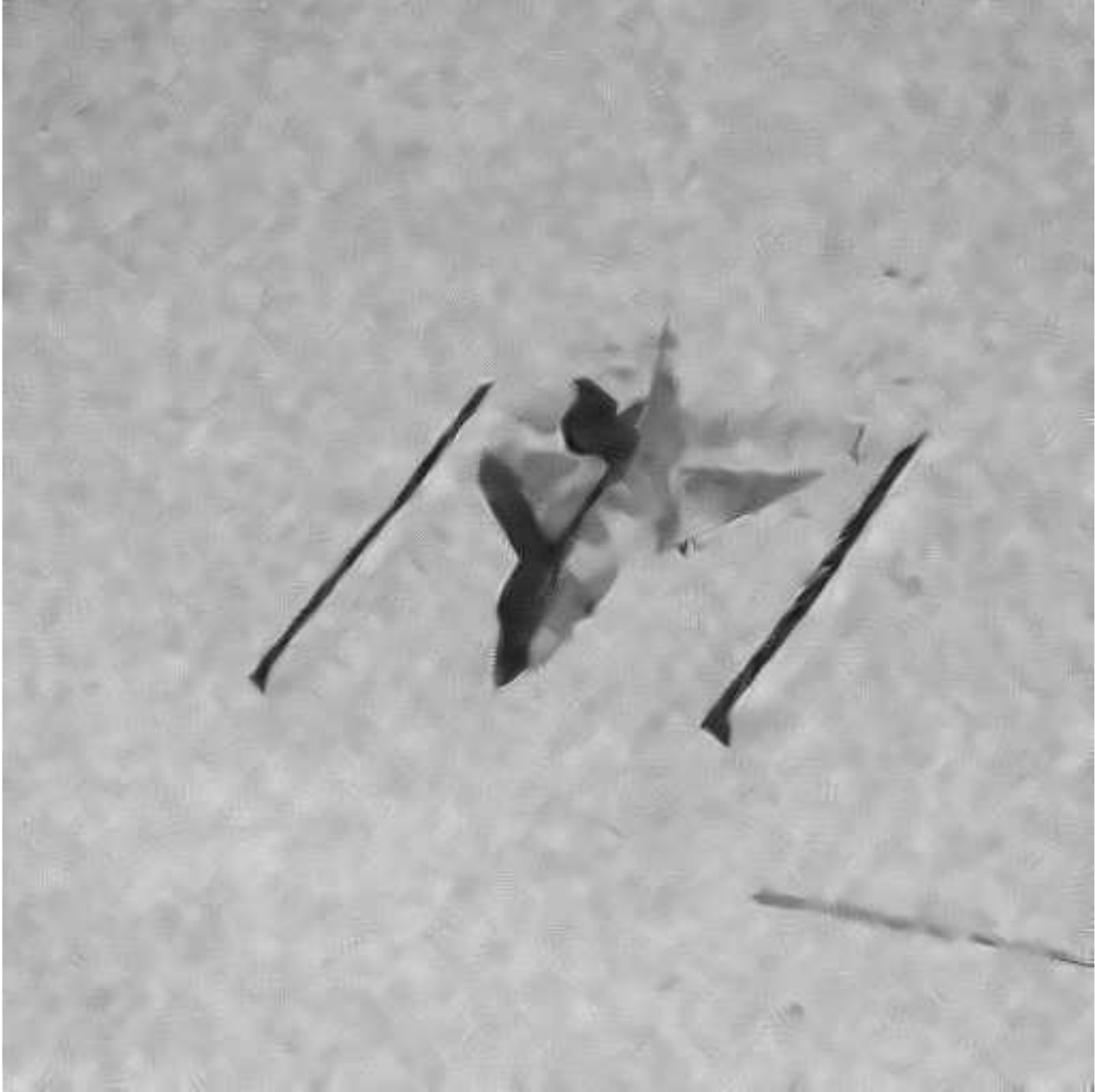}}\\
\subfigure[]{\includegraphics[width=.08\textwidth]{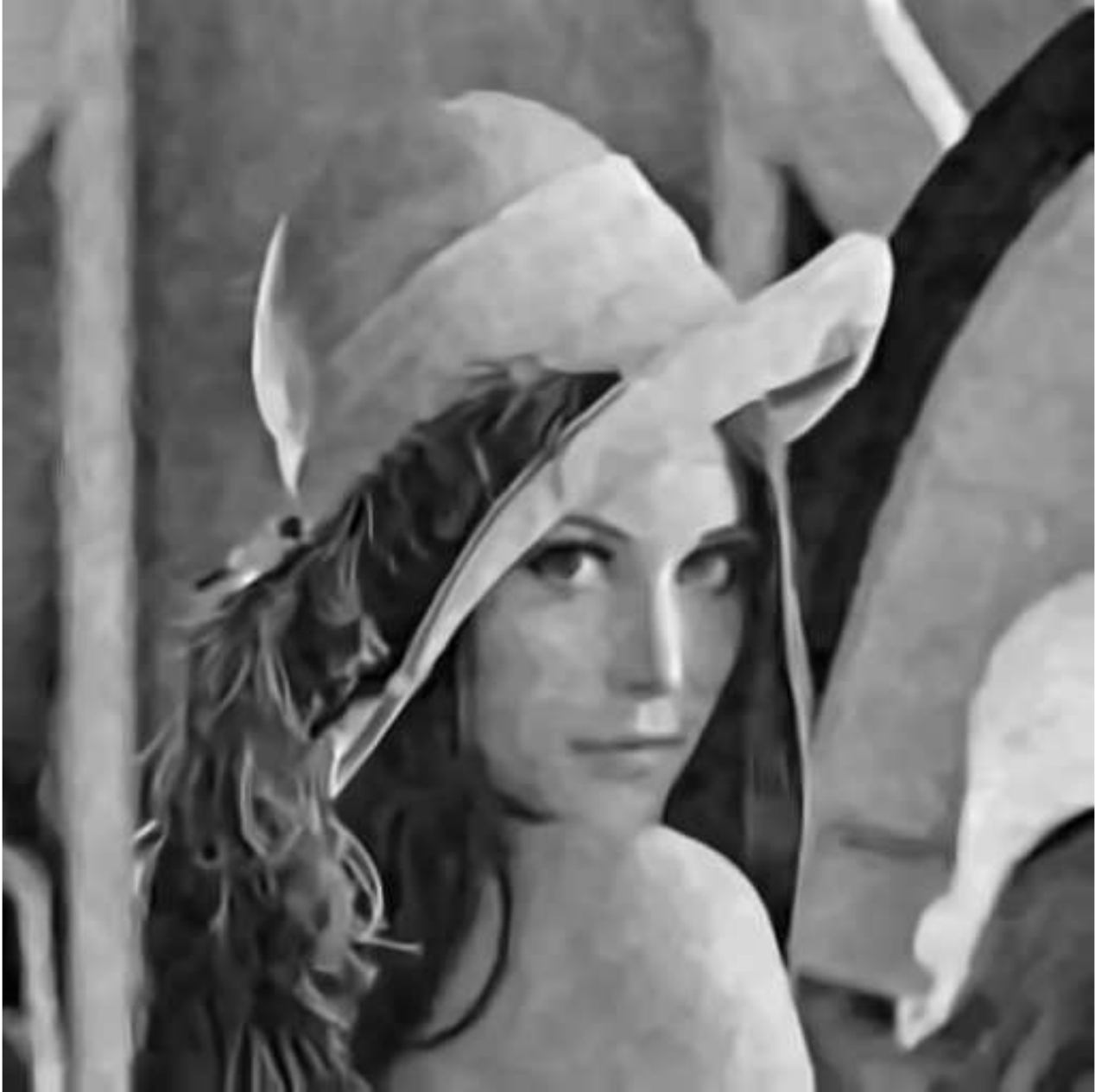}}~~
\subfigure[]{\includegraphics[width=.08\textwidth]{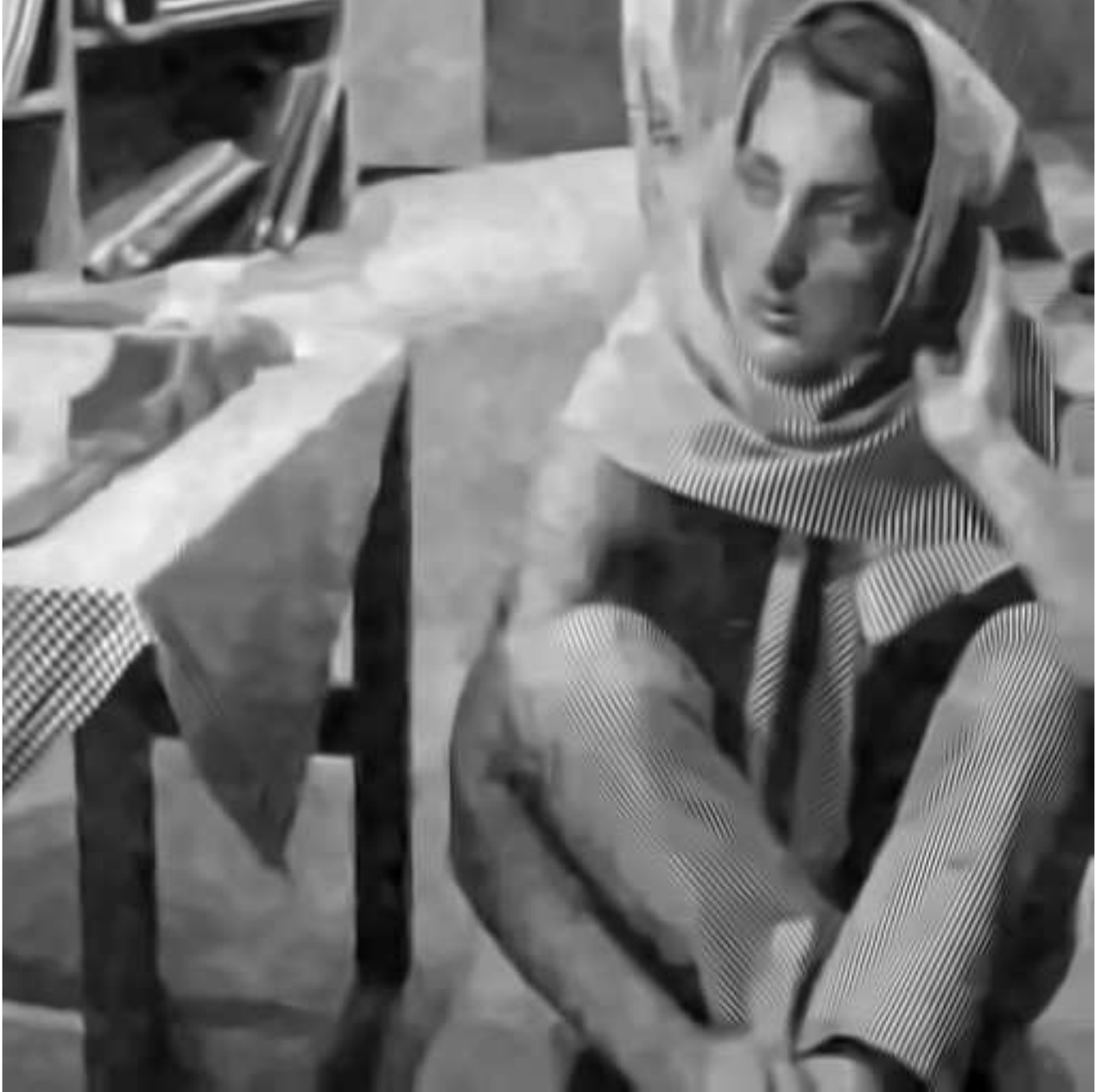}}~~
\subfigure[]{\includegraphics[width=.08\textwidth]{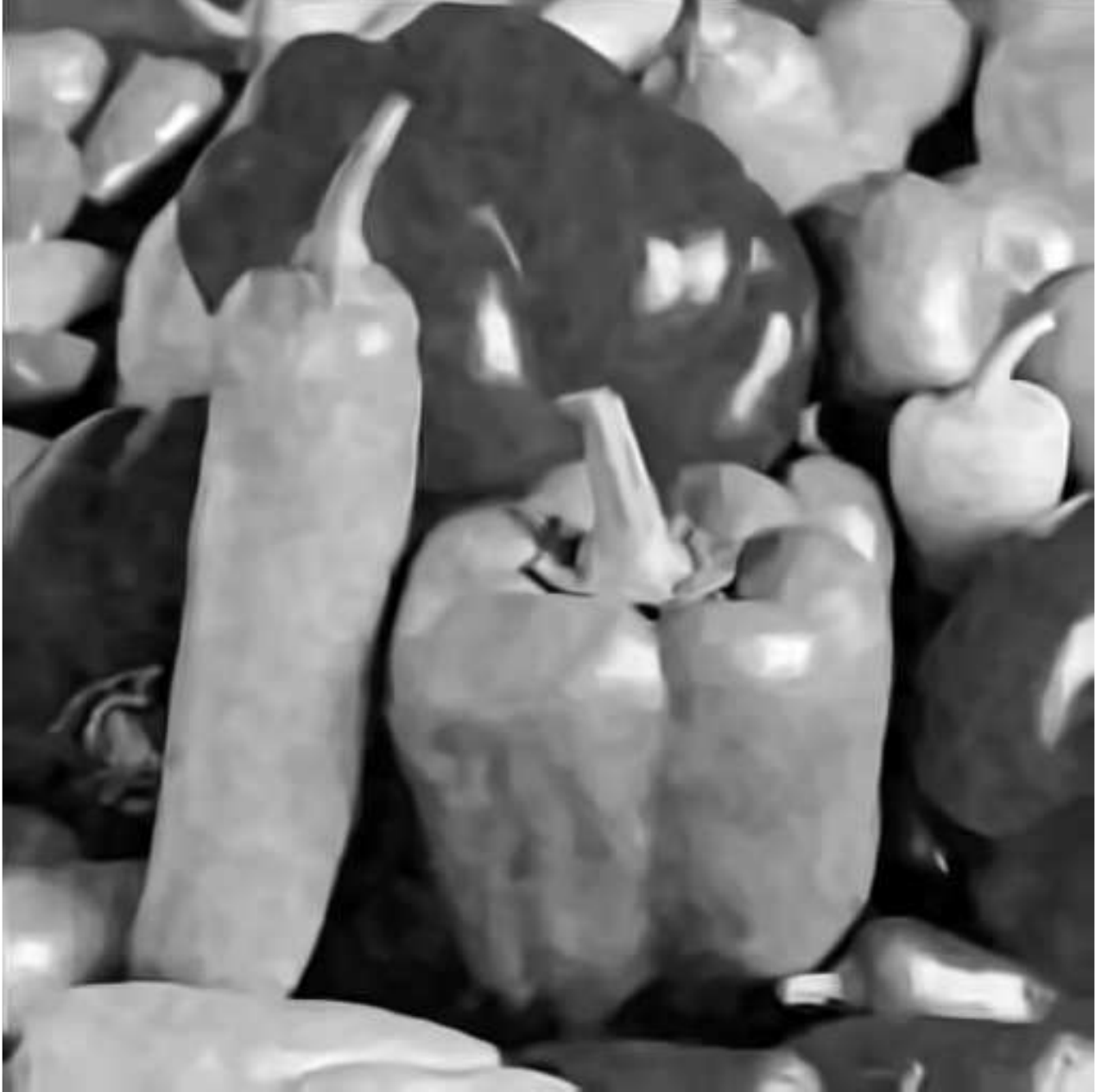}}~~
\subfigure[]{\includegraphics[width=.08\textwidth]{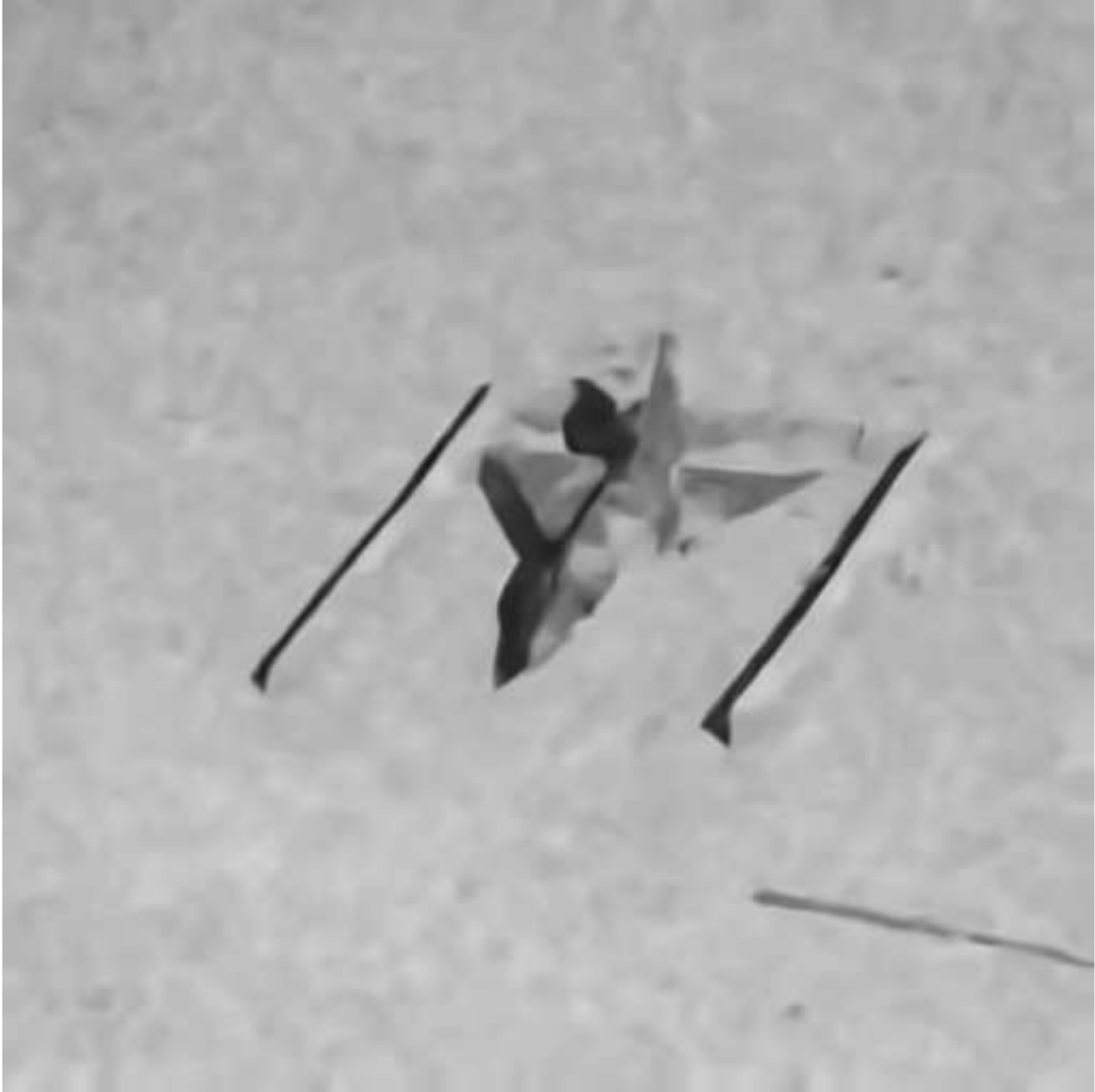}}
\end{center}
\caption{PR with CDP. Peak level $\nu=5.0\times10^{-3}$ for Poisson noise. First row: ``LS-PR''; Second row: ``TV-PR''; Third row: ``TGV-PR''; Fourth row: ``NLM-PR''; Fifth row: ``BM3D-PR''.}
\label{poicdp2}
\end{figure}

\begin{figure}
\begin{center}
\subfigure[]{\includegraphics[width=.08\textwidth]{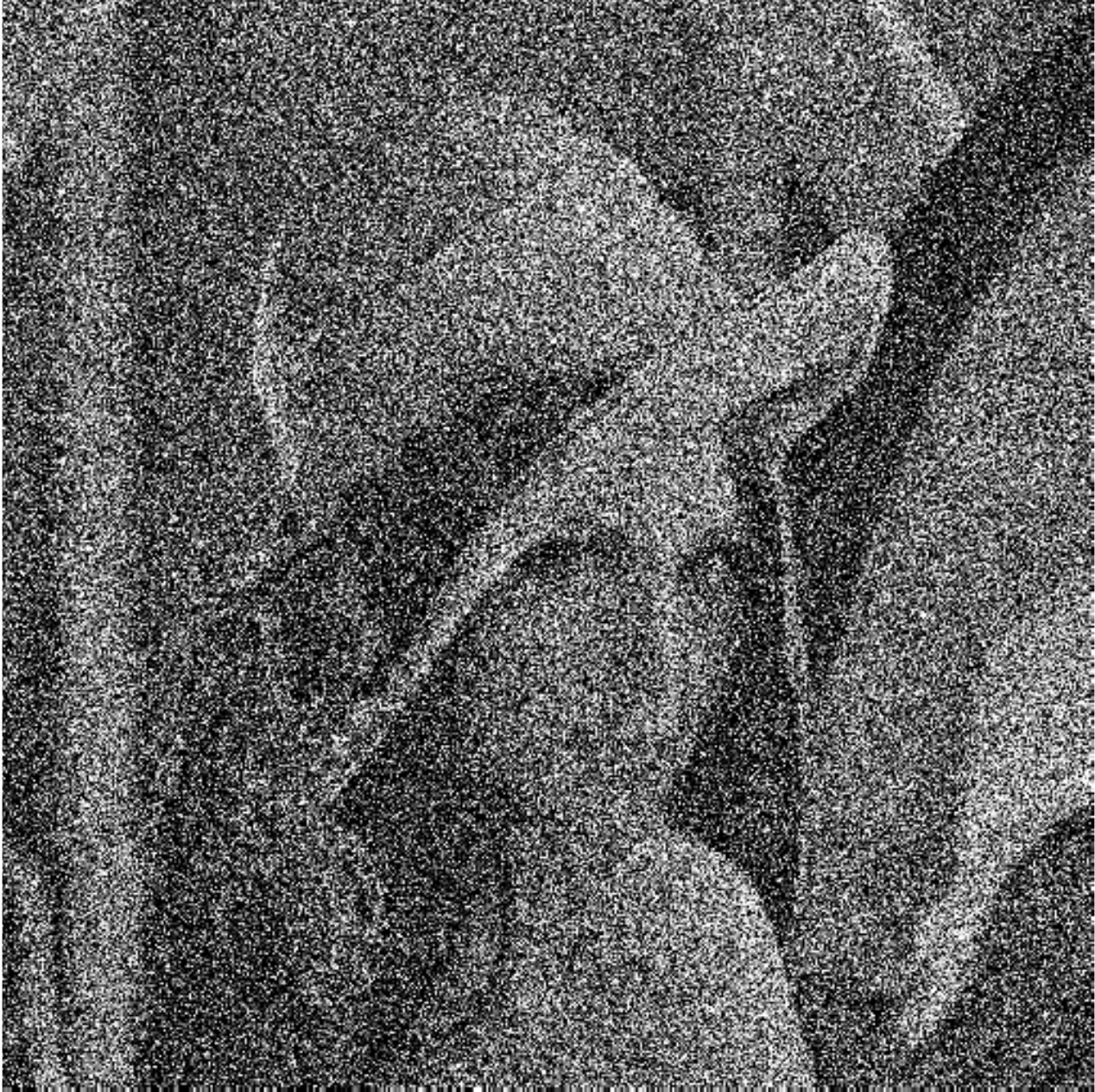}}~~
\subfigure[]{\includegraphics[width=.08\textwidth]{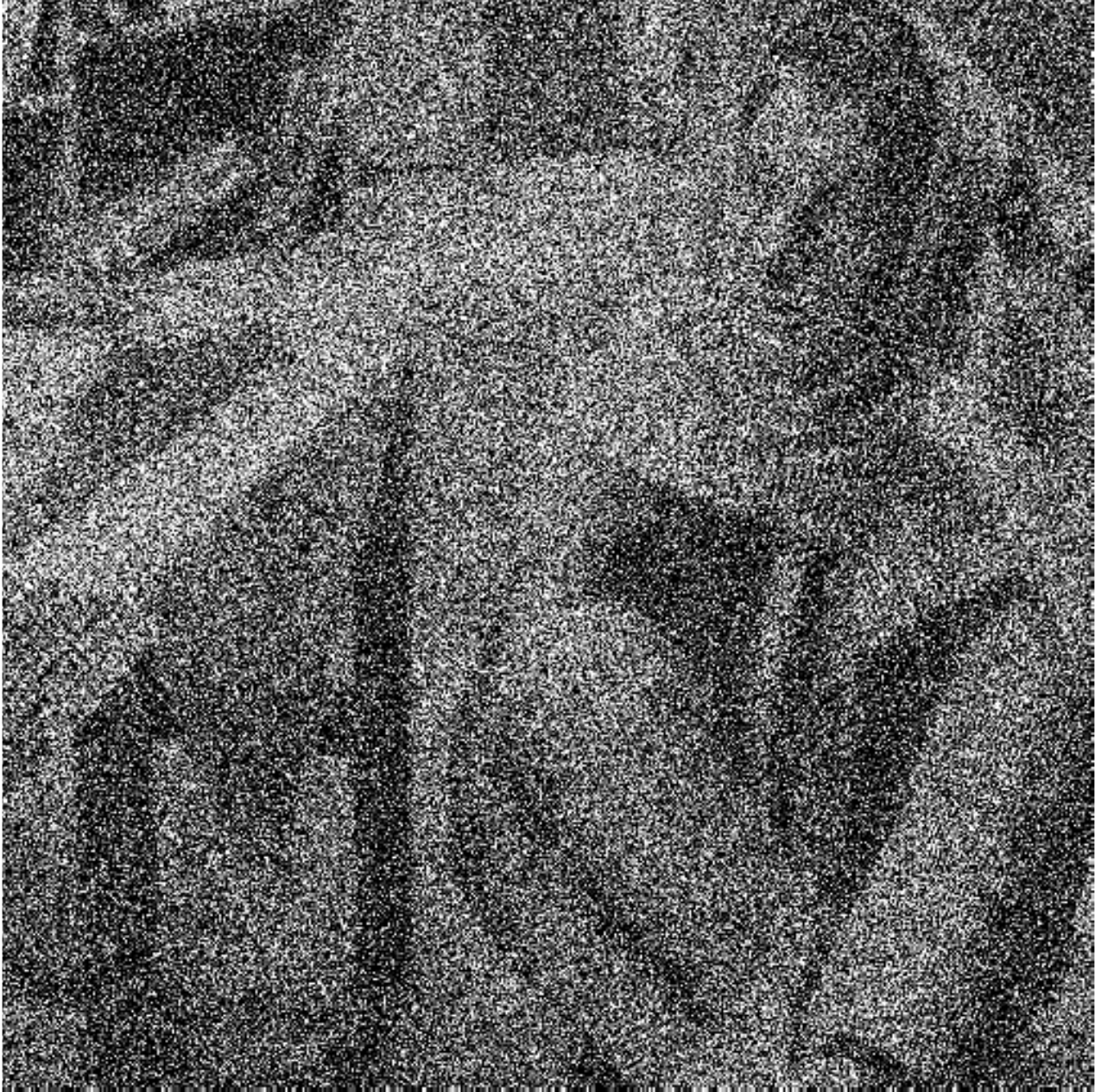}}~~
\subfigure[]{\includegraphics[width=.08\textwidth]{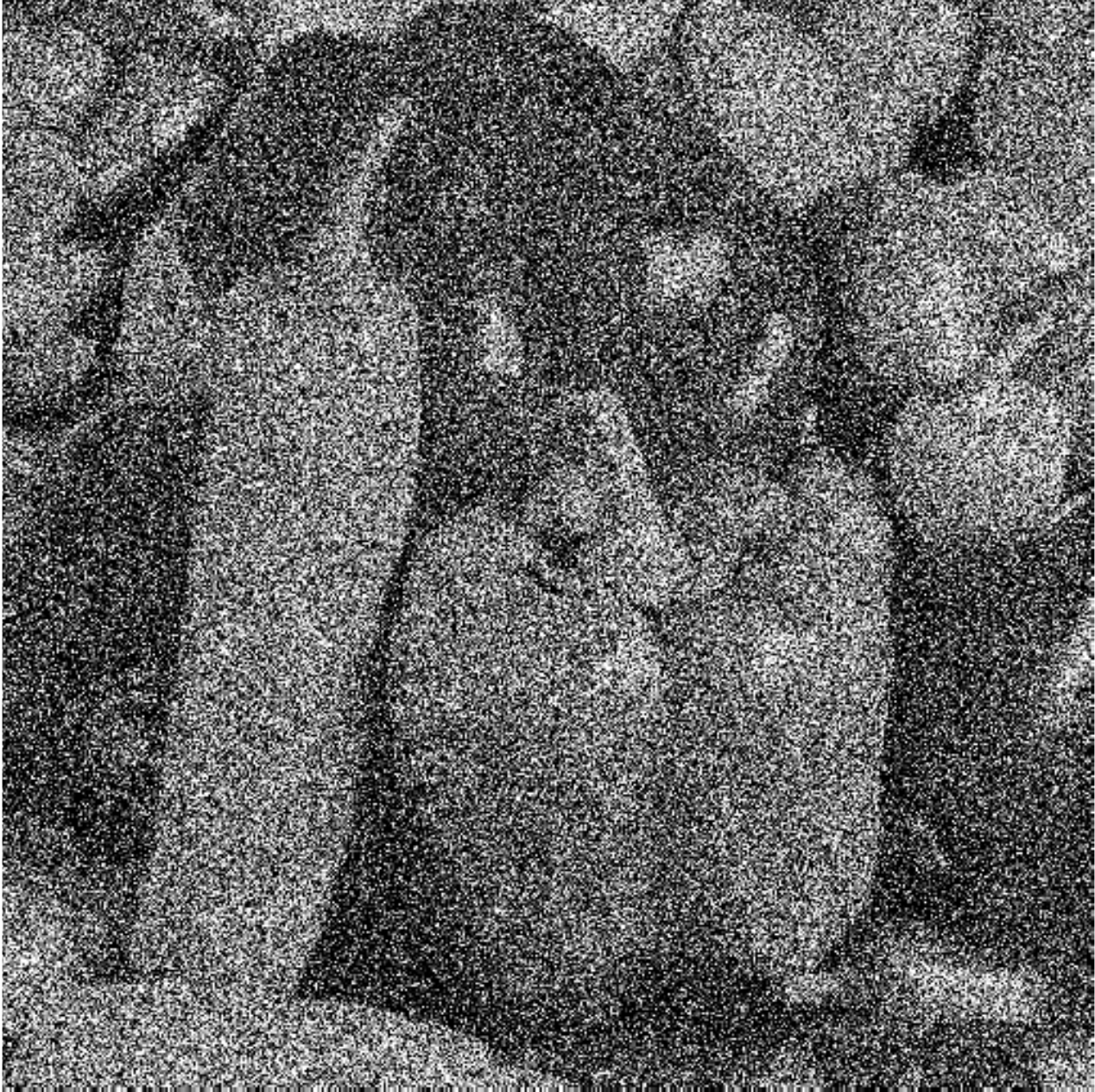}}~~
\subfigure[]{\includegraphics[width=.08\textwidth]{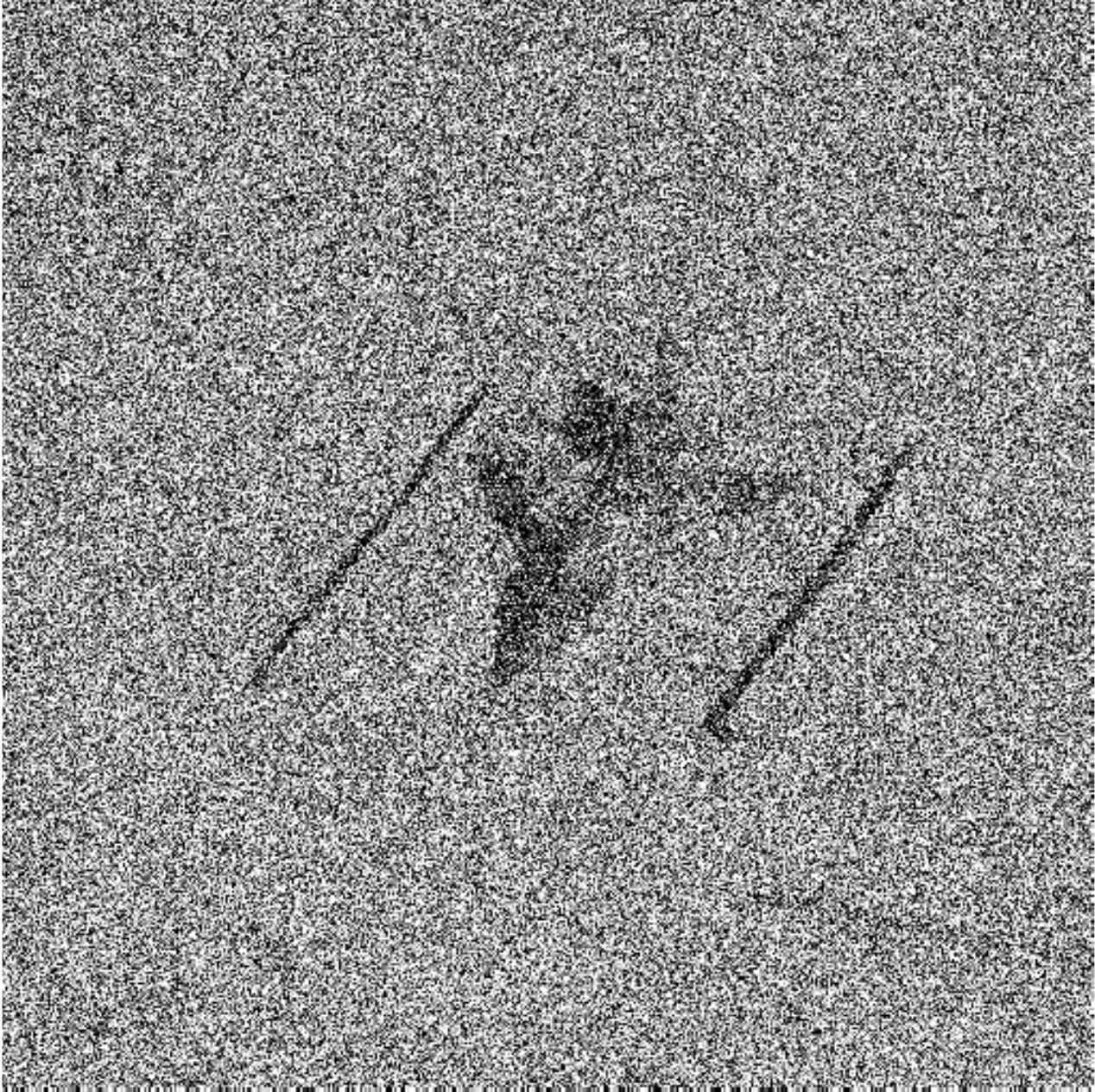}}\\
\subfigure[]{\includegraphics[width=.08\textwidth]{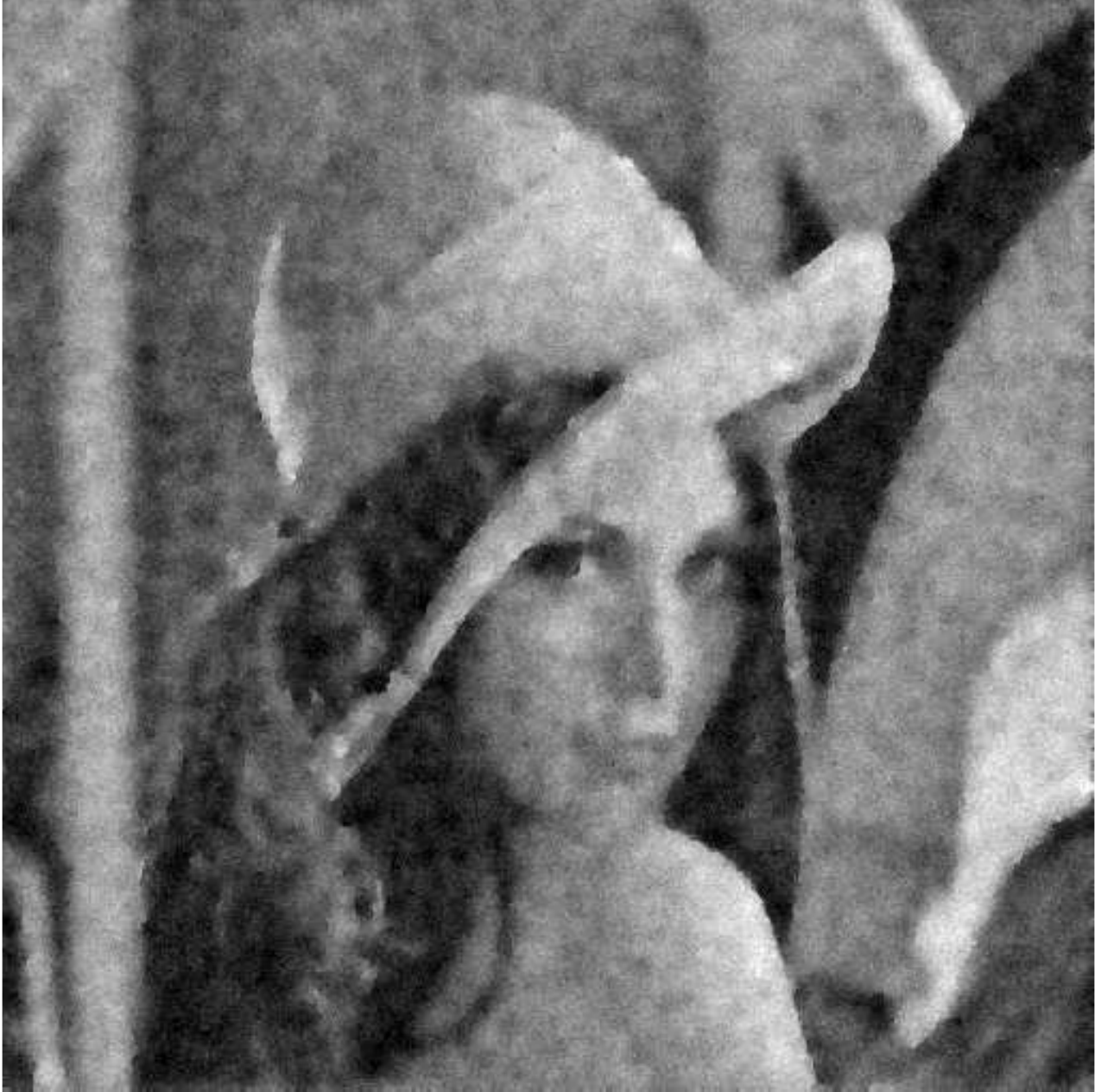}}~~
\subfigure[]{\includegraphics[width=.08\textwidth]{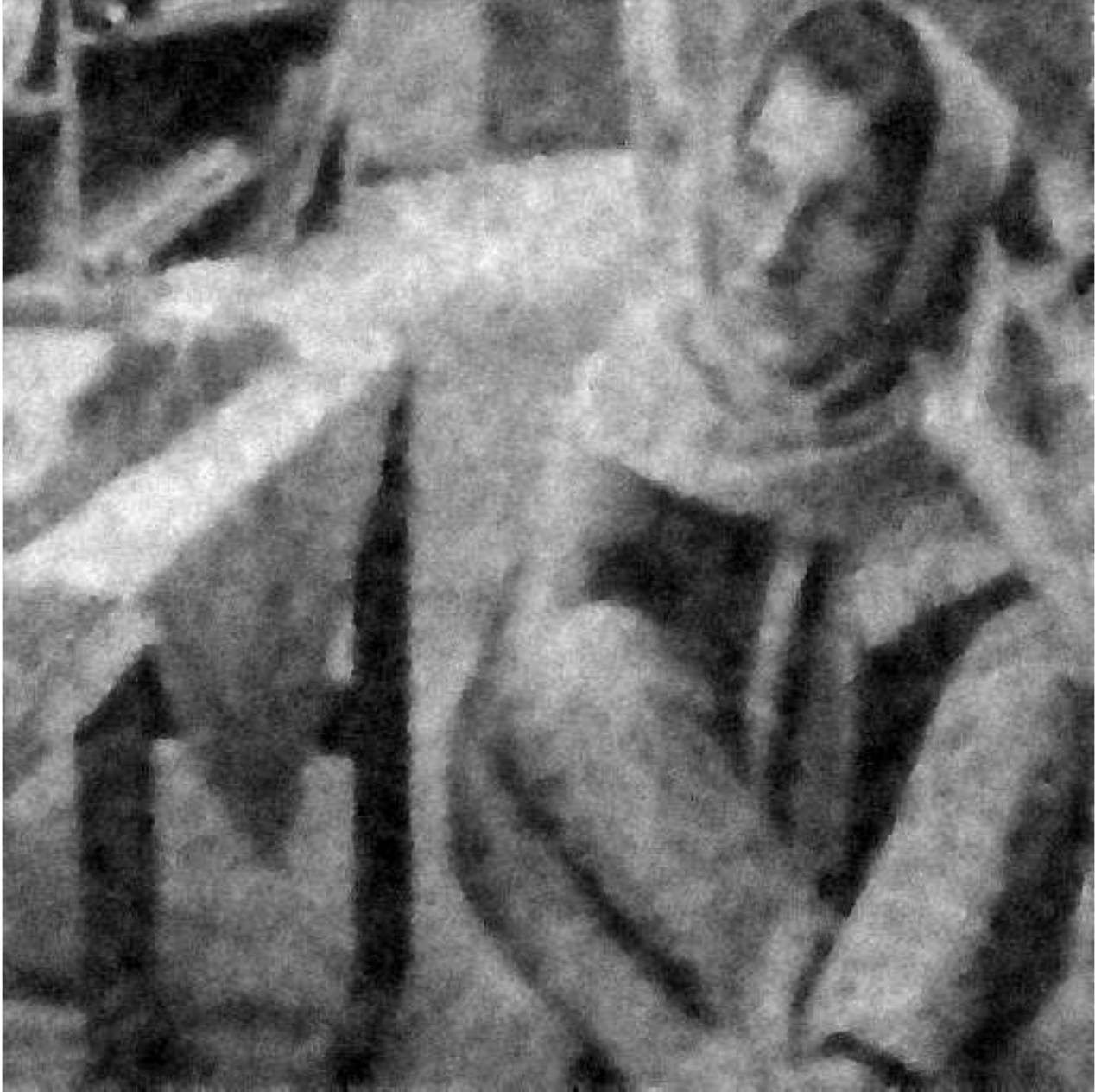}}~~
\subfigure[]{\includegraphics[width=.08\textwidth]{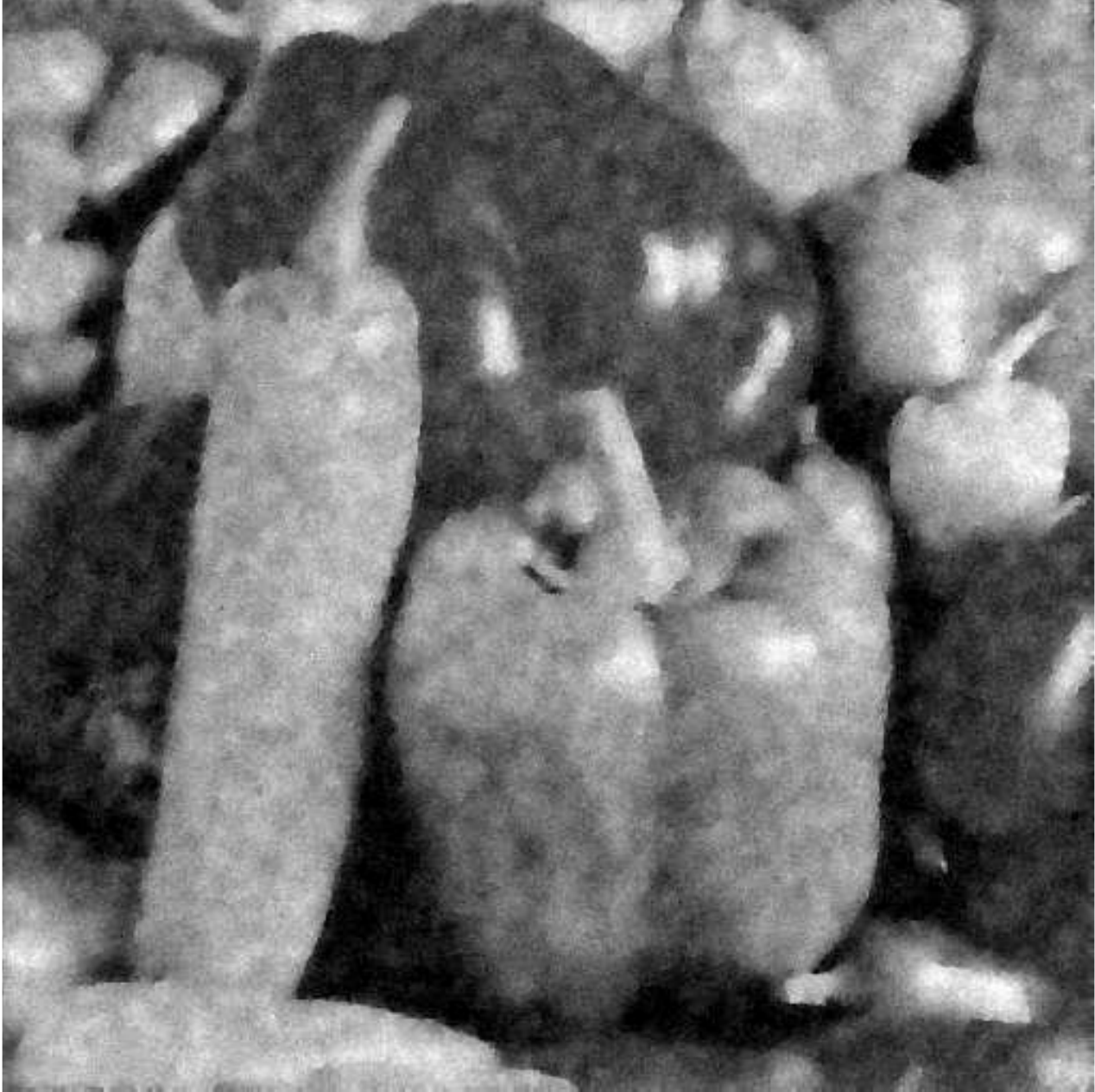}}~~
\subfigure[]{\includegraphics[width=.08\textwidth]{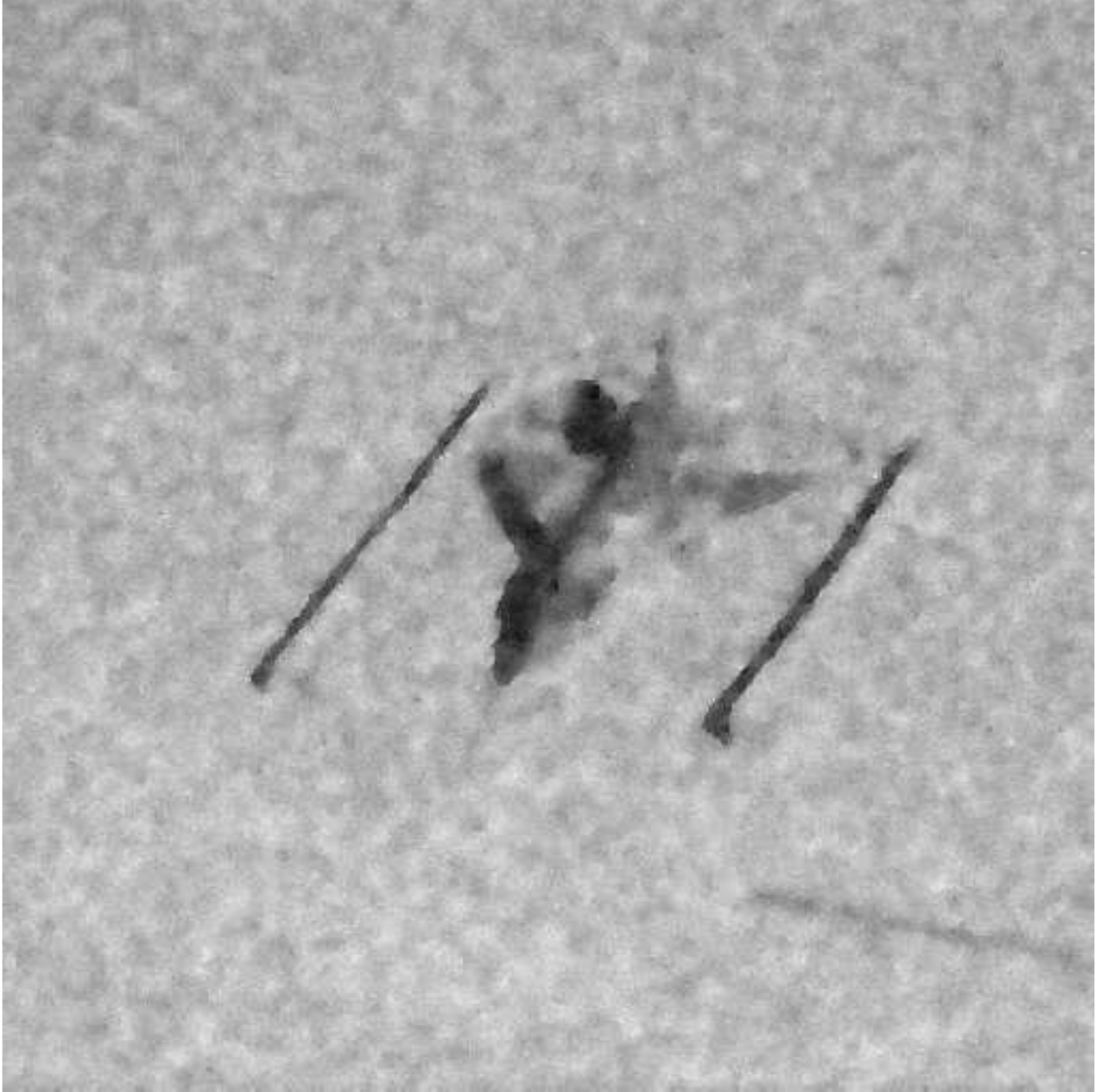}}\\
\subfigure[]{\includegraphics[width=.08\textwidth]{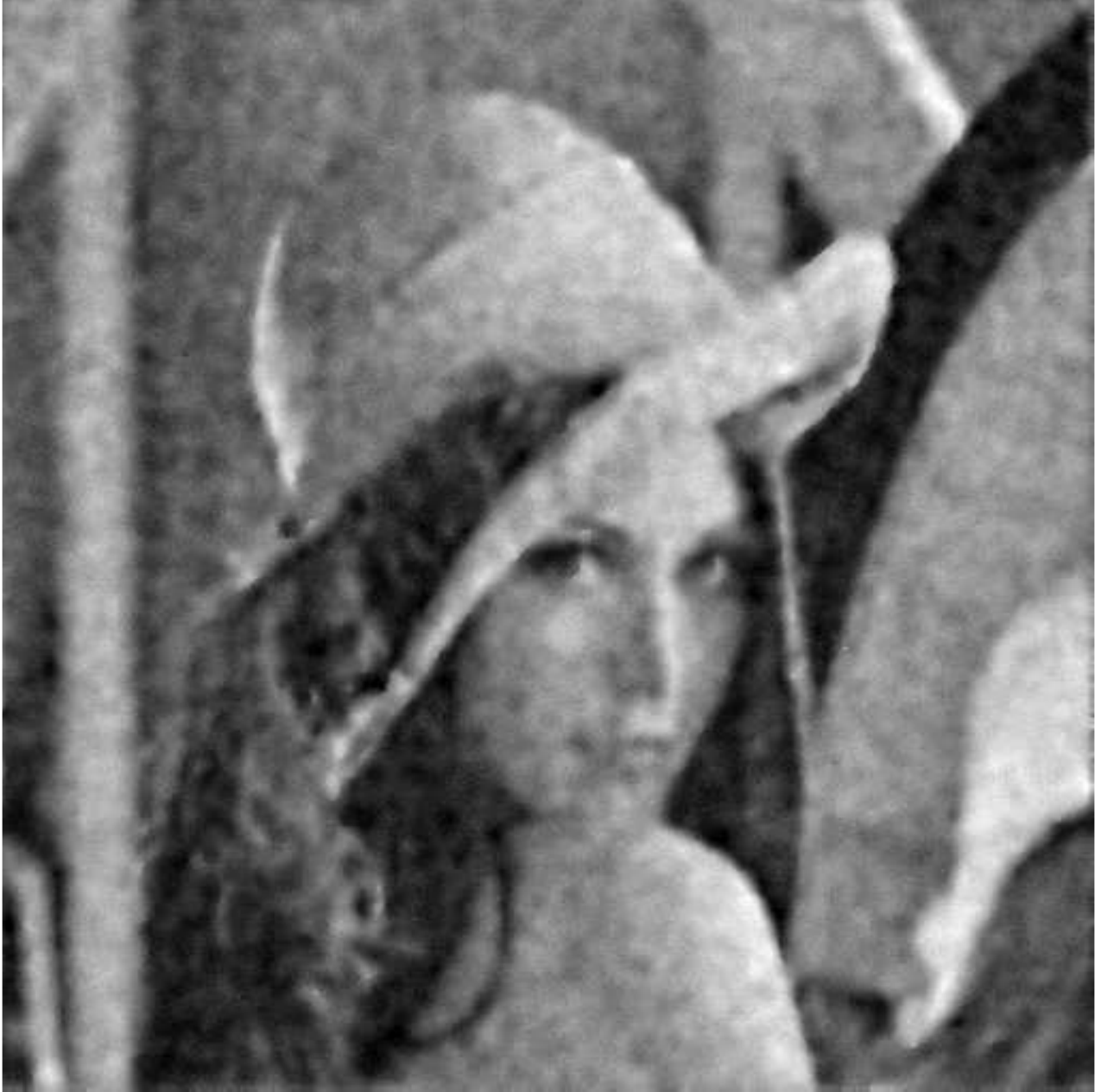}}~~
\subfigure[]{\includegraphics[width=.08\textwidth]{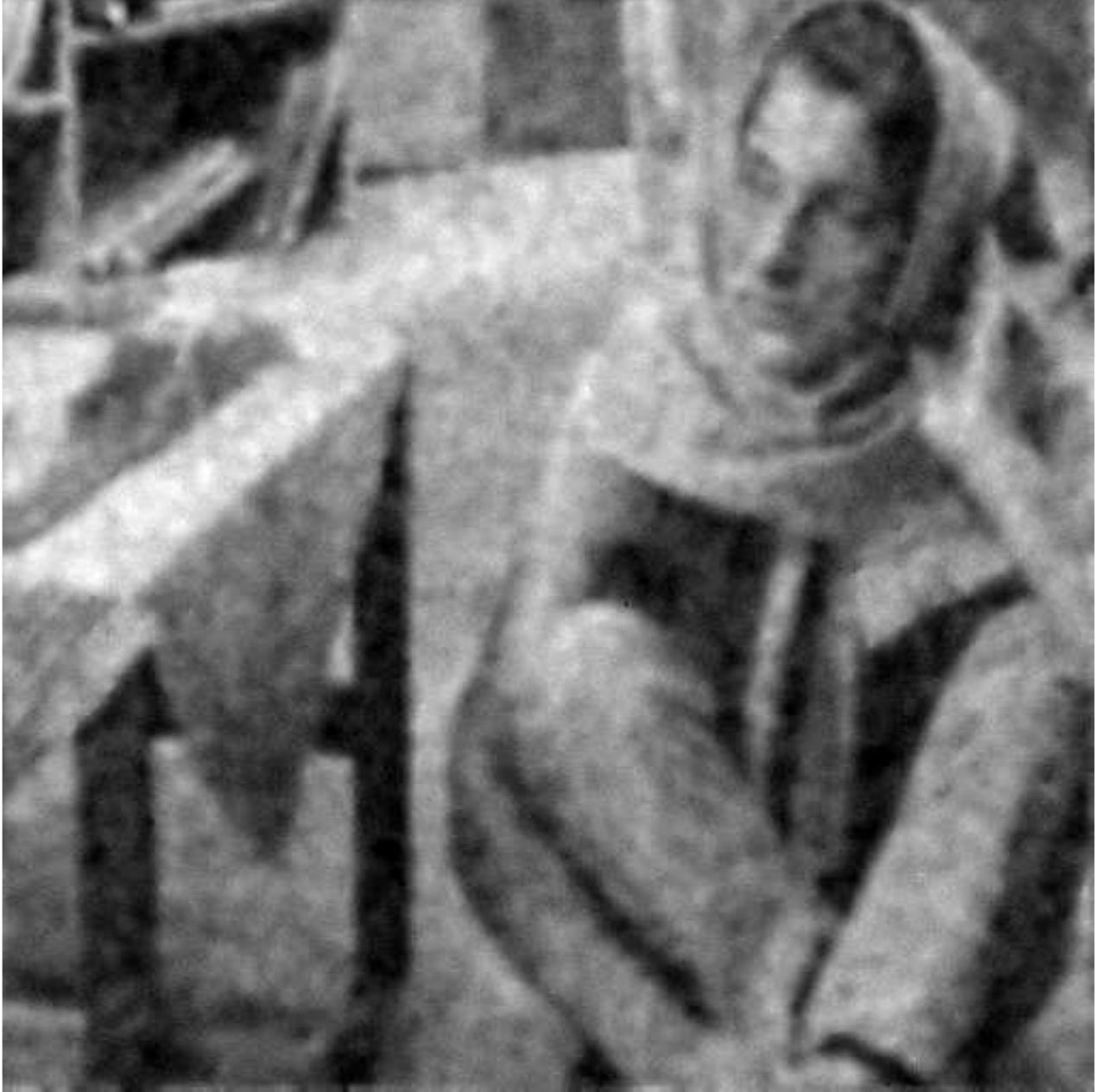}}~~
\subfigure[]{\includegraphics[width=.08\textwidth]{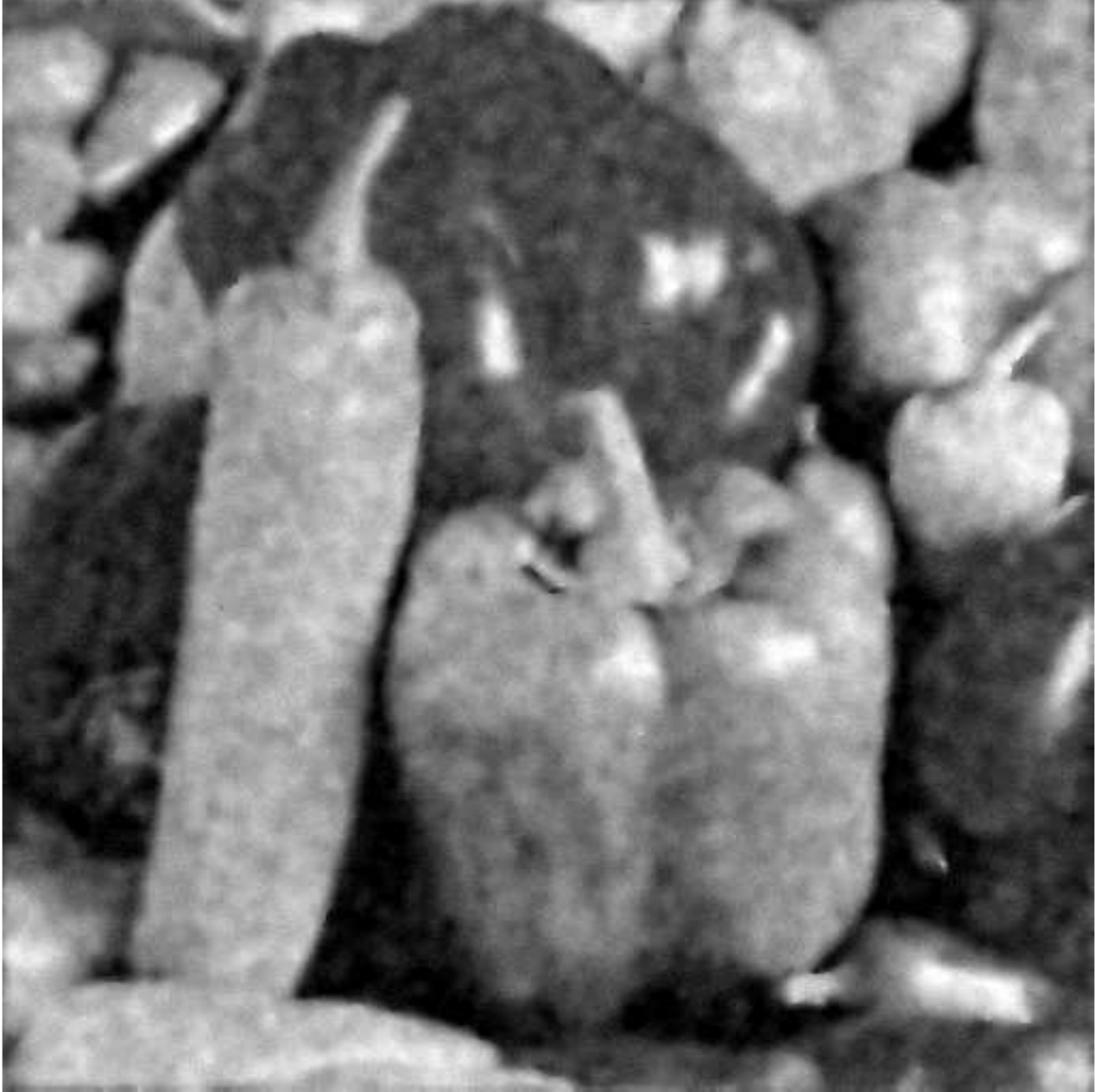}}~~
\subfigure[]{\includegraphics[width=.08\textwidth]{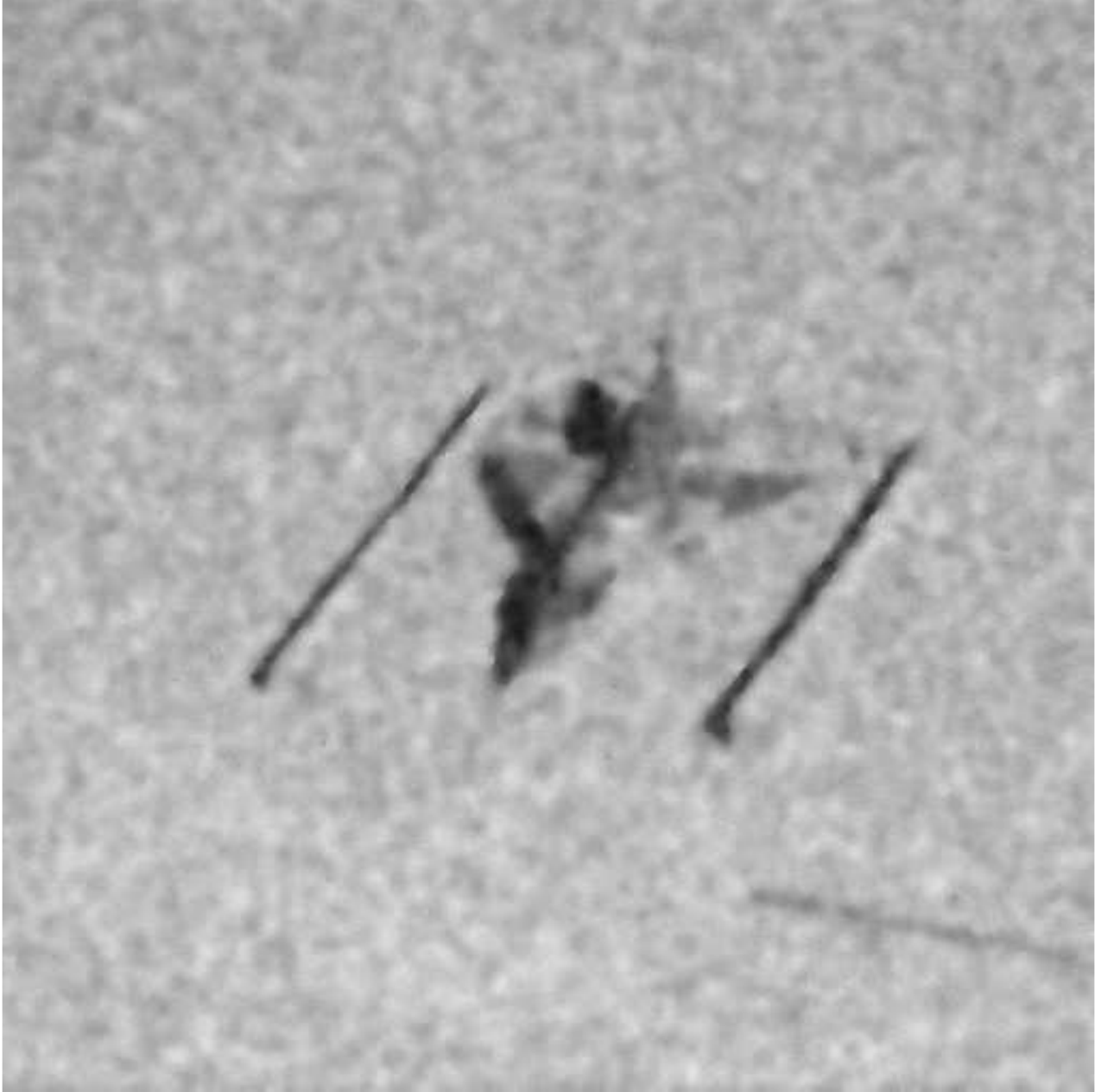}}\\
\subfigure[]{\includegraphics[width=.08\textwidth]{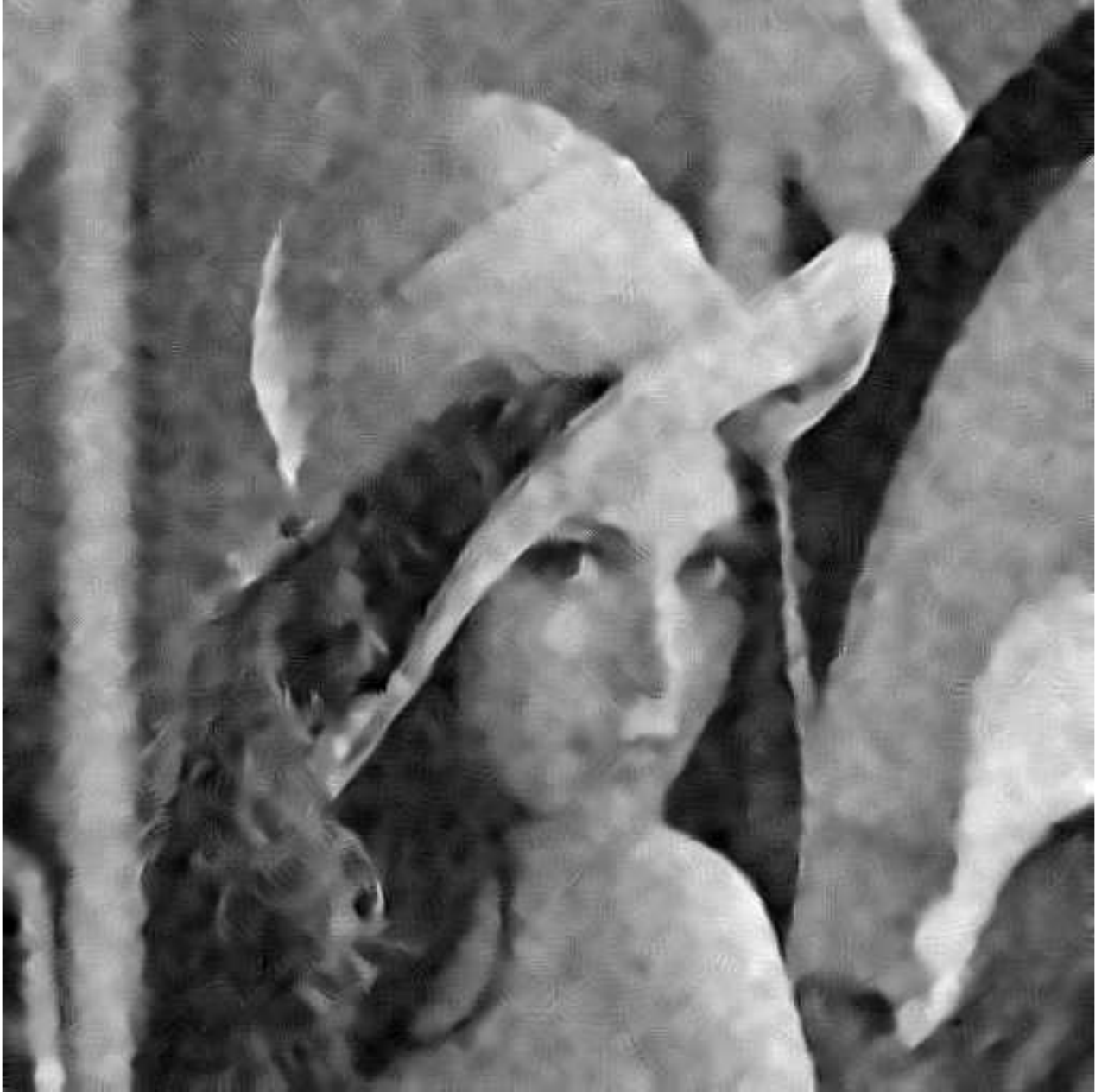}}~~
\subfigure[]{\includegraphics[width=.08\textwidth]{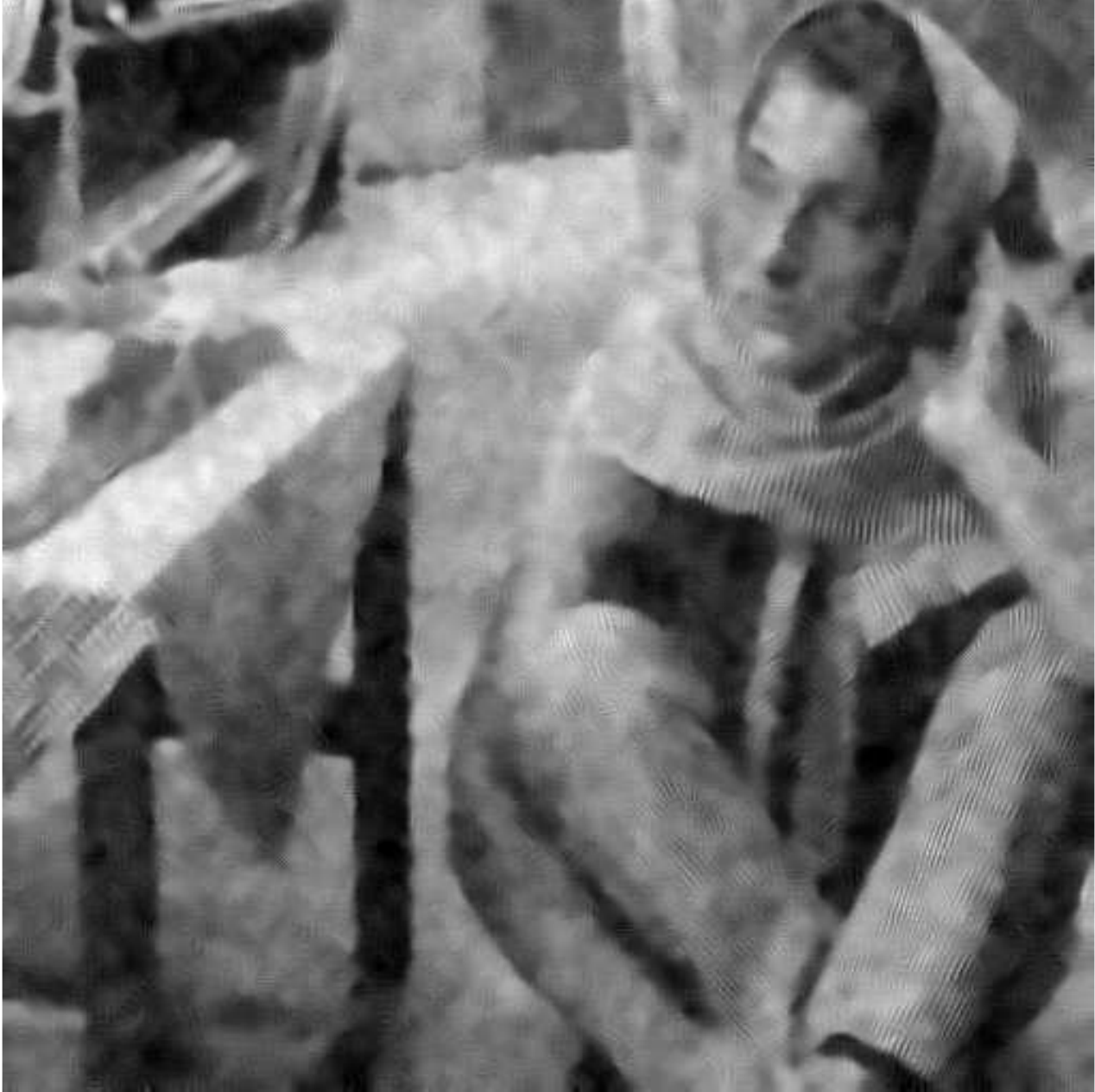}}~~
\subfigure[]{\includegraphics[width=.08\textwidth]{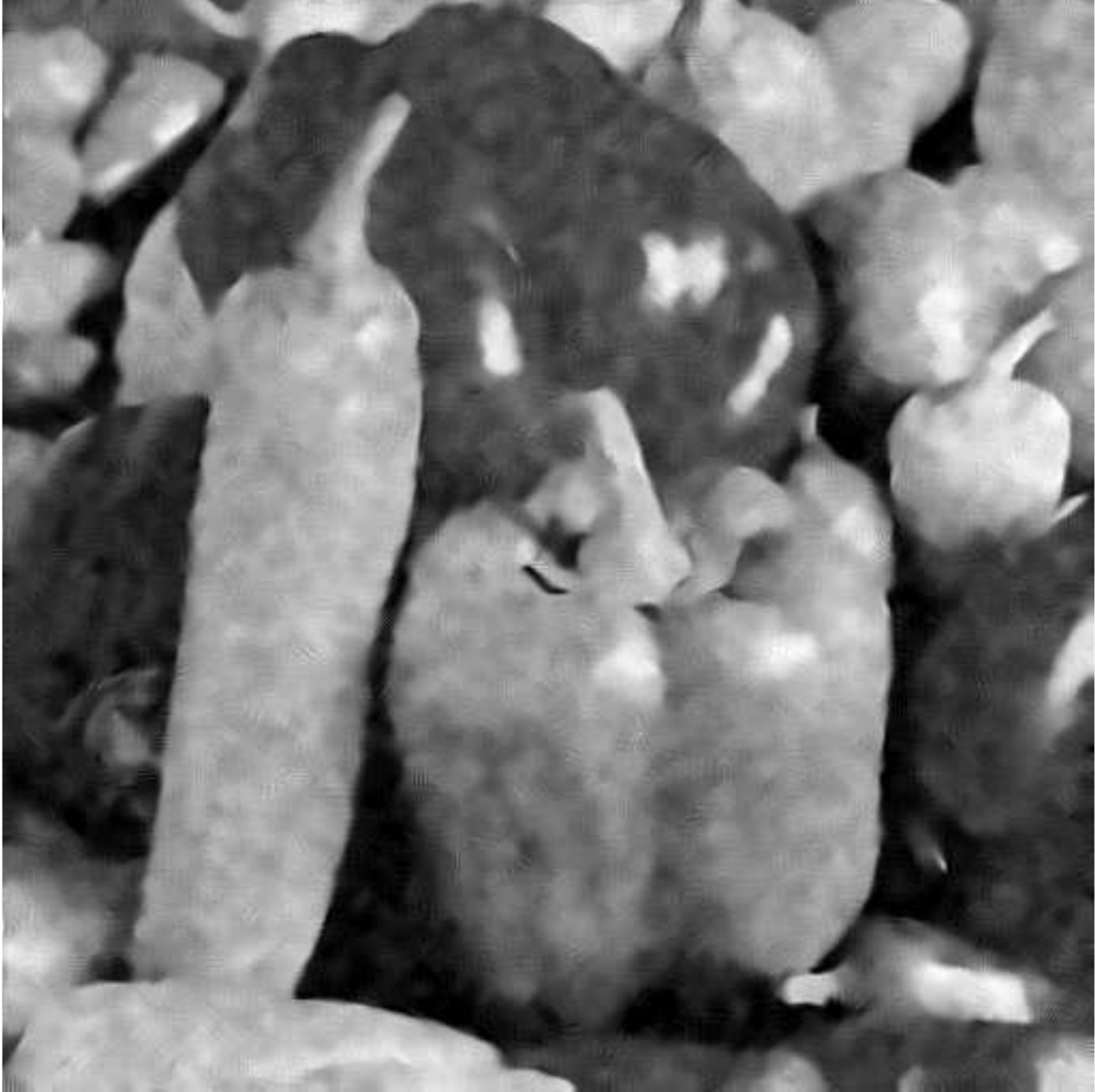}}~~
\subfigure[]{\includegraphics[width=.08\textwidth]{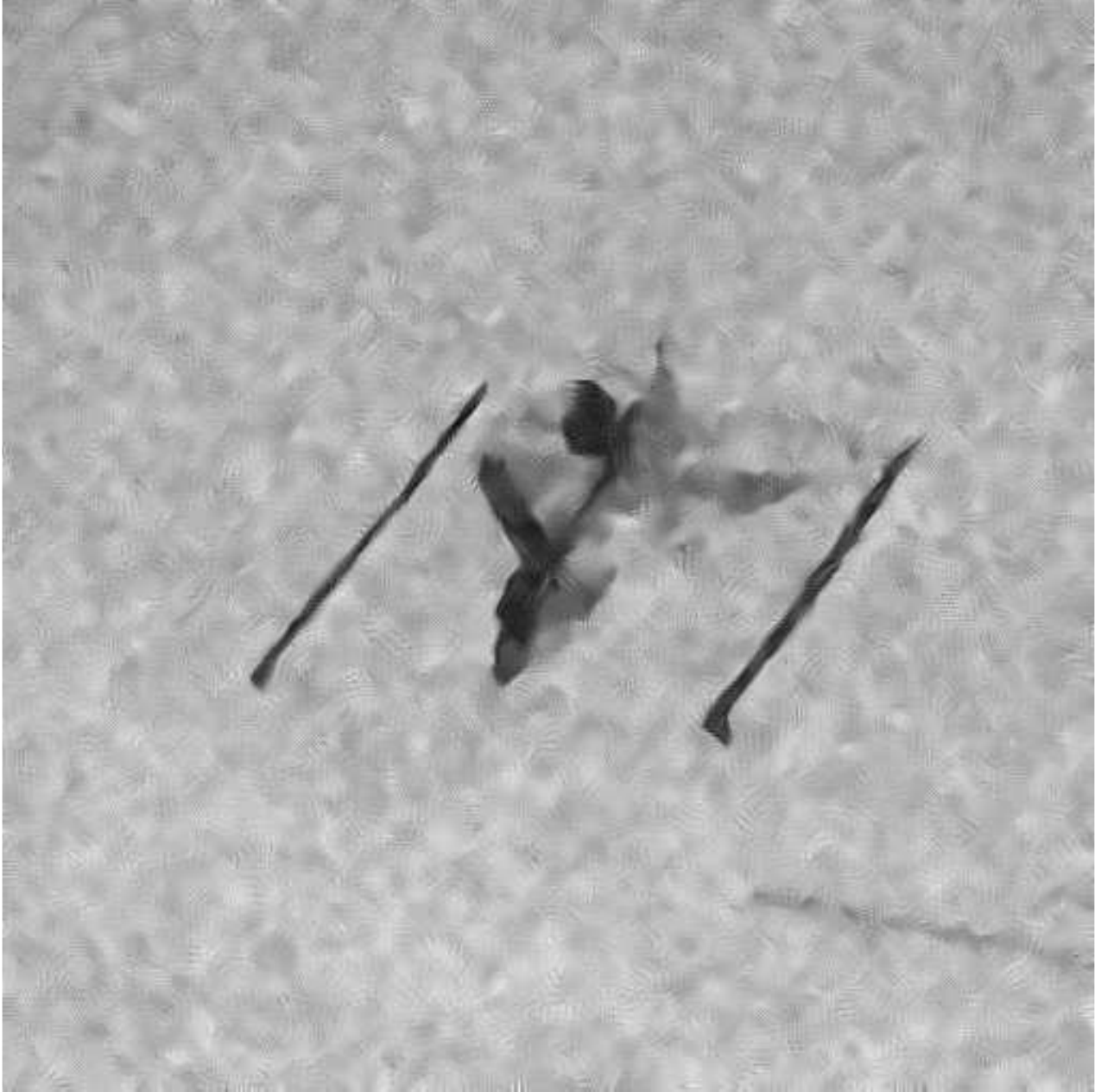}}\\
\subfigure[]{\includegraphics[width=.08\textwidth]{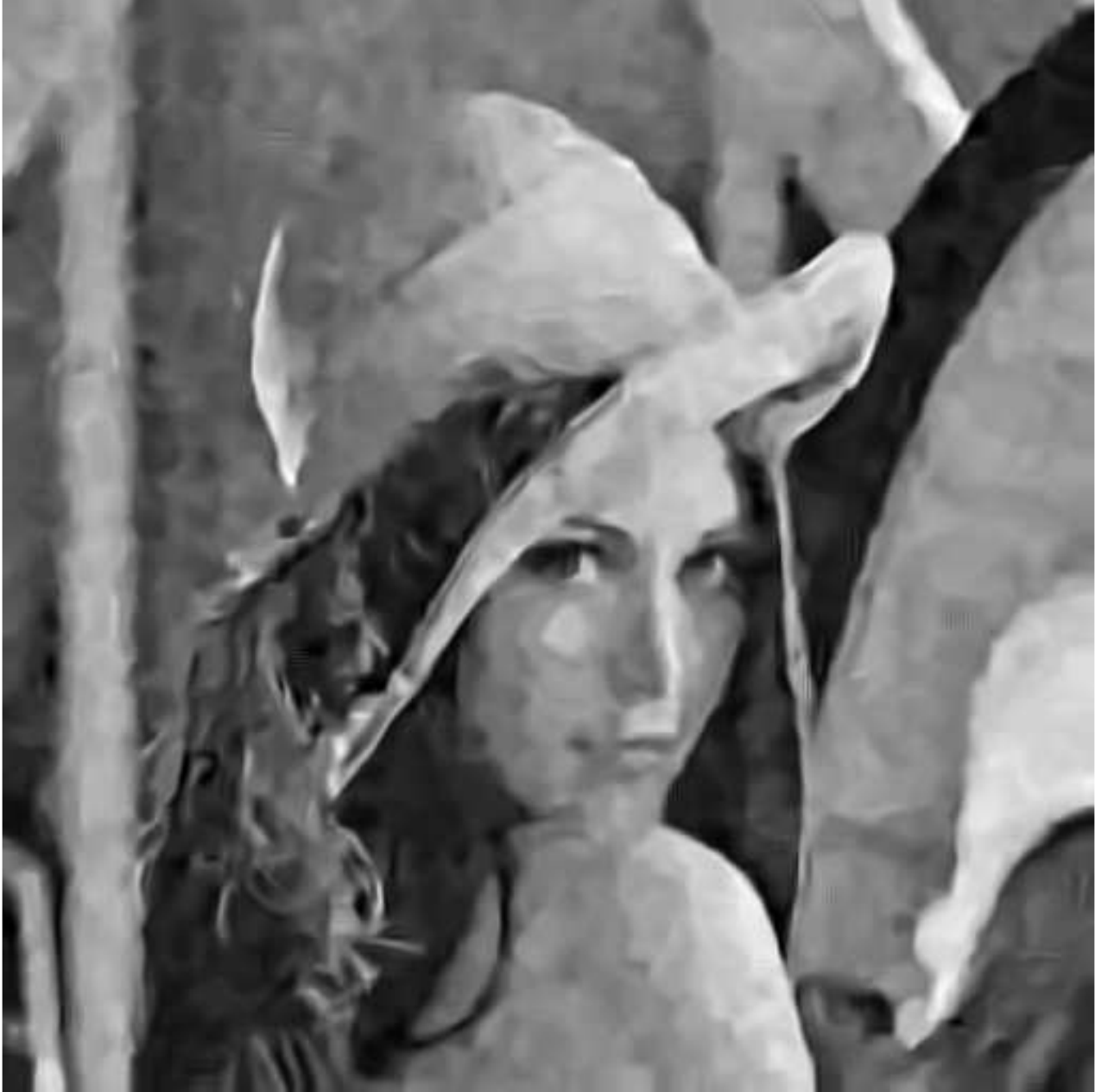}}~~
\subfigure[]{\includegraphics[width=.08\textwidth]{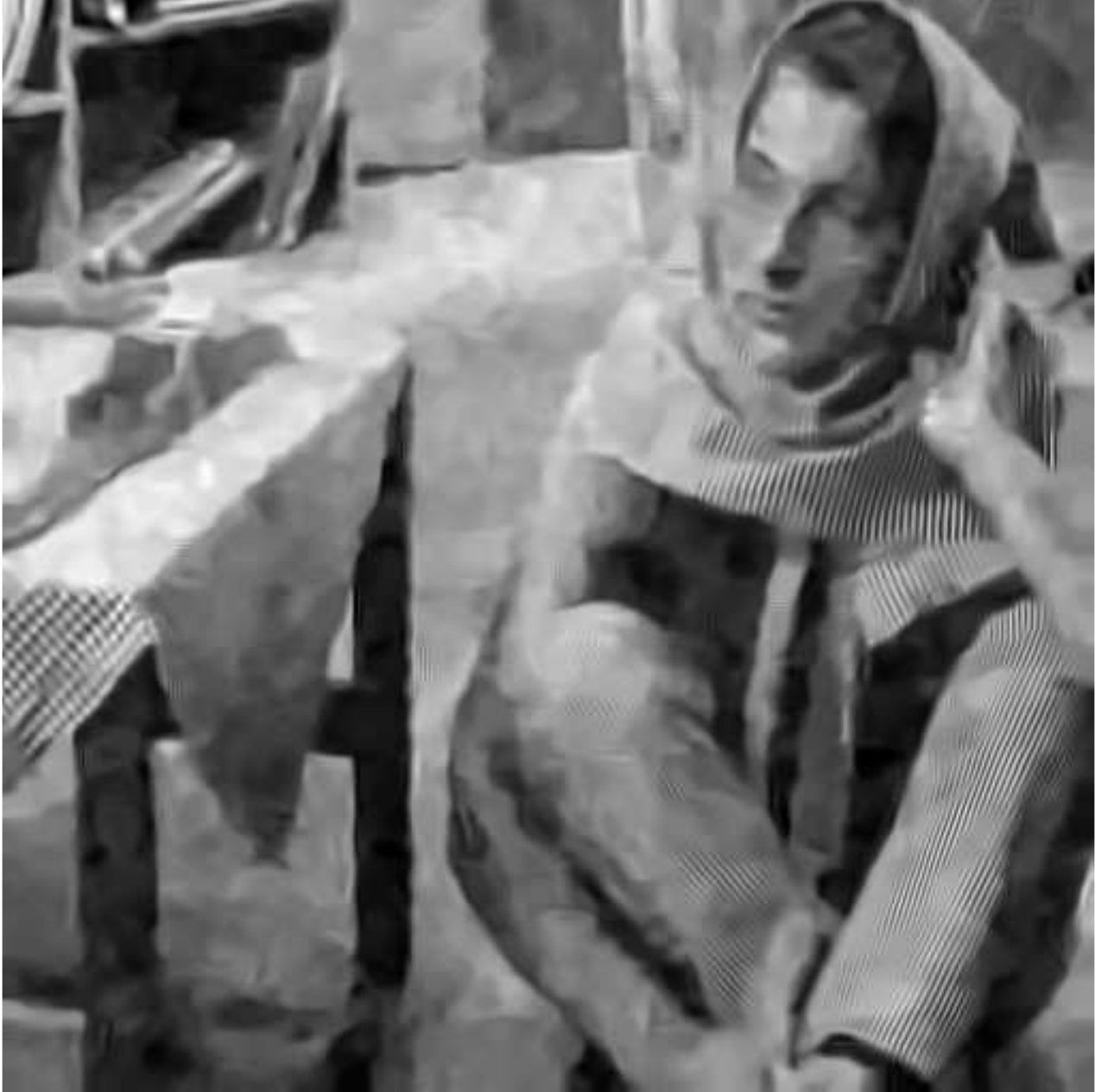}}~~
\subfigure[]{\includegraphics[width=.08\textwidth]{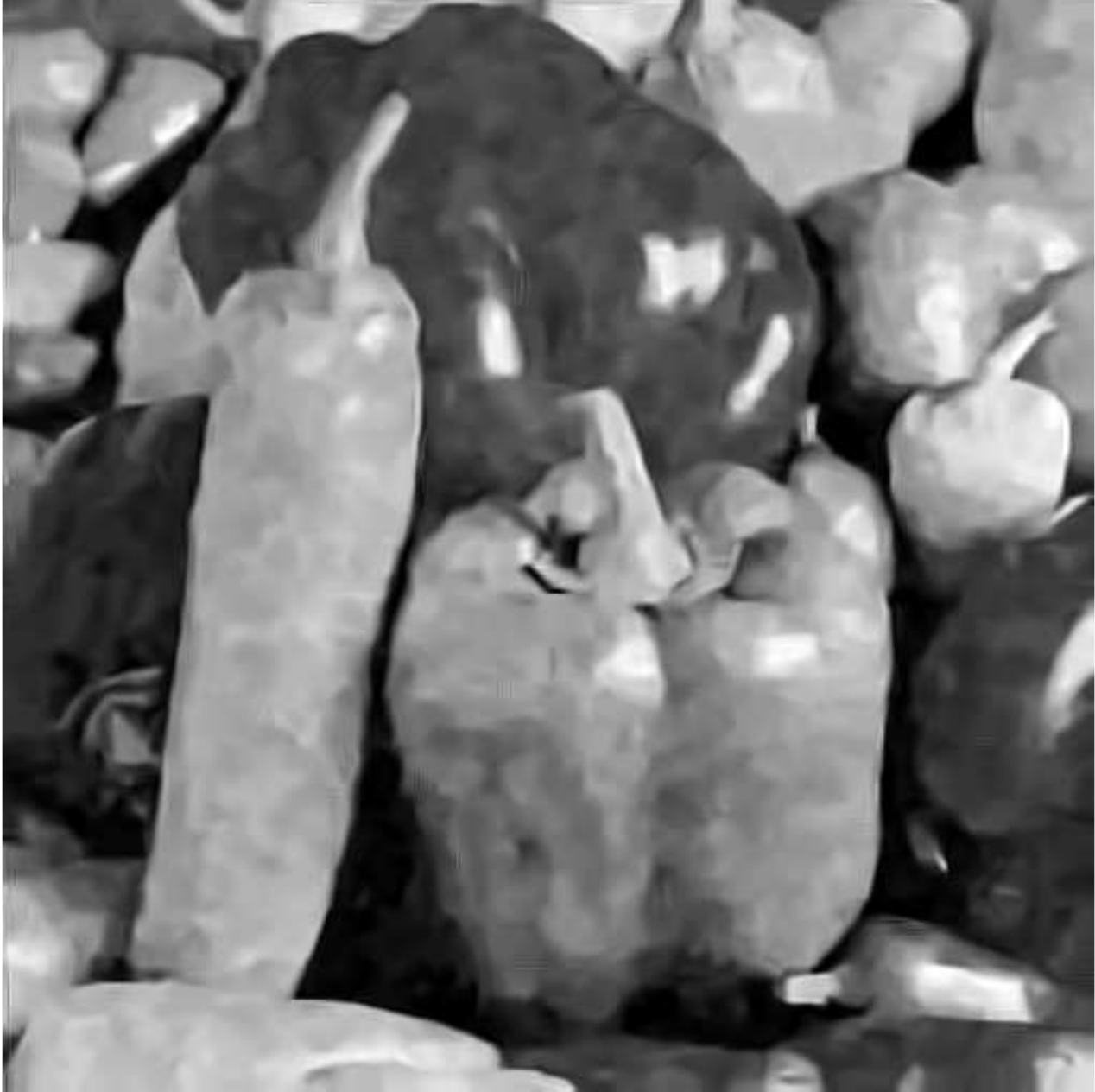}}~~
\subfigure[]{\includegraphics[width=.08\textwidth]{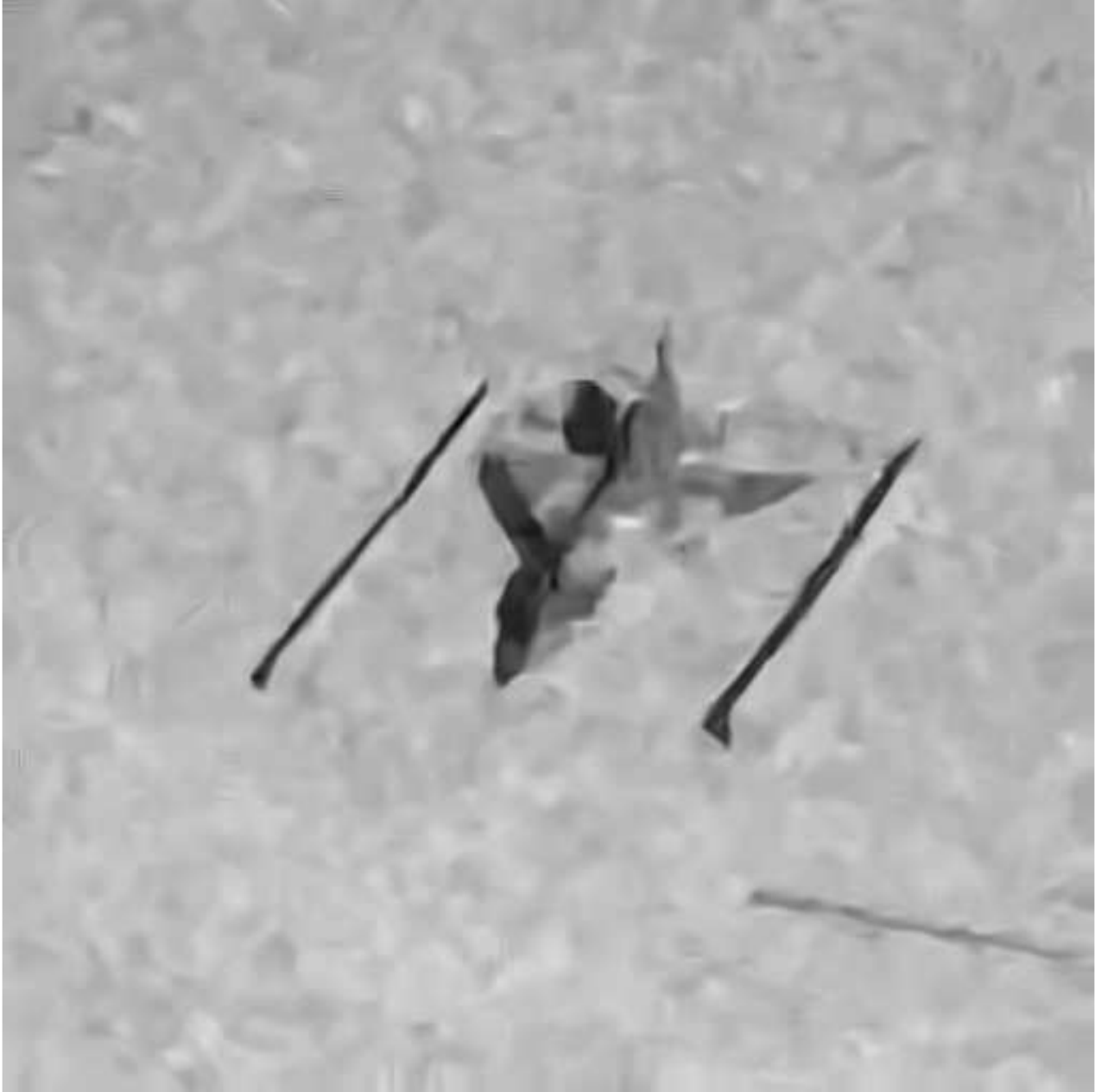}}
\end{center}
\caption{PR with CDP. Peak level $\nu=3.0\times10^{-3}$ for Poisson noise. First row: ``LS-PR''; Second row: ``TV-PR''; Third row: ``TGV-PR''; Fourth row: ``NLM-PR''; Fifth row: ``BM3D-PR''.}
\label{poicdp1}
\end{figure}

\begin{figure}
\begin{center}
\subfigure[``Lena'']{\includegraphics[width=.11\textwidth]{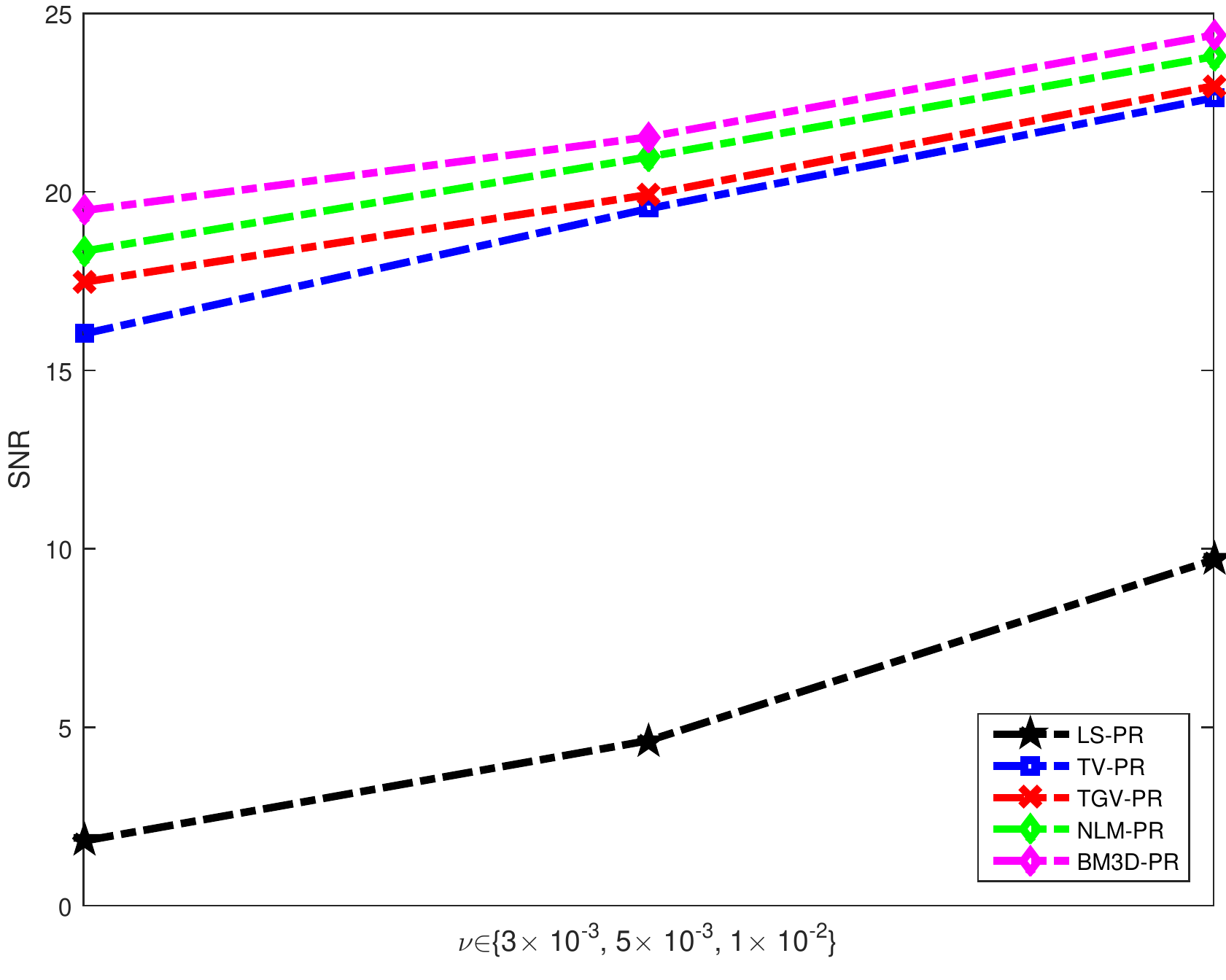}}
\subfigure[``Barbara'']{\includegraphics[width=.11\textwidth]{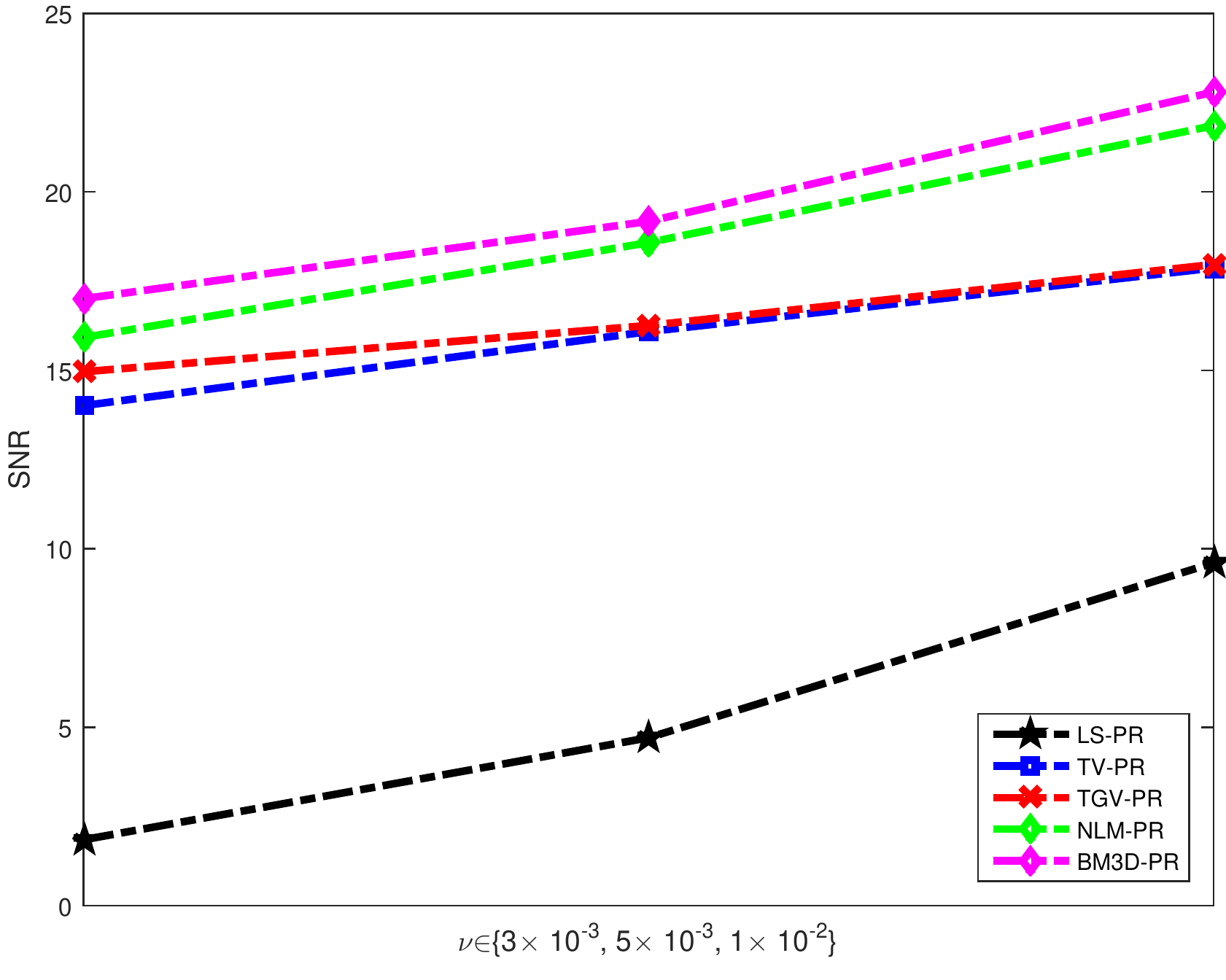}}
\subfigure[``Peppers'']{\includegraphics[width=.11\textwidth]{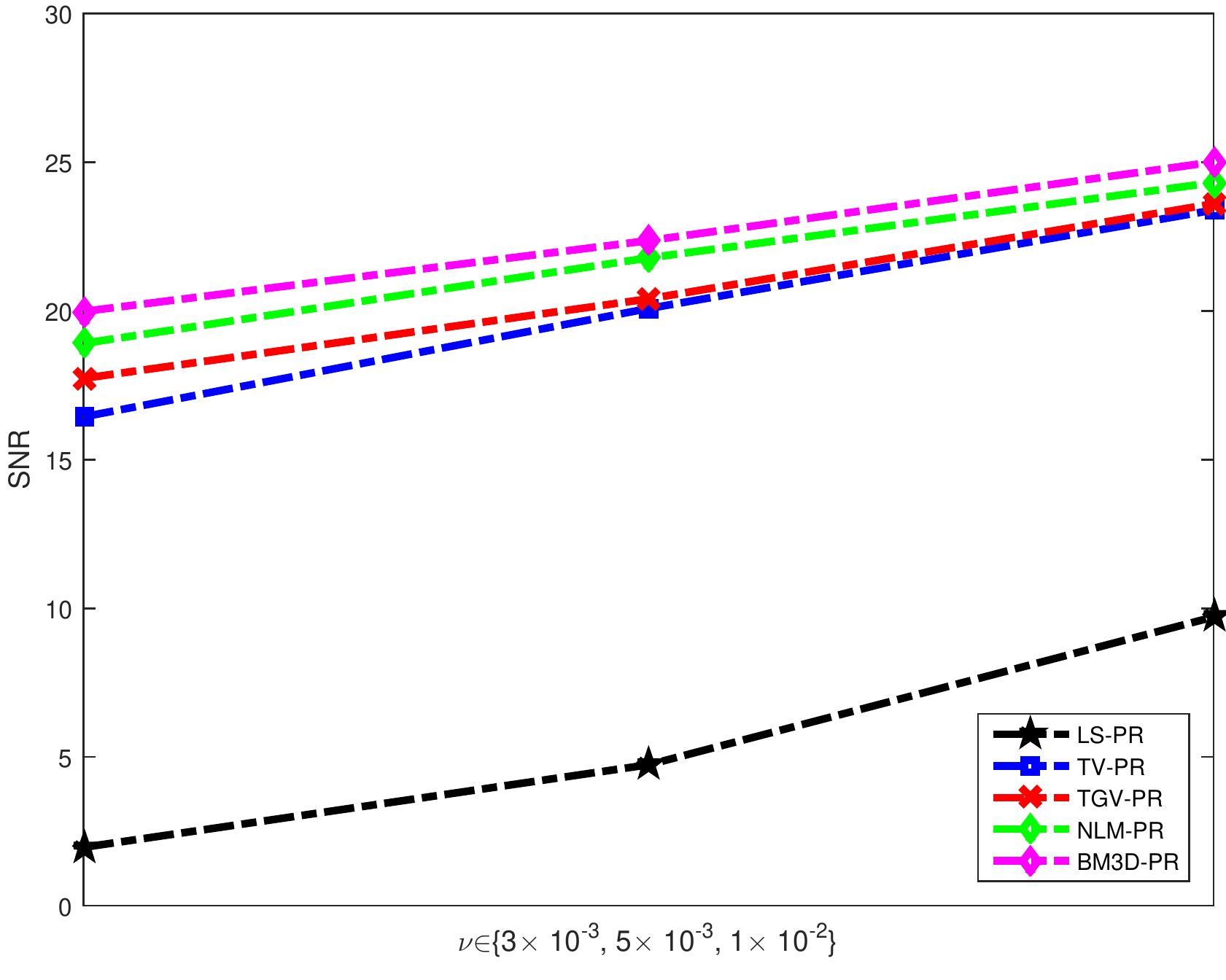}}
\subfigure[``Plane'']{\includegraphics[width=.11\textwidth]{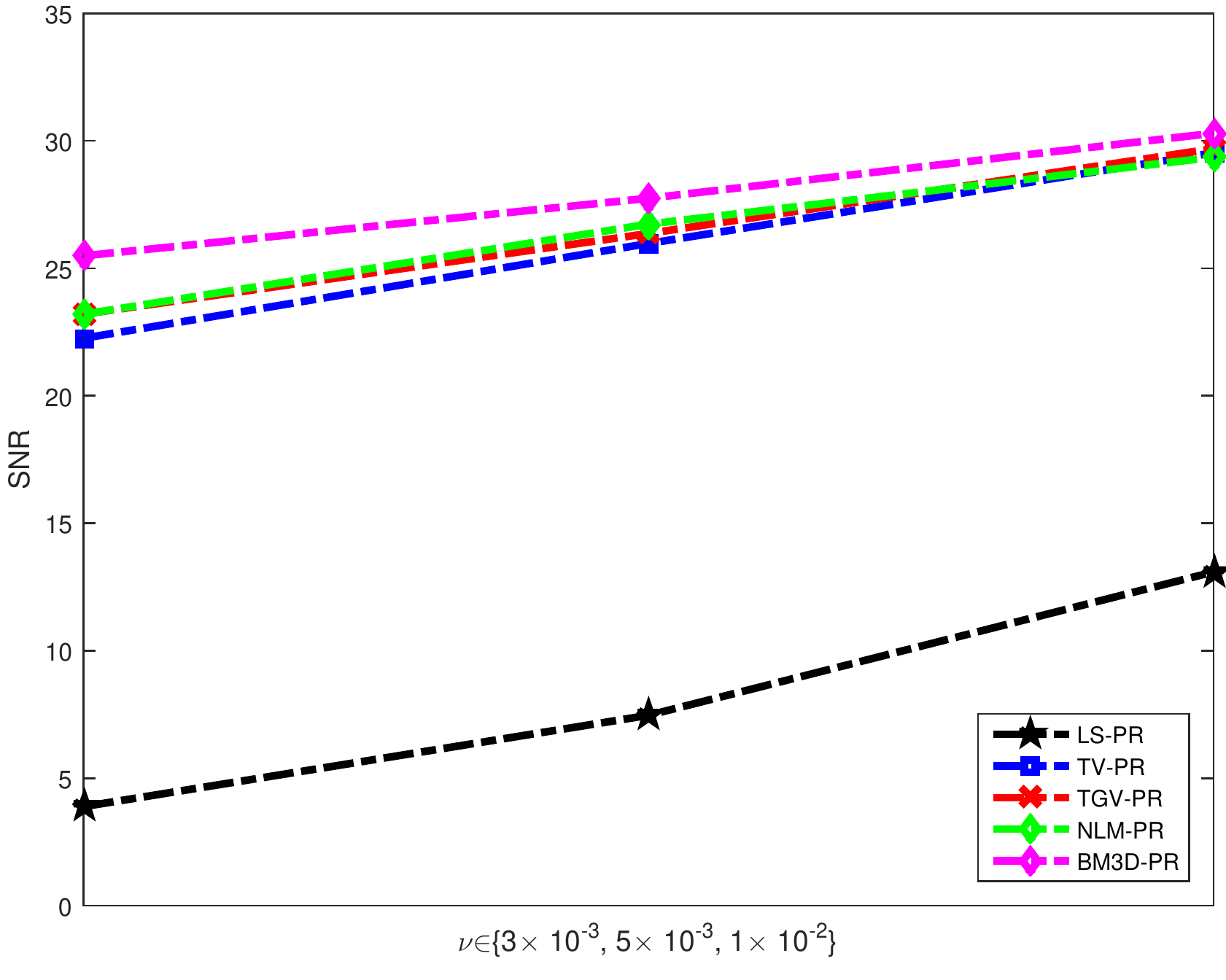}}
\end{center}
\caption{SNRs for images in Figure \ref{poicdp4} v.s. different peak levels $\nu$ by the methods with CDP.}
\label{snr2}
\end{figure}

\begin{table}
\begin{center}
\begin{spacing}{1.}
\begin{tabular}{|c|c|c|c|c|}
\hline
Method&Peak level $\nu$ &$\lambda$&r&$\eta$\\\hline
\multirow{3}{*}{TGV-PR}&$3.0\times 10^{-3}$&$7.0\times 10^{2}$&\multirow{3}{*}{$5.0\times 10^{5}$}&\multirow{9}{*}{$50$}\\
      &$5.0\times 10^{-3}$&$6.0\times 10^{2}$&&\\
      &$1.0\times 10^{-2}$&$5.0\times 10^{2}$&&\\\cline{1-4}
\multirow{3}{*}{NLM-PR}&$3.0\times 10^{-3}$&$6.0\times 10^{4}$&\multirow{3}{*}{$1.0\times 10^{6}$}&\\
      &$5.0\times 10^{-3}$&$5.0\times 10^{4}$&&\\
      &$1.0\times 10^{-2}$&$3.0\times 10^{4}$&&\\\cline{1-4}
\multirow{3}{*}{BM3D-PR}&$3.0\times 10^{-3}$&$1.5\times 10^{5}$&\multirow{3}{*}{$5.0\times 10^{5}$}&\\
      &$5.0\times 10^{-3}$&$2.0\times 10^{5}$&&\\
      &$1.0\times 10^{-2}$&$2.0\times 10^{5}$&&\\
      \hline
\end{tabular}
\end{spacing}
\end{center}
\caption{Parameters for Poisson noise removal of CDP on real-valued images for Figure \ref{poicdp3}-Figure \ref{poicdp1}.}
\label{tab1}
\end{table}

Numerical experiments for CDP is also performed on  the complex-valued image of Figure \ref{groundtruth} (e), and we show the recovery results in Figure \ref{poicdp4} with different noise level by setting the peak level $\nu\in\{5.0\times10^{-2},8.0\times 10^{-2},1.0\times 10^{-1}\}$. One can readily see that with high level noise as the first column of Figure \ref{poicdp4}, ``TV-PR'', and our proposed methods can remove the noise in the background. However, ``TV-PR'' and ``TGV-PR'' can not preserve the small structure at all. It seems that ``NLM-PR'' learns the wrong feature patterns.
Parameters used are put in Table \ref{tab2}. ``BM3D-PR'' is the most effective, which not only obtain clean background, but also recover some smaller structures. When the noise level decreases, ``NLM-PR'' still fails to find correct patterns for smaller structures, while the other methods can preserve both the larger and smaller scales features. We put the SNRs in  Figure \ref{snr1}, and one can observe the increase of SNRs by our proposed methods. The average SNRs are 2.64, 10.15, 11.78, 10.50 and 12.76 for all the five methods, and ``BM3D-PR'' improves the reconstructed images with highest SNRs among them, and about 2.5dB is increased compared with those by ``TV-PR'' even for such complicated image. Moreover,  the SNR increase of ``BM3D-PR'' comparing with those of ``LS-PR'' and ``TV-PR'' is the biggest among  the four different images in Figure \ref{snr1},  which implies that it is quite suitable for the images with textures.
\begin{figure}
\begin{center}
\subfigure[]{\includegraphics[width=.12\textwidth]{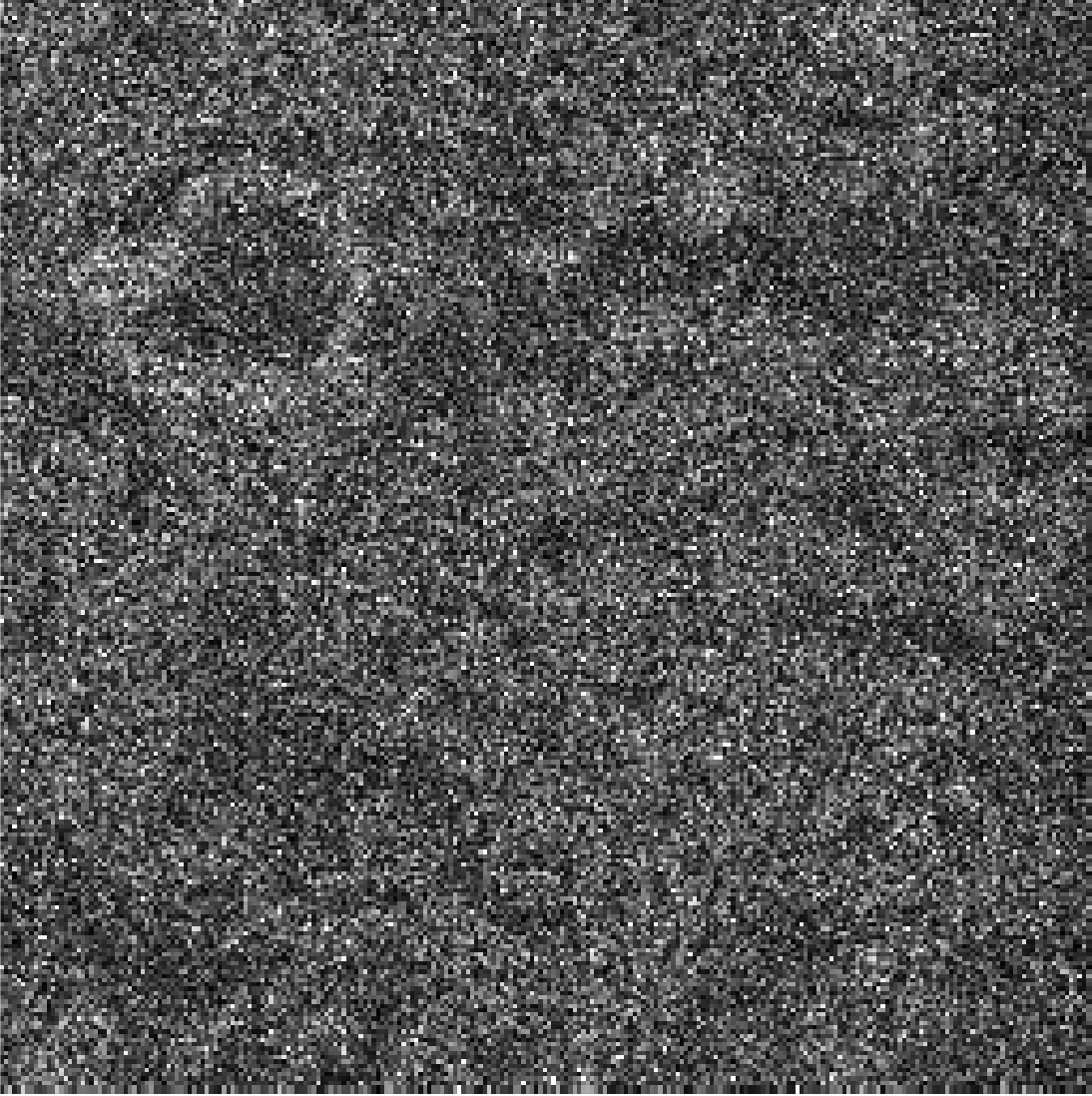}}
\subfigure[]{\includegraphics[width=.12\textwidth]{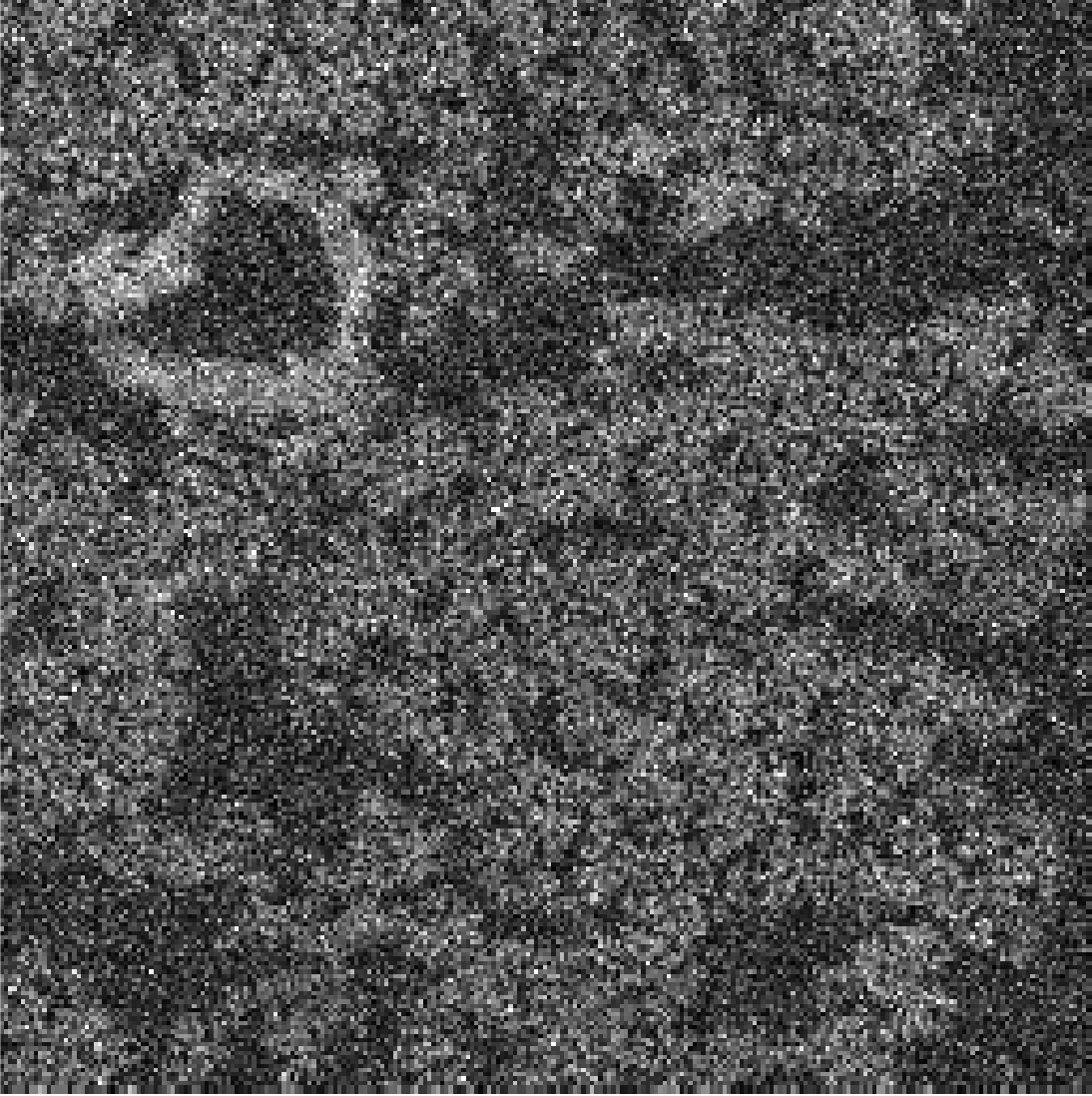}}
\subfigure[]{\includegraphics[width=.12\textwidth]{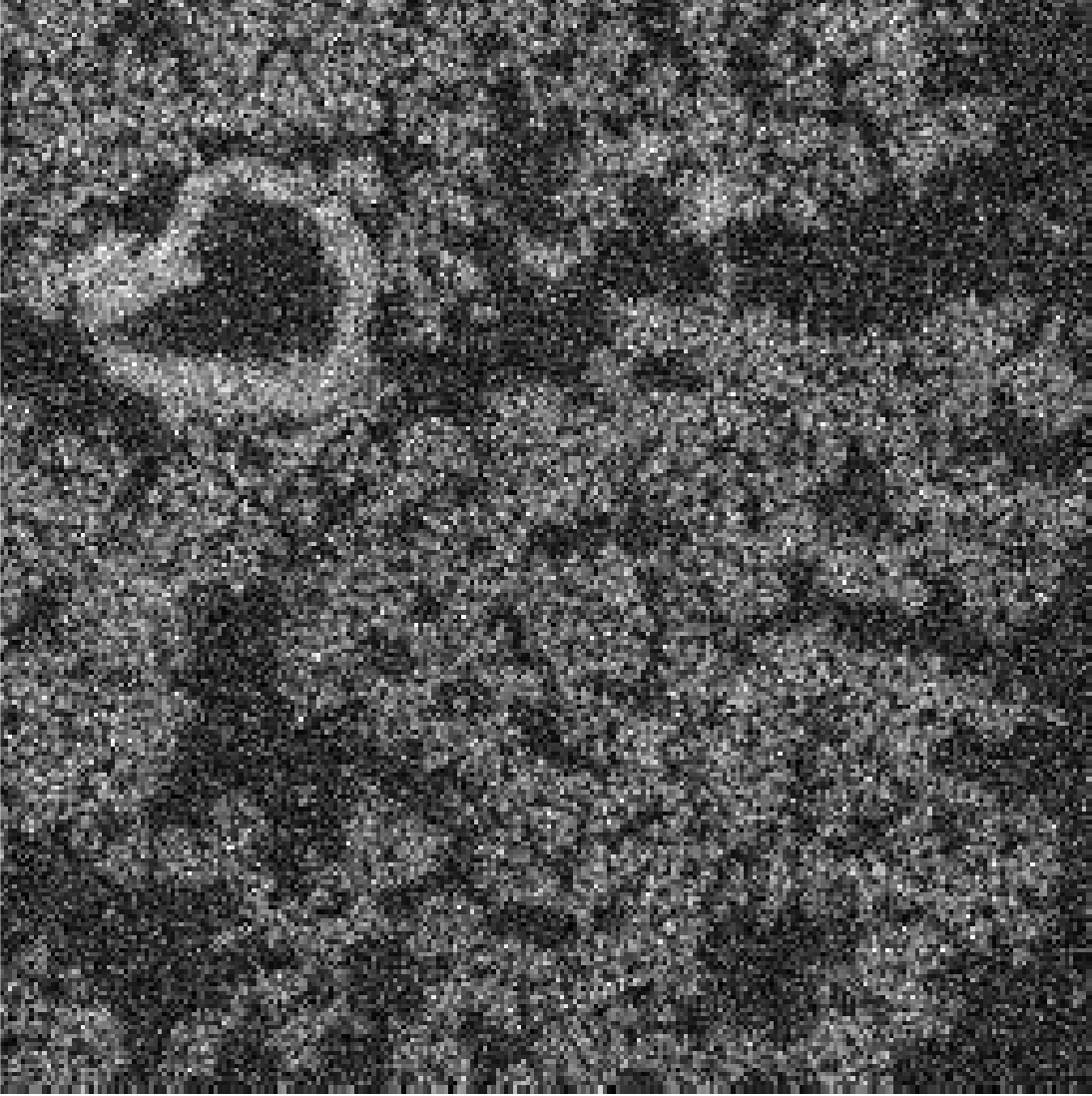}}\\
\subfigure[]{\includegraphics[width=.12\textwidth]{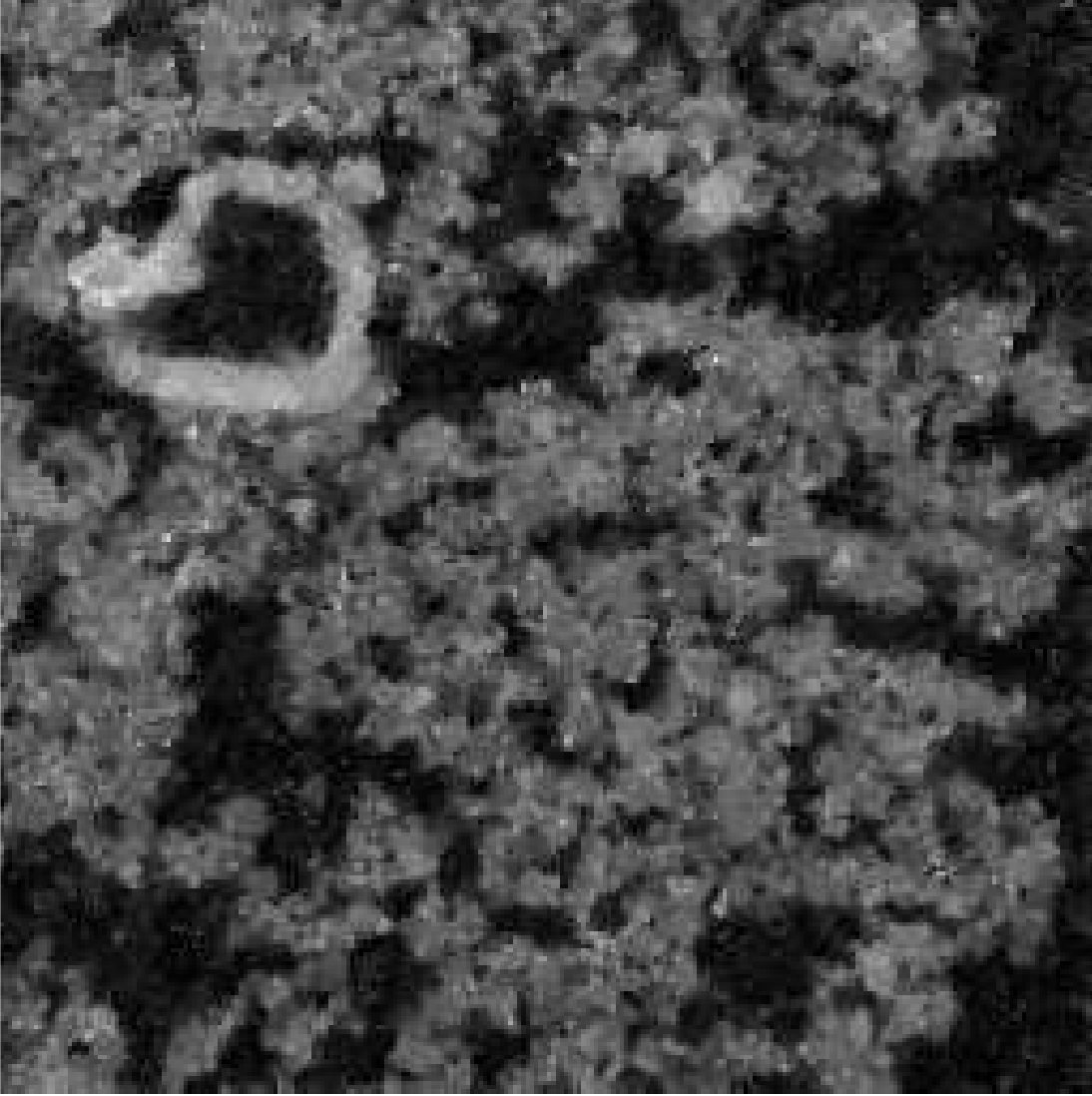}}
\subfigure[]{\includegraphics[width=.12\textwidth]{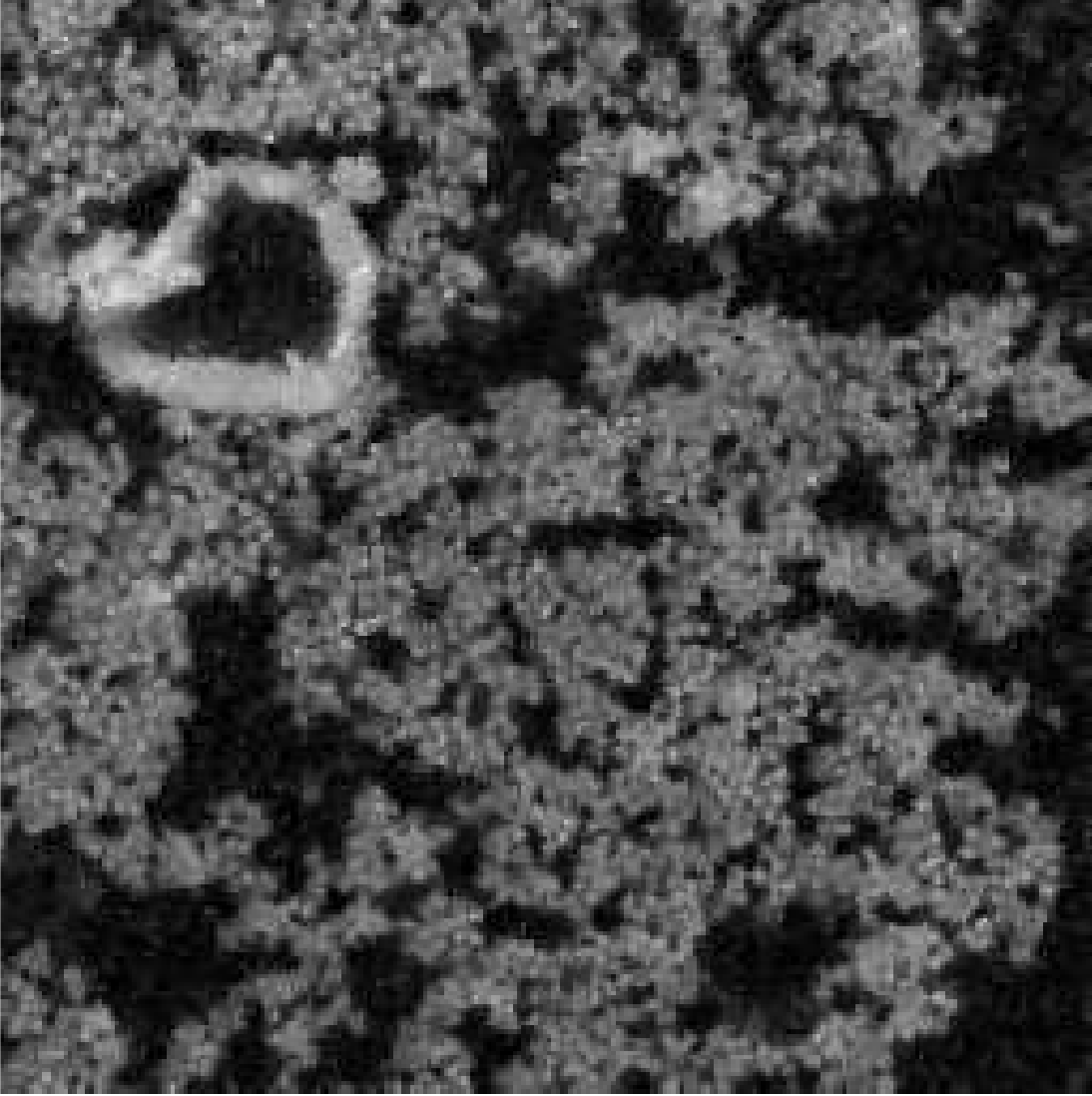}}
\subfigure[]{\includegraphics[width=.12\textwidth]{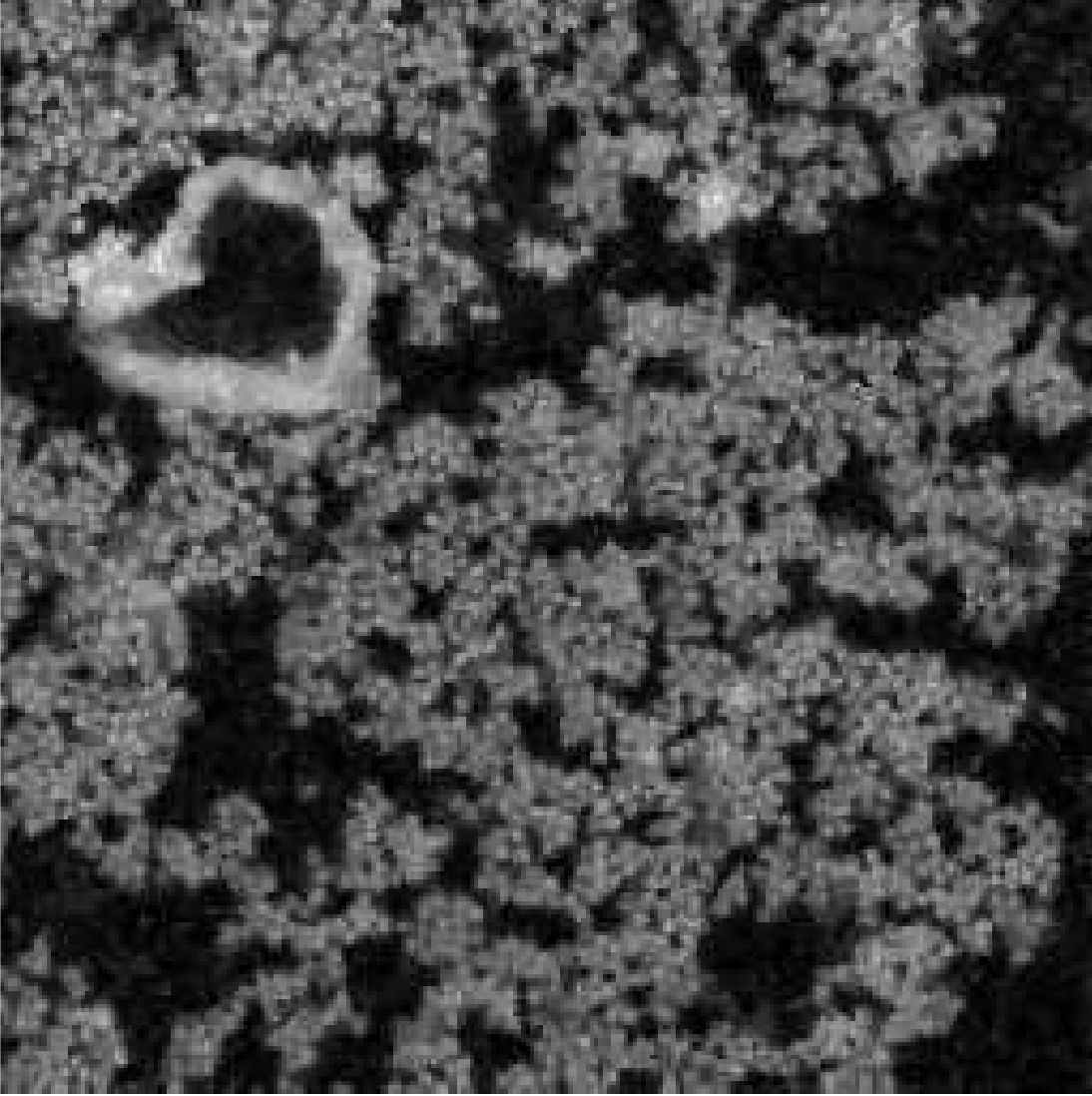}}\\
\subfigure[]{\includegraphics[width=.12\textwidth]{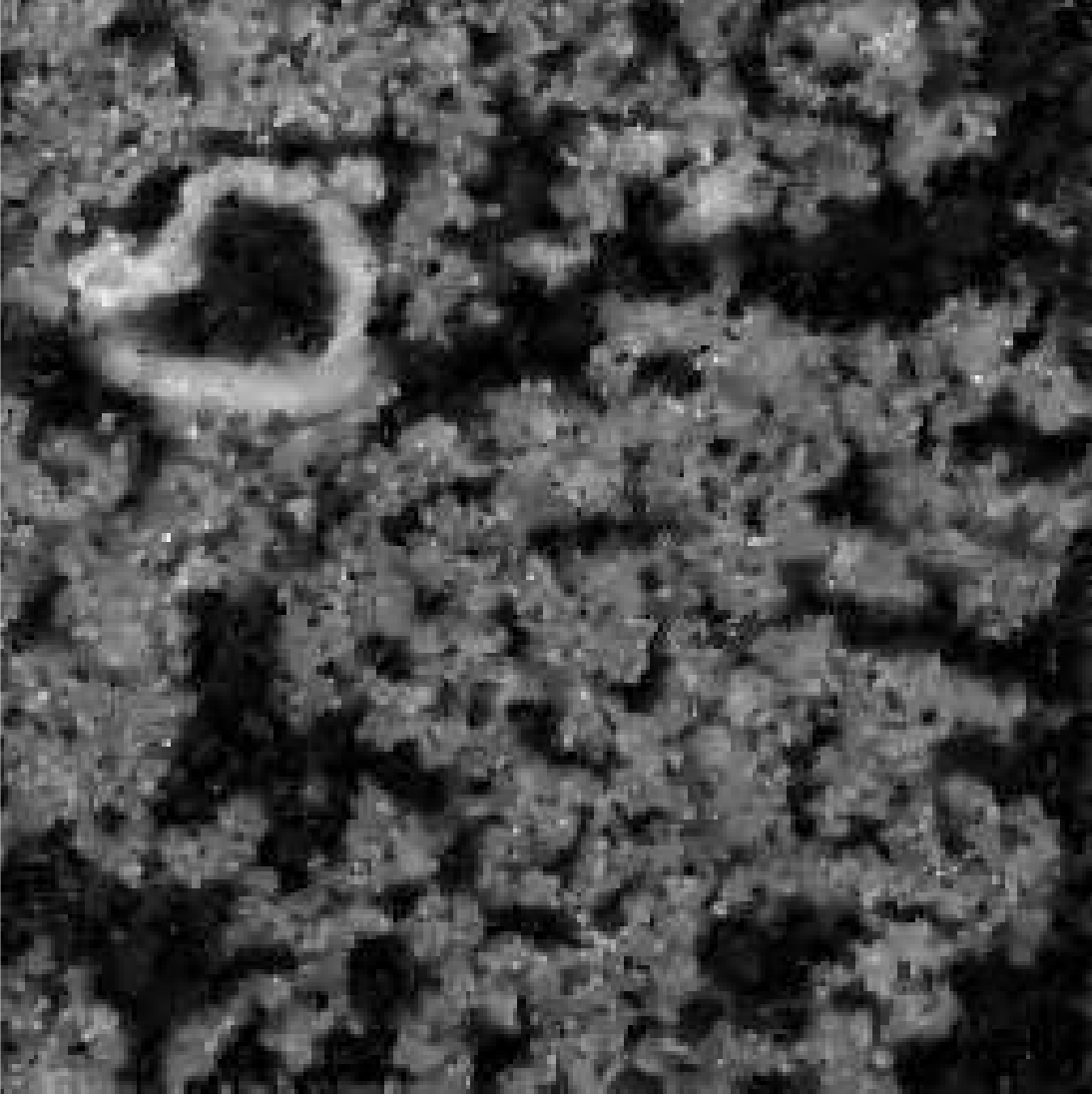}}
\subfigure[]{\includegraphics[width=.12\textwidth]{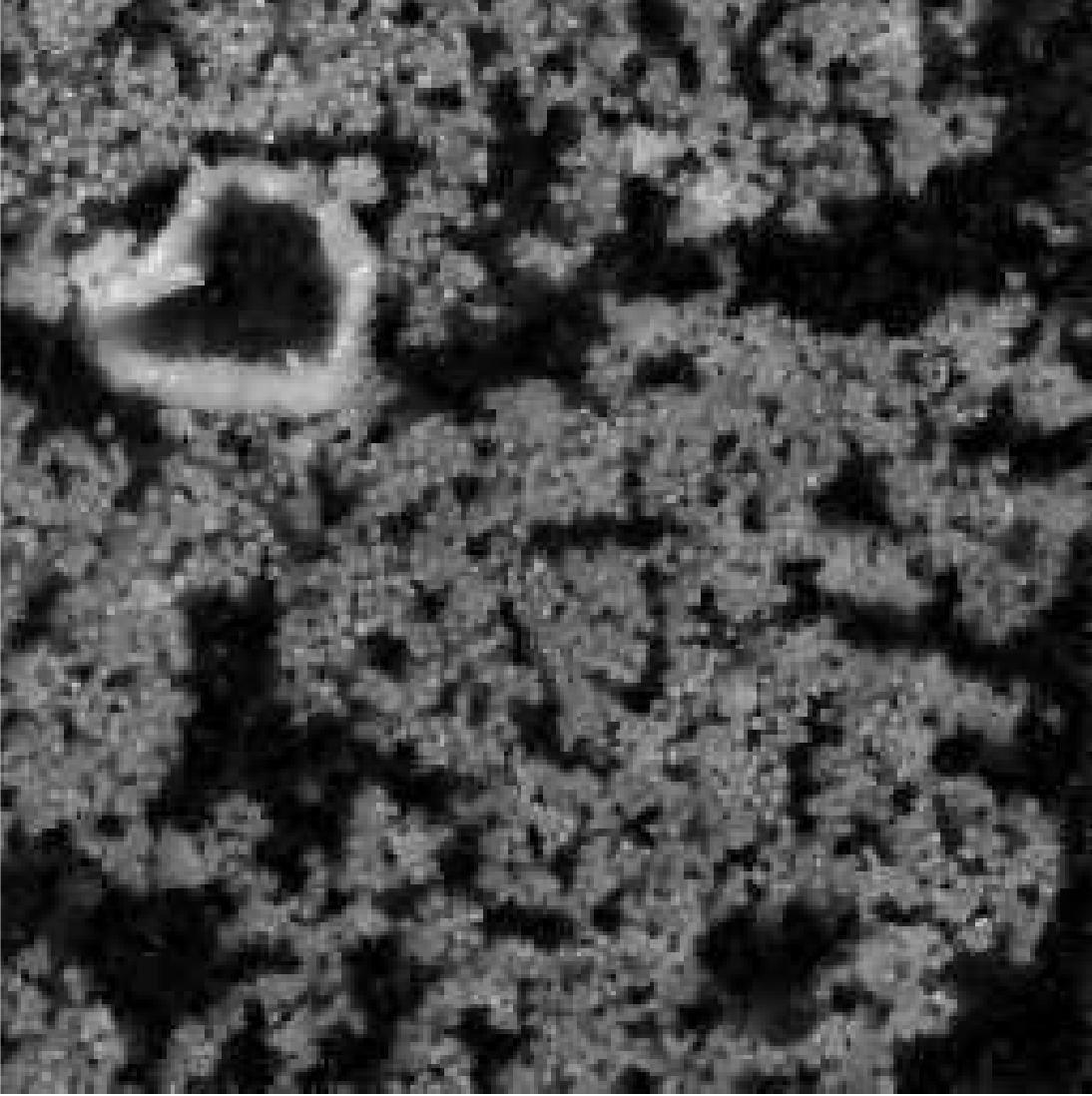}}
\subfigure[]{\includegraphics[width=.12\textwidth]{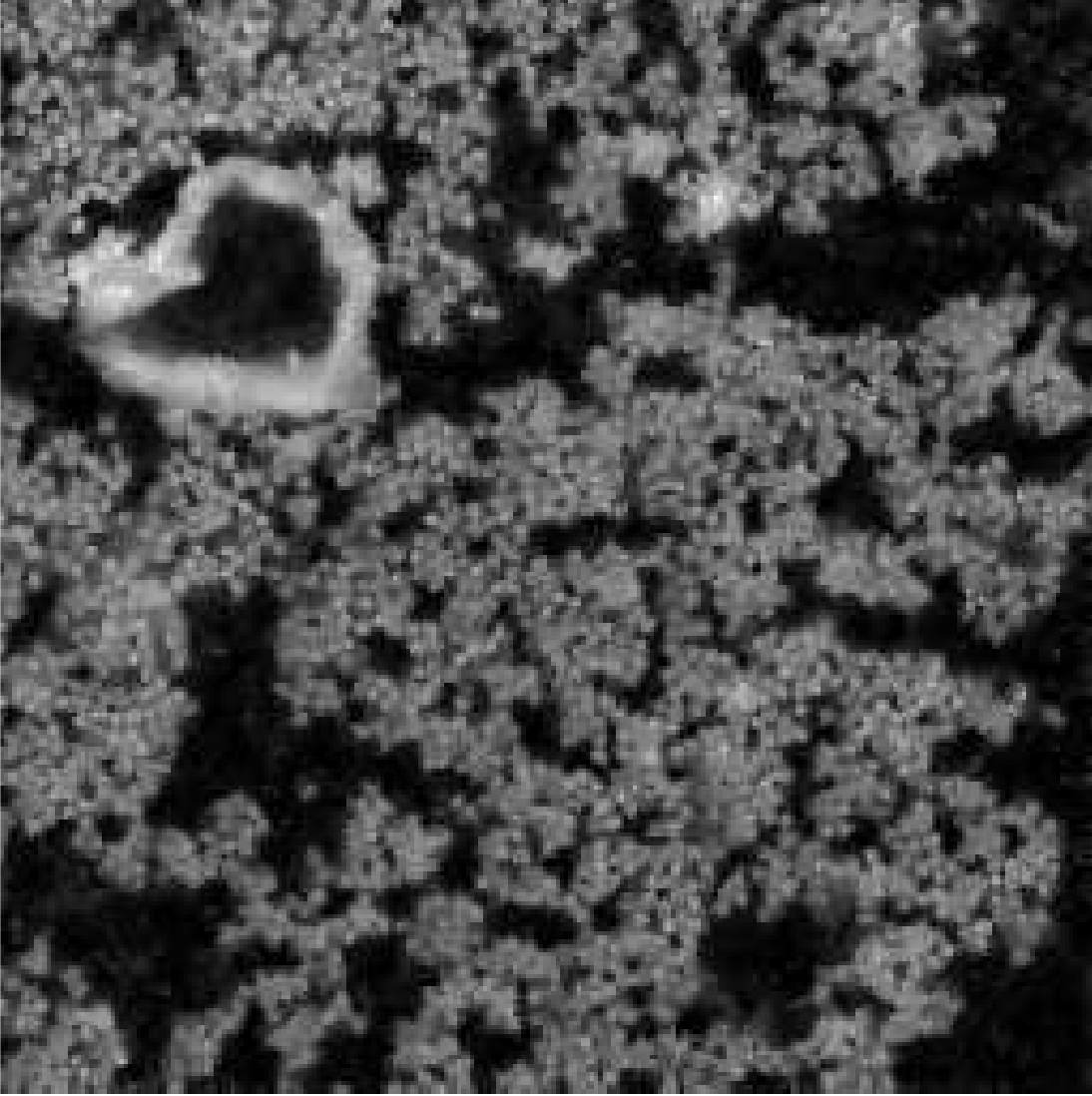}}\\
\subfigure[]{\includegraphics[width=.12\textwidth]{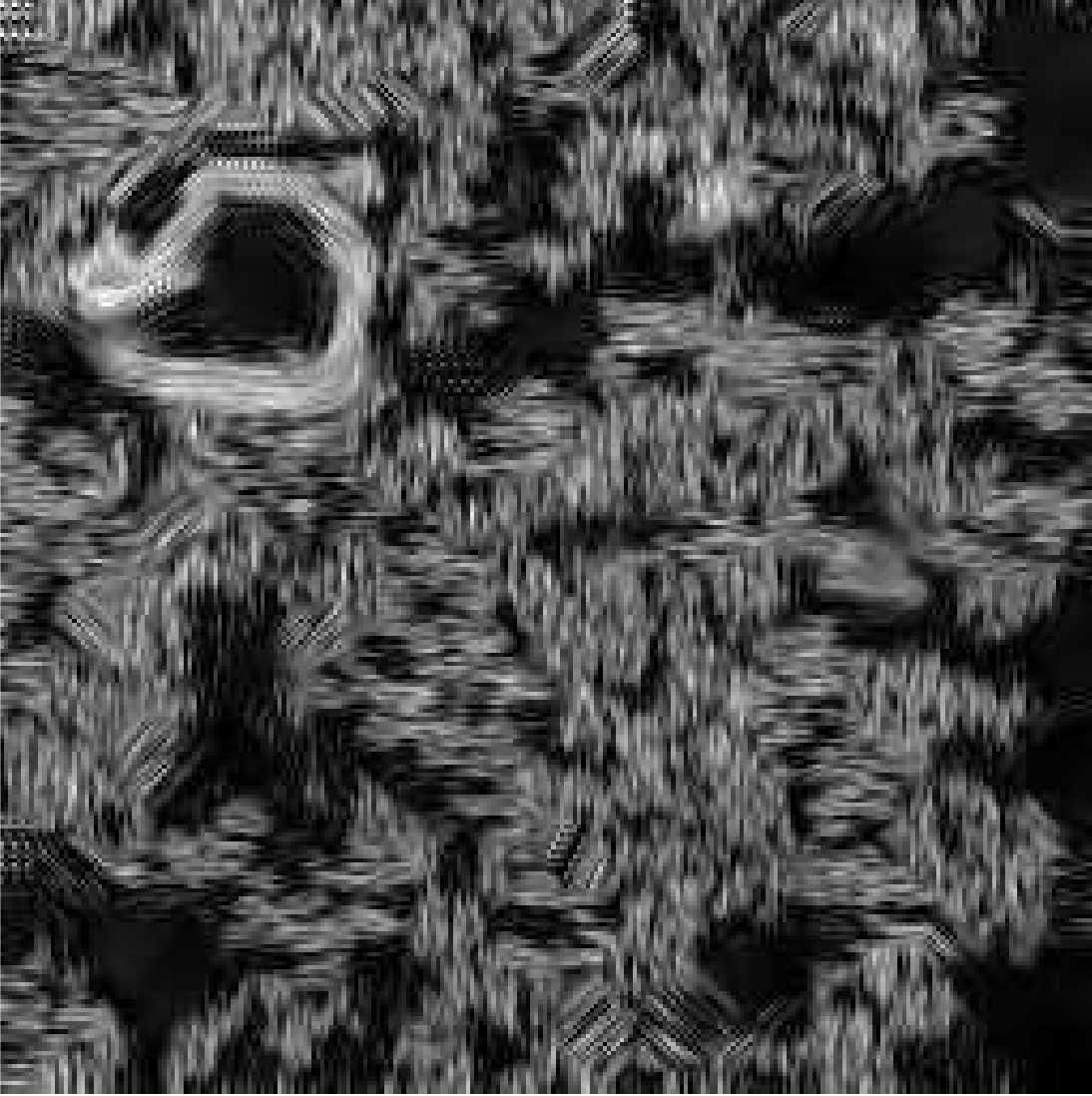}}
\subfigure[]{\includegraphics[width=.12\textwidth]{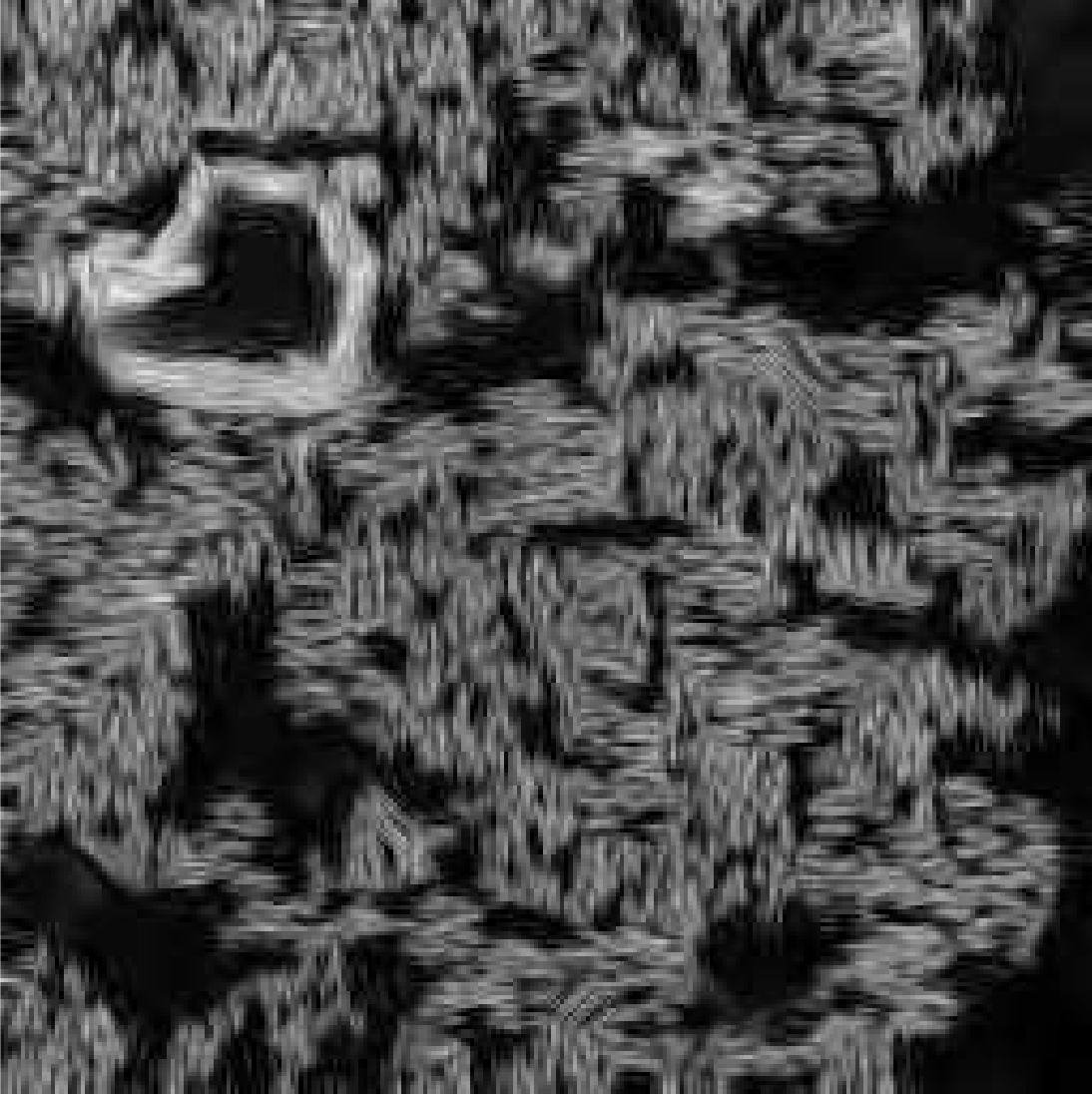}}
\subfigure[]{\includegraphics[width=.12\textwidth]{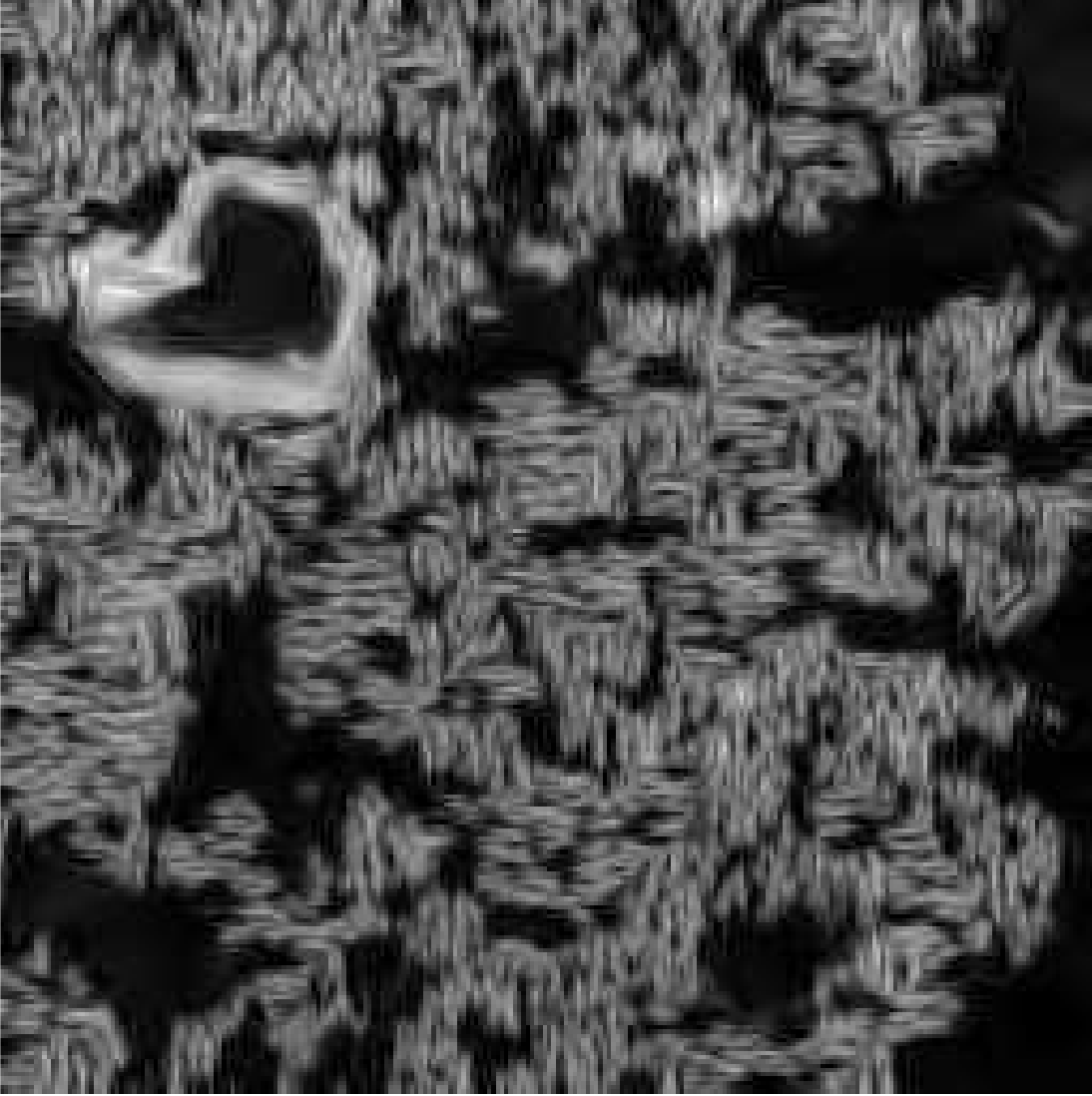}}\\
\subfigure[]{\includegraphics[width=.12\textwidth]{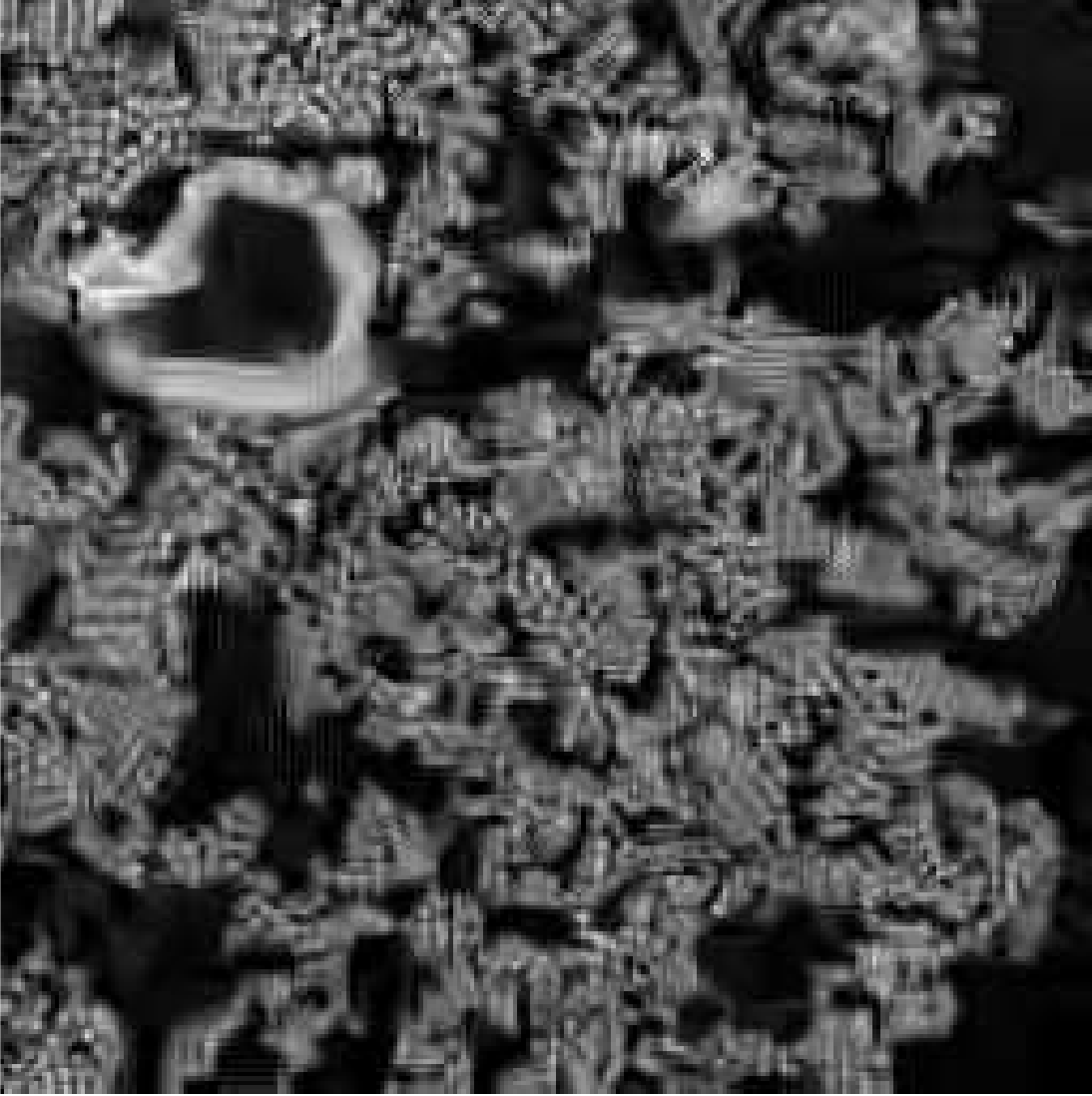}}
\subfigure[]{\includegraphics[width=.12\textwidth]{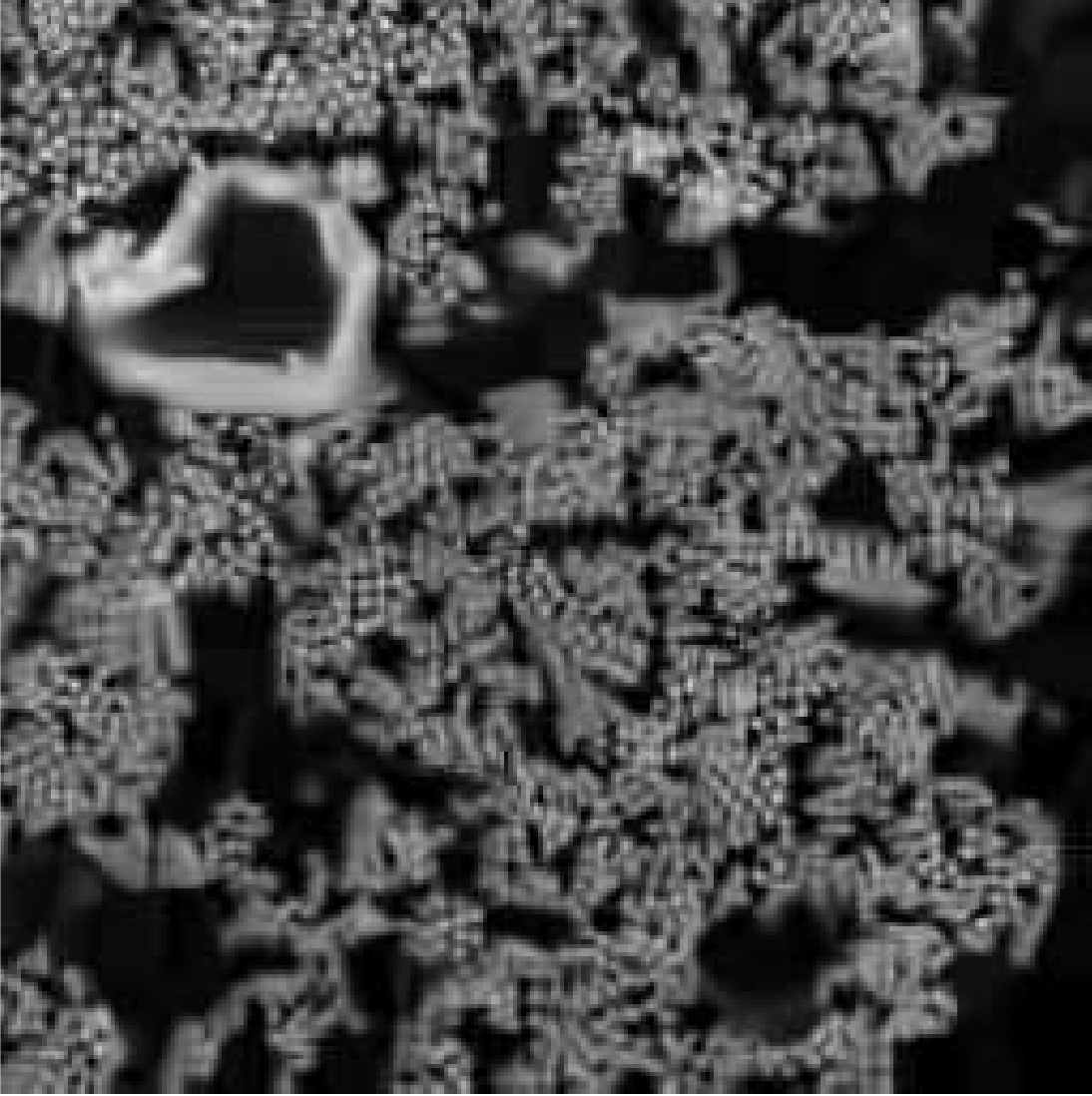}}
\subfigure[]{\includegraphics[width=.12\textwidth]{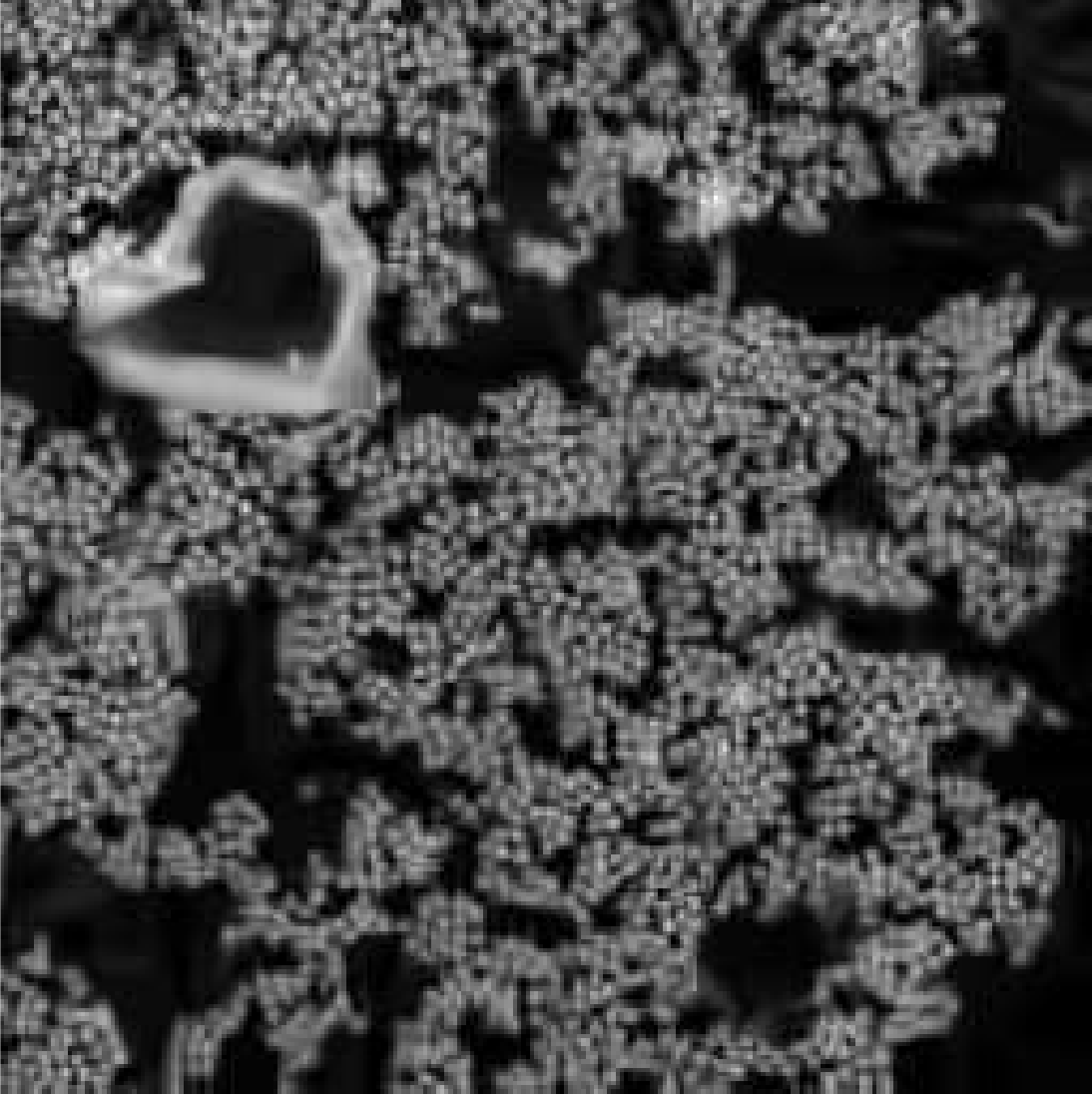}}
\end{center}
\caption{PR with CDP. Peak level $\nu=5.0\times10^{-2}, 8.0\times 10^{-2}, 1.0\times 10^{-1}$ for Poisson noise from left to right. First row: ``LS-PR''; Second row: ``TV-PR''; Third row: ``TGV-PR''; Fourth row: ``NLM-PR''; Fifth row: ``BM3D-PR''.}
\label{poicdp4}
\end{figure}

\begin{figure}
\begin{center}
\subfigure[]{\includegraphics[width=.15\textwidth]{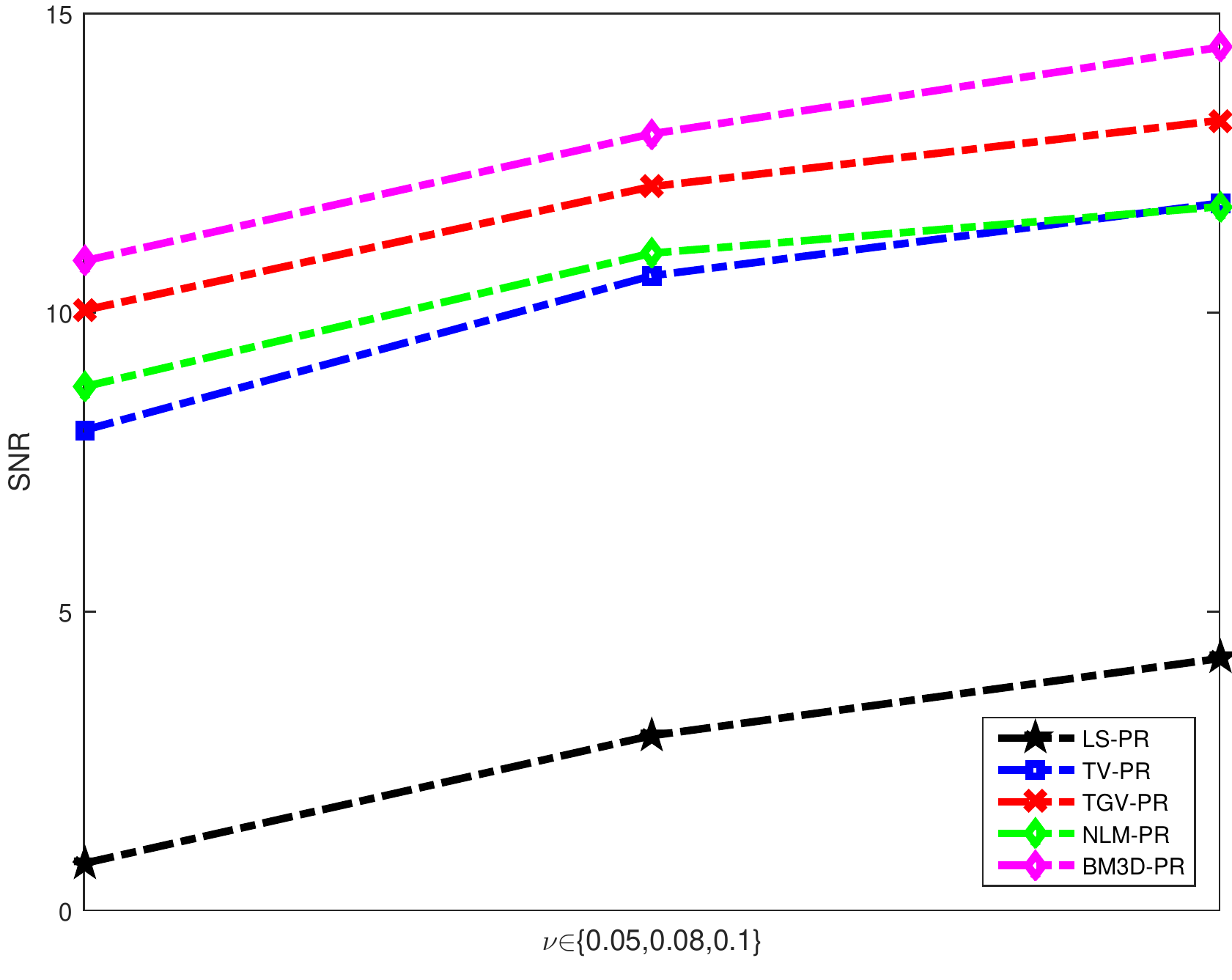}}
\end{center}
\caption{SNRs for reconstructed images in Figure \ref{poicdp4} v.s. peak level $\nu$ for different methods with CDP.}
\label{snr1}
\end{figure}

\begin{table}
\begin{center}
\begin{spacing}{1}
\begin{tabular}{|c|c|c|c|c|}
\hline
Method&Peak level $\nu$ &$\lambda$&r&$\eta$\\\hline
\multirow{3}{*}{TGV-PR}&$5.0\times 10^{-2}$&\multirow{3}{*}{2}&\multirow{3}{*}{$5.0\times 10^{5}$}&\multirow{3}{*}{$100$}\\
      &                 $8.0\times 10^{-2}$&&&\\
      &                 $1.0\times 10^{-1}$&&&\\\cline{1-5}
\multirow{3}{*}{NLM-PR}&$5.0\times 10^{-2}$& $4.0\times10^4$ &\multirow{3}{*}{$1.0\times10^6$}&\multirow{6}{*}{$50$}\\
      &$8.0\times 10^{-2}$&$3.5\times 10^{4}$& &\\
      &$1.0\times 10^{-1}$&$3.5\times 10^{4}$&&\\\cline{1-4}
\multirow{3}{*}{BM3D-PR}&$5.0\times 10^{-2}$&$8\times 10^{4}$&{$4.0\times 10^{5}$}&\\
      &$8.0\times 10^{-2}$&$1.5\times 10^{5}$&$1.0\times10^6$&\\
      &$1.0\times 10^{-1}$&$1.5\times 10^{5}$&$1.0\times10^6$&\\
      \hline
\end{tabular}
\end{spacing}
\end{center}
\caption{Parameters for Poisson noise removal of CDP on complex-valued images for Figure \ref{poicdp4}.}
\label{tab2}
\end{table}

\subsubsection{Ptychographic PR (Ptycho-PR)}
The complex-valued image in Figure \ref{groundtruth} (e) are tested to show the performance of our proposed methods with different noise level by setting the peak level $\nu\in\{0.2,0.5,0.8\}$, and see results in \ref{poipty1} and  SNRs in Figure \ref{snr3}, where we only show the performances of ``LS-PR'',``TV-PR'' and ``BM3D-PR''. We fix the parameters as $\lambda=7.0\times10^3, r=7.0\times10^4$ and $\eta=5$ for different noise levels. Reconstructed images are blurry by ``LS-PR'', which seems more challengeable than CDP with random masks. With high level noise, ``TV-PR'' recovers the images with sharp edges and clean background, but almost completely blurs the smaller scale features. ``BM3D-PR'' can work well and some of the features are recovered when one can hardly see any smaller features from the images by ``LS-PR''. When the noise level decreases, ``TV-PR'' and ``BM3D-PR'' can work well. Again ``BM3D-PR'' gains the highest SNRs  inferred from Figure \ref{snr4}. The average SNRs are 8.60, 13.67, and 15.89  for ``LS-PR'', ``TV-PR'', and ``BM3D-PR'' respectively, and our proposed method has about 2dB increase averagely compared with ``TV-PR''.

\begin{figure}
\begin{center}
\subfigure[]{\includegraphics[width=.12\textwidth]{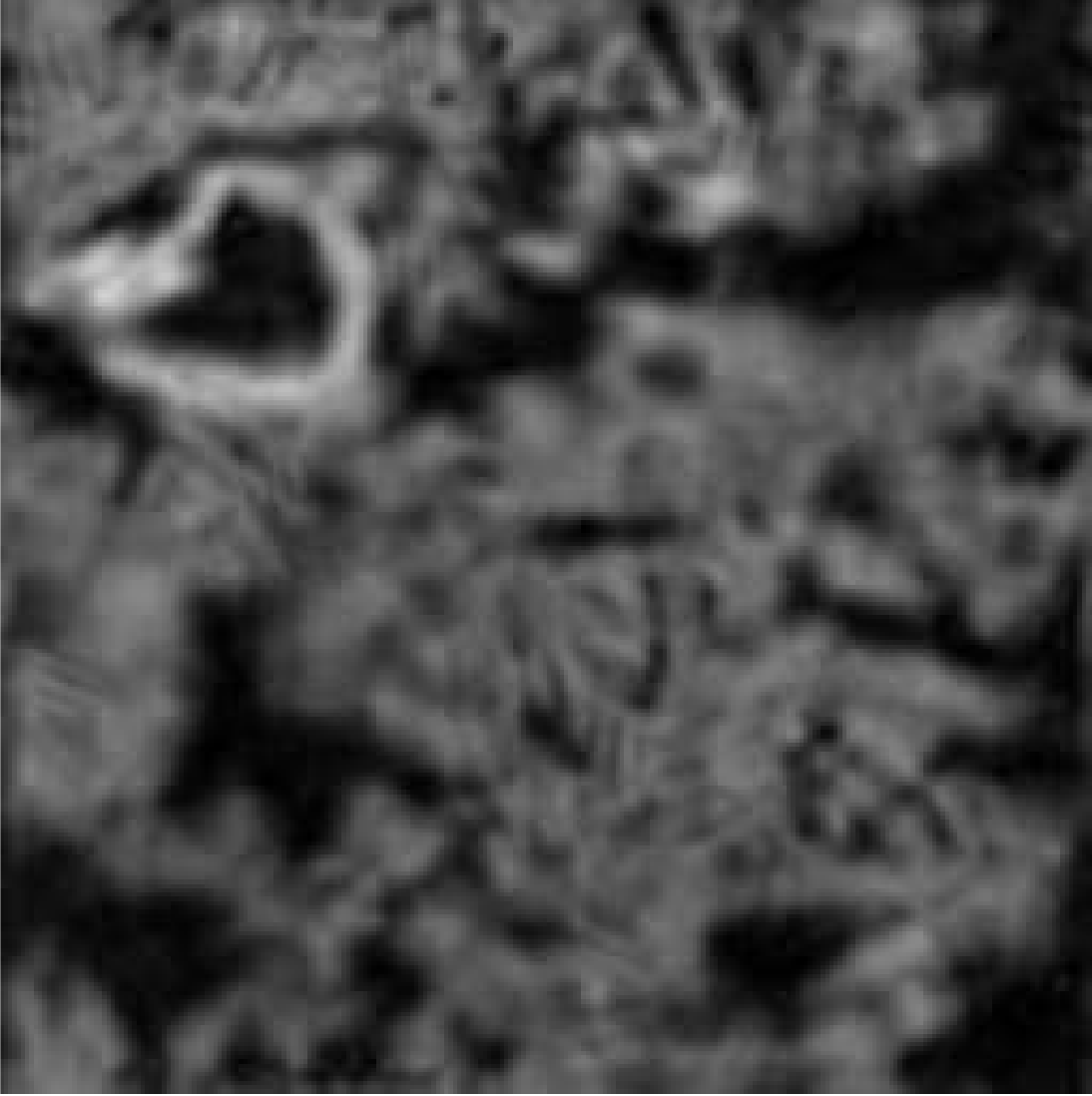}}
\subfigure[]{\includegraphics[width=.12\textwidth]{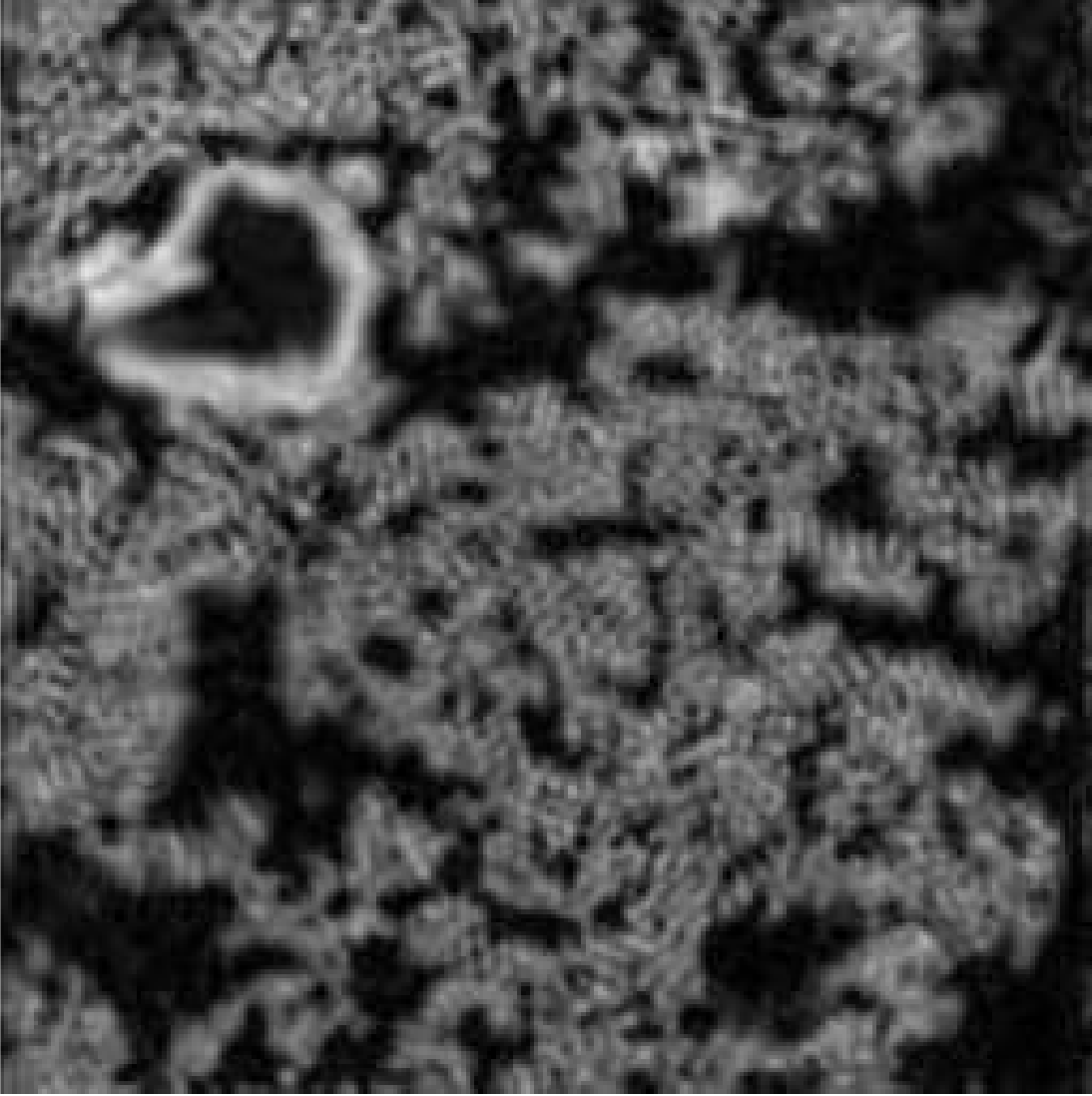}}
\subfigure[]{\includegraphics[width=.12\textwidth]{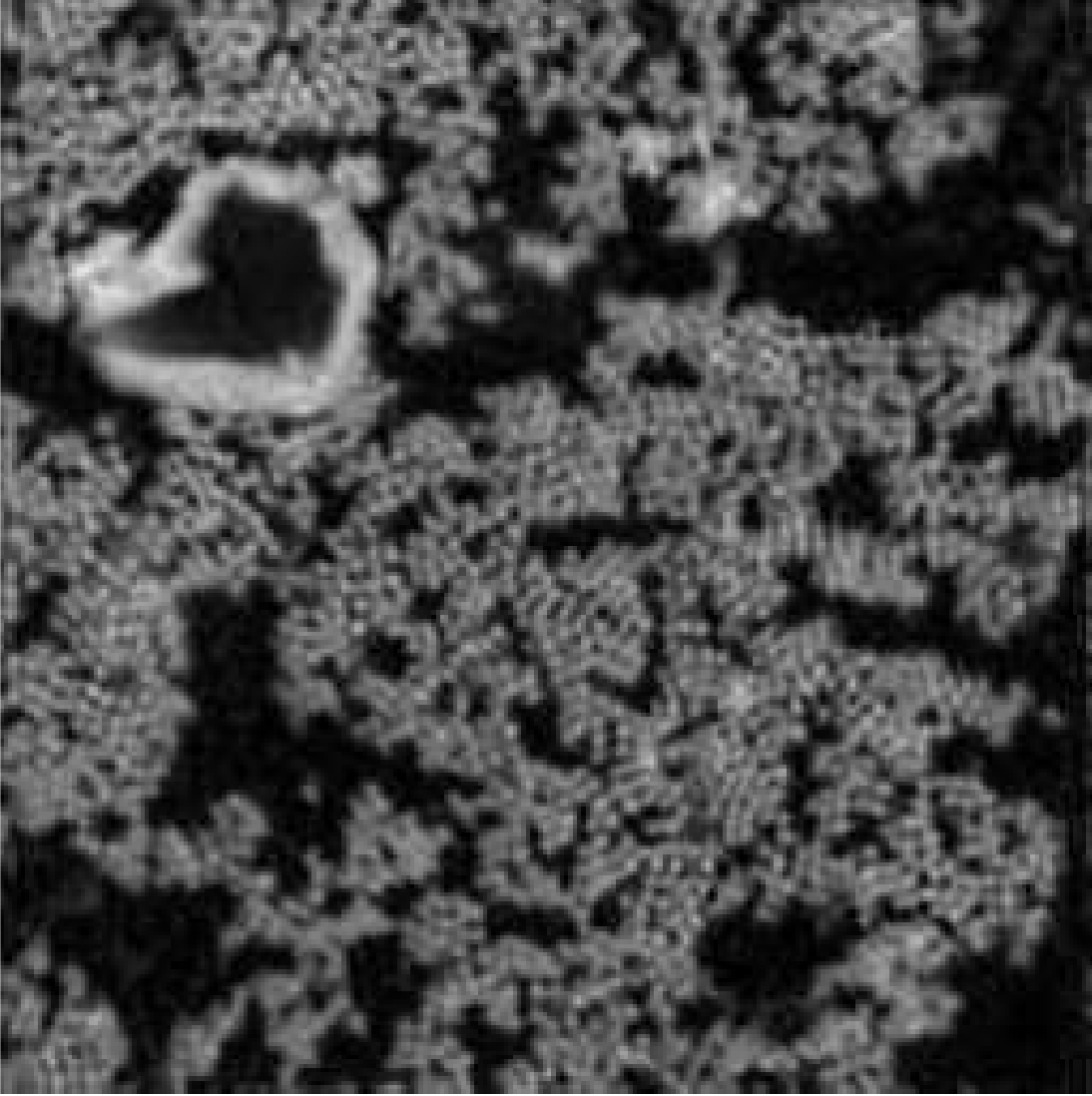}}\\
\subfigure[]{\includegraphics[width=.12\textwidth]{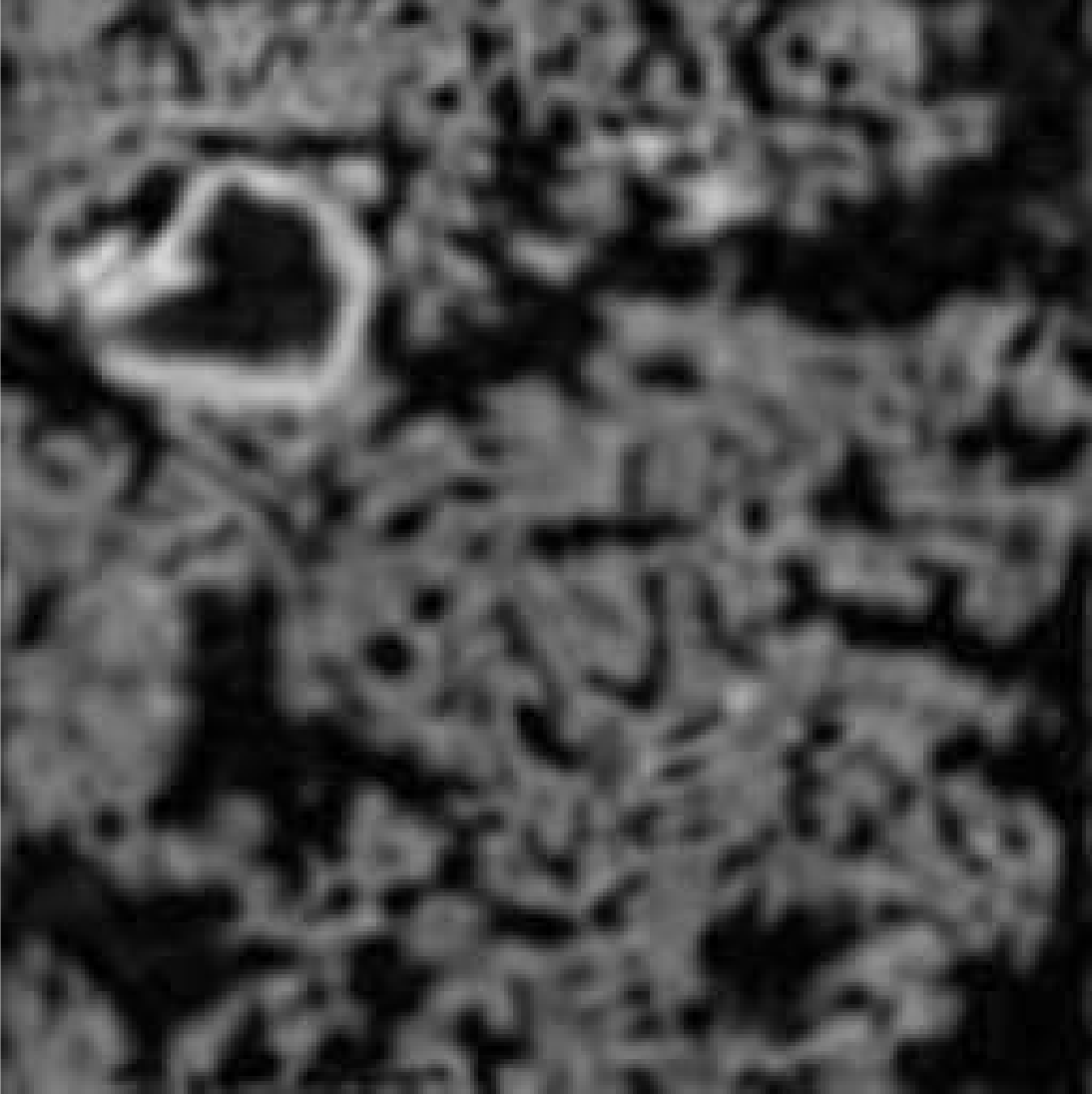}}
\subfigure[]{\includegraphics[width=.12\textwidth]{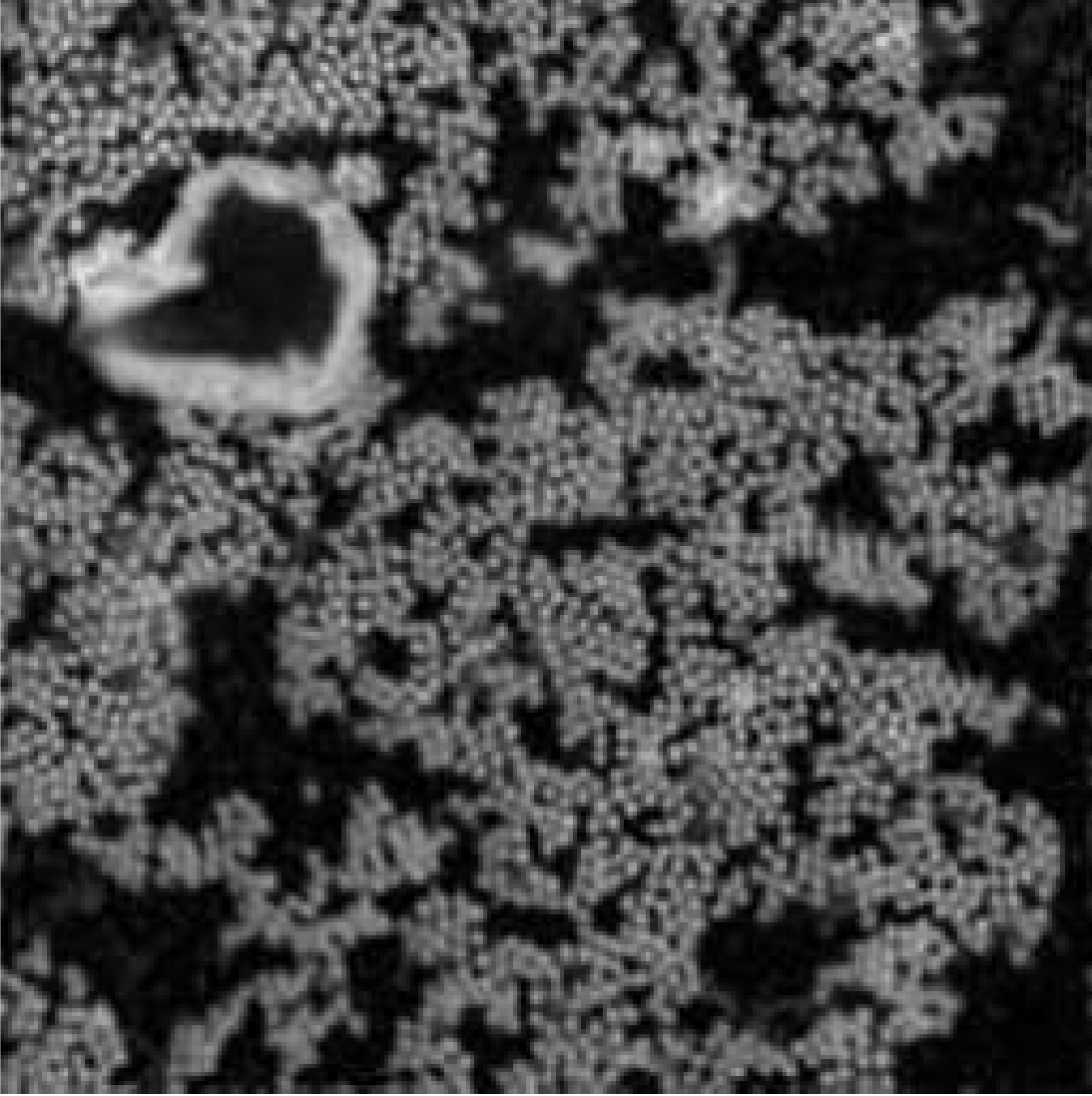}}
\subfigure[]{\includegraphics[width=.12\textwidth]{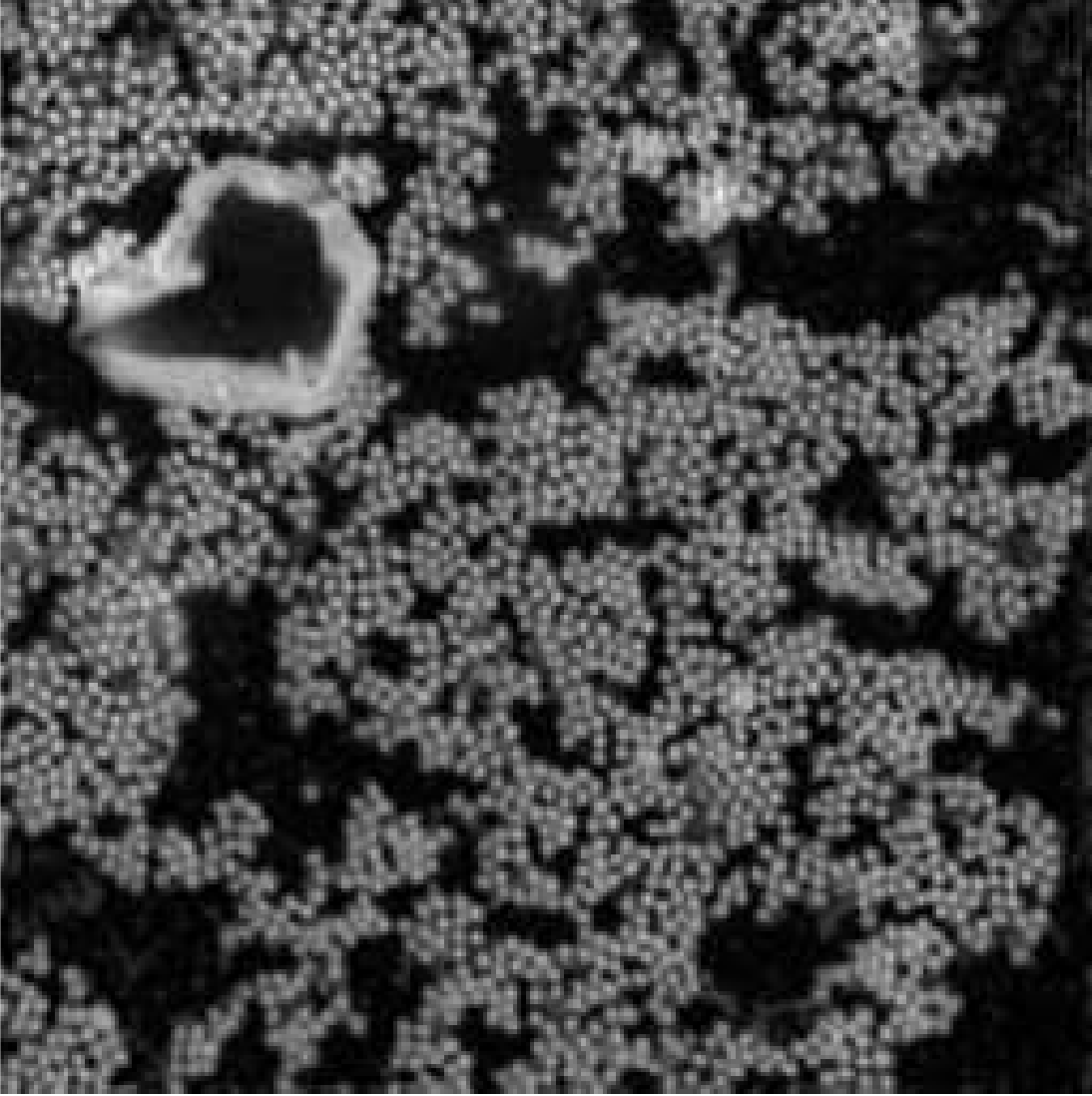}}\\
\subfigure[]{\includegraphics[width=.12\textwidth]{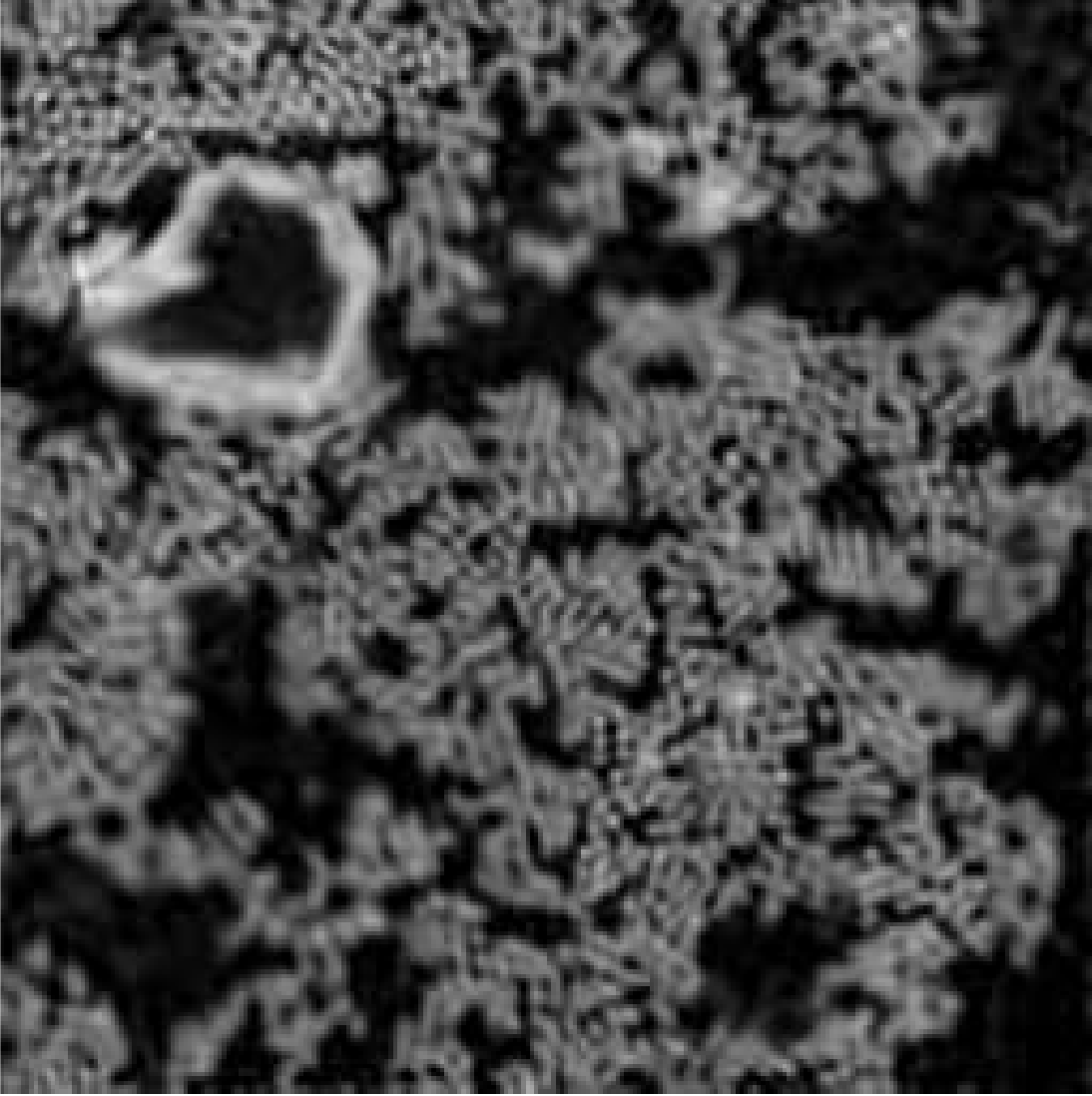}}
\subfigure[]{\includegraphics[width=.12\textwidth]{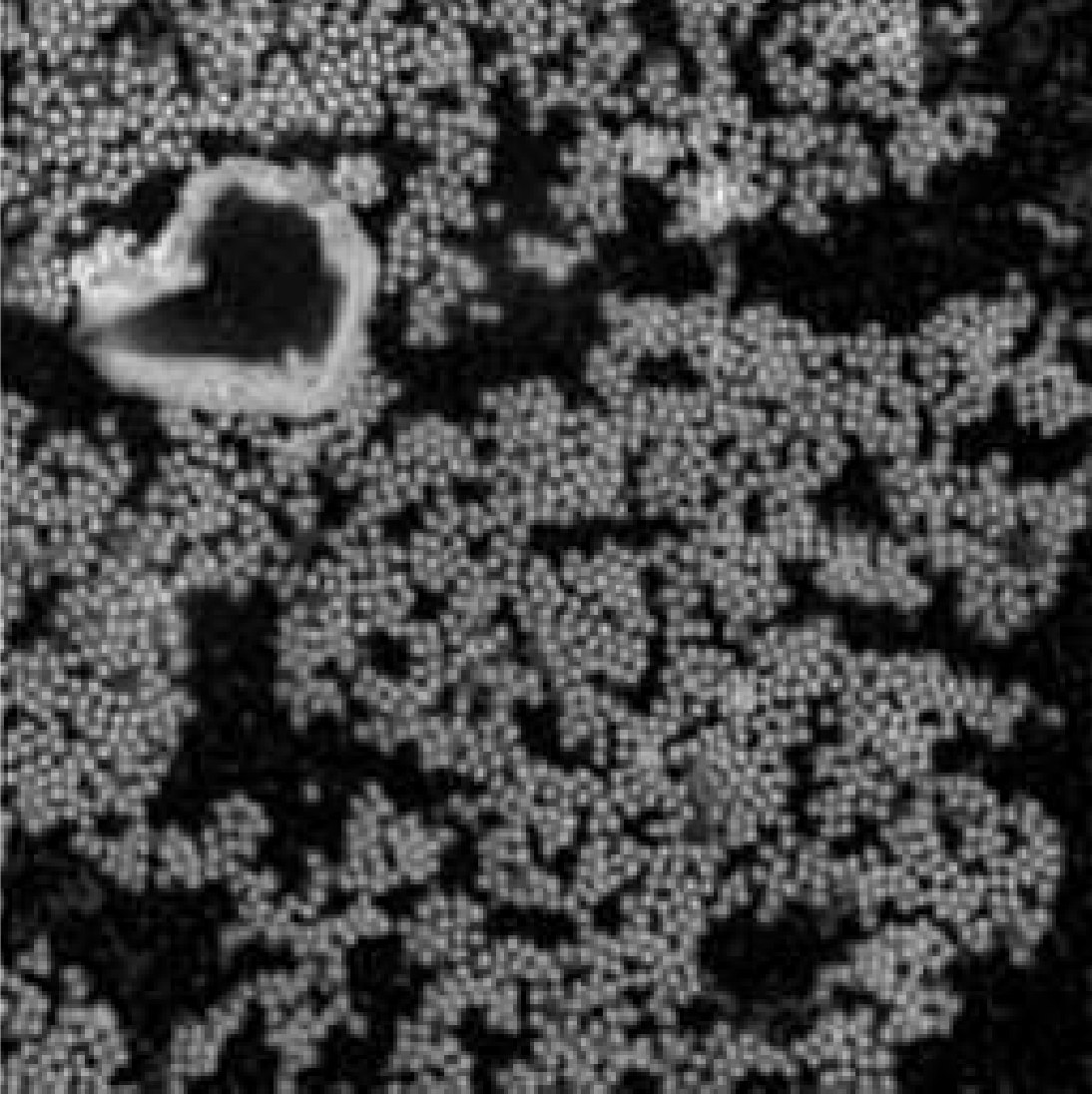}}
\subfigure[]{\includegraphics[width=.12\textwidth]{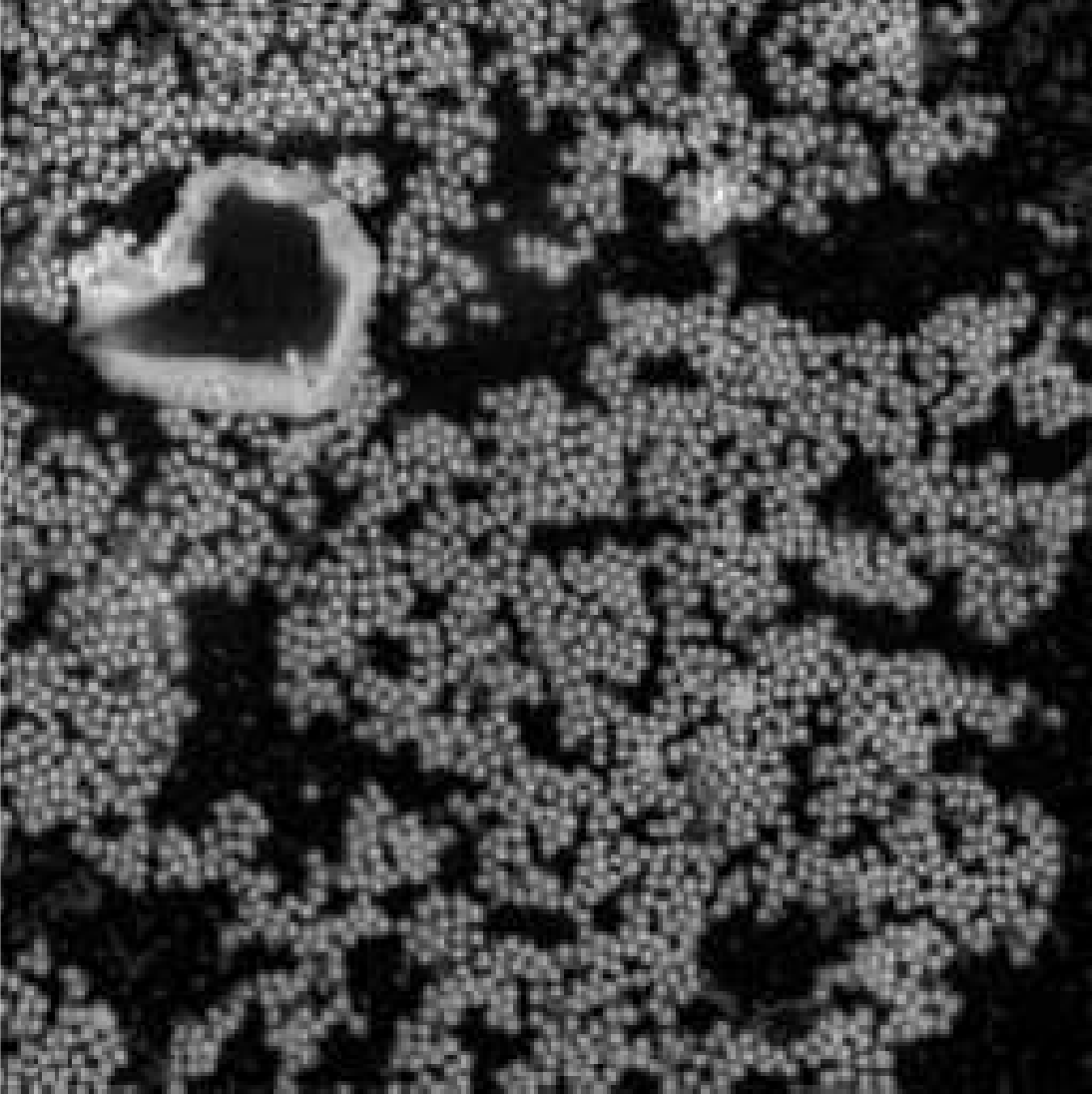}}
\end{center}
\caption{Ptychographic PR. Peak level $\nu=0.2,0.5,0.8$ for Poisson noise from left to right. First row: ``LS-PR''; Second row: ``TV-PR''; Third row: ``BM3D-PR''.}
\label{poipty1}
\end{figure}

\begin{figure}
\begin{center}
\subfigure[]{\includegraphics[width=.15\textwidth]{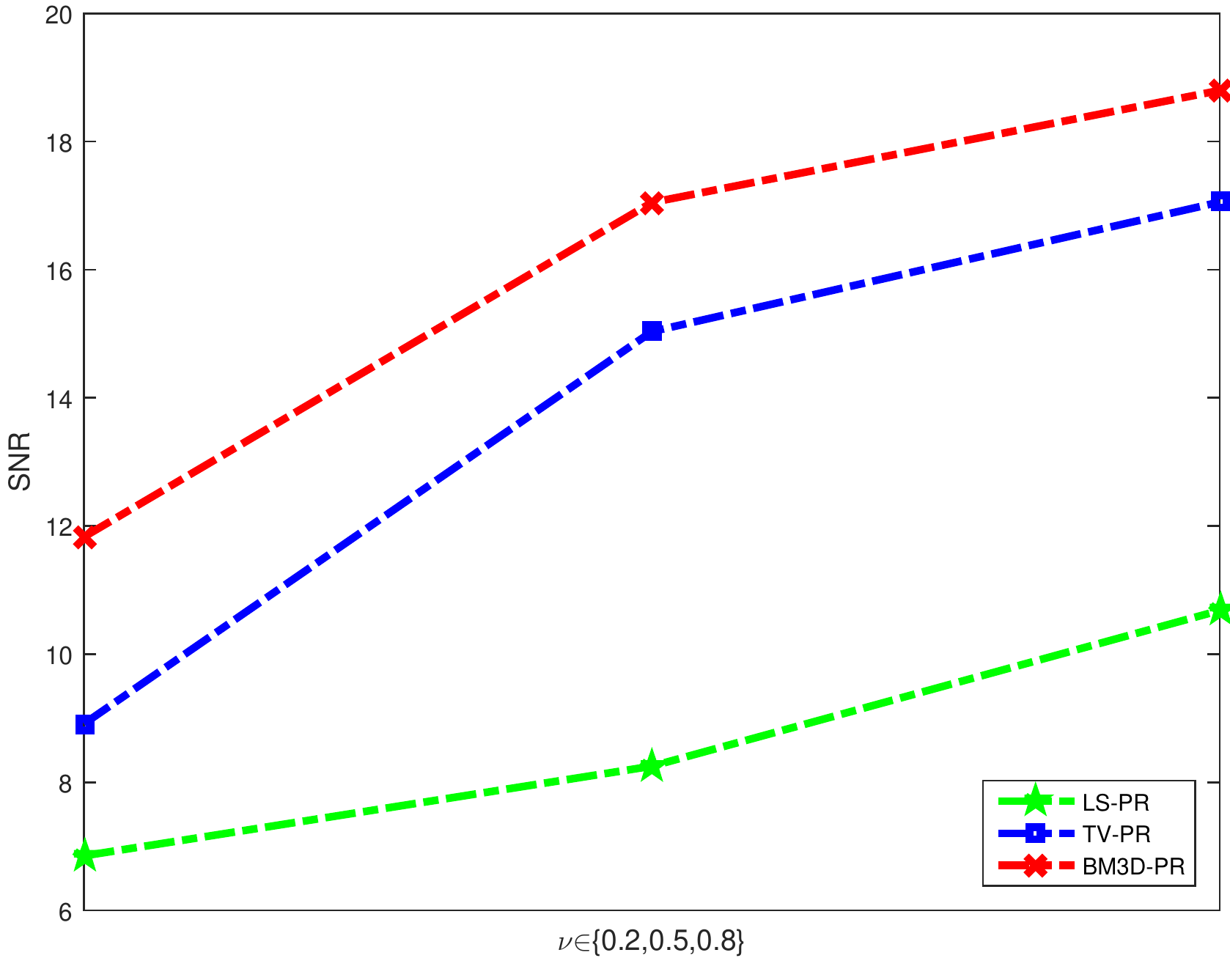}}
\end{center}
\caption{SNRs of the reconstructed images in Figure \ref{poipty1} v.s. peak level $\nu$ by different methods for ptychographic PR.}
\label{snr3}
\end{figure}
\subsection{Gaussian noise}

We only consider the ptychographic PR in this part, and see the results in Figure \ref{gaupty1}. Set parameters of ``BM3D-PR'' as  $\lambda=3.5\times 10^7, 1.5\times10^7, 1.0\times 10^7$ for the noisy measurements with SNRs to be 15, 20, and 30 respectively. Fix the parameter $r=2.5\times 10^7$, and $\eta=1.0\times 10^3$ with different noise levels.
The reconstructed images by ``LS-PR'' only keep the large-scale features, and completely lose the smaller ones. TV-PR just smooths the images, while also loses most of the smaller scale structures  in Figure \ref{gaupty1} (d) and (e).  With our proposed ``BM3D-PR'', it can recover almost all of the texture parts in Figure \ref{gaupty1} (h), and even contaminated  by very severe noise, some of the smaller scale features can be preserved in Figure \ref{gaupty1} (g). The SNRs are put in Figure \ref{snr4}, and one can readily observe the increase of the SNRs by our method. The average SNRs are 6.41, 9.65, and 11.09 for ``LS-PR'', ``TV-PR'' and ``BM3D-PR'' respectively, and our method gains about 3dB, 1.5dB increase averagely compared with ``LS-PR' and ``TV-PR'' respectively.

\begin{figure}
\begin{center}
\subfigure[]{\includegraphics[width=.12\textwidth]{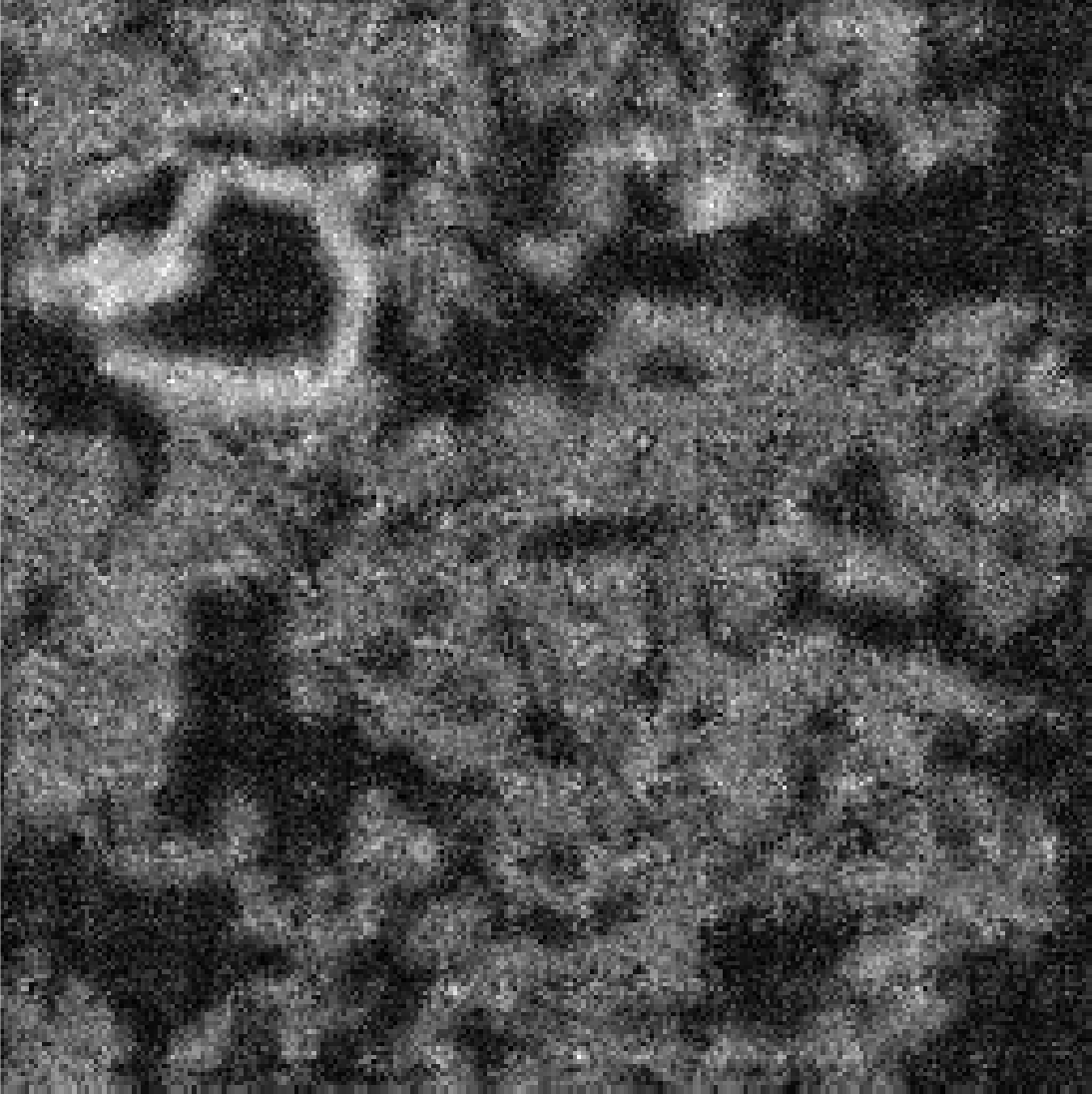}}
\subfigure[]{\includegraphics[width=.12\textwidth]{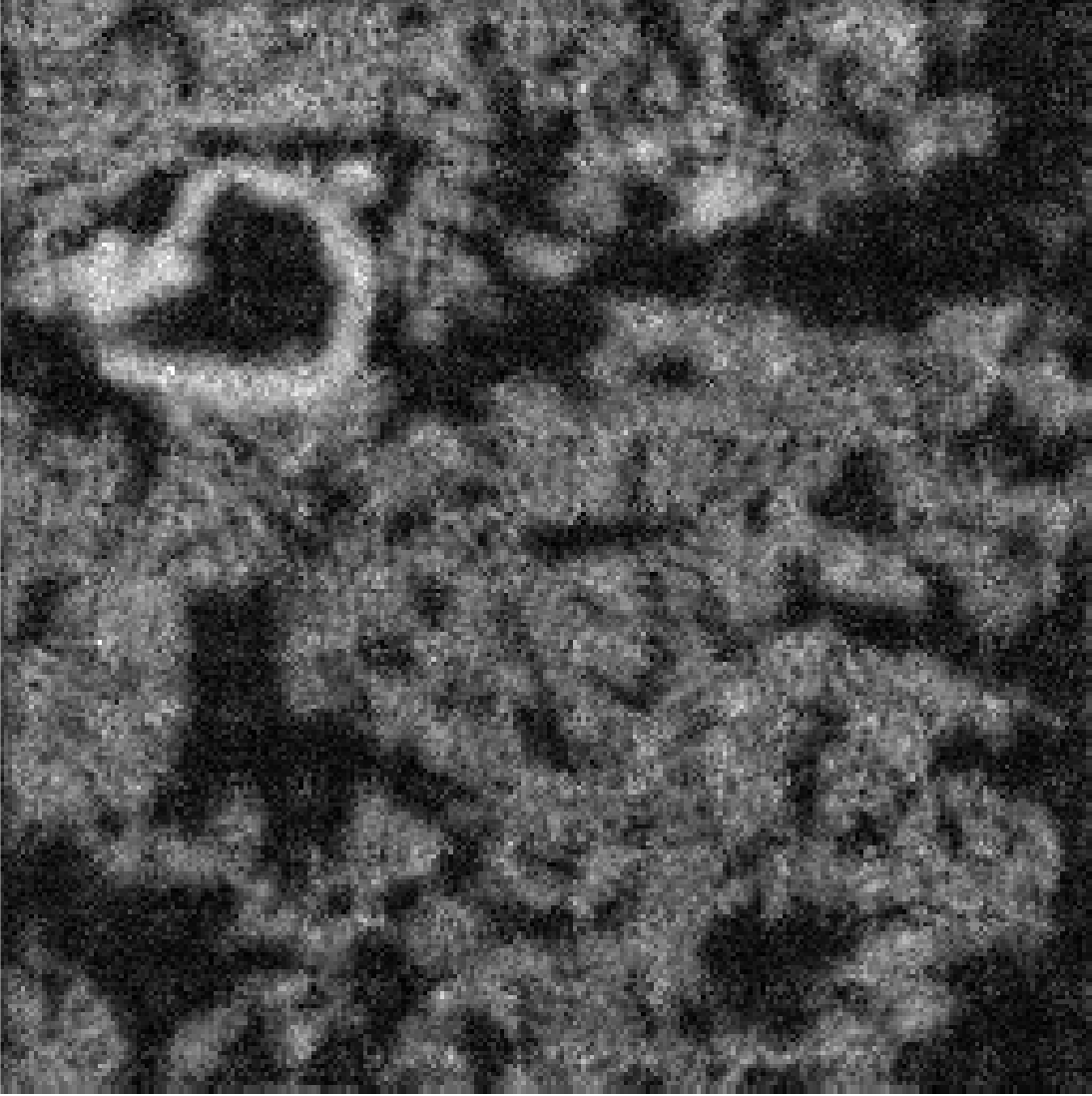}}
\subfigure[]{\includegraphics[width=.12\textwidth]{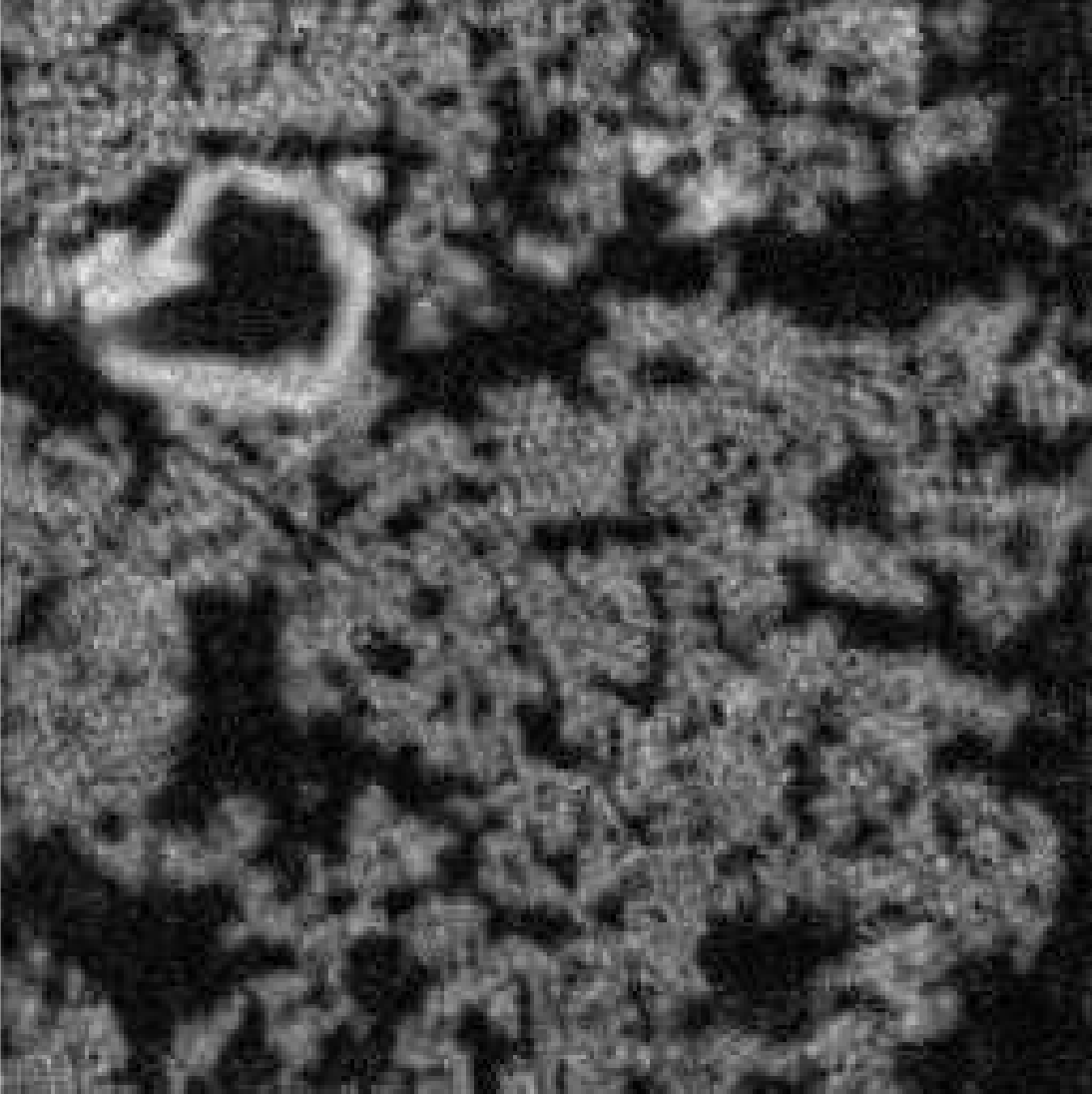}}\\
\subfigure[]{\includegraphics[width=.12\textwidth]{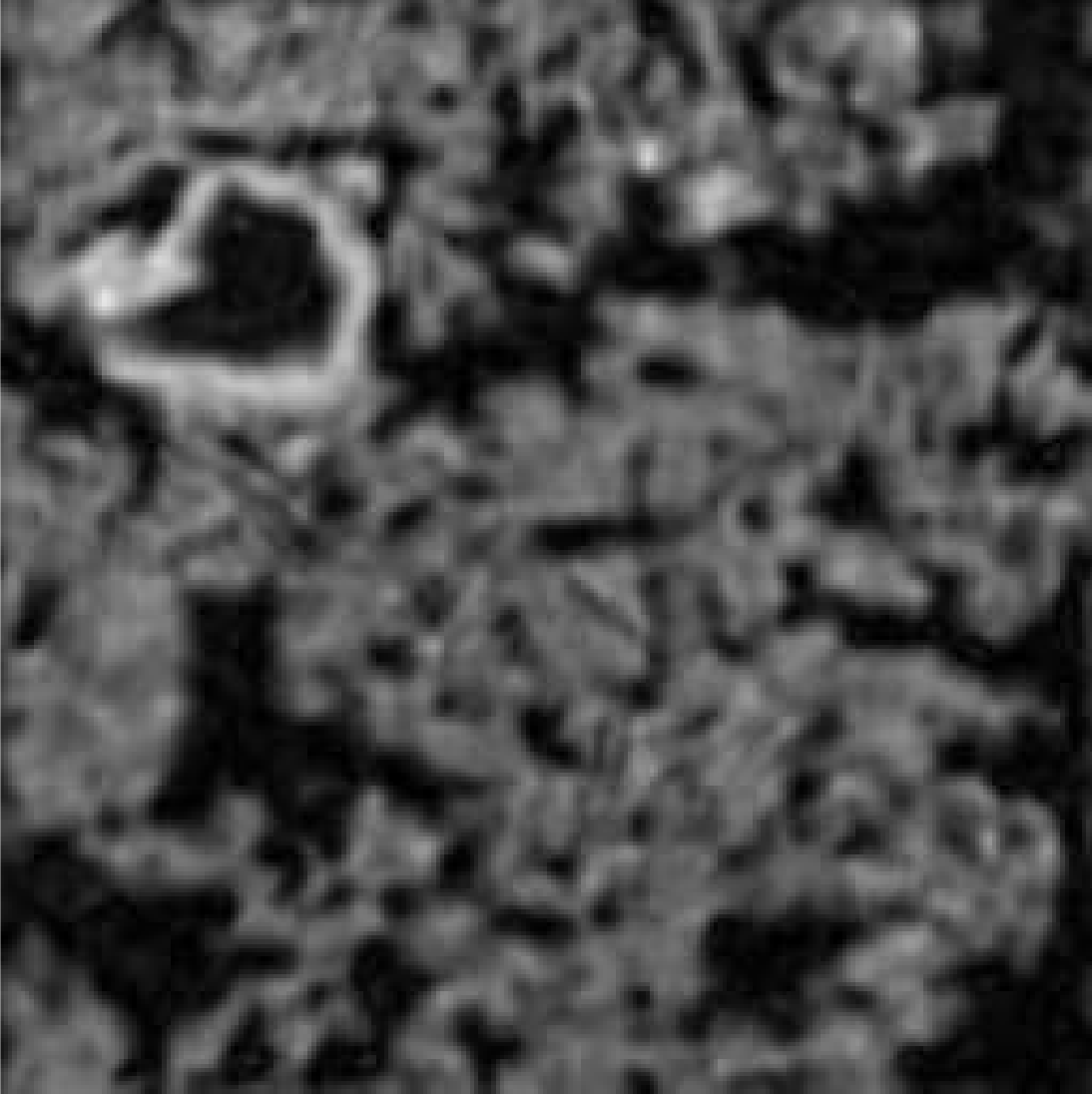}}
\subfigure[]{\includegraphics[width=.12\textwidth]{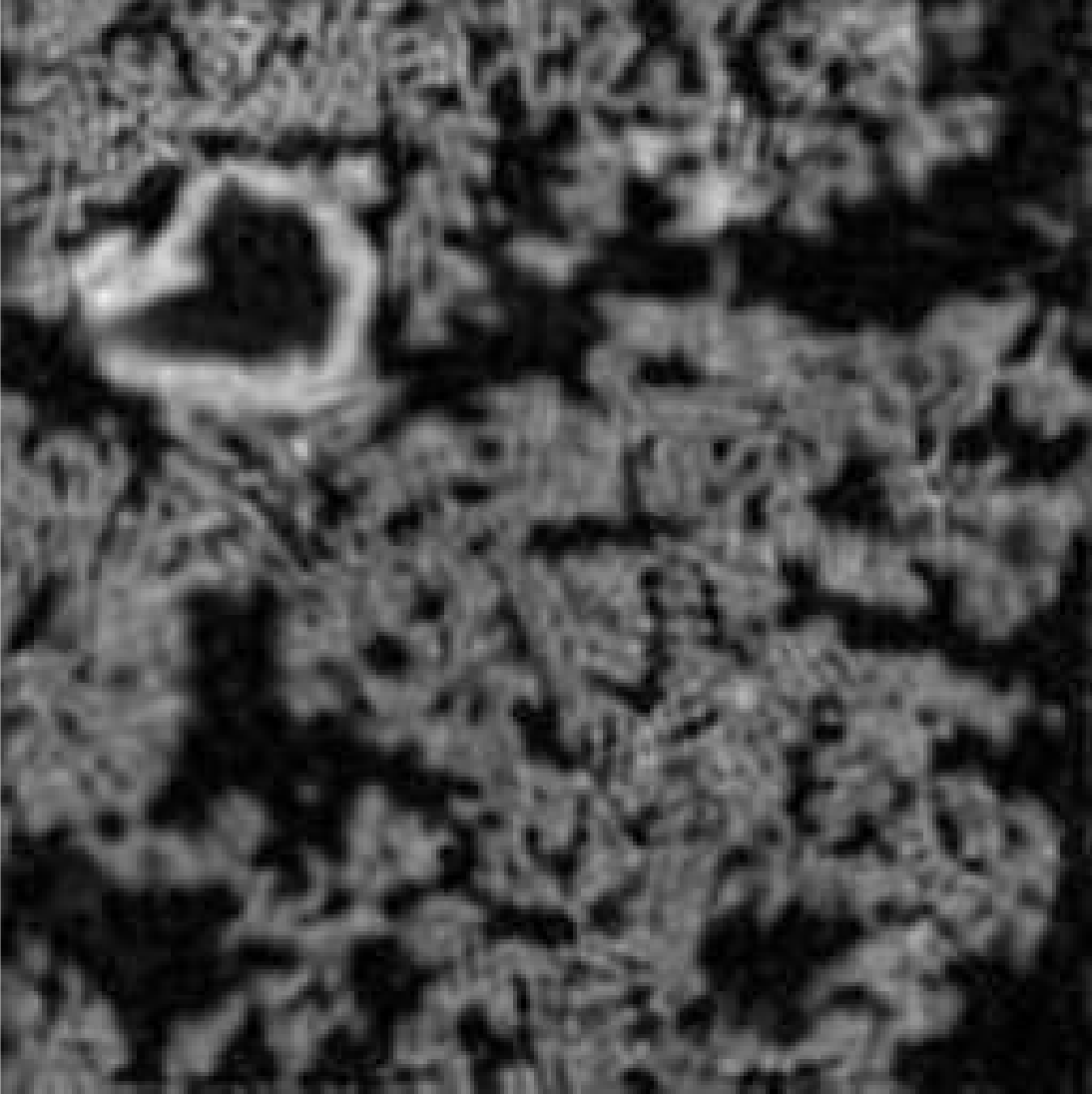}}
\subfigure[]{\includegraphics[width=.12\textwidth]{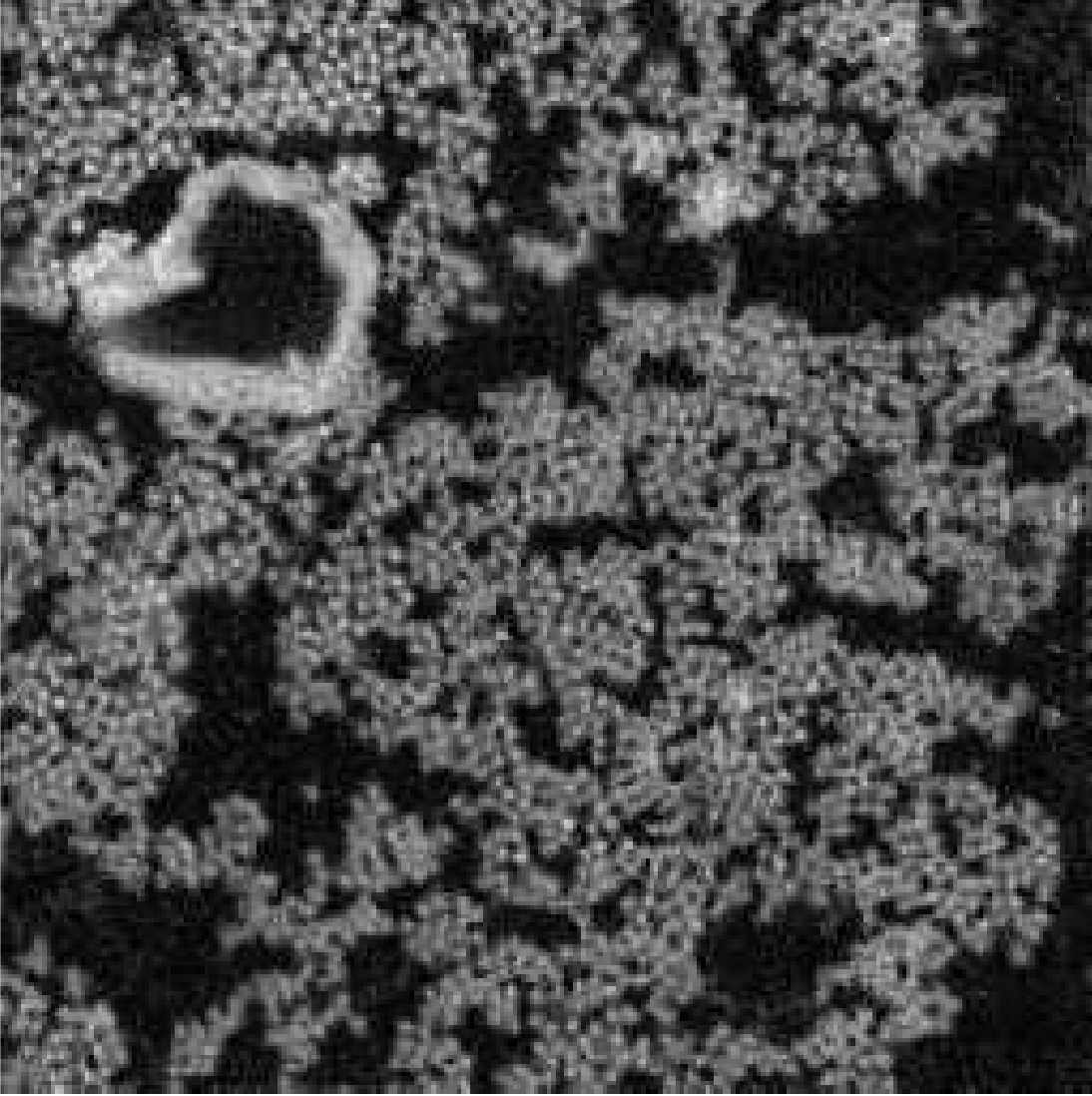}}\\
\subfigure[]{\includegraphics[width=.12\textwidth]{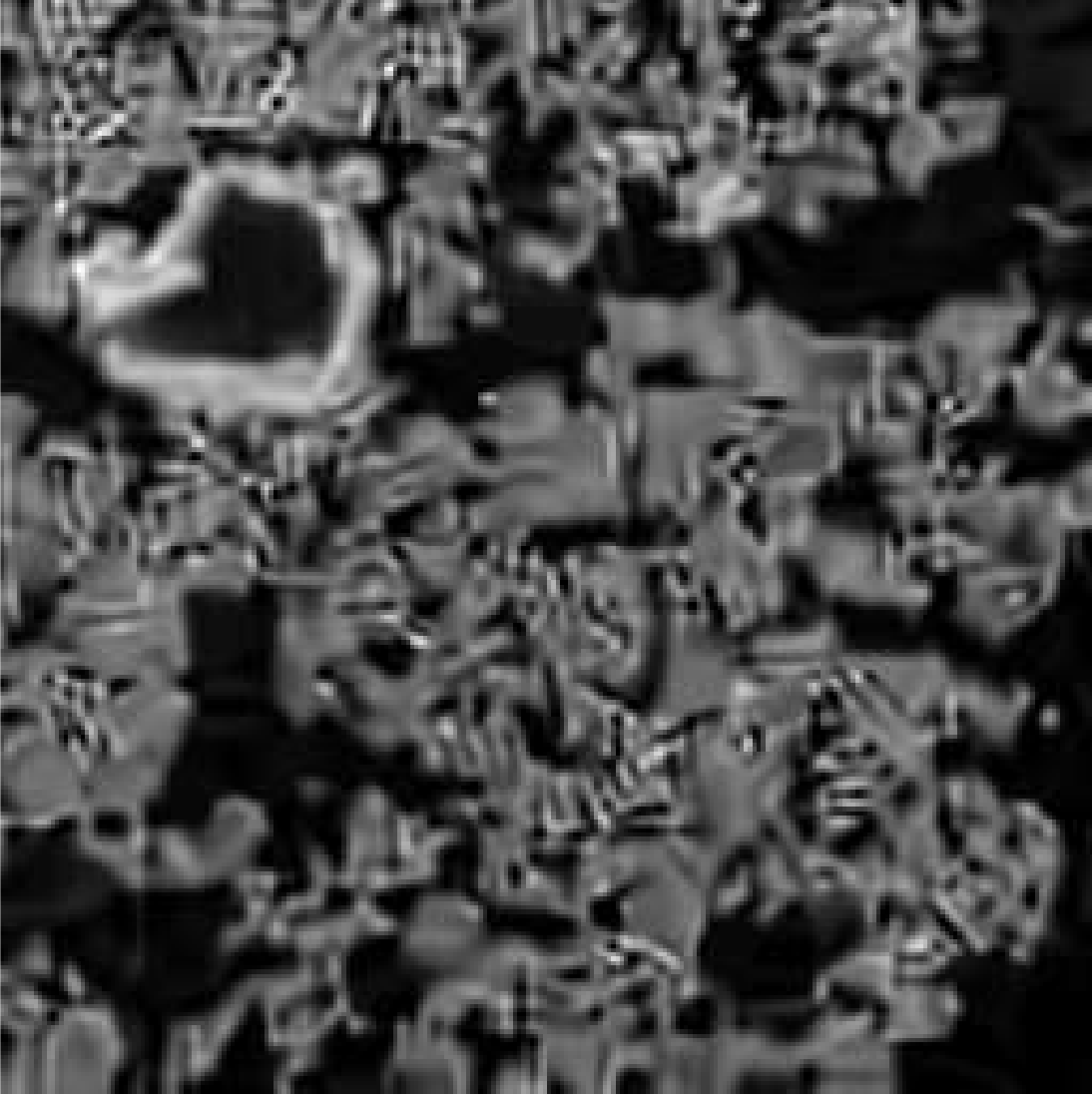}}
\subfigure[]{\includegraphics[width=.12\textwidth]{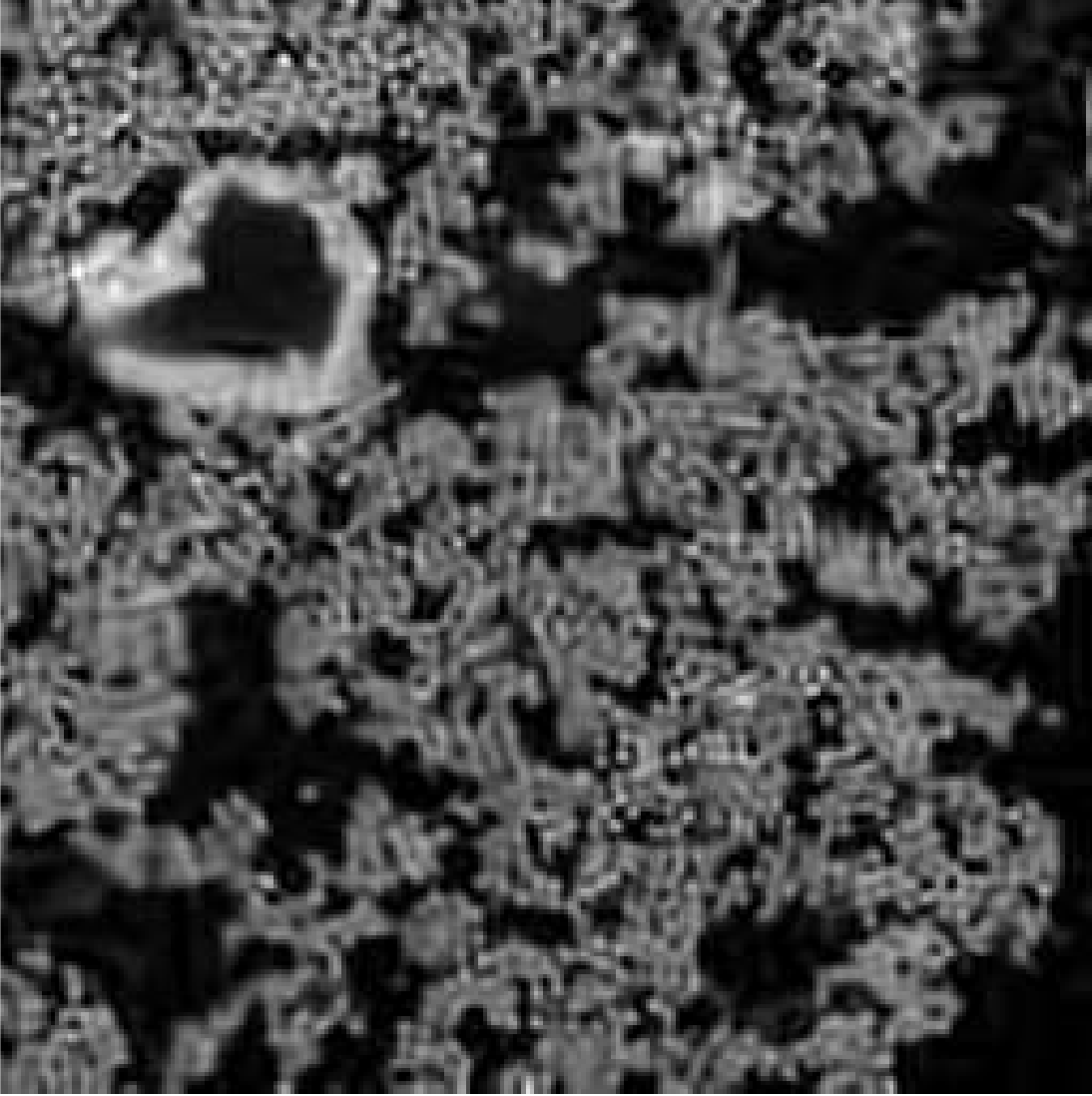}}
\subfigure[]{\includegraphics[width=.12\textwidth]{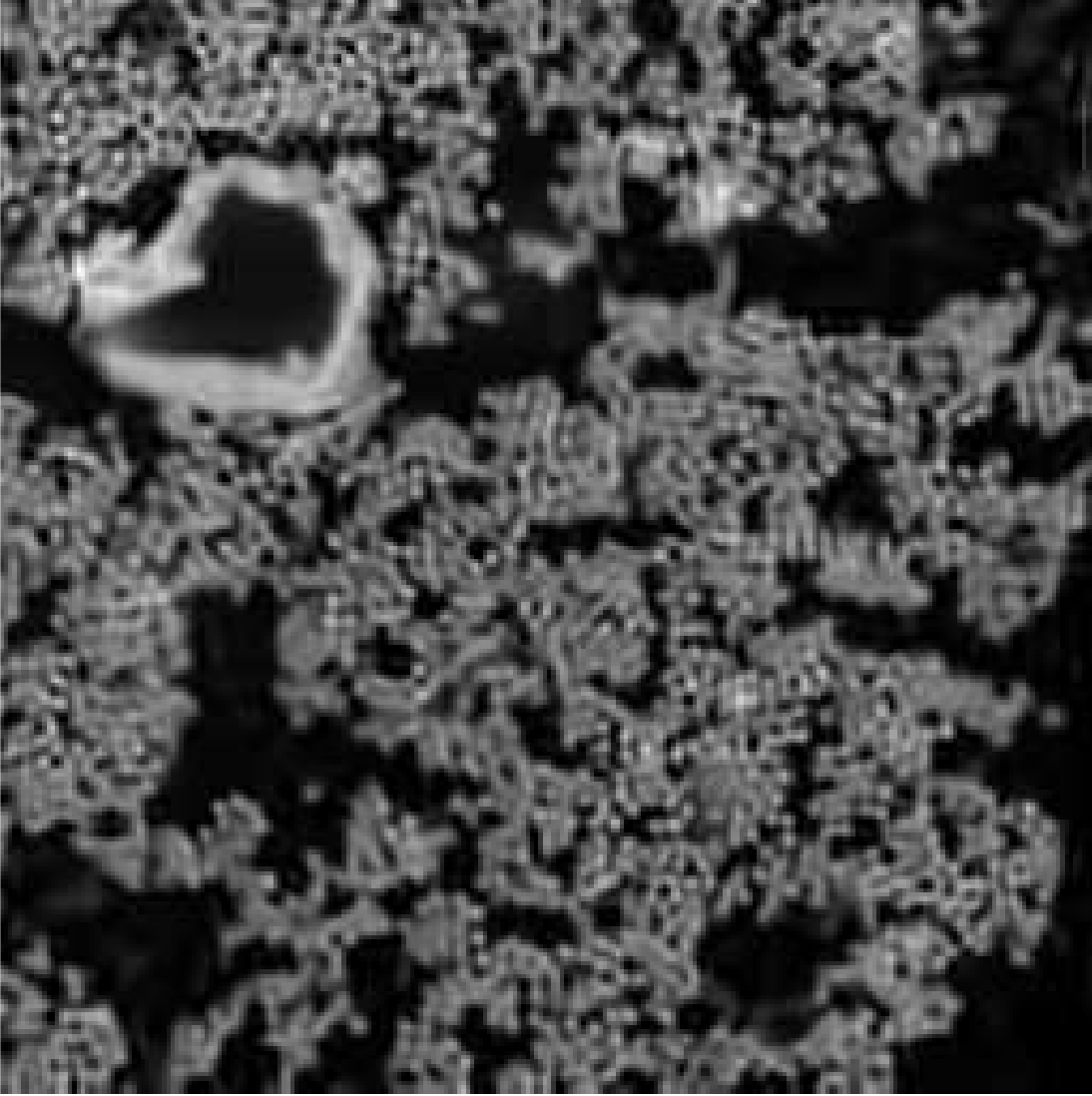}}
\end{center}
\caption{Ptychographic PR. Noisy level of measurements with $\mathrm{SNR}=15,20,30$ for Gaussian noise from left to right. First row: ``LS-PR''; Second row: ``TV-PR''; Third row: ``BM3D-PR''.}
\label{gaupty1}
\end{figure}

\begin{figure}
\begin{center}
\subfigure{\includegraphics[width=.15\textwidth]{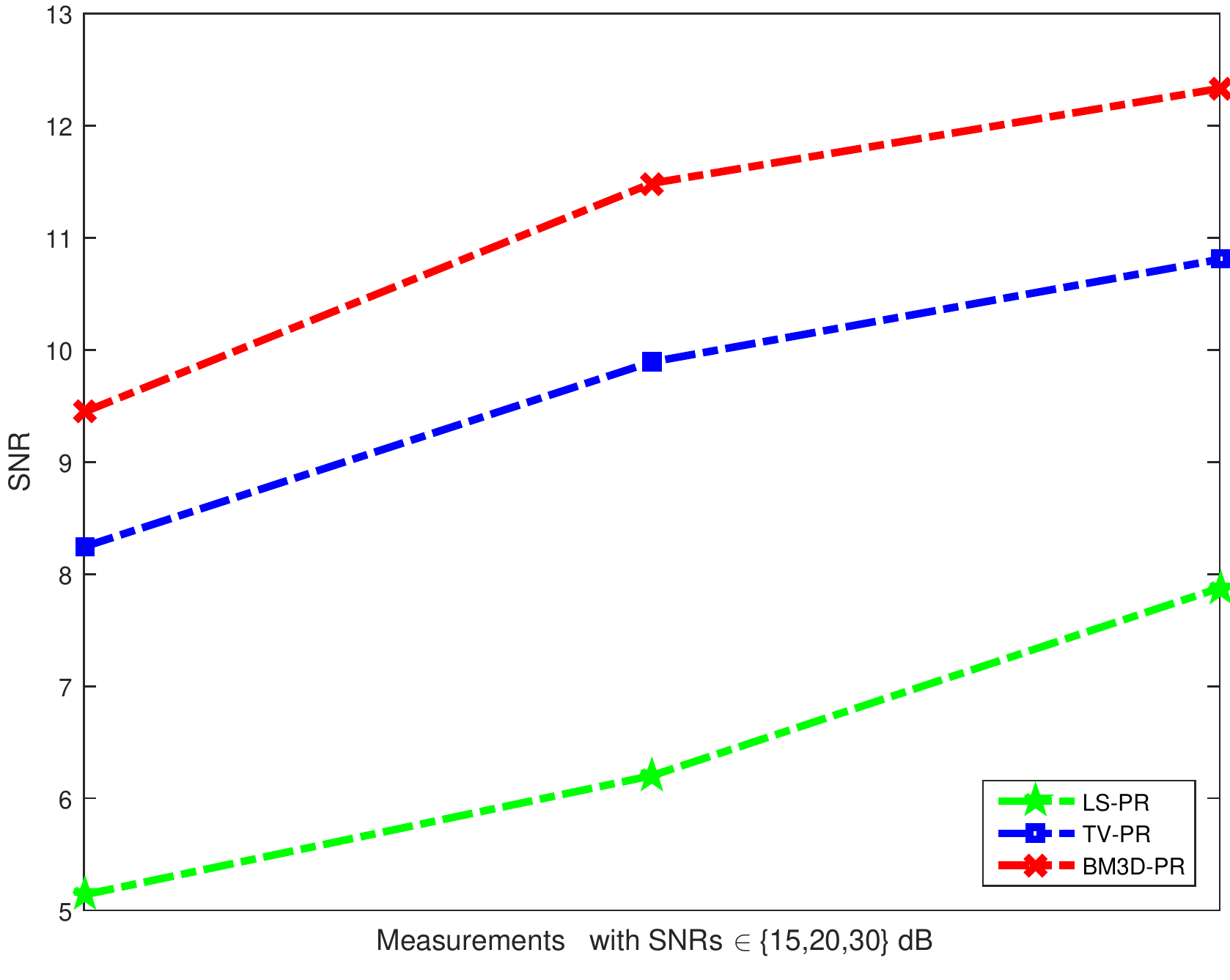}}
\end{center}
\caption{SNRs for reconstructed images in Figure \ref{gaupty1} v.s. measurements with $\mathrm{SNR}=15,20,30$ for ptychographic PR.}
\label{snr4}
\end{figure}

\subsection{Convergence}
To check the convergence of the iteration process, we monitor the histories of SNRs and relative errors of iterative solution $u^k$ w.r.t. the iteration number $k$, which are defined as $\frac{\|u^k-u^{k-1}\|}{\|u^k\|}.$
We show the histories of errors and SNRs for CDP for our proposed ``TGV-PR'', ``NLM-PR'', and ``BM3D-PR'' on images ``Lena'' with Poisson noise, and see  Figure \ref{hist} for detail. It demonstrates  that our proposed methods are convergent and stable.  For heavier noise, one needs more iterations to reach the final results with better image quality. Although we set the maximum iteration number as 30, about one third or half of the number is enough for real applications. It will be interesting to investigate how to set the optimal iteration number for different noise levels, and we put it as a future work.
\begin{figure}
\begin{center}
\subfigure{\includegraphics[width=.12\textwidth]{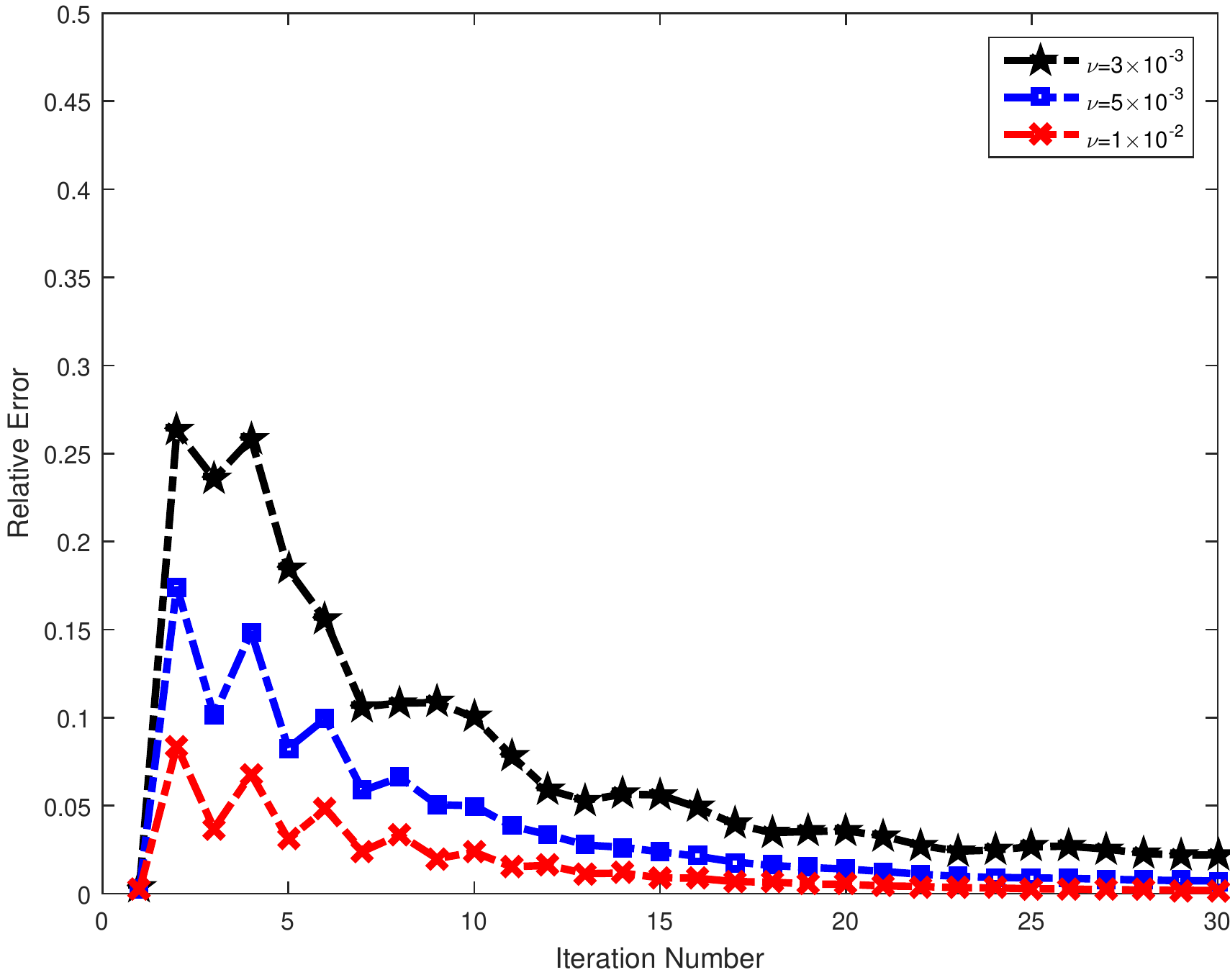}}
\subfigure{\includegraphics[width=.12\textwidth]{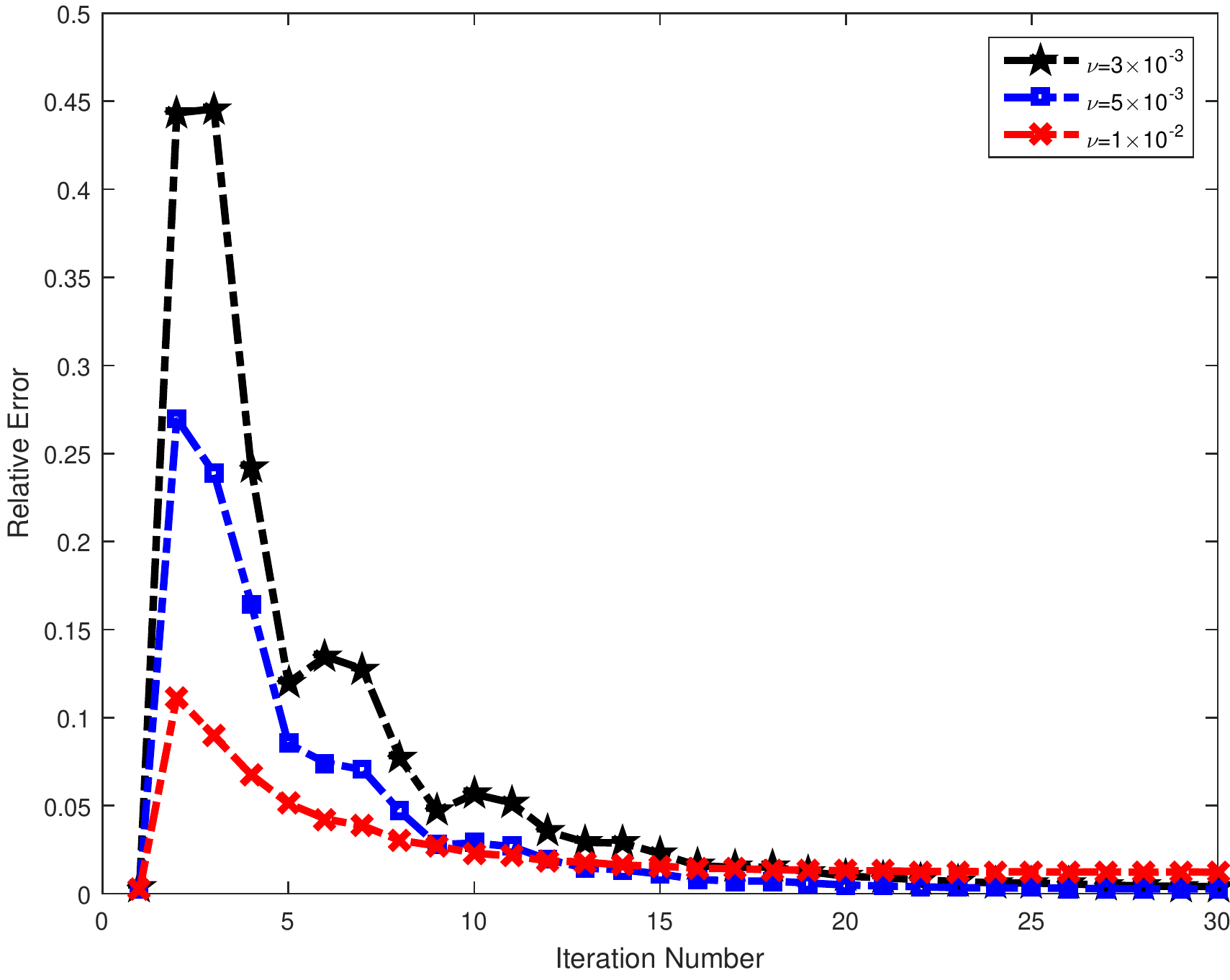}}
\subfigure{\includegraphics[width=.12\textwidth]{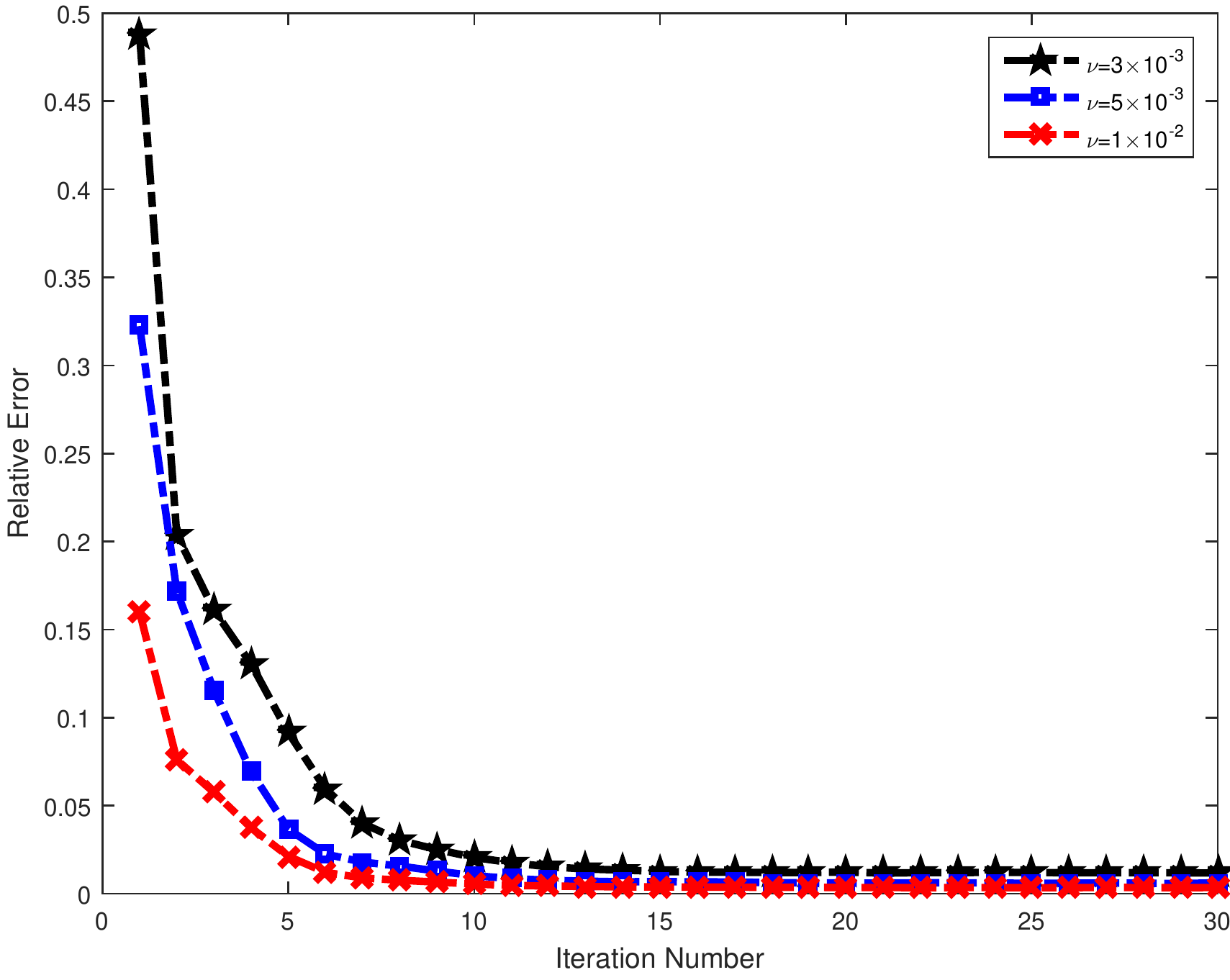}}\\
\vskip -.05in
\setcounter{subfigure}{0}
\subfigure[TGV-PR]{\includegraphics[width=.12\textwidth]{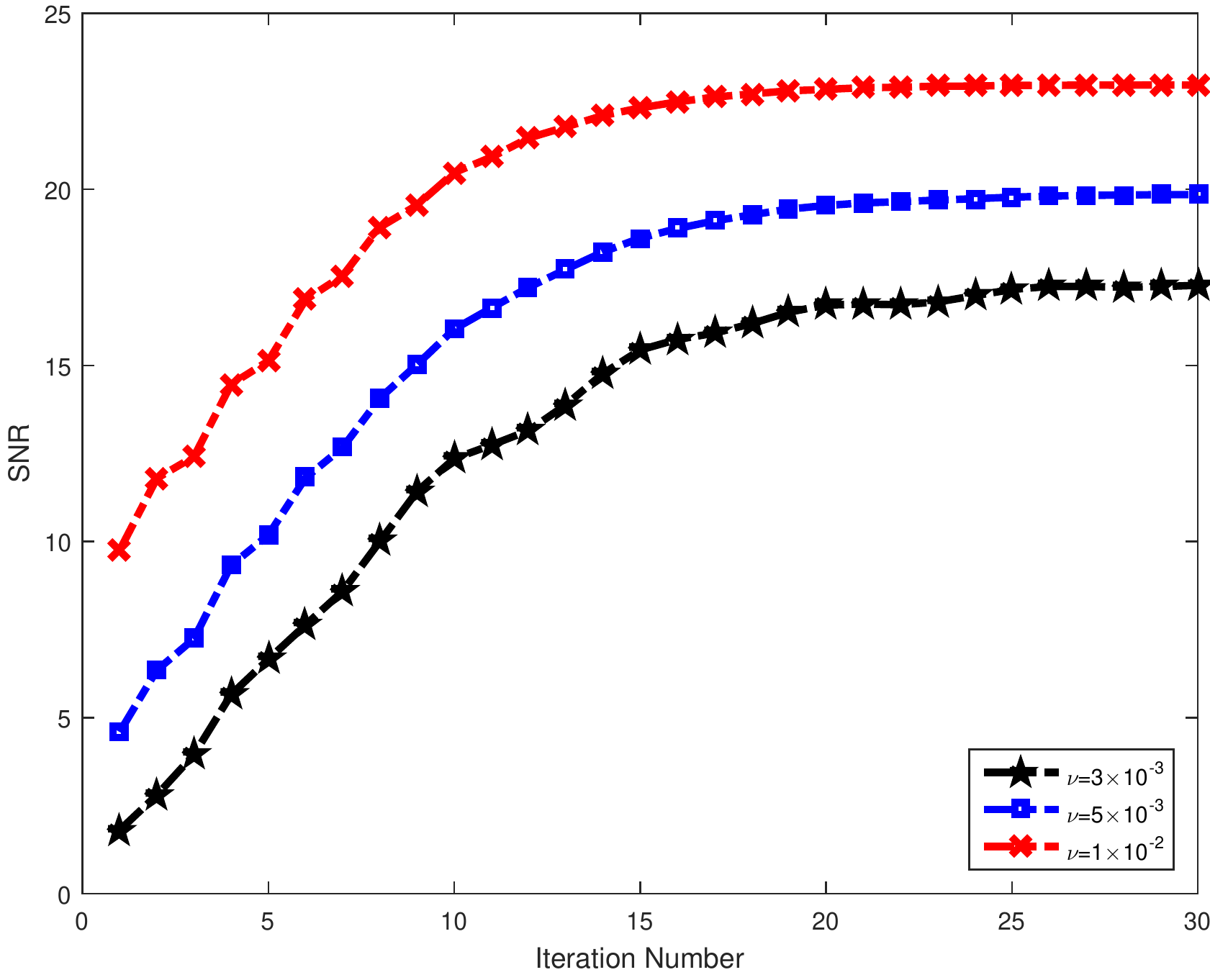}}
\subfigure[NLM-PR]{\includegraphics[width=.12\textwidth]{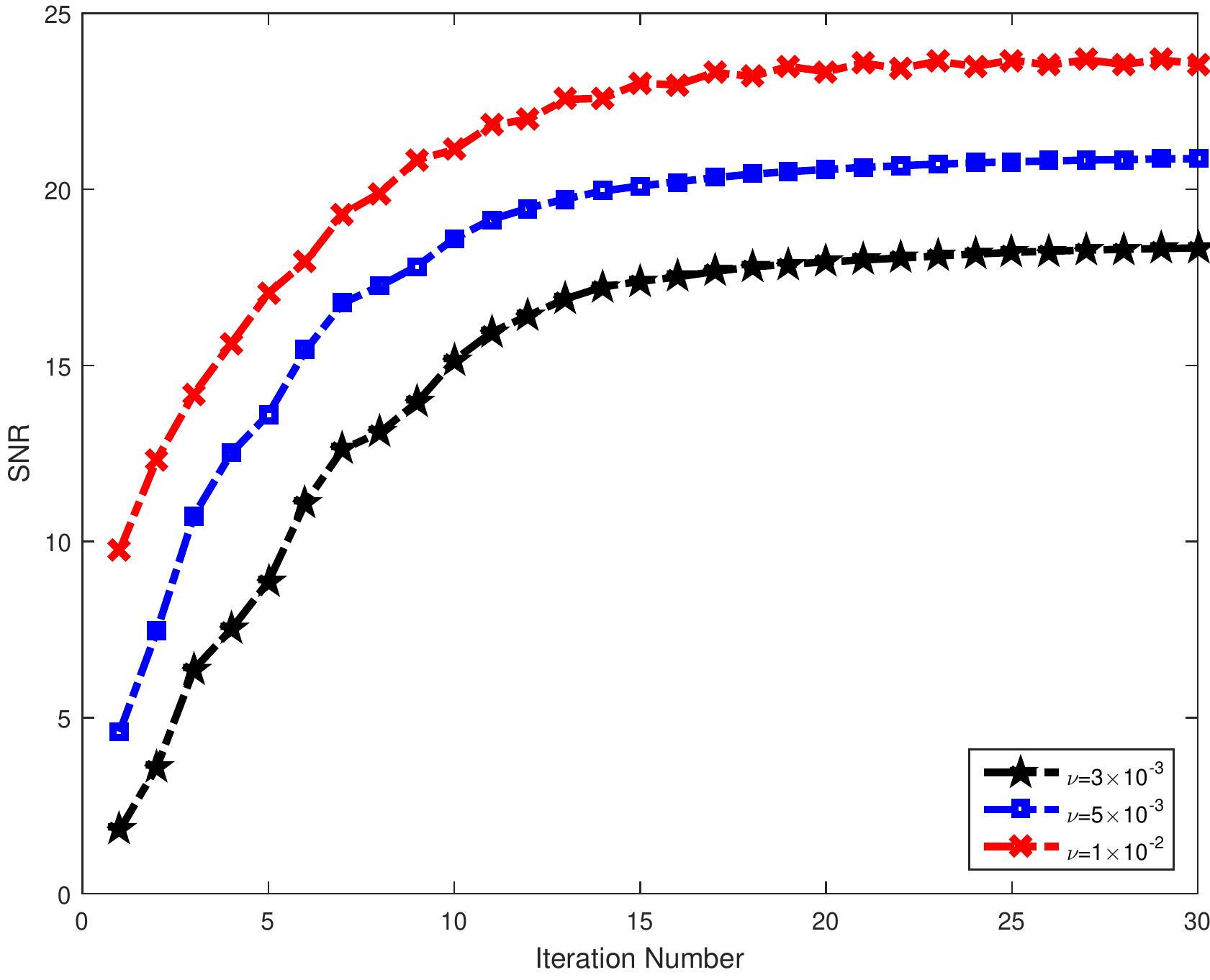}}
\subfigure[BM3D-PR]{\includegraphics[width=.12\textwidth]{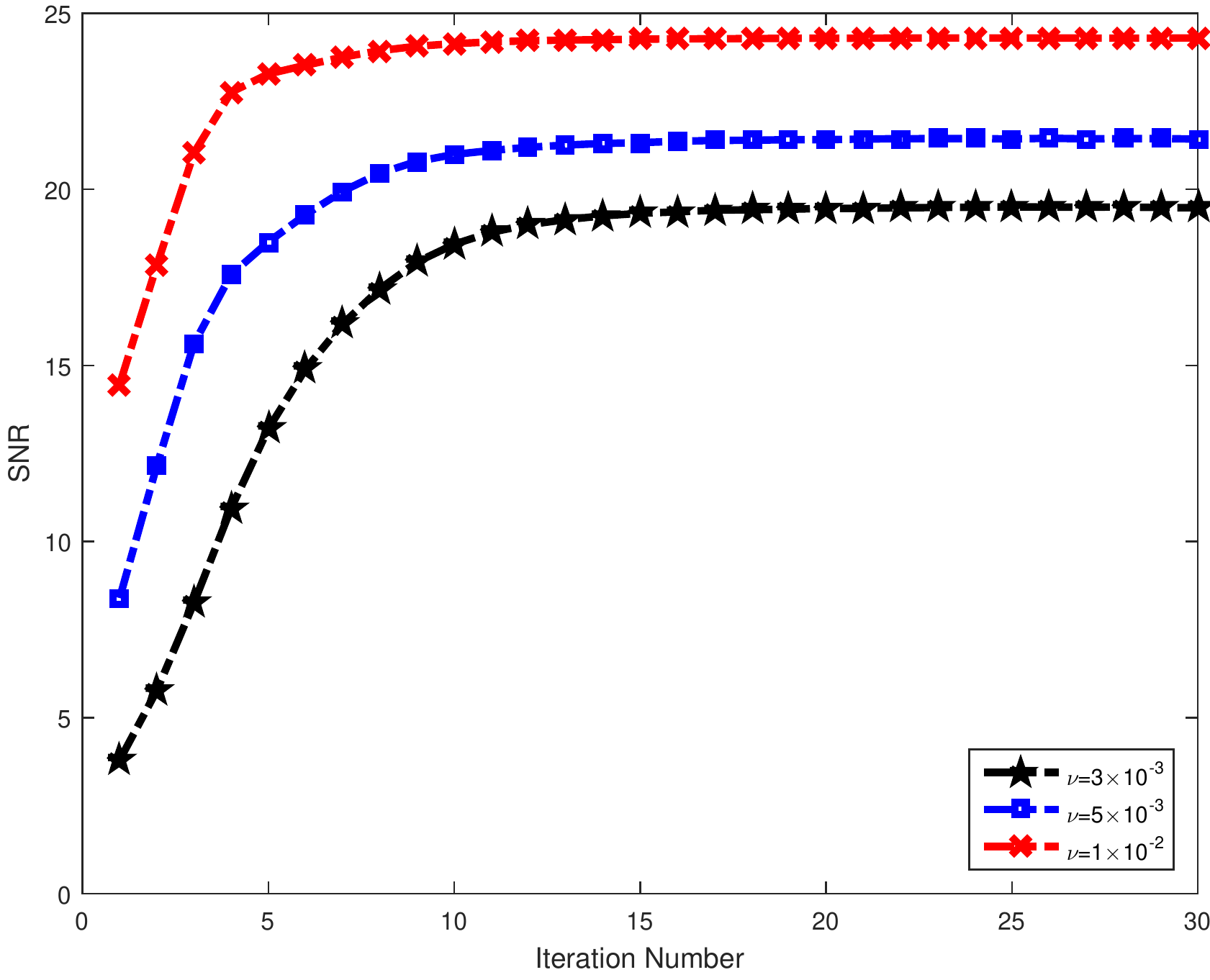}}
\end{center}
\caption{Convergence study by checking the histories of the SNRs and relative errors w.r.t. the iteration number. First row: The relative errors w.r.t. iteration number; Second row: The SNRs w.r.t. iteration number. From left to right: histories of relative errors and SNRs for ``TGV-PR'', ``NLM-PR'' and ``BM3D-PR'' respectively.}
\label{hist}
\end{figure}


We also perform an experiment to show the advantage of symmetric updating for our proposed methods compared with asymmetric style (remove Step 2 of \eqref{ADMMProximal}), and see the results in Figure \ref{hist1} for the history of the SNRs. One can readily observe the acceleration by introducing the symmetric iteration. A possible direction is to adopt the adaptive step to updating the multipliers in order to further speedup the algorithm, and we also put it as a future work.
\begin{figure}
\begin{center}
\subfigure[TGV-PR]{\includegraphics[width=.12\textwidth]{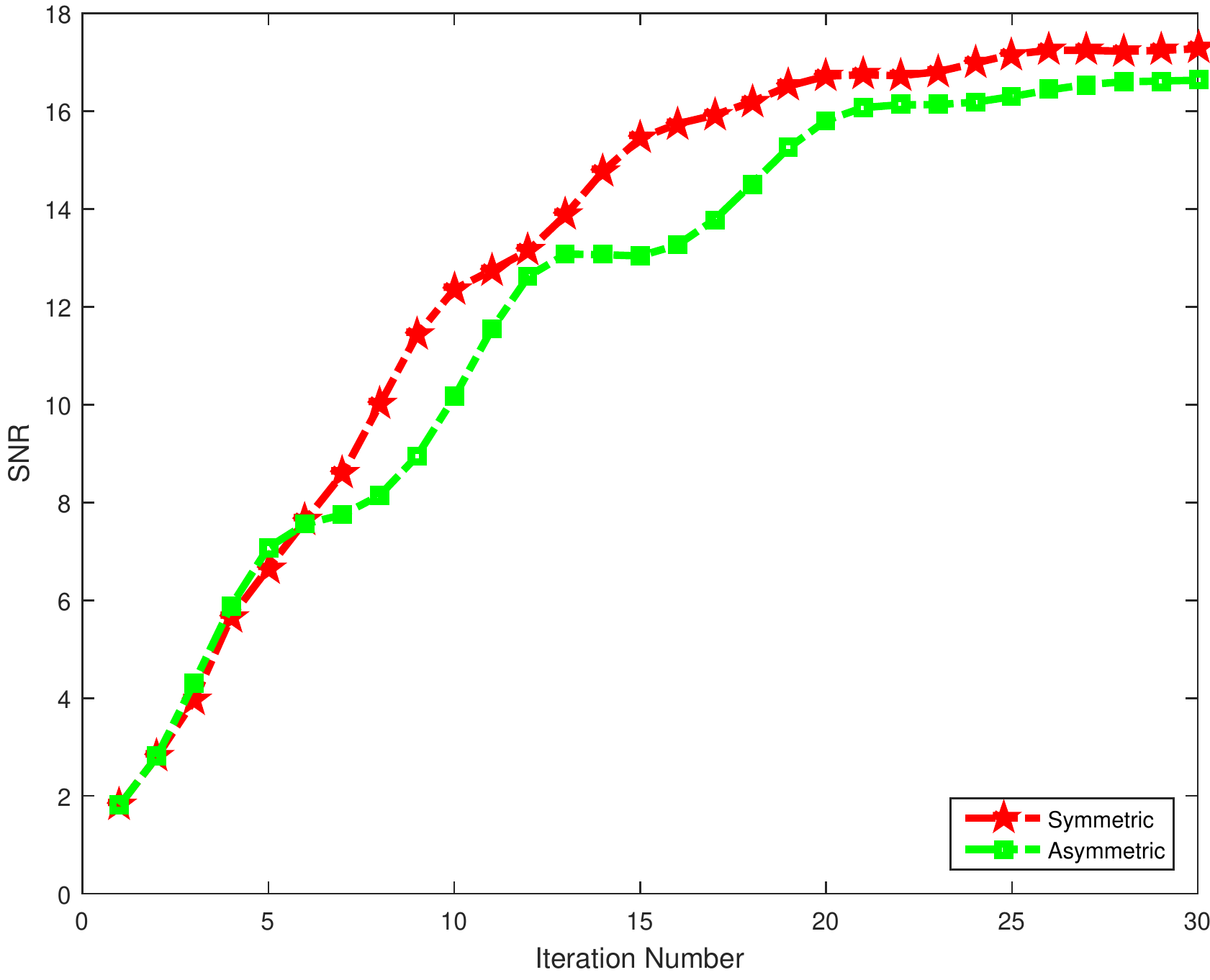}}
\subfigure[NLM-PR]{\includegraphics[width=.12\textwidth]{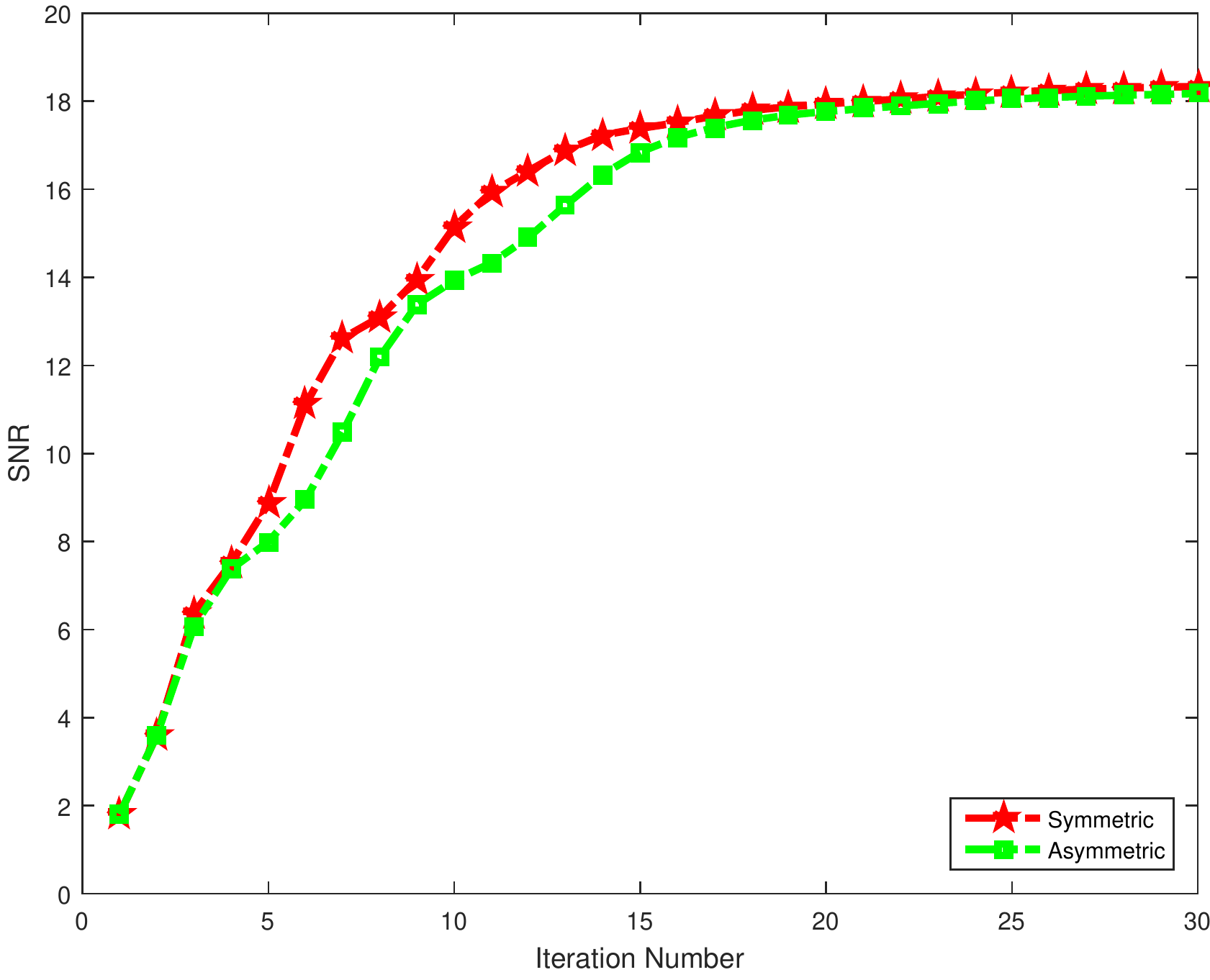}}
\subfigure[BM3D-PR]{\includegraphics[width=.12\textwidth]{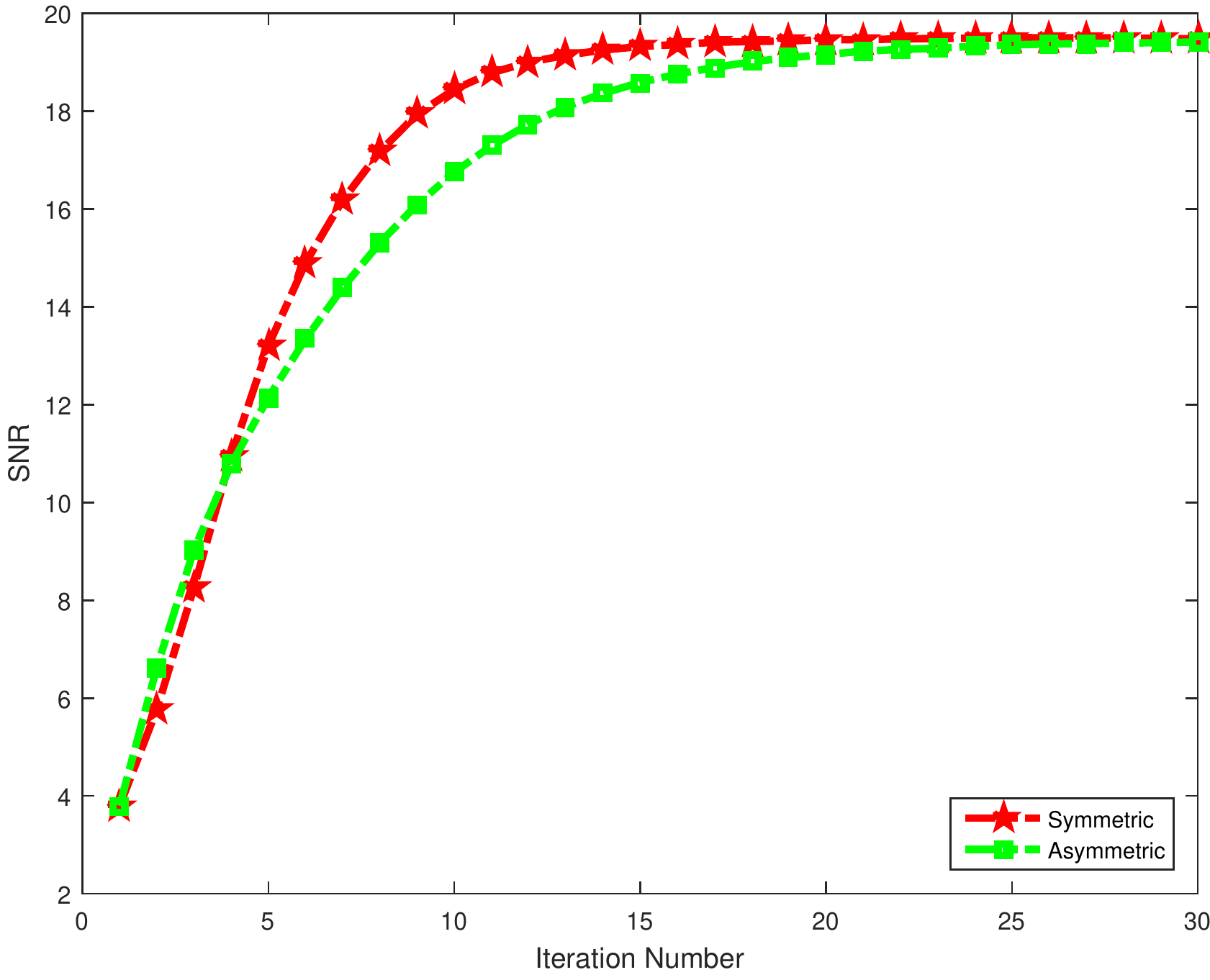}}
\end{center}
\caption{Performances comparison for  symmetric updating and asymmetric updating. From left to right: histories of SNRs for ``TGV-PR'', ``NLM-PR'' and ``BM3D-PR'' respectively.}
\label{hist1}
\end{figure}

\subsection{Stable performance for different parameters}
 Here the Poisson noise removal of CDP for real-valued image ``Lena'' is considered.   Set peak level as  $\nu=3.0\times10^{-3}$. $\lambda^0=7.0\times 10^2$, and $r^0=5.0\times 10^5$ as shown in Table \ref{tab1}. We test the performance of the proposed methods w.r.t the parameters $\lambda$ and $r$, where we vary one with  the other fixed. Particularly, we choose $\lambda$ from $\{\lambda^0\times2^{-l},\lambda^0\times2^{-l+1},\cdots,\lambda^0\times2^{l-1},\lambda^0\times2^{l}\}$ with $l=5$ and fix $r=r^0$, and see the resulted SNRs in Figure \ref{hist2} (a).  The parameter $\lambda$ play a role in balancing  the data fitting term and regularization term. If the parameter  $\lambda$ becomes very small, regularization effect get relative weak, and as a result, the recovery results contain more noise. If it is very large, the recovery results go far away from the measurements, and one also recover the images with lower quality.  Similarly, choose $r$ from $\{r^0\times2^{-l},r^0\times2^{-l+1},\cdots,r^0\times2^{l-1},r^0\times2^{l}\}$ with $l=5$ and fixed $\lambda=\lambda^0$ and plot the corresponding SNRs  of the reconstructed images in Figure \ref{hist2} (b). Since the proposed methods solve the nonconvex optimization problem just simply by  ADMM, it demonstrates we shall select a moderate value for parameter $r$.
\begin{figure}
\begin{center}
\subfigure{\includegraphics[width=.12\textwidth]{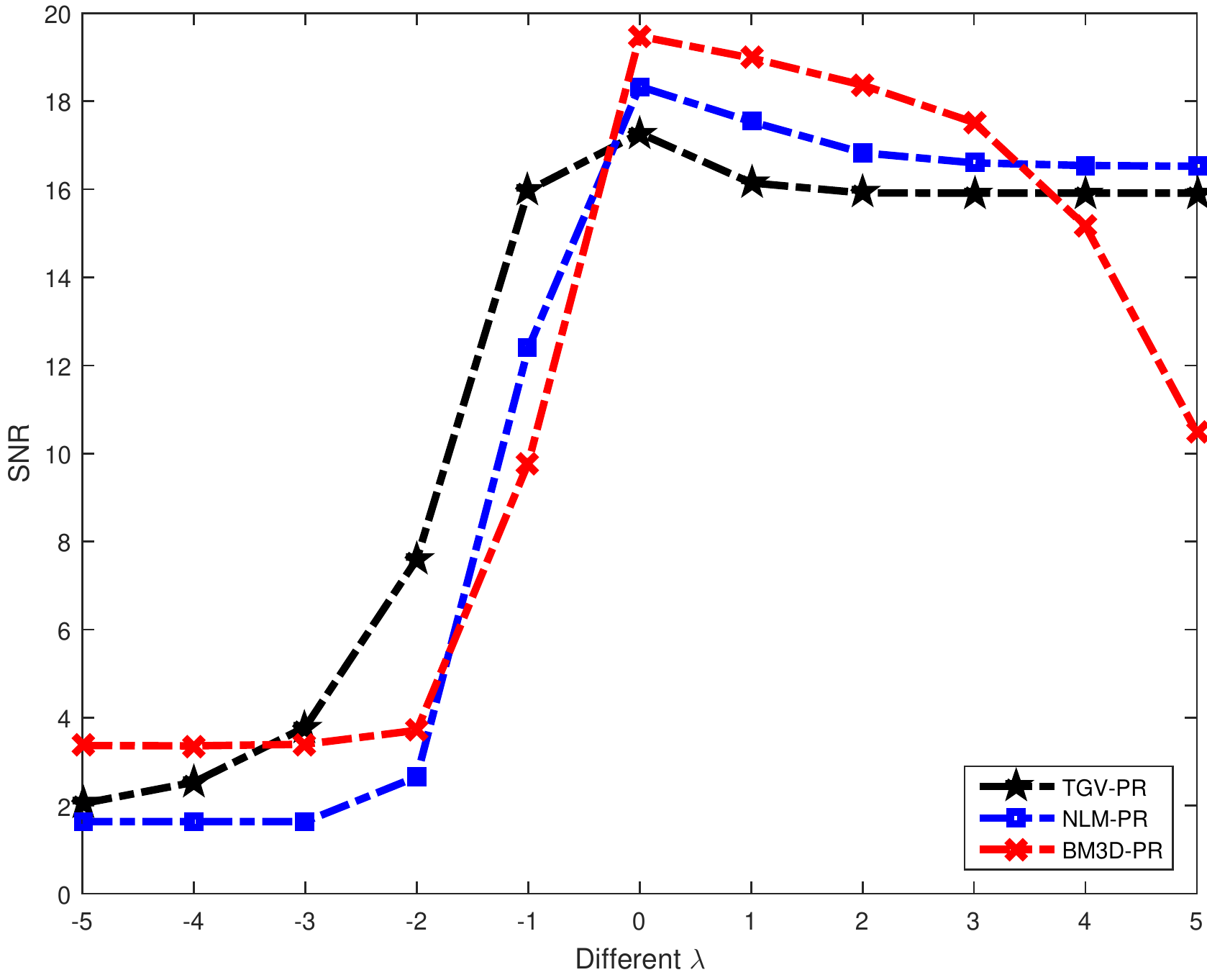}}\qquad
\subfigure{\includegraphics[width=.12\textwidth]{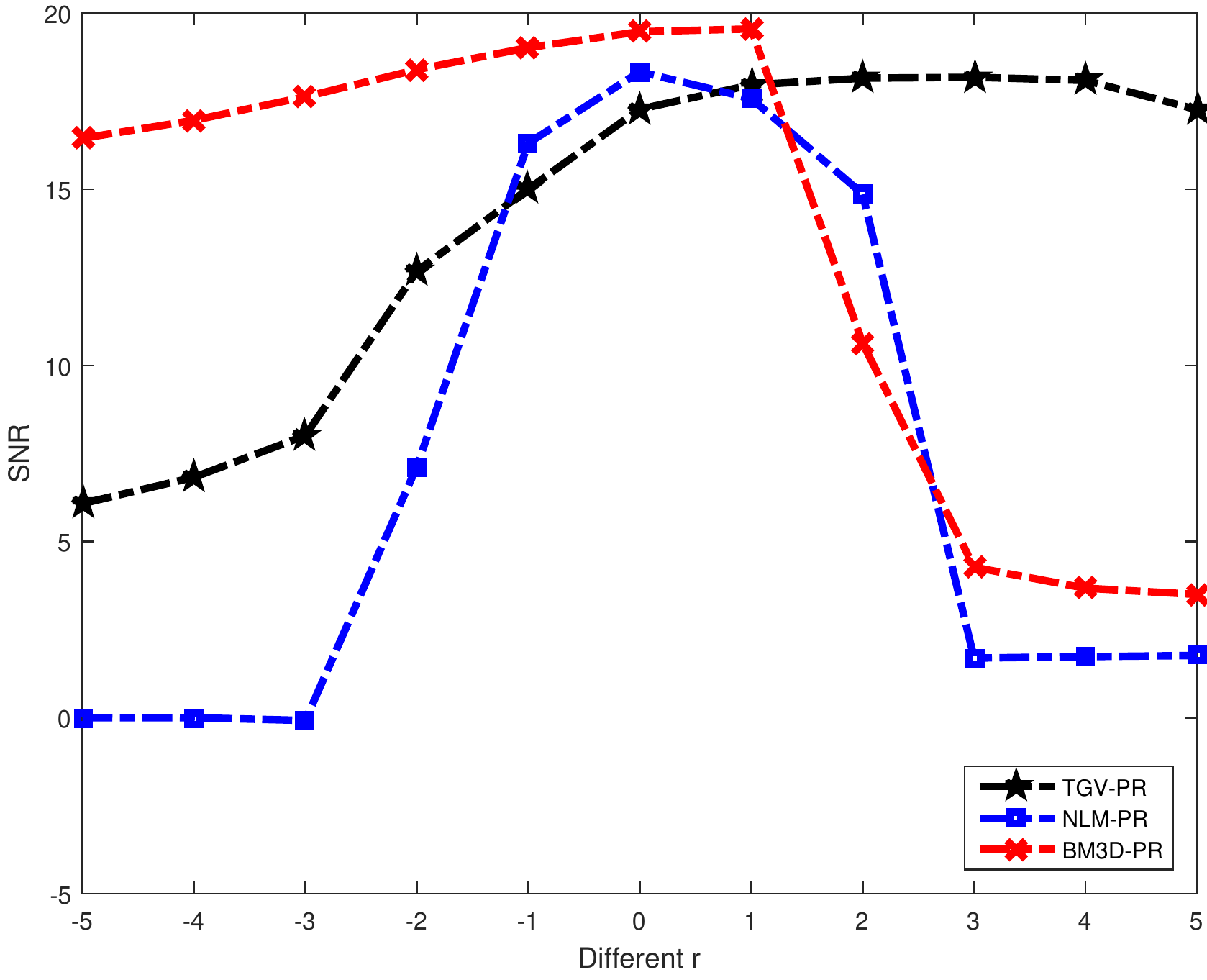}}
\end{center}
\caption{Performances w.r.t. parameters $\lambda$ and $r$.}
\label{hist2}
\end{figure}

\section{Conclusion}\label{sec:conclusions}
In this paper, we propose a simple and flexible framework to retrieve phase from noisy measurements contaminated by Poisson or Gaussian nose. By incorporating higher order TV as TGV and nonlocal sparsity based filters as NLM and BM3D, the image are recovered with sharp edges, clean background and repetitive features.  Numerical experiments show that even for the complicated complex-valued image with both large scale and repetitive smaller scale features, our proposed methods can work well compared with the traditional  PR algorithm without regularization, and simple ``TV-PR'' method. Our proposed methods rely on some important parameters, and in the future it is very interesting to explore an automatical approach for selecting optimal parameters in practise.

\section*{Acknowledgment}
Dr. H. Chang was partially supported by China Scholarship Council (CSC), National Natural Science Foundation of China (NSFC No. 11426165 and No. 11501413) and Tianjin 131 Talent Project.  This work was also partially
funded by the Center for Applied Mathematics
for Energy Research Applications, a joint ASCR-BES
funded project within the Office of Science, US Department
of Energy, under contract number DOE-DE-AC03-76SF00098.

\ifCLASSOPTIONcaptionsoff
  \newpage
\fi

\bibliographystyle{IEEEtran}
\bibliography{tip2014v1}

\end{document}